\documentclass{book}
\usepackage{amsmath,amssymb,amsthm,graphicx,epsfig,subfigure,float,url}
\usepackage[colorlinks=true]{hyperref}
\usepackage{enumitem}
\usepackage{pdfsync}

\usepackage{verbatim}
\usepackage[applemac]{inputenc}

\def\R{\mathrm{I\kern-0.21emR}}
\def\N{\mathrm{I\kern-0.21emN}}
\def\Z{\mathbb{Z}}
\newcommand{\C} {\mathbb{C}}

\newcommand{\Real}{\mathrm{Re}}
\newcommand{\OCP}{{\bf (OCP)}}

\renewcommand{\geq}{\geqslant}
\renewcommand{\leq}{\leqslant}

\newtheorem{theorem}{Theorem}  
\newtheorem{proposition}{Proposition}
\newtheorem{corollary}{Corollary}
\newtheorem{definition}{Definition}
\newtheorem{lemma}{Lemma}
\theoremstyle{definition}\newtheorem{example}{Example}
\theoremstyle{definition}\newtheorem{remark}{Remark}
\renewcommand{\Re}{\mathrm{Re}\,}

\author{Emmanuel Tr\'elat, lecture notes (master course)}

\title{\huge Control in finite and infinite dimension}

\begin{document}

\maketitle

\section*{Foreword}
This short book is the result of various master and summer school courses I have taught. The objective is to introduce the readers to mathematical control theory, both in finite and infinite dimension. 
In the finite-dimensional context, we consider controlled ordinary differential equations (ODEs); in this context, existence and uniqueness issues are easily resolved thanks to the Picard-Lindel\"of (Cauchy-Lipschitz) theorem. In infinite dimension, in view of dealing with controlled partial differential equations (PDEs), the concept of well-posed system is much more difficult and requires to develop a bunch of functional analysis tools, in particular semigroup theory -- and this, just for the setting in which the control system is written and makes sense. This is why I have splitted the book into two parts, the first being devoted to finite-dimensional control systems, and the second to infinite-dimensional ones. 

In spite of this splitting, it may be nice to learn basics of control theory for finite-dimensional linear autonomous control systems (e.g., the Kalman condition) and then to see in the second part how some results are extended to infinite dimension, where matrices are replaced by operators, and exponentials of matrices are replaced by semigroups. 
For instance, the reader will see how the Gramian controllability condition is expressed in infinite dimension, and leads to the celebrated Hilbert Uniqueness Method (HUM). 

Except the very last section, in the second part I have only considered linear autonomous control systems (the theory is already quite complicated), providing anyway several references to other textbooks for the several techniques existing to treat some particular classes of nonlinear PDEs. 
In contrast, in the first part on finite-dimensional control theory, there are much less difficulties to treat general nonlinear control systems, and I give here some general results on controllability, optimal control and stabilization. 

Of course, whether in finite or infinite dimension, there exist much finer results and methods in the literature, established however for specific classes of control systems. Here, my objective is to provide the reader with an introduction to control theory and to the main tools allowing to treat general control systems. 
I hope this will serve as motivation to go deeper into the theory or numerical aspects that are not covered here.

\vspace{\baselineskip}
\begin{flushright}\noindent
Paris, \hfill {\it $ $}\\
March 2023\hfill {\it Emmanuel Tr\'elat}\\
\end{flushright}

\tableofcontents

\part{Control in finite dimension}
%\noindent Use the template \emph{part.tex} together with the Springer document class SVMono (monograph-type books) or SVMult (edited books) to style your part title page and, if desired, a short introductory text (maximum one page) on its verso page in the Springer layout.

%\noindent Introduction and motivation.
%Presentation of control: controllability, optimal control, stabilization, observability.

\chapter{Controllability}\label{chap_cont}

Let $n$ and $m$ be two positive integers.
In this chapter we consider a control system in $\R^n$
\begin{equation}\label{general_control_system}
\dot{x}(t) = f(t,x(t),u(t))
\end{equation}
where $f:\R\times\R^n\times\R^m\rightarrow\R^n$ is of class $C^1$ with respect to $(x,u)$ and locally integrable with respect to $t$, and the controls are measurable essentially bounded functions of time taking their values in some measurable subset $\Omega$ of $\R^m$ (set of control constraints).

First of all, given an arbitrary initial point $x_0\in\R^n$, and an arbitrary control $u$, we claim that there exists a unique solution $x(\cdot)$ of \eqref{general_control_system} such that $x(0)=x_0$, maximally defined on some open interval of $\R$ containing $0$. 
We use here a generalization of the usual Picard-Lindel\"of theorem (sometimes called Carath\'eodory theorem), where the dynamics here can be discontinuous (because of the control). For a general version of this existence and uniqueness theorem, we refer to \cite[Theorem 5.3]{Hale} and \cite[Appendix C3]{Sontag}. 
We stress that the differential equation \eqref{general_control_system} holds for almost every $t$ in the maximal interval.
Given a time $T>0$ and an initial point $x_0$, we say that a control $u\in L^\infty([0,T],\R^m)$ is admissible if the corresponding trajectory $x(\cdot)$, such that $x(0)=x_0$, is well defined on $[0,T]$.

We say that the control system is \emph{linear} if $f(t,x,u) = A(t)x+B(t)u+r(t)$, with $A(t)$ a $n\times n$ matrix, $B(t)$ a $n\times m$ matrix (with real coefficients), $r(t)\in\R^n$, and in that case we will assume that $t\mapsto A(t)$, $t\mapsto B(t)$ and $t\mapsto r(t)$ are of class $L^\infty$ on every compact interval (actually, $L^1$ would be enough). The linear control system is said to be \emph{autonomous} if $A(t)\equiv A$ and $B(t)\equiv B$, otherwise it is said to be \emph{instationary} or \emph{time-varying}.
Note that, for linear control systems, there is no blow-up in finite time (i.e., admissibility holds true on any interval).

\begin{definition}\label{def_Ex0T}
Let $x_0\in\R^n$ and let $T>0$ arbitrary. A control $u\in L^\infty([0,T],\Omega)$ is said to be 
\textit{admissible} on $[0,T]$ if the trajectory $x_u(\cdot)$, solution of \eqref{general_control_system}, corresponding to the control $u$, and such that $x_u(0)=x_0$, is well defined on $[0,T]$.
The \textit{end-point mapping} $E_{x_0,T}$ is then defined by $E_{x_0,T}(u) = x_u(T)$.
\end{definition}

The set of admissible controls on $[0,T]$ is denoted by $\mathcal{U}_{x_0,T,\Omega}$. 
It is the domain of definition of $E_{x_0,T}$ (indeed one has to be careful with blow-up phenomena), when considering controls taking their values in $\Omega$.

\begin{definition}
The control system \eqref{general_control_system} is said to be (globally) \textit{controllable from $x_0$ in time $T$} if $E_{x_0,T}(\mathcal{U}_{x_0,T,\Omega})=\R^n$, i.e., if $E_{x_0,T}$ is surjective. 
\end{definition}

Accordingly, defining the \textit{accessible set} from $x_0$ in time $T$ by $\mathrm{Acc}_\Omega(x_0,T) = E_{x_0,T}(\mathcal{U}_{x_0,T,\Omega})$, the control system \eqref{general_control_system} is (globally) controllable from $x_0$ in time $T$ if $\mathrm{Acc}_\Omega(x_0,T) = \R^n$.

Since such a global surjectivity property is certainly a very strong property which may not hold in general, it is relevant to define \emph{local controllability}.

\begin{definition}\label{def_loccont}
Let $x_1=E_{x_0,T}(\bar u)$ for some $\bar u\in \mathcal{U}_{x_0,T,\Omega}$. The control system \eqref{general_control_system} is said to be \textit{locally controllable from $x_0$ in time $T$} around $x_1$ if $x_1$ belongs to the interior of $\mathrm{Acc}_\Omega(x_0,T)$, i.e., if $E_{x_0,T}$ is locally surjective around $x_1$.
\end{definition}

Other variants of controllability can be defined.
A clear picture will come from the geometric representation of the accessible set.

In this chapter we will provide several tools in order to analyze the controllability properties of a control system, first for linear (autonomous, and then instationary) systems, and then for nonlinear systems.

\section{Controllability of linear systems}\label{sec_cont_linear}
Throughout this section, we consider the linear control system $\dot x(t)=A(t)x(t)+B(t)u(t)+r(t)$, with $u\in L^\infty([0,+\infty),\Omega)$.
Since there is no finite-time blow-up for linear systems, we have $\mathcal{U}_{x_0,T,\Omega} = L^\infty([0,T],\Omega)$ for every $T>0$.

\subsection{Controllability of autonomous linear systems}\label{sec_cont_autonomous}
In this section, we assume that $A(t)\equiv A$ and $B(t)\equiv B$, where $A$ is a $n\times n$ matrix and $B$ is a $n\times m$ matrix.

\subsubsection{Without control constraints: Kalman condition}\label{sec_cont_kalman}
In this section, we assume that $\Omega=\R^m$ (no control constraint).
The celebrated Kalman theorem provides a necessary and sufficient condition for autonomous linear control systems without control constraint.

\begin{theorem}\label{thm_Kalman}
We assume that $\Omega=\R^m$ (no control constraint). The control system $\dot{x}(t)=Ax(t)+Bu(t)+r(t)$ is controllable (from any initial point, in arbitrary time $T>0$) if and only if the \textit{Kalman matrix}
$$K(A,B) = ( B , AB , \ldots , A^{n-1}B ) $$
(which is of size $n\times nm$) is of maximal rank $n$.
\end{theorem}

\begin{proof}[Proof of Theorem \ref{thm_Kalman}.]
Given any $x_0\in\R^n$, $T>0$ and $u\in L^\infty([0,T],\R^m)$, the Duhamel formula gives
\begin{equation}\label{duh}
E_{x_0,T}(u) = x_u(T) = e^{TA}x_0+\int_0^T e^{(T-t)A}r(t) \, dt + L_Tu
\end{equation}
where $L_T:L^\infty([0,T],\R^m)\rightarrow\R^n$ is the linear continuous operator defined by
$
L_Tu = \int_0^T e^{(T-t)A}Bu(t) \, dt.
$
Clearly, the system is controllable in time $T$ if and only if $L_T$ is surjective.
Then to prove the theorem it suffices to prove the following lemma.

\begin{lemma}\label{lemK}
The Kalman matrix $K(A,B)$ is of rank $n$ if and only if $L_T$ is surjective.
\end{lemma}

\begin{proof}[Proof of Lemma \ref{lemK}.]
We argue by contraposition.
If $L_T$ is not surjective, then there exists $\psi\in\R^n\setminus\{0\}$ which is orthogonal to the range of $L_T$, that is,
$$
\psi^\top \int_0^T e^{(T-t)A}Bu(t)\, dt=0\qquad\forall u\in L^\infty([0,T],\R^m).
$$
This implies that $\psi^\top e^{(T-t)A}B=0$, for every $t\in[0,T]$. Taking $t=T$ yields $\psi^\top B=0$. Then, derivating first with respect to $t$, and taking $t=T$ then yields $\psi^\top AB=0$. By immediate iteration we get that $\psi^\top A^kB=0$, for every $k\in\N$. In particular $\psi^\top K(A,B)=0$ and thus the rank of $K(A,B)$ is less than $n$.

Conversely, if the rank of $K(A,B)$ is less than $n$, then there exists $\psi\in\R^n\setminus\{0\}$ such that $\psi^\top K(A,B)=0$, and therefore $\psi^\top A^kB=0$, for every $k\in\{0,1,\ldots,n-1\}$. From the Hamilton-Cayley theorem, there exist real numbers $a_0,a_1,\ldots,a_{n-1}$ such that $A^n=\sum_{k=0}^{n-1} a_k A^k$. Therefore we get easily that $\psi^\top A^nB=0$. Then, using the fact that $A^{n+1}=\sum_{k=1}^{n} a_k A^k$, we get as well that $\psi^\top A^{n+1}B=0$. By immediate recurrence, we infer that $\psi^\top A^kB=0$, for every $k\in\N$, and therefore, using the series expansion of the exponential, we get that $\psi^\top e^{(T-t)A}B=0$, for every $t\in[0,T]$. We conclude that $\psi^\top L_T u=0$ for every $u\in L^\infty([0,T],\R^m)$ and thus that $L_T$ is not surjective.
\end{proof}
Theorem \ref{thm_Kalman} is proved.
\end{proof}

\begin{remark}
Note that the Kalman condition is purely algebraic and is easily checkable.

The Kalman condition does neither depend on $T$ nor on $x_0$. Therefore, if an autonomous linear control system is controllable from $x_0$ in time $T>0$, starting at $x_0$, then it is controllable from any other $x_0'$ in any time $T'>0$, in particular, in arbitrarily small time. This is due to the fact that there are no control constraints. When there are some control constraints one cannot hope to have such a property.
\end{remark}

\begin{example}
Consider an RLC circuit with a resistor with resistance $R$, a coil with inductance $L$ and a capacitor with capacitance $C$, connected in series. We control the input voltage $u(t)$ of the electrical circuit. Denoting by $i(t)$ the intensity, by additivity of voltages we have
$$
Ri(t) + L\frac{di}{dt}(t) + \frac{1}{C}\int^t i(s)ds = u(t).
$$
Setting $x_1(t)=\int^t i(s)ds$, $x_2(t)=\dot x_1(t)=i(t)$ and $x(t)=\begin{pmatrix}x_1(t)\\ x_2(t)\end{pmatrix}$, we find the control system $\dot x(t)=Ax(t)+Bu(t)$ with 
$$
A=\begin{pmatrix}0&1\\ \frac{-1}{LC} & \frac{-R}{L}\end{pmatrix}, \qquad
B=\begin{pmatrix}0\\1\end{pmatrix} .
$$
The Kalman condition is then easily checked. 

This simple but illuminating example shows the importance of the RLC device in electricity. Note that the RLC circuit is paradigmatic for the three main operations in mathematical analysis: mutliplication operator, derivation operator and integration operator.
\end{example}

\begin{example}[Kalman condition computations]\ 
\begin{itemize}[parsep=1mm,itemsep=0mm,topsep=1mm,leftmargin=*]
\item Let $m>0$ and $d,k\geq 0$. Prove that the system consisting of the controlled spring: $m\ddot x+d\dot x+kx=u$, % (for any $m,d,k>0$).
is controllable.

\item Let $k_1,k_2\geq 0$. Prove that the system of coupled springs (two-car train) given by
$$
\ddot x_1=-k_1x_1+k_2(x_2-x_1),\qquad \ddot x_2=-k_2(x_2-x_1)+u,
$$
is controllable if and only if $k_2>0$.

\item Let $(b_1,b_2)\in\R^2\setminus\{(0,0)\}$. Prove that the control system
$$
\dot x_1 = x_2+b_1 u, \qquad
\dot x_2 = -x_1+b_2 u, 
$$
is controllable.

\item Prove that the control system
$$
\dot x_1=2x_1+(\alpha-3)x_2+u_1+u_2, \qquad \dot x_2=2x_2+\alpha(\alpha-1)u_1,
$$
is controllable if and only if $\alpha(\alpha-1)\neq 0$.

\item Let $N,m\in\N\setminus\{0\}$, let $A=(a_{ij})_{1\leq i,j\leq N}$ be a $N\times N$ real-valued matrix and $B=(b_{ij})_{1\leq i\leq N,\ 1\leq j\leq m}$ be a $N\times m$ real-valued matrix, such that the pair $(A,B)$ satisfies the Kalman condition.
Let $d\in\N\setminus\{0\}$. Prove that the control system in $(\R^d)^N$ given by
\begin{equation*}
\dot v_i(t) = \sum_{j=1}^N a_{ij} v_j(t) + \sum_{j=1}^m b_{ij} u_j(t), \qquad i=1\ldots N,
\end{equation*}
where $v_i(t)\in\R^d$ and $u_j(t)\in\R^d$, is controllable. 
%\begin{quote}
%\textit{Corrig\'e:}
%On utilise la notation du produit de Kronecker de matrices: le syst\`eme s'\'ecrit sous la forme $\dot v(t)=(A\otimes I_d)v(t)+(B\otimes I_d)u(t)$, o\`u par d\'efinition les matrices $A\otimes I_d$ et $B\otimes I_d$ sont constitu\'ees des blocs $d\times d$ respectifs $a_{ij} I_d$ et $b_{ij} I_d$. On voit facilement que, pour le produit habituel de matrices: $(A\otimes I_d)^k(B\otimes I_d)=A^kB\otimes I_d$, et la condition de Kalman s'ensuit pour le syst\`eme en dimension $dN$.
%\end{quote}
%
\end{itemize}
\end{example}

Many other examples are given in \cite{Trelat}.

\begin{remark}[Hautus test]
The following assertions are equivalent:
\begin{enumerate}
\item[(1)] The pair $(A,B)$ satisfies Kalman's condition $\mathrm{rank}(K(A,B))=n$.
\item[(2)] $\forall\lambda\in\C\quad\mathrm{rank}(\lambda I-A,B)=n$.
\item[(3)] $\forall\lambda\in\mathrm{Spec}(A)\quad\mathrm{rank}(\lambda I-A,B)=n$.
\item[(4)] $\forall z$ eigenvector of $A^\top,\quad B^\top z\neq 0$.
\item[(5)] $\exists c>0\ \mid\ \forall\lambda\in\C\quad\forall z\in\R^n\quad\Vert(\lambda I-A^\top)z\Vert^2+\Vert B^\top z\Vert^2\geq c\Vert z\Vert^2$.
\end{enumerate}
Indeed, $(2)\Leftrightarrow (3)$, $(2)\Leftrightarrow (5)$, and not (4) $\Rightarrow$ not (1), are easy. We also easily get $(3)\Leftrightarrow (4)$ by contradiction. The implication not (1) $\Rightarrow$ not (4) is proved as follows. We set $N=\{z\in\R^n\ \mid\ z^\top A^kB=0\quad\forall k\in\N\}$. It is easy to establish that $A^\top N\subset N$. Then, not (1) $\Rightarrow N\neq\{0\}$, and then to conclude it suffices to note that $A^\top$ must have an eigenvector in $N$.

The condition (5) of the Hautus test is particularly interesting because it is amenable to generalizations in the infinite-dimensional setting.
\end{remark}

\begin{remark}
Let us finally comment on how to generalize the Kalman condition in infinite dimension. 
This will be done properly in Lemma \ref{Kalman_infinitedimension} in Section \ref{sec_Kalman_infinitedimension}, but we can already anticipate and provide the reader with a flavor of the new difficulties arising in infinite dimension. 
Let us replace $\R^n$ with a Banach space $X$ and $\R^m$ with a Banach space $U$. The matrix $A$ becomes an operator $A:D(A)\rightarrow X$ (for instance, a Laplacian) and $B\in L(U,X)$ is a linear bounded operator (for instance, modelling an internal control for the heat equation). The Duhamel formula \eqref{duh} remains valid provided $e^{tA}$ is replaced with the semigroup generated by $A$ (we stress that this is not an exponential if $A$ is an unbounded operator). At this stage, the reader is advised to admit this notion and continue reading. The essential elements of semigroup theory are recalled in Chapter \ref{chap_semigroup}.

The argument of the proof of Theorem \ref{thm_Kalman} remains essentially the same, except of course the application of the Hamilton-Cayley theorem, and the Kalman matrix is replaced with a matrix of operators with an infinite number of columns ($k$ does not stop at $n-1$ but goes to $+\infty$), thus giving the statement of Lemma \ref{Kalman_infinitedimension}.
There is however a serious and deep difference: at the beginning of the proof of Lemma \ref{lemK} we have used the fact that, when the vector space $\mathrm{Ran}(L_T)$ is a proper subset of $\R^n$, there exists a nontrivial vector $\psi$ vanishing it: this is a separation argument. In infinite dimension this argument may dramatically fail, because a proper subset of $X$ might be dense in $X$: in such a case, separation is not possible.
Then, to make the argument valid, one has to consider the closure $\overline{\mathrm{Ran}(L_T)}$ of $\mathrm{Ran}(L_T)$, as done in Lemma \ref{Kalman_infinitedimension}, leading to a result of \textit{approximate} controllability.

Actually, as we will see in Part \ref{part2} of that book, in infinite dimension we must distinguish (at least) between \textit{approximate} and \textit{exact} controllability. This distinction does not exist in finite dimension. At this stage, let us just say that, for linear autonomous control systems, exact controllability means $\mathrm{Ran}(E_{T,x_0})=X$ while approximate controllability means $\overline{\mathrm{Ran}(E_{T,x_0})}=X$. For example, a heat equation settled on a domain with internal control exerted on a proper subset of this domain is approximately controllable in $X=L^2$ (see Part \ref{part2}) but can never be exactly controllable in this state space because of the smoothing property of the heat semigroup.
\end{remark}

\subsubsection{With control constraints}\label{sec_cont_constraints}
An easy adaptation of the proof of Theorem \ref{thm_Kalman} yields the following result.

\begin{proposition}\label{prop_Kalman_constraints}
We assume that $r=0$, that $0\in\overset{\circ}{\Omega}$, and that the Kalman condition holds true. 
Let $x_0\in\R^n$ be arbitrary.
For every $T>0$, the accessible set $Acc_\Omega(x_0,T)$ contains an open neighborhood of the point $e^{TA}x_0$. 
In other words, the control system is locally controllable in time $T$ from $x_0$ around $e^{TA}x_0$.

More precisely, for every $T>0$ there exists a neighborhood $V$ of $e^{TA}x_0$ such that, for every $x_1\in V$, there exists $u\in L^\infty([0,T],\Omega)$ (which is close to $0$ in $L^\infty$ topology) such that $E_{x_0,T}(u)=x_1$.
Conversely, this openness property implies the Kalman condition.
\end{proposition}

Many variants of controllability properties can be obtained under various additional assumptions. For instance we have the following easy result.

\begin{proposition}
We assume that $r=0$, that $0\in\overset{\circ}{\Omega}$, that the Kalman condition holds true, and that all eigenvalues of $A$ have negative real part. 
Let $x_0\in\R^n$ be arbitrary.
There exists a time $T>0$ and a control $u\in L^\infty([0,T],\Omega)$ such that the solution of $\dot x(t)=Ax(t)+Bu(t)$, $x(0)=x_0$, satisfies $x(T)=0$.
\end{proposition}

The time $T$ in the above result may be large. The strategy of proof consists of taking $u=0$ and letting the trajectory converge asymptotically to $0$, and then as soon as it is sufficiently close to $0$, we apply the controllability result with controls having a small enough norm.

\subsubsection{Similar systems}\label{sec_similar}
Let us investigate the effect of a change of basis in linear autonomous control systems.

\begin{definition}
The linear control systems $\dot{x}_1=A_1x_1+B_1u_1$ and $\dot{x}_2=A_2x_2+B_2u_2$ are said to be \textit{similar} whenever there exists $P\in GL_n(\R)$ such that $A_2=PA_1P^{-1}$
and $B_2=PB_1$. We have then $x_2=Px_1$ and $u_2=u_1$.

We also say that the pairs $(A_1,B_1)$ and $(A_2,B_2)$ are similar.
\end{definition}

\begin{remark}
The Kalman property is intrinsic, that is
$$(B_2,A_2B_2,\ldots,A_2^{n-1}B_2) =
P(B_1,A_1B_1,\ldots,A_1^{n-1}B_1),$$
In particular the rank of the Kalman matrix is invariant under a similar transform.
\end{remark}

\begin{proposition}\label{prop_controllernormalform}
Let $A$ be a matrix of size $n\times n$, and let $B$ be a matrix of size $n\times m$.
Then the pair $(A,B)$ is similar to the pair $(A',B')$, with
$$
A'=\begin{pmatrix} A_1' & A_3' \\ 0 & A_2' \end{pmatrix} \quad\textrm{and}\quad
B'=\begin{pmatrix} B_1' \\ 0 \end{pmatrix}
$$
where $A_1$ is of size $r\times r$, $B_1$ is of size $r\times m$, and $r=\mathrm{rank}\, K(A,B)=\mathrm{rank}\, K(A_1',B_1')$.
\end{proposition}

In other words, this result says the following. Denoting by $y=\begin{pmatrix}y_1\\ y_2\end{pmatrix}$ the new coordinates, with $y_1$ of dimension $r$ and $y_2$ of dimension $n-r$, the control system in the new coordinates is written as
\begin{equation*}
\begin{split}
\dot y_1 &= A_1' y_1 + B_1'u + A_3'y_2 \\
\dot y_2 &= A_2'y_2 
\end{split}
\end{equation*}
Since the pair $(A_1',B_1')$ satisfies the Kalman condition, it follows that the part of the system in $y_1$ is controllable: it is called the \textit{controllable part} of the system. The part in $y_2$ is uncontrolled and is called the \textit{uncontrollable part} of the system.

\begin{proof}[Proof of Proposition \ref{prop_controllernormalform}.]
We assume that the rank of $K(A,B)$ is less than $n$ (otherwise there is nothing to prove). The subspace 
$$
F=\mathrm{Ran}\, K(A,B)=\mathrm{Ran}\, B+\mathrm{Ran}\, AB+\cdots+\mathrm{Ran}\, A^{n-1}B 
$$
is of dimension $r$, and is invariant under $A$ (this can be seen by using the Hamilton-Cayley theorem). Let $G$ be a subspace of $\R^n$ such that $\R^n=F\oplus G$ and let $(f_1,\ldots,f_r)$ be a basis of $F$ and $(f_{r+1},\ldots,f_n)$ be a basis of $G$. Let $P$ be the change-of-basis matrix from the basis $(f_1,\ldots,f_n)$ to the canonical basis of $\R^n$. Since $F$ is invariant under $A$, we get
$$
A'=PAP^{-1}=\begin{pmatrix} A_1' & A_3' \\ 0 & A_2'\end{pmatrix}
$$
and since $\mathrm{Ran}\, B\subset F$, we must have
$B'=PB=\begin{pmatrix} B_1' \\ 0 \end{pmatrix}.$
Finally, it is clear that the rank of $K(A_1',B_1')$ is equal to the rank of $K(A,B)$.
\end{proof}

\begin{theorem}[Brunovski normal form]\label{thmbrunovski}
Let $A$ be a matrix of size $n\times n$ and let $B$ be a matrix of size $n\times 1$ (note that $m=1$ here) be such that $(A,B)$ satisfies the Kalman condition. Then the pair $(A,B)$ is similar to the pair $(\tilde{A},\tilde{B})$, with
$$
\tilde{A}=\begin{pmatrix}
0  & 1 & \cdots & 0 \\
\vdots & \ddots & \ddots & \vdots \\
0 & \cdots & 0 & 1 \\
-a_n & -a_{n-1} & \cdots & -a_1
\end{pmatrix} \quad\textrm{and}\quad
\tilde{B} = \begin{pmatrix}
0 \\ \vdots \\ 0 \\ 1 \end{pmatrix} 
$$
where the coefficients $a_i$ are those of the characteristic polynomial of $A$, that is $\chi_A(X) = X^n + a_1 X^{n-1} +\cdots + a_{n-1}X + a_n$.
\end{theorem}

Note that the matrix $\tilde A$ is the companion matrix of characteristic polynomial $\chi_A$.
Theorem \ref{thmbrunovski} means that, in the new coordinates, the control system is equivalent to the scalar differential equation of order $n$ with scalar control
$x^{(n)}(t)+a_1x^{(n-1)}(t)+\cdots+a_nx(t)=u(t) .$

\begin{proof}[Proof of Theorem \ref{thmbrunovski}.]
First, let us note that, if there exists a basis $(f_1,\ldots,f_n)$ in which the pair $(A,B)$ takes the form $(\tilde{A},\tilde{B})$, then we must have $f_n=B$ (up to scaling) and
$$Af_n=f_{n-1}-a_1f_n,\ \ldots,\ Af_2=f_1-a_{n-1}f_n,\ Af_1=-a_nf_n.$$
Let us then define the vectors $f_1,\ldots,f_n$ by
$$f_n=B,\ f_{n-1}=Af_n+a_1f_n,\ \ldots,\ f_1=Af_2+a_{n-1}f_n.$$
The $n$-tuple $(f_1,\ldots,f_n)$ is a basis of $\R^n$, since
\begin{equation*}
\begin{split}
\mathrm{Span}\,\{ f_n \} &= \mathrm{Span}\,\{ B \}  \\
\mathrm{Span}\,\{ f_n,f_{n-1} \} &= \mathrm{Span}\,\{ B,AB \}  \\
&\vdots \\
\mathrm{Span}\,\{ f_n,\ldots,f_1 \} &= \mathrm{Span}\,\{ B,\ldots,A^{n-1}B \}=\R^n . 
\end{split}
\end{equation*}
It remains to check that $Af_1=-a_nf_n$. We have
\begin{multline*}
Af_1 = A^2f_2+a_{n-1}Af_n 
= A^2(Af_3+a_{n-2}f_n)+a_{n-1}Af_n \\
%= A^3f_3 + a_{n-2}A^2f_n + a_{n-1}Af_n 
= \ldots 
= A^nf_n + a_1A^{n-1}f_n + \cdots + a_{n-1}Af_n 
= -a_nf_n
\end{multline*}
since by the Hamilton-Cayley theorem we have $A^n=-a_1A^{n-1}-\cdots-a_nI$. In the basis $(f_1,\ldots,f_n)$, the pair $(A,B)$ takes the form $(\tilde{A},\tilde{B})$.
\end{proof}

\begin{remark}
This theorem can be generalized to the case $m>1$ but the normal form is not that simple. 
%More precisely, if the pair $(A,B)$ satisfies the Kalman condition, then it is similar to some pair $(\tilde{A},\tilde{B})$ such that
%$$
%\tilde{A}=\begin{pmatrix}
%\tilde{A}_1 & * & \cdots & * \\
%0 & \tilde{A}_2 & \ddots & \vdots \\
%\vdots & \ddots & \ddots & * \\
%0 & \cdots & 0 & \tilde{A}_s
%\end{pmatrix}
%\quad\textrm{and}\quad
%\tilde{B}G=\begin{pmatrix} \tilde{B}_1 \\ \vdots \\ \tilde{B}_s \end{pmatrix}
%$$
%where the matrices $\tilde{A}_i$ are companion matrices, $G$ is a matrix of size $m\times s$, and for every $i\in\{1,\ldots,s\}$, all coefficients of $\tilde{B}_i$ are equal to zero except the one in the last row, column $i$, which is equal to $1$.
\end{remark}

\subsection{Controllability of time-varying linear systems}\label{sec_controllability_time-varying}
%In what follows, we denote by $M(\cdot)$ the \textit{resolvent} of the linear system $\dot x(t)=A(t)x(t)$, that is, the unique solution of the Cauchy problem $\dot M(t)=A(t)M(t)$, $M(0)=I_n$. Note that, in the autonomous case $A(t)\equiv A$, we have $M(t)=e^{tA}$. But in general the resolvent cannot be computed explicitly.
In what follows, we denote by $R(t,s)$ the \textit{state-transition matrix} of the linear system $\dot x(t)=A(t)x(t)$, that is, the unique solution of 
\begin{equation}\label{defR}
\partial_t R(t,s)=A(t)R(t,s), \qquad R(s,s)=I_n, 
\end{equation}
for $t,s\in\R$. Note that, in the autonomous case $A(t)\equiv A$, we have $R(t,s)=e^{(t-s)A}$. But in general the state-transition matrix cannot be computed explicitly.
Recall that $R(t,s)R(s,\tau)=R(t,\tau)$ for all $t,s,\tau\in\R$; in particular, $R(t,s)=R(s,t)^{-1}$.

\subsubsection{Case without control constraints}

\begin{theorem}\label{controlabiliteinstationnaire}
We assume that $\Omega=\R^m$ (no constraints on the control).
The control system $\dot{x}(t)=A(t)x(t)+B(t)u(t)+r(t)$ is controllable in time $T$ (from any initial point $x_0$) if and only if the \textit{Gramian matrix}
$$
%G_T = M(T) \int_0^T M(t)^{-1}B(t) {B(t)^\top}{M(t)^{-1}}^\top dt \ M(T)^\top
G_T = \int_0^T R(T,t)B(t) {B(t)^\top} R(T,t)^\top dt 
$$
is invertible.
\end{theorem}

\begin{remark}\label{remGT'_1}
The invertibility condition of $G_T$ depends on $T$ but not on the initial point. Therefore, if a linear instationary control system is controllable from $x_0$ in time $T$ then it is controllable from any other initial point (but with the same time $T$). It may fail to be controllable in time $T'<T$ (take for instance $B(t)=0$ for $0\leq t\leq T'$). 
%Actually, it follows from Remark \ref{remGT'_2} further that it is controllable in time $T'>T$.
Anyway, controllability in time $T$ implies controllability in any time $T'\geq T$: indeed, at time $T$ the range of the end-point mapping is equal to the whole $\R^n$ and it cannot decrease in larger time.
%Another argument is given in Remark \ref{remGT'_2} further. 
\end{remark}

%\begin{remark}
%In the definition of the Gramian, the fact to have multiplied the integral to the left by $M(T)$ and to the right by $M(T)^\top$ has absolutely no importance in finite dimension, but it is more convenient to prepare the generalization to autonomous infinite dimensional systems (see Remark \ref{rem_auton}).
%\end{remark}

\begin{remark}\label{remGT}
Note that $G_T=G_T^\top$ and that
$$
\psi^\top G_T\psi = \langle G_T\psi,\psi\rangle = \int_0^T \Vert B(t)^\top R(T,t)^\top \psi\Vert^2\, dt \geq 0\qquad\forall\psi\in\R^n ,
$$
i.e., $G_T$ is a symmetric nonnegative matrix of size $n\times n$. Theorem \ref{controlabiliteinstationnaire} states that the system is controllable in time $T$ if and only if $G_T$ is positive definite.
By a diagonalization argument, this is equivalent to the existence of $C_T>0$ (the lowest eigenvalue of $G_T$) such that
\begin{equation}\label{inegobsdimfinie}
\int_0^T \Vert B(t)^\top R(T,t)^\top \psi\Vert^2\, dt \geq C_T \Vert \psi\Vert^2\qquad\forall\psi\in\R^n .
\end{equation}
This is an \textit{observability inequality}.
%
%Although we have not defined the concept of observability, we underline that 
The system is controllable in time $T$ if and only if the above observability inequality holds. %: this is the well known \textit{controllability -- observability duality}. 

The important concept of observability is not developed in this book. We recall it briefly for a linear control system $\dot x(t)=A(t)x(t)+B(t)u(t)+r(t)$ and we refer to \cite{LeeMarkus, Sontag, Trelat} for more details.
Assume that, at any time $t$, we cannot observe the whole state $x(t)$ but only a part of it, $y(t)=C(t)x(t)$, where $C(t)$ is a $m\times n$ matrix. The control system, with its output $y(\cdot)$, is said to be observable in time $T$ if a given output $y(\cdot)$ observed on $[0,T]$ can be generated by only one initial condition $x_0$. While controllability corresponds to a surjectivity property, observability thus corresponds to an injectivity property. 
Actually, the system $\dot x(t)=A(t)x(t)+B(t)u(t)+r(t)$ with output $y(t)=C(t)x(t)$ is observable in time $T$ if and only if the control system $\dot x(t)=A(t)^\top x(t) + C(t)u(t)$ is controllable in time $T$: this is the so-called controllability -- observability duality (see Lemma \ref{lemgenanafonc} in Part \ref{part2}, Section \ref{sec_duality}, for a general mathematical statement establishing this duality in a general context). Hence, observability in time $T$ is characterized by the observability inequality \eqref{inegobsdimfinie} with $B(t)^\top$ replaced by $C(t)$.
\end{remark}

%\begin{remark}[continued from Remark \ref{remGT'_1}]\label{remGT'_2}
%Using that $R(T,t)=R(T,0)R(0,t)$, we have $G_T = R(T,0) \tilde G_T R(T,0)^\top$ where $\tilde G_T = \int_0^T R(0,t)B(t) {B(t)^\top} R(0,t)^\top dt$. Of course, $G_T$ is invertible if and only if $\tilde G_T$ is invertible. Sometimes, this is the matrix $\tilde G_T$ that is called the controllability Gramian of the system. We have the following property: if $T'\geq T$ then $\psi^\top\tilde G_{T'}\psi = \int_0^{T'} \Vert B(t)^\top R(0,t)^\top \psi\Vert^2\, dt  \geq \int_0^T \Vert B(t)^\top R(0,t)^\top \psi\Vert^2\, dt = \psi^\top\tilde G_T\psi$ for any $\psi\in\R^n$.
%This is why, if controllability holds in time $T$ then it holds in any time $T'\geq T$.
%\end{remark}

\begin{proof}[Proof of Theorem \ref{controlabiliteinstationnaire}.]
Let $x_0\in\R^n$ arbitrary. Any solution of the control system, associated with some control $u$ and starting at $x_0$, satisfies at time $T$
$$
x_u(T)=x^*+\int_0^T R(T,t)B(t)u(t) \, dt
$$
with $x^*=R(T,0)x_0+\int_0^T R(T,t)r(t) \, dt$.

Let us assume that $G_T$ is invertible and let us prove that the control system is controllable in time $T$. Let $x_1\in\R^n$ be any target point. We seek an appropriate control $u$ in the form $u(t)={B(t)}^\top R(t,T)^\top\psi$, with $\psi\in\R^n$ to be chosen such that $x_u(T)=x_1$. With this control, we have $x_u(T)=x^*+G_T\psi$, and since $G_T$ is invertible it suffices to take $\psi=G_T^{-1}(x_1-x^*)$.

Conversely, if $G_T$ is not invertible, then by Remark \ref{remGT} there exists $\psi\in\R^n\setminus\{0\}$ such that $\psi^\top G_T\psi=\int_0^T\Vert {B(t)}^\top R(t,T)^\top \psi\Vert^2dt=0$,
hence $\psi^\top R(T,t)B(t)=0$ for almost every $t\in [0,T]$. It follows that $\psi^\top \int_0^T R(T,t)B(t)u(t)\, dt=0$
for every $u\in L^\infty([0,T],\R^m)$, and thus $\psi^\top (x_u(T)-x^*)=0$, 
which means that $x_u(T)$ belongs to a proper affine subspace of $\R^n$ (namely, $x^*+\psi^\perp$) as $u$ varies. Hence the system is not controllable in time $T$.
\end{proof}

\begin{remark}\label{rem_hum}
This theorem can be proved in an easier and more natural way with the Pontryagin maximum principle (PMP), within an optimal control viewpoint: anticipating a bit, $p(t)=R(t,T)^\top \psi$ is the adjoint vector, solution of $\dot p(t)=-A(t)^\top p(t)$ such that $p(T)=\psi$, obtained by applying the PMP with the cost being the square of the $L^2$ norm of the control, and actually the control used in the above proof is optimal for the $L^2$ norm. The above proof also leads in the infinite-dimensional setting to the \textit{HUM method} (see Part \ref{part2}).
\end{remark}

\begin{remark}\label{rem_auton}
If the system is autonomous ($A(t)\equiv A$, $B(t)\equiv B$) then $R(t,s)= e^{(t-s)A}$ and thus
$$
G_T=\int_0^T e^{(T-t)A}B{B}^\top e^{(T-t)A^\top}dt=\int_0^T e^{tA}B{B}^\top e^{tA^\top}dt .
$$
In that case, since the controllability (Kalman) condition does not depend on the time, it follows that $G_{T_1}$ is invertible if and only if $G_{T_2}$ is invertible, which is not evident from the above integral form (this fact is not true in general in the instationary case).

In the autonomous case, the observability inequality \eqref{inegobsdimfinie} can be written as
\begin{equation}\label{inegobsdimfinie_autonomous}
\int_0^T \Vert B^\top\underbrace{ e^{(T-t)A^\top} \psi}_{p(t)}\Vert^2\, dt \geq C_T \!\!\underbrace{\Vert\psi\Vert^2}_{\Vert p(T)\Vert^2} \qquad\forall\psi\in\R^n .
\end{equation}
Note that, setting $\lambda(t)=p(T-t)$, we have $\dot\lambda(t)=A^\top\lambda(t)$, $\lambda(0)=\psi$, and \eqref{inegobsdimfinie_autonomous} can also be written as
$\int_0^T \Vert B^\top \lambda(t)\Vert^2\, dt \geq C_T \Vert\lambda(0)\Vert^2$.

%Although we will not develop here a general theory for observability, it
This observability inequality is appropriate to be generalized in the infinite-dimensional setting, replacing $e^{tA}$ with a semigroup, and will be of instrumental importance in the derivation of the so-called HUM method (see Part \ref{part2}).
\end{remark}

Let us now provide an ultimate theorem which generalizes the Kalman condition in the instationary case.

\begin{theorem}\label{controlabiliteinstationnaire2}
We assume that $\Omega=\R^m$ (no constraint on the control).
Consider the control system
$\dot{x}(t)=A(t)x(t)+B(t)u(t)+r(t)$
where $t\mapsto A(t)$ and $t\mapsto B(t)$ are of class $C^\infty$. We define the sequence of matrices
$$
B_0(t)=B(t),\quad B_{k+1}(t)=A(t)B_k(t)-\frac{dB_{k}}{dt}(t),\quad k\in\N.
$$
\begin{enumerate}
\item If there exists $t\in[0,T]$ such that
\begin{equation}\label{kalman_time}
\mathrm{Span}\,\{B_k(t)v\ \vert\ v\in\R^m,\ k\in\N \}=\R^n
\end{equation}
then the system is controllable in time $T$.
\item If $t\mapsto A(t)$ and $t\mapsto B(t)$ are moreover analytic (i.e., expandable in a convergent power series at any $t$), then the system is controllable in time $T$ if and only if \eqref{kalman_time} is satisfied for every $t\in[0,T]$.
\end{enumerate}
\end{theorem}

%We do not prove this theorem now. 
The proof of Theorem \ref{controlabiliteinstationnaire2} readily follows from the Hamiltonian characterization of singular trajectories (see \cite{BonnardChyba,Trelat}, see also the proof of the weak Pontryagin Maximum Principle in Section \ref{sec_weakPMP}).

\begin{example}
Thanks to Theorem \ref{controlabiliteinstationnaire2}, it is easy to prove that the control system $\dot{x}(t)=A(t)x(t)+B(t)u(t)$, with
$$
A(t)=\begin{pmatrix}t&1&0\\ 0&t^3&0\\ 0&0&t^2\end{pmatrix},\
\textrm{and}\ B(t)=\begin{pmatrix}0\\ 1\\ 1\end{pmatrix},
$$
is controllable in any time $T>0$, while the control system
\begin{equation*}
\left\{\begin{split}
\dot{x}(t)&=-y(t)+u(t)\cos t,\\
\dot{y}(t)&=x(t)+u(t)\sin t,
\end{split}\right.
\end{equation*}
is never controllable (Theorem \ref{controlabiliteinstationnaire} can also be applied).
\end{example}

\subsubsection{Case with control constraints}
When there are some control constraints, we can easily adapt Theorems \ref{controlabiliteinstationnaire} and \ref{controlabiliteinstationnaire2}, like in Proposition \ref{prop_Kalman_constraints}, to obtain local controllability results.

\begin{proposition}
We assume that $r=0$ and that $0\in\overset{\circ}{\Omega}$. 
Let $x_0\in\R^n$ and $T>0$ be arbitrary.
\begin{itemize}[parsep=1mm,itemsep=0mm,topsep=1mm,leftmargin=*]
\item The control system $\dot x(t)=A(t)x(t)+B(t)u(t)$ is locally controllable in time $T$ around the point $R(T,0)x_0$ if and only if the Gramian matrix $G_T$ is invertible.
\item Assume that $t\mapsto A(t)$ and $t\mapsto B(t)$ are $C^\infty$. If \eqref{kalman_time} is satisfied then the control system $\dot x(t)=A(t)x(t)+B(t)u(t)$ is locally controllable in time $T$ around $R(T,0)x_0$; the converse is true if $t\mapsto A(t)$ and $t\mapsto B(t)$ are analytic.
\end{itemize}
\end{proposition}

\subsection{Geometry of accessible sets}

\begin{theorem}\label{access}
Consider the control system $\dot{x}(t)=A(t)x(t)+B(t)u(t)+r(t)$ in $\R^n$
with controls $u$ taking their values in a compact subset $\Omega$ of $\R^m$.
For every $x_0\in\R^n$, for every $t\geq 0$, the accessible set $\mathrm{Acc}_\Omega(x_0,t)$ is compact, convex and depends continuously on $t$ for the Hausdorff topology.\footnote{Denoting by $d$ the Euclidean distance of $\R^n$, given any two compact subsets $K_1$ and $K_2$ of $\R^n$, the Hausdorff distance $d_H$ between $K_1$ and $K_2$ is defined by
$$
d_H(K_1,K_2)=\sup \Big( \sup_{y \in K_2} d(y,K_1), \sup_{y \in K_1}d(y,K_2) \Big) .
$$
}
\end{theorem}

\begin{remark}
Note that the convexity of the accessible set holds true even though $\Omega$ is not assumed to be convex. This property is not obvious and follows from a Lyapunov lemma (itself based on the Krein-Milman theorem in infinite dimension; see \cite{HermesLaSalle}). Actually this argument leads to $\mathrm{Acc}_\Omega(x_0,t)=\mathrm{Acc}_{\textrm{Conv}(\Omega)}(x_0,t)=\mathrm{Acc}_{\partial\Omega}(x_0,t)$, where $\textrm{Conv}(\Omega)$ is the convex closure of $\Omega$ and $\partial\Omega$ is the boundary of $\Omega$.
This illustrates the so-called \textit{bang-bang principle} (see Section \ref{sec241}).

In infinite dimension those questions are much more difficult (see \cite{Wang_book}).
\end{remark}

\begin{proof}[Proof of Theorem \ref{access}.]
We first assume that $\Omega$ is convex. In this case, we have
$$
\mathrm{Acc}_{\Omega}(x_0,t) = R(t,0)x_0+\int_0^tR(t,s)r(s)\, ds + L_t ( L^\infty([0,T],\Omega) ) 
$$
where the linear continuous operator $L_t:L^\infty([0,T],\R^m)\rightarrow\R^n$ is defined by
$
L_t u = \int_0^t R(t,s) B(s) u(s)\, ds .
$
The convexity of $\mathrm{Acc}_{\Omega}(x_0,t)$ follows by linearity from the convexity of the set $L^\infty([0,T],\Omega)$.

Let us now prove the compactness of $\mathrm{Acc}_{\Omega}(x_0,t)$. Let $(x_n^1)_{n\in\N}$ be a sequence of points of $\mathrm{Acc}_{\Omega}(x_0,t)$. For every $n\in\N$, let $u_n\in L^\infty([0,T],\Omega)$ be a control steering the system from $x_0$ to $x_n^1$ in time $t$, and let $x_n(\cdot)$ be the corresponding trajectory. We have
\begin{equation}\label{eqdem1}
x_n^1=x_n(t)=R(t,0)x_0+\int_0^tR(t,s)(B(s)u_n(s)+r(s))\, ds.
\end{equation}
Since $\Omega$ is compact, the sequence $(u_n)_{n\in\N}$ is bounded in $L^2([0,T],\R^m)$. Since this space is reflexive (see \cite{Brezis}), by weak compactness we infer that a subsequence of $(u_{n})_{n\in\N}$ converges weakly to some $u\in L^2([0,T],\R^m)$. Since $\Omega$ is assumed to be convex, we have moreover that $u\in L^2([0,T],\Omega)$ (note that one has also $u\in L^\infty([0,T],\Omega)$ because $\Omega$ is compact). Besides, using \eqref{eqdem1} and the control system we easily see that the sequence $(x_n(\cdot))_{n\in\N}$ is bounded in $H^1([0,t],\R^n)$. Since this Sobolev space is reflexive and is compactly imbedded in $C^0([0,t],\R^n)$, we deduce that a subsequence of $(x_n(\cdot))_{n\in\N}$ converges uniformly to some $x(\cdot)$ on $[0,t]$. Passing to the limit in \eqref{eqdem1}, we get
$$
x(t)=R(t,0)x_0+\int_0^tR(t,s)(B(s)u(s)+r(s))\,ds
$$
and in particular (a subsequence of) $x_n^1=x_n(t)$ converges to $x(t)\in \mathrm{Acc}_{\Omega}(x_0,t)$. The compactness property is proved.

Let us prove the continuity in time of $\mathrm{Acc}_{\Omega}(x_0,t)$ in Hausdorff topology, i.e., that for any $\varepsilon>0$ there exists $\delta >0$ such that, for all $t_1,t_2\in\R$ satisfying $0 \leq t_1<t_2$ and $t_2-t_1\leq\delta$, we have: %$d_H(\mathrm{Acc}_{\Omega}(x_0,t_1),\mathrm{Acc}_{\Omega}(x_0,t_2)) \leq \varepsilon$.
%It suffices to prove that 
%\begin{enumerate}
%\item $\forall y \in \mathrm{Acc}_{\Omega}(x_0,t_2) \quad d(y,\mathrm{Acc}_{\Omega}(x_0,t_1)) \leq \varepsilon,$
%\item $\forall y \in \mathrm{Acc}_{\Omega}(x_0,t_1) \quad d(y,\mathrm{Acc}_{\Omega}(x_0,t_2)) \leq \varepsilon.$
%\end{enumerate}
\begin{enumerate}[parsep=1mm,itemsep=1mm,topsep=1mm]
\item[(1)] $d(y,\mathrm{Acc}_{\Omega}(x_0,t_1))$ for every $y \in \mathrm{Acc}_{\Omega}(x_0,t_2)$;
\item[(2)] $d(y,\mathrm{Acc}_{\Omega}(x_0,t_2))$ for every $y \in \mathrm{Acc}_{\Omega}(x_0,t_1)$.
\end{enumerate}
Let us prove (1) ((2) being similar). Let $y\in \mathrm{Acc}_{\Omega}(x_0,t_2)$. It suffices to prove that there exists $z\in \mathrm{Acc}_{\Omega}(x_0,t_1)$ such that $d(y,z) \leq \varepsilon$. By definition of $\mathrm{Acc}_{\Omega}(x_0,t_2)$, there exists $u\in L^\infty([0,T],\Omega)$ such that the corresponding trajectory, starting at $x_0$, satisfies $x(t_2)=y$. Then $z=x(t_1)$ is suitable. Indeed, %using that $R(t_i,s)=R(t_i,0)R(0,s)$, we have
\begin{equation*}
\begin{split}
x(t_2)-x(t_1) &= R(t_2,0)x_0+\int_0^{t_2}R(t_2,s)(B(s)u(s)+r(s))\, ds \\
& \quad - R(t_1,0)x_0- \int_0^{t_1}R(t_1,s) (B(s)u(s)+r(s))\, ds \\
&= \int_{t_1}^{t_2} R(t_2,s)(B(s)u(s)+r(s))\, ds \\
&  \quad +\left( R(t_2,0)-R(t_1,0) \right) \left(x_0+\int_0^{t_1} R(0,s)(B(s)u(s)+r(s))\, ds \right)  
\end{split}
\end{equation*}
where we have used that $R(t_i,s)=R(t_i,0)R(0,s)$.
If $t_2-t_1$ is small then the first term of the above sum is small by continuity, and the second term is small by continuity of $t\mapsto R(t,0)$. The result follows.

\medskip

In the general case where $\Omega$ is only compact (but not necessarily convex), the proof is more difficult and uses the Lyapunov lemma in measure theory (see, e.g., \cite[Lemma 4A p. 163]{LeeMarkus}) and more generally the Aumann theorem (see, e.g., \cite{HermesLaSalle}), from which, recalling that $L_T u = \int_0^T R(T,t) B(t) u(t)\, dt$, it follows that
%\begin{equation*}
%\begin{split}
%& \bigg\{ \int_0^T R(T,t) B(t) u(t)\, dt\ \bigm|\ u\in L^\infty([0,T],\Omega) \bigg\} \\
%=& \bigg\{ \int_0^T R(T,t) B(t) u(t)\, dt\ \bigm|\ u\in L^\infty([0,T],\partial\Omega) \bigg\}\\
%=& \bigg\{ \int_0^T R(T,t) B(t) u(t)\, dt\ \bigm|\ u\in L^\infty([0,T],\mathrm{Conv}(\Omega)) \bigg\}
%\end{split}
%\end{equation*}
\begin{multline*}
\left\{ L_Tu \ \mid\ u\in L^\infty([0,T],\Omega) \right\} 
= \left\{ L_Tu\ \mid\ u\in L^\infty([0,T],\partial\Omega) \right\} \\
= \left\{ L_Tu\ \mid\ u\in L^\infty([0,T],\mathrm{Conv}(\Omega)) \right\}
\end{multline*}
and moreover that these sets are compact and convex. The result follows.
\end{proof}

\begin{remark}
%If the set of control constraints $\Omega$ is compact, then the accessible is compact (and convex), and evolves continuously in time. In such conditions obviously $\mathrm{Acc}_{\Omega}(x_0,t)$ can never be equal to $\R^n$. In other words, the system is never globally controllable in time $t$. This is natural in view of the control constraints. Actually t
Theorem \ref{access} allows one to define the concept of \textit{minimal time}: when there are some compact control constraints, due to the compactness of the accessible set, one cannot steer the control system from $x_0$ to another point $x_1$ in arbitrary small time; a minimal positive time is due. %We will come back on this issue later.

Another question of interest is to know whether the control system is controllable, in time not fixed, that is: when is the union of all sets $\mathrm{Acc}_{\Omega}(x_0,t)$, over $t\geq 0$, equal to the whole $\R^n$? This question is difficult.
\end{remark}

\section{Controllability of nonlinear systems}\label{sec_nonlinear}

\subsection{Local controllability results}
\paragraph{Preliminaries: end-point mapping.}
%The first important concept to define is the end-point mapping.

It is easy\footnote{This follows from usual finite-time blow-up arguments on ordinary differential equations, and from the usual Picard-Lindel\"of (Cauchy-Lipschitz) theorem with parameters, the parameter being here a control in a Banach set (see for instance \cite{Sontag}).} to establish that the set $\mathcal{U}_{x_0,T,\R^m}$, endowed with the standard topology of $L^\infty([0,T],\R^m)$, is open, and that the end-point mapping $E_{x_0,T}$ (see Definition \ref{def_Ex0T}) is of class $C^1$ on $\mathcal{U}_{x_0,T,\R^m}$ (it is $C^p$ whenever $f$ is $C^p$).

Note that, for every $t\geq 0$, the accessible set is $\mathrm{Acc}_{\Omega}(x_0,t) = E_{x_0,t}(\mathcal{U}_{x_0,t,\Omega})$.

In what follows we often denote by $x_u(\cdot)$ a trajectory solution of \eqref{general_control_system} corresponding to the control $u$.

\begin{proposition}\label{prop_diffFrechetE}
Let $x_0\in\R^n$ and let $u\in \mathcal{U}_{x_0,T,\R^m}$. The (Fr\'echet) differential\footnote{We recall that, given two Banach spaces $X$ and $Y$, a mapping $F:X\rightarrow Y$ is said to be Fr\'echet differentiable at $x\in X$ if there exists a linear continuous mapping $dF(x):X\rightarrow Y$ such that $F(x+h)=F(x)+dF(x).h+o(h)$ for every $h\in X$. Here, the notation $dF(x).h$ means that the linear operator $dF(x)$ is applied to $h$.}
$dE_{x_0,T}(u):L^\infty([0,T],\R^m)\rightarrow\R^n$ is given by
$$
dE_{x_0,T}(u).\delta u = \delta x(T) = \int_0^TR(T,t)B(t)\delta u(t)\,dt
$$
where $\delta x(\cdot)$ is the solution of the so-called \textit{linearized system} along $(x_u(\cdot),u(\cdot))$,
$$
\delta\dot x(t) = A(t)\delta x(t)+B(t)\delta u(t),\quad \delta x(0)=0,
$$
with
$$
A(t)=\frac{\partial f}{\partial x}(t,x_u(t),u(t)),\quad B(t)=\frac{\partial f}{\partial u}(t,x_u(t),u(t)) ,
$$
(which are respectively of size $n\times n$ and $n\times m$), 
and $R(\cdot,\cdot)$ is the state-transition matrix of the linearized system, defined by \eqref{defR}. %as the unique $n\times n$ matrix solution of $\partial_tR(t,s)=A(t)R(t,s)$, $R(s,s)=I_n$, for any $s\in\R$.
\end{proposition}

\begin{proof}[Proof of Proposition \ref{prop_diffFrechetE}.]
We have
$
E_{x_0,T}(u+\delta u)=E_{x_0,T}(u)+dE_{x_0,T}(u).\delta u+\mathrm{o}(\delta u)
$
by definition of the Fr\'echet differential.
In this first-order Taylor expansion, $E_{x_0,T}(u)=x_u(T)$ and $E_{x_0,T}(u+\delta u)=x_{u+\delta u}(T)$. We want to compute $dE_{x_0,T}(u).\delta u$, which is equal, at the first order, to $x_{u+\delta u}(T)-x_u(T)$. In what follows, we set $\delta x(t)=x_{u+\delta u}(t)-x_u(t)$. We have
\begin{align*}
\delta\dot x(t) &= f(t,x_{u+\delta u}(t),u(t)+\delta u(t)) - f(t,x_u(t),u(t)) \\
&= f(t,x_u(t)+\delta x(t),u(t)+\delta u(t)) - f(t,x_u(t),u(t)) \\
&= \frac{\partial f}{\partial x}(t,x_u(t),u(t)).\delta x(t) + \frac{\partial f}{\partial u}(t,x_u(t),u(t)).\delta  u(t) + \mathrm{o}(\delta x(t),\delta u(t)) 
\end{align*}
so that, at the first order, we identify the linearized system. By integration (note that the remainder terms can be rigorously handled by standard Gronwall arguments, not detailed here), we get $\delta x(T)=\int_0^TR(T,t)B(t)\delta u(t)\, dt$, as expected.
Note that this term provides a linear continuous operator and thus is indeed the Fr\'echet differential of the end-point mapping.
\end{proof}

\begin{remark}\label{remark_loccont}
This theorem says that the differential of the end-point mapping at $u$ is the end-point mapping of the linearized system along $(x_u(\cdot),u(\cdot))$. This is similar to the well known result in dynamical systems theory, stating that the differential of the flow is the flow of the linearized system. This remark has interesting consequences in terms of local controllability properties.
\end{remark}

\paragraph{Local controllability results along a trajectory.}
Let $x_0\in\R^n$ and let $\bar u\in \mathcal{U}_{x_0,T,\R^m}$ be arbitrary. According to Remark \ref{remark_loccont}, if the linearized system along $(x_{\bar u}(\cdot),\bar u(\cdot))$ is controllable in time $T$, then the end-point mapping of the linearized system is surjective, meaning that the linear continuous mapping $dE_{x_0,T}(\bar u):L^\infty([0,T],\R^m)\rightarrow\R^n$ is surjective. It follows from an implicit function argument (surjective mapping theorem) that the end-point mapping itself, $E_{x_0,T}:\mathcal{U}_{x_0,T,\R^m}\rightarrow\R^n$, %is a \textit{local submersion}, and thus, in particular, 
is locally surjective and locally open at $u$.

The above argument works because we have considered controls taking their values in the whole $\R^m$. The argument still works whenever one considers a set $\Omega$ of control constraints, provided that we have room to consider local variations of $\bar u$: this is true as soon as $\bar u$ is in the interior of $L^\infty([0,T],\Omega)$ for the topology of $L^\infty([0,T],\R^m)$ (note that this condition is stronger than requiring that $\bar u$ takes its values in the interior of $\Omega$).
%The local surjectivity of $E_{x_0,T}$ means that the general control system \eqref{general_control_system} is locally controllable at $x_u(T)$.
We have thus obtained the following result (see Definition \ref{def_loccont}).

\begin{theorem}\label{thm_localcont_traj}
Let $x_0\in\R^n$ and let $\bar u\in \mathcal{U}_{x_0,T,\Omega}$. 
We denote by $\bar x(\cdot) = x_{\bar u}(\cdot)$ the trajectory solution of \eqref{general_control_system}, corresponding to the control $\bar u(\cdot)$, such that $\bar x(0)=x_0$, and we set $\bar x_1=\bar x(T)=E_{x_0,T}(\bar u)$.
We assume that the function $\bar u(\cdot)$ is in the interior of $L^\infty([0,T],\Omega)$ for the topology of $L^\infty([0,T],\R^m)$. 

If the linearized system along $(\bar x(\cdot),\bar u(\cdot))$ is controllable in time $T$, then the nonlinear control system \eqref{general_control_system} is locally controllable from $x_0$ in time $T$ around $x_1$.
More precisely, there exists an open neighborhood $V$ of $\bar x_1$ in $\R^n$ and an open neighborhood $U$ of $\bar u(\cdot)$ in $L^\infty([0,T],\R^m)$, satisfying $U\subset \mathcal{U}_{x_0,T,\Omega}\subset L^\infty([0,T],\Omega)$, such that, for every $x_1\in V$, there exists a control $u\in U$ such that $x_u(0)=x_0$ and $x_1=x_u(T)=E_{x_0,T}(u)$.
\end{theorem}

%Note that, in the above statement, $v$ is in a $L^\infty$ neighborhood of $u$ and $x_v(\cdot)$ is in a neighborhood of $x_u(\cdot)$ in $C^0$ topology. 
%The implicit function argument even implies that, in some appropriate coordinates, $E_{x_0,T}$ is a linear projection.

\begin{remark}
The controllability in time $T$ of the linearized system $\delta\dot x(t)=A(t)\delta x(t)+B(t) \delta u(t)$ along $(\bar x(\cdot),\bar u(\cdot))$ can be characterized thanks to Theorems \ref{controlabiliteinstationnaire} and \ref{controlabiliteinstationnaire2}. We thus have explicit sufficient conditions for local controllability. Note that the conditions are not necessary (for instance, in $1D$, $\dot{x}(t)=u(t)^3$, along $\bar u=0$).
\end{remark}

\begin{example} %% exam 2023
Consider the Reeds-Shepp control system
$$
\dot x(t) = v(t)\cos\theta(t), \qquad
\dot y(t) = v(t)\sin\theta(t), \qquad
\dot\theta(t) = u(t),
$$
where the controls $u$ and $v$ are subject to the constraints $\vert u\vert\leq 1$ and $\vert v\vert\leq 1$.
We call \emph{segment} any connected piece of trajectory along which $u=\mathrm{Cst}=0$ and $v=\mathrm{Cst}=\bar v$ with $\bar v\neq 0$.
We call \emph{arc of a circle} any connected piece of trajectory along which $u=\mathrm{Cst}=\bar u$ and $v=\mathrm{Cst}=\bar v$ with $\bar u\neq 0$ and $\bar v\neq 0$.
\begin{itemize}[parsep=1mm,itemsep=1mm,topsep=1mm,leftmargin=*]
\item Prove that the control system is locally controllable along any segment such that $\vert\bar v\vert<1$ (in time equal to that of the segment).
\item Prove that the control system is locally controllable along any arc of a circle such that $\vert\bar u\vert<1$ and $\vert\bar v\vert<1$ (in time equal to that of the arc of a circle).
\item Deduce that the system is globally controllable: for all $(x_0,y_0,\theta_0)\in\R^3$ and $(x_1,y_1,\theta_1) \in\R^3$, there exist $T>0$ and controls $u$ et $v$ satisfying the constraints and generating a trajectory steering the system from $(x_0,y_0,\theta_0)$ to $(x_1,y_1,\theta_1)$ in time $T$. Describe such a strategy in a very simple way.
\end{itemize}
\end{example}

Let us next provide two important applications of Theorem \ref{thm_localcont_traj} (which are particular cases): local controllability around a point, and the return method.

\paragraph{Local controllability around an equilibrium point.}
Assume that the general control system \eqref{general_control_system} is \textit{autonomous}, i.e., that $f$ does not depend on $t$. Assume that $(\bar x,\bar u)\in\R^n\times\R^m$ is an equilibrium point of $f$, i.e., $f(\bar x,\bar u)=0$. In that case, the constant trajectory defined by $\bar x(t)=\bar x$ and $\bar u(t)=\bar u$ is a solution of \eqref{general_control_system}. The linearized system along this (constant) trajectory is given by
$$
\delta\dot x(t)=A\delta x(t)+B \delta u(t) 
$$
with $A=\frac{\partial f}{\partial x}(\bar x,\bar u)$ and $B=\frac{\partial f}{\partial u}(\bar x,\bar u)$.
It follows from Theorem \ref{general_control_system} that, if this linear autonomous control system is controllable (in time $T$) then the nonlinear control system is locally controllable in time $T$ around the point $\bar x$, i.e., $\bar x$ can be steered in time $T$ to any point in some neighborhood. By reversing the time (which is possible because we are here in finite dimension), the converse can be done.
We have thus obtained the following result.

\begin{corollary}\label{cor_loc_cont}
With the above notations, assume that $\mathrm{rank}\, K(A,B)=n$ and that $\bar u\in\mathring{\Omega}$ (interior of $\Omega$).
Then, for every $T>0$, the control system $\dot{x}(t)=f(x(t),u(t))$ is locally controllable in time $T$ around the point $\bar x$ in the following sense: 
for every $T>0$ there exist an open neighborhood $V$ of $\bar x$ in $\R^n$ and an open neighborhood $W$ of $\bar u$ in $\R^m$, satisfying $W\subset\Omega$ and $L^\infty([0,T],W)\subset \mathcal{U}_{x_0,T,\Omega}\subset L^\infty([0,T],\Omega)$, such that, for all $x_0,x_1\in V$, there exists a control $u\in L^\infty([0,T],W)$ such that $x_u(0)=x_0$ and $x_u(T)=x_1$.
\end{corollary}

\begin{example}
Consider the control system
$$
\dot x(t) = x(t)(1-x(t))(x(t)-\theta(t)) , \qquad 
\dot \theta(t) = u(t) ,
$$
where the control $u$ is subject to the constraint $\vert u\vert\leq 1$. The state $x(t)$ represents the density of the population of an ecological system and the state $\theta(t)$, of which we control the derivative, is called the \emph{Allee parameter} of the system. Let $\bar\theta\in(0,1)$ and $T>0$. Prove that the control system is locally controllable in time $T$ around the equilibrium point $(x,\theta,u)=(\bar\theta,\bar\theta,0)$.
\end{example}

\begin{example}\label{ex_pendulum}
Let us consider the famous example of the inverted pendulum (see Figure \ref{figpendulum}) of mass $m$, attached to a carriage of mass $M$ whose acceleration $u(t)$ is controlled, with constraint $\vert u(t)\vert\leq 1$. 
\begin{figure}[h]
\begin{center}
\resizebox{5cm}{!}{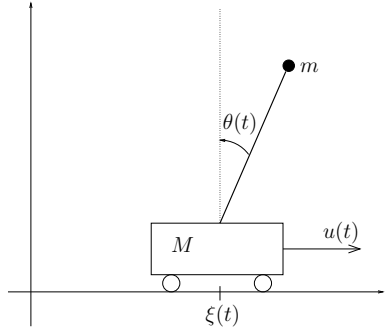}
\end{center}
\caption{Inverted pendulum}\label{figpendulum}
\end{figure}
The control system is
\begin{equation*}
\begin{split}
\ddot{\xi}&=\frac{ml\dot{\theta}^2\sin\theta- mg\cos\theta\sin\theta+u}{M+m\sin^2\theta},\\
\ddot{\theta}&=\frac{-ml\dot{\theta}^2\sin\theta\cos\theta +(M+m)g\sin\theta-u\cos\theta}{l(M+m\sin^2\theta)} ,
\end{split}
\end{equation*}
i.e., this is a system in dimension $4$, of state $x=(\xi,\dot\xi,\theta,\dot\theta)^\top$ (see \cite{Andrea,Trelat} for the derivation of those equations).
The linearized control system at any equilibrium point $(\bar\xi,0,0,0)^\top$ is given by the pair of matrices
$$
A=\begin{pmatrix}
0 & 1 & 0 & 0 \\
0 & 0 & -\frac{mg}{M} & 0 \\
0 & 0 & 0 & 1 \\
0 & 0 & \frac{(M+m)g}{lM} & 0
\end{pmatrix}
\qquad\textrm{and} \qquad
B=\begin{pmatrix} 0\\ \frac{1}{M}\\ 0\\
-\frac{1}{lM}\end{pmatrix}.
$$
The Kalman condition is easily verified. 
Corollary \ref{cor_loc_cont} implies that, for any $T>0$, this control system is locally controllable in time $T$ around the equilibrium.
\end{example}

\begin{example}\label{ex_Maxwell-Bloch}
Consider the control system
$$
\dot x_1(t) = x_2(t) + u_1(t) ,\quad
\dot x_2(t) = x_1(t)x_3(t) + u_2(t), \quad
\dot x_3(t) = -x_1(t)x_2(t) ,
$$
where the state is $x(t)=(x_1(t),x_2(t),x_3(t))^\top\in\R^3$ and the control is $u(t)=(u_1(t),u_2(t))\in\R^2$.
This system stands for the real-valued Maxwell-Bloch equations which model the interaction between light and matter and describe the dynamics of a real-valued two-state quantum system interacting with the electromagnetic mode of an optical resonator. 
%\begin{enumerate}[parsep=1mm,itemsep=1mm,topsep=1mm]%,leftmargin=*
%\item Prove that all equilibrium points $(\bar x,\bar u)$ of the system are given by the parametrized families
%$\mathcal{F}_1 = \{ \bar x=(0,b,c),\ \bar u=(-b,0)\ \mid\ b,c\in\R \}$ and
%$\mathcal{F}_2 = \{ \bar x=(a,0,c),\ \bar u=(0,-ac)\ \mid\ a,c\in\R \}$.
%\item Linearize the system at a point of $\mathcal{F}_1$ and derive a necessary and sufficient condition for controllability of this linearized system.
%\item Linearize the system at a point of $\mathcal{F}_2$ and derive a necessary and sufficient condition for controllability of this linearized system.
%\item Deduce that the system is locally controllable around any equilibrium point that is not of the form $\bar x=(0,0,c)$, $\bar u=(0,0)$.
%\end{enumerate}
\begin{enumerate}[parsep=1mm,itemsep=1mm,topsep=1mm,leftmargin=*]
\item[(1)] First, we see that all equilibrium points $(\bar x,\bar u)$ of the system are given by the parametrized families
$\mathcal{F}_1 = \{ \bar x=(0,b,c),\ \bar u=(-b,0)\ \mid\ b,c\in\R \}$ and
$\mathcal{F}_2 = \{ \bar x=(a,0,c),\ \bar u=(0,-ac)\ \mid\ a,c\in\R \}$.

\item[(2)] Computing the linearized system at a point of $\mathcal{F}_1$ (resp., of $\mathcal{F}_2$),
% is given by the matrices 
%$
%A=\begin{pmatrix}
%0 & 1 & 0 \\
%c & 0 & 0 \\
%-b & 0 & 0
%\end{pmatrix}$ and $B=\begin{pmatrix}
%1 & 0 \\
%0 & 1 \\
%0 & 0
%\end{pmatrix}$,
we check that the Kalman condition implies that $b\neq 0$ (resp., $a\neq 0$) is a necessary and sufficient condition for the controllability of this linearized system.

%\item[(2.b)] The linearized system at a point of $\mathcal{F}_2$ is given by the matrices
%$
%A=\begin{pmatrix}
%0 & 1 & 0 \\
%c & 0 & a \\
%0 & -a & 0
%\end{pmatrix}$ and $B=\begin{pmatrix}
%1 & 0 \\
%0 & 1 \\
%0 & 0
%\end{pmatrix}$,
%and the Kalman condition implies that $a\neq 0$ is a necessary and sufficient condition for the controllability of this linearized system.

\item[(3)] It follows from Corollary \ref{cor_loc_cont} that the Maxwell-Bloch system is locally controllable around any equilibrium point that is not of the form $\bar x=(0,0,c)$, $\bar u=(0,0)$.
\end{enumerate}
\end{example}

Many applications of Corollary \ref{cor_loc_cont} can be found in the literature. It can be noted that the proof, based on the implicit function theorem, is robust and withstands some generalizations in infinite dimension (possibly replacing the implicit function theorem with more appropriate results, such as the Kakutani fixed point theorem).
The interested reader will easily find many examples and applications. We do not have room here to list or provide some of them.

\paragraph{The return method.}
In Corollary \ref{cor_loc_cont}, the sufficient condition is that the linearized system at the equilibrium point be controllable. Assume now that we are in a situation where the linearized system at the equibrium point is not controllable, and however we would like to prove, using alternative conditions, that the nonlinear control system is locally controllable.
The idea of the so-called \textit{return method}\footnote{The return method was invented by J.-M. Coron first with the aim of stabilizing control systems with smooth instationary feedbacks. Then it was applied to many problems of control of PDEs in order to establish controllability properties. We refer the reader to \cite{Coron}.} is to assume that there exists a nontrivial loop trajectory of the control system, going from $x_0$ to $x_0$ in time $T$, along which the linearized control system is controllable. Then, Theorem \ref{thm_localcont_traj} implies that the control system is locally controllable around $x_0$.

Note that the method is not restricted to equilibrium points. We have the following corollary.

\begin{corollary}
Let $x_0\in\R^n$. Assume that there exists a trajectory $\bar x(\cdot)$ of the control system \eqref{general_control_system}, corresponding to a control $\bar u(\cdot)$ on $[0,T]$, such that $\bar x(0)=\bar x(T)=x_0$. Assume that the function $\bar u(\cdot)$ is in the interior of $L^\infty([0,T],\Omega)$ for the topology of $L^\infty([0,T],\R^m)$.
If the linearized system along $(\bar x(\cdot),\bar u(\cdot))$ is controllable in time $T$, then the nonlinear control system \eqref{general_control_system} is locally controllable in time $T$ around the point $x_0$:
there exists an open neighborhood $V$ of $x_0$ in $\R^n$ and an open neighborhood $U$ of $\bar u(\cdot)$ in $L^\infty([0,T],\R^m)$, satisfying $U\subset \mathcal{U}_{x_0,T,\Omega}\subset L^\infty([0,T],\Omega)$, such that, for every $x_1\in V$, there exists a control $u\in U$ such that $x_u(0)=x_0$ and $x_1=x_u(T)=E_{x_0,T}(u)$.
\end{corollary}

\begin{example}
Consider the Dubins car model
%\begin{equation*}
%\begin{split}
%\dot{x}_1(t) &= \cos\theta(t),\qquad x_1(0)=0,\\
%\dot{x}_2(t) &= \sin\theta(t),\qquad\, x_2(0)=0,\\
%\dot\theta(t) &= u(t),\quad\quad\qquad \theta(0)=0\ [2\pi].
%\end{split}
%\end{equation*}
$$
\dot{x}_1(t) = \cos\theta(t),\qquad 
\dot{x}_2(t) = \sin\theta(t),\qquad
\dot\theta(t) = u(t),
$$
with initial point $x_1(0)=x_2(0)=0$ and $\theta(0)=0\ [2\pi]$. The control is subject to the constraint $\vert u\vert\leq M$ for some fixed $M>0$.
Let us prove that this control system is locally controllable at $(0,0,0\ [2\pi])$ in any time $T>\frac{2\pi}{M}$. Consider the reference trajectory given by
$$
\bar x_1(t) = \frac{T}{2\pi}\sin\frac{2\pi t}{T},\quad \bar x_2(t) = \frac{T}{2\pi}\left( 1-\cos\frac{2\pi t}{T}\right),\quad \bar \theta(t)=\frac{2\pi t}{T},\quad \bar u(t)=\frac{2\pi}{T} 
$$
(the inequality $T>\frac{2\pi}{M}$ implies that $\vert\bar u\vert<M$).
The linearized system along this trajectory is represented by the matrices
$$
A(t) = \begin{pmatrix}
0 & 0 & -\sin\frac{2\pi t}{T} \\ 0 & 0 & \phantom{-}\cos\frac{2\pi t}{T} \\ 0 & 0 & 0 
\end{pmatrix}
, \quad
B(t) = \begin{pmatrix}
0 \\ 0 \\ 1
\end{pmatrix}
$$
and it is easy to check (using Theorem \ref{controlabiliteinstationnaire2}) that this system is controllable in any time $T>0$. The local controllability property in time $T>\frac{2\pi}{M}$ follows. 

Actually, it can be proved that the Dubins system is \emph{not} locally controllable at $(0,0,0\ [2\pi])$ in time $T\in[0,\frac{2\pi}{M}]$.
\end{example}

In the previous paragraphs, we have derived, thanks to simple implicit function arguments, \textit{local} controllability results. The local feature of these results is not a surprise, having in mind that, from a general point of view, it is expected that showing the global surjectivity of the nonlinear mapping $E_{x_0,T}$ is difficult.

In view of that, we next provide two further issues.

\subsection{Geometry of the accessible set}
In view of better understanding why it is hopeless, in general, to get global controllability, it is useful to have a clear geometric picture of the accessible set. We have the following result, similar to Theorem \ref{access}.

\begin{theorem}\label{ensacccompact}
Let $x_0\in\R^n$ and let $T>0$.
We assume that:
\begin{itemize}[parsep=1mm,itemsep=1mm,topsep=1mm]%,leftmargin=*
\item $\Omega$ is compact;
\item there exists $b>0$ such that, for every admissible control $u\in\mathcal{U}_{x_0,T,\Omega}$, one has $\Vert x_u(t)\Vert \leq b$ for every $t\in [0,T]$;
\item there exists $c>0$ such that $\Vert f(t,x,u)\Vert\leq c$ for every $t\in[0,T]$, for every $x\in\R^n$ such that $\Vert x\Vert\leq b$ and for every $u\in\Omega$;
\item the set of velocities $V(t,x)=\{f(t,x,u)\ \vert\ u\in\Omega\}$ is convex, for all $(t,x)$.
\end{itemize}
Then the set $\mathrm{Acc}_\Omega(x_0,t)$ is compact and varies continuously in time on $[0,T]$ (for the Hausdorff topology).
\end{theorem}

\begin{remark}
The second assumption (uniform boundedness of trajectories) is done to avoid blow-up of trajectories. It is satisfied for instance if the dynamics $f$ is sublinear at infinity. The third assumption is done for technical reasons in the proof, because at the beginning we assumed that $f$ is locally integrable, only, with respect to $t$.
The assumption of convexity of $V(t,x)$ is satisfied for instance for control-affine systems (that is, whenever $f$ is affine in $u$) and if $\Omega$ is moreover convex.
\end{remark}

\begin{proof}[Proof of Theorem \ref{ensacccompact}.]
First of all, note that $V(t,x)$ is compact, for all $(t,x)$, because $\Omega$ is compact. Let us prove that $Acc(x_0,t)$ is compact for every $t\in[0,T]$. It suffices to prove that every sequence $(x_n)$ of points of $Acc(x_0,t)$ has a converging subsequence. For every integer $n$, let $u_n\in\mathcal{U}_{x_0,t,\Omega}$ be a control steering the system from $x_0$ to $x_n$ in time $t$ and let $x_n(\cdot)$ be the corresponding trajectory. We have
$
x_n=x_n(t)=x_0+\int_0^t f(s,x_n(s),u_n(s))\,ds.
$
Setting $g_n(s)=f(s,x_n(s),u_n(s))$ for $s\in[0,t]$, using the assumptions, the sequence of functions $(g_n(\cdot))_{n\in\N}$ is bounded in $L^\infty([0,t],\R^n)$, and therefore, up to some subsequence, it converges to some function $g(\cdot)$ for the weak star topology of $L^\infty([0,t],\R^n)$ (see \cite{Brezis}). For every $\tau\in[0,t]$, we set $x(\tau)=x_0+\int_0^\tau g(s)\, ds$. Clearly, $x(\cdot)$ is absolutely continuous on $[0,t]$, and $\lim_{n\rightarrow +\infty} x_n(s)=x(s)$ for every $s\in[0,t]$, that is, the sequence $(x_n(\cdot))_{n\in\N}$ converges pointwise to $x(\cdot)$. The objective is now to prove that $x(\cdot)$ is a trajectory associated with a control $u$ taking its values in $\Omega$, that is, to prove that $g(s)=f(s,x(s),u(s))$ for almost every $s\in[0,t]$.

To this aim, for every integer $n$ and almost every $s\in[0,t]$, we set $h_n(s)=f(s,x(s),u_n(s))$, and we define the set
$${\cal V} = \{ h(\cdot)\in L^2([0,t],\R^n)\ \mid\ h(s)\in V(s,x(s))\ \textrm{for a.e.}\ s\in[0,t] \} . $$
Note that $h_n\in{\cal V}$ for every integer $n$. For all $(t,x)$, the set $V(t,x)$ is compact and convex, and, using the fact that, from any sequence converging strongly in $L^2$ we can substract a subsequence converging almost everywhere, we infer that ${\cal V}$ is convex and closed in $L^2([0,t],\R^n)$ for the strong topology. Therefore ${\cal V}$ is closed as well in $L^2([0,t],\R^n)$ for the weak topology (see \cite{Brezis}).
But like $(g_n)_{n\in\N}$, the sequence $(h_n)_{n\in\N}$ is bounded in $L^2$, hence up to some subsequence it converges to some function $h$ for the weak topology, and $h$ must belong to ${\cal V}$ since ${\cal V}$ is weakly closed.

Finally, let us prove that $g=h$ almost everywhere. We have
\begin{equation}\label{temp1}
\int_0^t\varphi(s)g_n(s)ds = \int_0^t\varphi(s)h_n(s)\,ds +\int_0^t\varphi(s)\left( g_n(s)-h_n(s)\right) \,ds
\end{equation}
for every $\varphi\in L^2([0,t],\R)$.
By assumption, $f$ is globally Lipschitz in $x$ on $[0,T]\times \bar{B}(0,b) \times \Omega$, and hence by the mean value inequality, there exists $C>0$ such that $\Vert g_n(s)-h_n(s)\Vert\leq C\Vert x_n(s)-x(s)\Vert$ for almost every $s\in[0,t]$. The sequence $(x_n)$ converge pointwise to $x(\cdot)$, hence, using the dominated convergence theorem, we infer that
$\int_0^t\varphi(s)\left( g_n(s)-h_n(s)\right) \, ds \rightarrow 0$
as $n\rightarrow +\infty$.
Passing to the limit in \eqref{temp1}, we obtain $\int_0^t\varphi(s)g(s)ds = \int_0^t\varphi(s)h(s)\,ds$ for every $\varphi\in L^2([0,t],\R)$ and therefore $g=h$ almost everywhere on $[0,t]$.

In particular, we have $g\in{\cal V}$, and hence for almost every $s\in[0,t]$ there exists $u(s)\in\Omega$ such that $g(s)=f(s,x(s),u(s))$. Applying a measurable selection lemma in measure theory (note that $g\in L^\infty([0,t],\R^n)$), $u(\cdot)$ can be chosen to be measurable on $[0,T]$ (see \cite[Lemmas 2A, 3A p. 161]{LeeMarkus}).

In conclusion, the trajectory $x(\cdot)$ is associated on $[0,t]$ with the control $u$ taking its values in $\Omega$, and $x(t)$ is the limit of the points $x_n$. This shows the compactness of $\mathrm{Acc}_\Omega(x_0,t)$.

It remains to establish the continuity of the accessible set with respect to time. Like in Theorem \ref{access}, let $t_1$ and $t_2$ be two real numbers such that $0<t_1<t_2\leq T$ and let $x_2\in \mathrm{Acc}_\Omega(x_0,t_2)$. By definition, there exists a control $u$ taking its values in $\Omega$, generating a trajectory $x(\cdot)$, such that $x_2=x(t_2)=x_0+\int_0^{t_2} f(t,x(t),u(t))\,dt$.
The point $x_1=x(t_1)=x_0+\int_0^{t_1} f(t,x(t),u(t))\,dt $ belongs to $\mathrm{Acc}_\Omega(x_0,t_1)$, and using the assumptions on $f$, we have $\Vert x_2-x_1\Vert\leq \mathrm{Cst}\vert t_2-t_1\vert $. We conclude easily.
\end{proof}

\subsection{Global controllability results}
One can find global controllability results in the existing literature, established for particular classes of control systems. Let us provide here controllability results in the important class of control-affine systems.

We say that a control system is \textit{control-affine} whenever the dynamics $f$ is affine in $u$, in other words the control system is
$$
\dot{x}(t) = f_0(x(t))+\sum_{i=1}^m u_i(t) f_i(x(t))
$$
where the mappings $f_i:\R^n\rightarrow\R^n$, $i=0,\ldots,m$ are smooth. The term $f_0(x)$ is called a \textit{drift}.
Here, there is a crucial insight coming from differential geometry. We consider the mappings $f_i$ are vector fields on $\R^n$. Such vector fields generate some flows, some integral curves, and at this point geometric considerations come into the picture.

There are many existing results in the literature, providing local or global controllability results under conditions on the Lie brackets of the vector fields.

We recall that the Lie bracket of two vector fields $X$ and $Y$ is defined either by $[X,Y](x)=dY(x).X(x)-dX(x).Y(x)$, or, recalling that a vector field is a first-order derivation on $C^\infty(\R^n,\R)$ defined by $(Xf)(x)=df(x).X(x)$ for every $f\in C^\infty(\R^n,\R)$ (Lie derivative), by $[X,Y]=XY-YX$ (it is obvious to check that it is indeed a first-order derivation).
We also mention that, denoting by $\mathrm{exp}(tX)$ and $\mathrm{exp}(tY)$ the flows generated by the vector fields $X$ and $Y$, the flows commute, i.e., $\mathrm{exp}(t_1X)\circ\mathrm{exp}(t_2Y)=\mathrm{exp}(t_2Y)\circ\mathrm{exp}(t_1X)$ for all times $t_1$ and $t_2$, if and only if $[X,Y]=0$. If the Lie bracket is nonzero then the flows do not commute, but we have the asymptotic expansion
$$
\mathrm{exp}(-tY)\circ\mathrm{exp}(-tX)\circ\mathrm{exp}(tY)\circ\mathrm{exp}(tX)(x) = x+\frac{t^2}{2}[X,Y](x)+\mathrm{o}(t^2)
$$
as $t\rightarrow 0$. The left-hand side of that equality is the point obtained by starting at $x$, following the vector field $X$ during a time $t$, then the vector field $Y$ during a time $t$, then $-X$ during a time $t$, and then $-Y$ during a time $t$. What it says is that this loop is not closed! The lack of commutation is measured through the Lie bracket $[X,Y]$.
For more details on Lie brackets, we refer the reader to any textbook of differential geometry. Without going further, we mention that the Campbell-Hausdorff formula gives a precise series expansion of $Z$, defined by $\mathrm{exp}(X)\circ\mathrm{exp}(Y)=\mathrm{exp}(Z)$, in terms of iterated Lie brackets of $X$ and $Y$. The first terms are $Z=X+Y+\frac{1}{2}[X,Y]+\cdots$.

Finally, we recall that the Lie algebra generated by a set of vector fields is the set of all possible iterated Lie brackets of these vector fields.

For control-affine systems without drift, we have the following well-known Chow-Rashevski theorem (also called H\"ormander condition, or Lie Algebra Rank Condition), whose early versions can be found in \cite{Chow,Rashevski}.

\begin{theorem}\label{thm_liebrackets}
Consider a control-affine system without drift in $\R^n$. Assume that $\Omega=\R^m$ (no constraint on the control) and that the Lie algebra generated by the vector fields $f_1,\ldots,f_m$ is equal to $\R^n$ (at any point). Then the system is globally controllable, in any time $T$.
\end{theorem}

\begin{proof}[Proof of Theorem \ref{thm_liebrackets}.]
We sketch the proof in the case $n=3$ and $m=2$, assuming that $\mathrm{rank}(f_1,f_2,[f_1,f_2])=3$ at any point. Let $\lambda\in\R$. We define the mapping
$$
\varphi_{\lambda}(t_1,t_2,t_3) = \mathrm{exp}(\lambda f_1) \, \mathrm{exp}(t_3 f_2) \, \mathrm{exp}(- \lambda f_1)\, \mathrm{exp}(t_2 f_2)\, \mathrm{exp} (t_1 f_1)(x_0) . $$
We have $\varphi_{\lambda}(0)=x_0$. Let us prove that, for $\lambda \neq 0$ small enough, $\varphi_{\lambda}$ is a local diffeomorphism at $0$. From the Campbell-Hausdorff formula, we infer that
$
\varphi_{\lambda}(t_1,t_2,t_3)=\mathrm{exp}(t_1f_1+(t_2+t_3)f_2+\lambda t_3[f_1,f_2]+ \cdots) , 
$
hence
$\partial_{\partial t_1}\varphi_\lambda(0)=f_1(x_0)$,
$\partial_{\partial t_2}\varphi_\lambda(0)=f_2(x_0)$ and
$\partial_{\partial t_3}\varphi_\lambda(0)=f_2(x_0)+\lambda[f_1,f_2](x_0)+\textrm{o}(\lambda)$.
By assumption, it follows that $d\varphi_\lambda$ is an isomorphism, and therefore $\varphi_\lambda$ is a local diffeomorphism at $0$. We conclude by an easy connectedness argument.
\end{proof}

We approach here the \textit{geometric control theory}. The theorem above is one of the many existing results that can be obtained with Lie bracket considerations. We refer the reader to the textbook \cite{Jurdjevic} for many results which are of a geometric nature. In particular this reference contains some material in order to treat the case of control-affine systems with drift (see also \cite{Coron}).
Note that, in presence of a drift $f_0$, an easy sufficient condition ensuring global controllability is that the Lie algebra generated by the controlled vector fields $f_1,\ldots,f_m$ be equal to $\R^n$ (at any point); 
indeed, the rough idea is that, taking large controls, in some sense the drift term can be compensated and then one can apply Theorem \ref{thm_liebrackets}.

\begin{example}
The Heisenberg system in $\R^3$
$$
\dot{x}(t) = u_1(t),\quad \dot{y}(t)=u_2(t),\quad\dot{z}(t)=u_1(t)y(t)-u_2(t)x(t),
$$
is represented by the two vector fields
$f_1 = \partial x + y \partial z$ and $f_2 = \partial y-x\partial z$.
We have $[f_1,f_2]=-2\partial z$ and thus the Lie algebra condition is satisfied. Therefore this system is controllable.
\end{example}

\begin{remark}
In practice, Lie brackets are often realized thanks to oscillating functions, like the sine function taken with a sufficiently large frequency (see \cite{Coron}). One can find in the literature many results concerning the \emph{motion planning problem}, which consists of designing simple enough controls realizing controllability. We refer to \cite{Jean} for a survey on such techniques.
\end{remark}

%\begin{remark}
%The Lie condition above is also called the H\"ormander condition, due to the well-known following result due to H\"ormander (see \cite{Ho-67}). Let $L=\sum_{i=1}^m f_i^2$ be an operator, defined as a sum of squares of vector fields: this is a second-order differential operator, which can be called a \textit{sub-Laplacian}.\footnote{If $m=n$ and if the $f_i$ coincide with the canonical basis of $\R^n$ then $L$ is exactly the usual Laplacian. We speak of sub-Laplacian when $m<n$.} We say that $L$ is \textit{hypoelliptic} if $g$ is $C^\infty$ whenever $Lg$ is $C^\infty$. Note that the usual Laplacian is hypoelliptic, but the question is not obvious if $m<n$.
%The famous result by H\"ormander states that $L$ is hypoelliptic if and only if the Lie algebra generated by the vector fields $f_1,\ldots,f_m$ is equal to $\R^n$ (at any point).
%
%Here, we recover an idea of controllability because, in order to prove that if $Lg$ is $C^\infty$ then $g$ is $C^\infty$, we have to ``control" the derivatives of $g$ in any possible direction of $\R^n$.
%\end{remark}

\chapter{Optimal control}\label{chap_opt}

In Chapter \ref{chap_cont}, we have provided controllability properties for general classes of control systems. Considering some control problem of trying to reach some final configuration for the control system \eqref{general_control_system}, from some initial configuration, with an admissible control, it happens that, in general, there exists an infinite number of controls making the job (think of all possibilities of realizing a parallel parking, for instance). Among this infinite number of controls, we now would like to select (at least) one control, achieving the desired controllability problem, and moreover minimizing some cost criterion (for instance, one would like to realize the parallel parking by minimizing the time, or by minimizing the fuel consumption). This is then an optimal control problem.

\medskip

The main objective of this chapter is to formulate the \textit{Pontryagin maximum principle}, which is the milestone of optimal control theory. It provides first-order necessary conditions for optimality, which allow one to compute or at least parametrize the optimal trajectories.

\medskip

Let us give the general framework that will be used throughout the chapter.

Let $n$ and $m$ be two positive integers.
%As in Chapter \ref{chap_cont}, 
We consider a control system in $\R^n$
\begin{equation}\label{gencontsyst}
\dot{x}(t) = f(t,x(t),u(t))
\end{equation}
where $f:\R\times\R^n\times\R^m\rightarrow\R^n$ is of class $C^1$, and the controls are measurable essentially bounded functions of time taking their values in some measurable subset $\Omega$ of $\R^m$ (set of control constraints).

Let $f^0:\R\times\R^n\times\R^m\rightarrow\R$ and $g:\R\times\R^n\rightarrow\R$ be functions of class $C^1$. For every $x_0\in\R^n$, for every $t_f\geq 0$, and for every admissible control $u\in\mathcal{U}_{x_0,t_f,\Omega}$ (see Section \ref{sec_nonlinear}), the \textit{cost} of the trajectory $x(\cdot)$, solution of \eqref{gencontsyst}, corresponding to the control $u$ and such that $x(0)=x_0$, is defined by
\begin{equation}\label{def_cost}
C_{x_0,t_f}(u) = \int_0^{t_f} f^0(t,x(t),u(t))\,dt+g(t_f,x(t_f)).
\end{equation}
%Most often, in what follows, if the initial point is fixed or is clear from the context, we will simply denote the cost by $C(t_f,u)$ or by $C_{t_f}(u)$.
Many variants of a cost can be given, anyway the one above is already quite general and covers a very large class of problems. If needed, one could easily add some term penalizing the initial point. Note also that the term $g(t_f,x(t_f))$ could as well be written in integral form and thus be put in the definition of the function $f^0$; however we prefer to keep this formulation that we find convenient in many situations.

%Let us now define the optimal control problem that we will consider. 
Let $M_0$ and $M_1$ be two measurable subsets of $\R^n$.
We consider the optimal control problem (denoted in short $\OCP$ in what follows) of determining a trajectory $x(\cdot)$, defined on $[0,t_f]$ (where the final time $t_f$ can be fixed or not in \OCP), corresponding to an admissible control $u\in\mathcal{U}_{x(0),t_f,\Omega}$, solution of \eqref{gencontsyst}, such that
$x(0)\in M_0$ and $x(t_f)\in M_1$
and minimizing the cost \eqref{def_cost} over all possible trajectories steering the control system from $M_0$ to $M_1$ in time $t_f$.

\medskip

This is a general nonlinear optimal control problem, but without any state constraints. We could restrict the set of trajectories by imposing some pointwise constraints on $x(t)$ (a region of the state space may be forbidden). Such constraints are however not easily tractable in the Pontryagin maximum principle and make the analysis much more difficult (see Section \ref{sec_generalizations_PMP}). %We will however comment further on this issue, but for simplicity we will mainly not consider state constraints.

\section{Existence of an optimal control}
Although it is not very useful, let us state a general result ensuring the existence of an optimal solution of \OCP.
In the theorem below, note that we can assume that $f^0$ and $g$ are only continuous. Here, there is no additional difficulty in adding some state constraints.

\begin{theorem}\label{thmfilippov}
We consider $\OCP$ and we assume that:
\begin{itemize}[parsep=1mm,itemsep=1mm,topsep=1mm]%,leftmargin=*
\item $\Omega$ is compact, $M_0$ and $M_1$ are compact;
\item $M_1$ is reachable from $M_0$, that is, there exists a trajectory (corresponding to an admissible control) steering the system from $M_0$ to $M_1$;
\item there exist $b>0$ and $c>0$ such that $\Vert f(t,x,u)\Vert+\vert f^0(t,x,u)\vert\leq c$ for every $t\in[0,b]$, every $x\in\R^n$ satisfying $\Vert x\Vert\leq b$, and every $u\in\Omega$; moreover, for every trajectory $x(\cdot)$ defined on $[0,t_f]$ and steering the system from $M_0$ to $M_1$, one has $t_f\leq b$ and $\Vert x(t)\Vert \leq b$ for every $t\in [0,t_f]$ (no blow-up);
\item the epigraph of extended velocities
\begin{equation}\label{optconv}
\tilde{V}(t,x)=\left\{
\begin{pmatrix}
f(t,x,u)\\
f^0(t,x,u)+\gamma
\end{pmatrix}
\ \Bigm|\ u\in\Omega,\ \gamma\geq 0\right\}
\end{equation}
is convex, for all $(t,x)$.
\end{itemize}
We assume moreover in $\OCP$ that the trajectories are subject to state constraints $c_i(t,x(t))\leq 0$, where the $c_i$, $i\in\{1,\ldots,r\}$ are continuous functions defined on $\R\times\R^n$.

Then $\OCP$ has at least one solution.
\end{theorem}

If the final time is fixed in $\OCP$ then we assume that $M_1$ is reachable from $M_0$ exactly in time $t_f$.
Note that it is easy to generalize this result to more general situations, for instance the sets $M_0$, $M_1$ and $\Omega$ could depend on $t$ (see \cite{LeeMarkus}, and more generally see \cite{Cesari} for many variants of existence results).

Such existence results are however often difficult to apply in practice because of the strong assumption \eqref{optconv} (not satisfied in general as soon as $f$ is ``too much" nonlinear). In practice, we often apply the Pontryagin maximum principle (that we will see next), without being sure a priori that there exist an optimal solution. If we can solve the resulting necessary conditions, then this often gives a way for justifying that indeed, a posteriori, there exists an optimal solution.

\begin{proof}[Proof of Theorem \ref{thmfilippov}.]
The proof is similar to the one of Theorem \ref{ensacccompact}.

Let $\delta$ be the infimum of costs $C(u)$ over the set of admissible controls $u\in L^\infty(0,t(u);\Omega)$ generating trajectories such that $x(0)\in M_0$, $x(t(u))\in M_1$ and satisfying the state constraints $c_i(x(\cdot))\leq 0$, $i=1,\ldots,r$.
Let us consider a minimizing sequence of trajectories $x_n(\cdot)$ associated with controls $u_n$, that is, a sequence of trajectories satisfying all constraints and such that $C(u_n)\rightarrow\delta$ as $n\rightarrow+\infty$.
For every integer $n$, we set
$$
\tilde F_n(t)=\begin{pmatrix}
f(t,x_n(t),u_n(t))  \\ f^0(t,x_n(t),u_n(t)) 
\end{pmatrix} = \begin{pmatrix}
F_n(t)  \\ F^0_n(t) 
\end{pmatrix}
$$
for almost every $t\in[0,t(u_n)]$.
From the assumptions, the sequence of functions $(\tilde F_n(\cdot))_{n\in\N}$ (extended by $0$ on $(t_n(u),b]$) is bounded in $L^\infty(0,b;\R^n)$, and hence, up to some subsequence, it converges to some function
%$\tilde F(\cdot)=\begin{pmatrix} F(\cdot)\\ F^0(\cdot)\end{pmatrix}$
$\tilde F(\cdot)=\left( F(\cdot), F^0(\cdot)\right)^\top$
for the weak star topology of $L^\infty(0,b;\R^{n+1})$ (see \cite{Brezis}). Also, up to some subsequence, the sequence $(t_n(u_n))_{n\in\N}$ converges to some $T\geq 0$, and we have $\tilde F(t)=0$ for $t\in(T,b]$. Finally, by compactness of $M_0$, up to some subsequence, the sequence $(x_n(0))_{n\in\N}$ converges to some $x_0\in M_0$. For every $t\in[0,T]$, we set
$x(t)=x_0+\int_0^t F(s)\,ds ,$
and then $x(\cdot)$ is absolutely continuous on $[0,T]$. Moreover, for every $t\in[0,T]$, we have $\lim_{n\rightarrow +\infty} x_n(t)=x(t)$, that is, the sequence $(x_n(\cdot))_{n\in\N}$ converges pointwise to $x(\cdot)$. As in the proof of Theorem \ref{ensacccompact}, the objective is then to prove that the trajectory $x(\cdot)$ is associated with a control $u$ taking its values in $\Omega$, and that this control is moreover optimal.

We set
%$
%\tilde h_n(t)=\begin{pmatrix}
%f(t,x(t),u_n(t)) \\ f^0(t,x(t),u_n(t))
%\end{pmatrix} ,
%$
$
\tilde h_n(t)=\left( f(t,x(t),u_n(t)) , f^0(t,x(t),u_n(t)) \right)^\top
$
for every integer $n$ and for almost every $t\in[0,t(u_n)]$.
If $T>t(u_n)$, then we extend $\tilde h_n$ on $[0,T]$ by
%$
%\tilde h_n(t)=\begin{pmatrix}
%f(t,x(t),v) \\ f^0(t,x(t),v)
%\end{pmatrix} ,
%$
$
\tilde h_n(t)=\left( f(t,x(t),v) , f^0(t,x(t),v) \right)^\top
$
for some arbitrary $v\in\Omega$. Besides, we define (note that $\Omega$ is compact)
$$
\beta = \max\{ \vert f^0(t,x,u)\vert\ \vert\ 0\leq t\leq b,\ \Vert x\Vert\leq b,\ u\in\Omega \}.
$$
For every $(t,x)\in\R^{1+n}$, we then slightly modify the definition of $\tilde V(t,x)$ to make it compact (keeping it convex), by setting
$$
\tilde{V}_\beta(t,x)=\left\{
\begin{pmatrix}
f(t,x,u)\\
f^0(t,x,u)+\gamma
\end{pmatrix}
\ \Big\vert\ u\in\Omega,\ \gamma\geq 0,\ \vert f^0(t,x,u)+\gamma\vert\leq \beta\right\}.
$$
We define
$\tilde{\cal V} = \{ \tilde h(\cdot)\in L^2([0,T],\R^{n+1})\ \vert\ h(t)\in \tilde V_\beta(t,x(t))\ \textrm{for a.e.}\ t\in[0,T] \} . $
By construction, we have $\tilde h_n\in\tilde{\cal V}$ for every integer $n$. 
At this step, we need a lemma:

\begin{lemma}\label{lemcaltildeVconv}
The set $\tilde{\cal V}$ is convex and strongly closed in $L^2([0,T],\R^{n+1})$.
\end{lemma}

\begin{proof}[Proof of Lemma \ref{lemcaltildeVconv}]
Let us prove that $\tilde{\cal V}$ is convex. Let $\tilde r_1,\tilde r_2\in\tilde{\cal V}$, and let $\lambda\in[0,1]$. By definition, for almost every $t\in [0,T]$, we have $\tilde r_1(t)\in \tilde V_\beta(t,x(t))$ and $\tilde r_2(t)\in \tilde V_\beta(t,x(t))$. Since $\tilde V_\beta(t,x(t))$ is convex, it follows that $\lambda\tilde r_1(t)+(1-\lambda)\tilde r_2(t)\in \tilde V_\beta(t,x(t))$. Hence $\lambda\tilde r_1+(1-\lambda)\tilde r_2\in \tilde{\cal V}$.

Let us prove that $\tilde{\cal V}$ is strongly closed in $L^2([0,T],\R^n)$. Let $(\tilde r_n)_{n\in\N}$ be a sequence of $\tilde{\cal V}$ converging to $\tilde r$ for the strong topology of $L^2([0,T],\R^n)$. Let us prove that $\tilde r\in \tilde{\cal V}$. Up to some subsequence, $(\tilde r_n)_{n\in\N}$ converges almost everywhere to $\tilde r$, but by definition, for almost every $t\in [0,T]$ we have $\tilde r_n(t)\in \tilde V_\beta(t,x(t))$, and $\tilde V_\beta(t,x(t))$ is compact, hence $\tilde r(t)\in \tilde V_\beta(t,x(t))$ for almost every $t\in [0,T]$.
Lemma \ref{lemcaltildeVconv} is proved.
\end{proof}

We now continue the proof of Theorem \ref{thmfilippov}. By Lemma \ref{lemcaltildeVconv}, the set $\tilde{\cal V}$ is convex and strongly closed and thus also weakly closed in $L^2([0,T],\R^{n+1})$ (see \cite{Brezis}). The sequence $(\tilde h_n)_{n\in\N}$ being bounded in $L^2([0,T],\R^{n+1})$, up to some subsequence, it converges weakly to some $\tilde h\in\tilde{\cal V}$ since this set is weakly closed.

Let us prove that $\tilde F=\tilde h$ almost everywhere. We have
\begin{equation}\label{temp1__1}
\int_0^T\varphi(t)\tilde F_n(t)\, dt = \int_0^T\varphi(t)\tilde h_n(t)\, dt +
\int_0^T\varphi(t)\left(\tilde  F_n(t)-\tilde h_n(t)\right) dt
\end{equation}
for every $\varphi\in L^2(0,T)$.
By assumption, the mappings $f$ and $f^0$ are globally Lipschitz in $x$ on $[0,T]\times \bar{B}(0,b) \times \Omega$, hence there exists $C>0$ such that $\Vert \tilde F_n(t)-\tilde h_n(t)\Vert\leq C\Vert x_n(t)-x(t)\Vert$ for almost every $t\in[0,T]$.
Since the sequence $(x_n(\cdot))_{n\in\N}$ converges pointwise to $x(\cdot)$, by the dominated convergence theorem we infer that 
$\int_0^T\varphi(t)\left(\tilde F_n(t)-\tilde h_n(t)\right) dt \rightarrow 0$
as $n\rightarrow +\infty$.
Passing to the limit in \eqref{temp1__1}, it follows that $\int_0^T\varphi(t)\tilde F(t)\, dt = \int_0^T\varphi(t)\tilde h(t)\, dt$ for every $\varphi\in L^2(0,T)$, and therefore $\tilde F=\tilde h$ almost everywhere on $[0,T]$.

In particular, $\tilde F\in\tilde {\cal V}$, and hence for almost every $t\in[0,T]$ there exist $u(t)\in\Omega$ and $\gamma(t)\geq 0$ such that
%$
%\tilde F(t)=\begin{pmatrix}
%f(t,x(t),u(t)) \\
%f^0(t,x(t),u(t))+\gamma(t)\end{pmatrix}.
%$
$
\tilde F(t)=\left( f(t,x(t),u(t)) , f^0(t,x(t),u(t))+\gamma(t)\right)^\top .
$
Applying a measurable selection lemma (noting that $\tilde F\in L^\infty([0,T],\R^{n+1})$), the functions $u(\cdot)$ and $\gamma(\cdot)$ can moreover be chosen to be measurable on $[0,T]$ (see \cite[Lem. 2A, 3A p.~161]{LeeMarkus}).

It remains to prove that the control $u$ is optimal for $\OCP$. First of all, since $x_n(t_n(u_n))\in M_1$, by compactness of $M_1$ and using the convergence properties established above, we get that $x(T)\in M_1$. Similarly, we get, clearly, that $c_i(x(\cdot))\leq 0$, $i=1,\ldots,r$. Besides, by definition $C(u_n)$ converges to $\delta$, and, using the convergence properties established above, $C(u_n)$ converges as well to $\int_0^T ( f^0(t,x(t),u(t))+\gamma(t) )\, dt+g(T,x(T))$. Since $\gamma$ takes nonnegative values, this implies that
$\int_0^T  f^0(t,x(t),u(t))\, dt +g(T,x(T)) 
\leq  \int_0^T ( f^0(t,x(t),u(t))+\gamma(t) )\, dt +g(T,x(T)) \leq C(v),
$
for every admissible control $v$ generating a trajectory steering the system from $M_0$ to $M_1$ and satisfying all constraints. 
Hence $u$ is optimal and $\gamma=0$.
Theorem \ref{thmfilippov} is proved.
\end{proof}

\section{Pontryagin maximum principle (PMP)}
\subsection{General statement}
The Pontryagin maximum principle (in short, PMP) states first-order necessary conditions for optimality. 

\begin{theorem}\label{PMP}
If $(x(\cdot),u(\cdot))$ is an optimal solution of \OCP on $[0,t_f]$,
then there exist an absolutely continuous function $p(\cdot): [0,t_f]\longrightarrow \R^n$ called adjoint vector and a real number $p^0\leq 0$, with $(p(\cdot),p^0)\neq (0,0)$, such that\footnote{With rigorous notations, we should write \eqref{systPMP} in the form $\dot x=\frac{\partial H}{\partial p}^\top$ and $\dot p=-\frac{\partial H}{\partial x}^\top$, or equivalently, $\dot x=\nabla_p H$ and $\dot p=-\nabla_x H$ (as these are vectors in $\R^n$). But we keep the writing \eqref{systPMP},  used in classical mechanics for Hamiltonian systems. In coordinates, this means that $\dot x_i=\frac{\partial H}{\partial p_i}$ and $\dot p_i=-\frac{\partial H}{\partial x_i}$ for every $i\in\{1,\ldots,n\}$.}
\begin{equation} \label{systPMP}
\dot{x}(t)=\frac{\partial H}{\partial p}(t,x(t),p(t),p^0,u(t)),\quad
\dot{p}(t)=-\frac{\partial H}{\partial x}(t,x(t),p(t),p^0,u(t)),
\end{equation}
for almost every $t\in[0,t_f]$, where the function $H:\R\times\R^n\times\R^n\times\R\times\R^m\rightarrow \R$, called Hamiltonian of \OCP, is defined by
\begin{equation}\label{def_H}
H(t,x,p,p^0,u)=\langle p,f(t,x,u)\rangle+p^0f^0(t,x,u)
\end{equation}
and we have the maximization condition
\begin{equation} \label{contraintePMP}
H(t,x(t),p(t),p^0,u(t))=\max_{v\in\Omega} H(t,x(t),p(t),p^0,v)
\end{equation}
for almost every $t\in[0,t_f]$.

If the final time $t_f$ is not fixed in \OCP, then we have moreover
\begin{equation}\label{condannul}
\max_{v\in\Omega} H(t_f,x(t_f),p(t_f),p^0,v) = -p^0\frac{\partial g}{\partial t}(t_f,x(t_f)).
\end{equation}

Moreover, the adjoint vector can be chosen such that we have the so-called transversality conditions (if they make sense)
\begin{equation}\label{condt1}
p(0)\ \bot\ T_{x(0)}M_0
\end{equation}
\begin{equation}\label{condt2}
p(t_f)-p^0\nabla_x g(t_f,x(t_f))\ \bot\ T_{x(t_f)}M_1
\end{equation}
where the notation $T_xM$ stands for the usual tangent space to $M$ at the point $x$ (these conditions can be written as soon as the tangent space is well defined).
\end{theorem}

\begin{remark}\label{rem_normalization}
%We will see in the proof of the weak PMP (see Section \ref{sec_weakPMP}) that $(p(T),p^0)=(\psi,\psi^0)$ is a Lagrange multiplier. It is defined up to scaling. Therefore, 
If $(p(\cdot),p^0)$ is a given adjoint vector satisfying the various conclusions stated in Theorem \ref{PMP}, then, for every $\lambda>0$, $(\lambda p(\cdot),\lambda p^0)$ is also an adjoint vector satisfying the statements.

Note that we cannot take $\lambda<0$ since this would lead to a change of sign in the Hamiltonian and thus would impact the maximization condition \eqref{contraintePMP}. Actually, the historical choice made by Pontryagin is to take $p^0\leq 0$ in the statement: this leads to the \textit{maximum principle} (the choice $p^0\geq 0$ is valid as well but in that case leads to a \textit{minimum principle}).
\end{remark}

A quadruple $(x(\cdot),p(\cdot),p^0,u(\cdot))$ satisfying \eqref{systPMP} and \eqref{contraintePMP} is called an \textit{extremal}.
The PMP says that every optimal trajectory $x(\cdot)$, associated with a control $u(\cdot)$, is the projection onto $\R^n$ of an extremal $(x(\cdot),p(\cdot),p^0,u(\cdot))$.
\begin{itemize}[parsep=1mm,itemsep=1mm,topsep=1mm]%,leftmargin=*
\item If $p^0 < 0$, the extremal is said to be \textit{normal}. In that case, it is usual (but not mandatory) to \emph{normalize} the adjoint vector so that $p^0=-1$.
\item If $p^0=0$, the extremal is said to be \textit{abnormal}.
\end{itemize}

The historical proof of the PMP stated in Theorem \ref{PMP} can be found in \cite{Pontryagin}. As in \cite{AgrachevSachkov, BonnardChyba, HermesLaSalle, LeeMarkus}, it is based on the use of needle-like variations combined with a Brouwer fixed point argument. It is interesting to note that there are other proofs, based on the Ekeland variational principle (see \cite{Ekeland}), on the Hahn-Banach theorem (see \cite{BressanPiccoli}). A concise sketch of proof, based on an implicit function argument (and using needle-like variations) can be found in \cite{HaberkornTrelat}.
As discussed in \cite{BourdinTrelat}, all these different approaches of proof have their specificities. One approach or another may be preferred when trying to derive a PMP in a given context (for instance, the Ekeland approach is well adapted to derive versions of the PMP with state constraints, or in infinite dimension).

\medskip

In Section \ref{sec_weakPMP}, we give a proof of the PMP in the simplified context where $\Omega=\R^m$ (no control constraint), or at least, under the assumption that the optimal control $u$ belongs to the interior of $L^\infty([0,t_f],\Omega)$. In this case, \eqref{contraintePMP} implies that
\begin{equation}\label{dHdu=0}
\frac{\partial H}{\partial u}(t,x(t),p(t),p^0,u(t))=0
\end{equation}
almost everywhere on $[0,t_f]$ and the corresponding statement is sometimes called ``weak PMP" (see \cite{BonnardChyba, Trelat}).
By the way, it is interesting to note that, in this context, a control $u$ that is the projection of an abnormal extremal ($p^0=0$) must satisfy $p(t_f)^\top dE_{x_0,t_f}(u)=0$ (see Section \ref{sec_weakPMP}), i.e., is a singularity of the end-point mapping; in other words, the linearized control system along $(x_u(\cdot),u(\cdot))$ is not controllable in time $t_f$.

In the general case where the optimal control $u$ may saturate the constraints, the proof is more difficult and requires more technical developments such as needle-like variations, or a version of the implicit function theorem under constraints.

\begin{remark}\label{rem622}
In the conditions of Theorem \ref{PMP}, we have moreover 
\begin{equation}\label{tez}
\frac{d}{dt} \max_{v\in\Omega} H(t,x(t),p(t),p^0,v) = \frac{\partial H}{\partial t} (t,x(t),p(t),p^0,u(t)) 
\end{equation}
for almost every $t\in[0,t_f]$ (this can be proved by the Danskin theorem, using the fact that $u$ takes its values in a compact subset of $\Omega$).

In particular if $\OCP$ is autonomous, that is, if $f$ and $f^0$ do not depend on $t$, then $H$ does not depend on $t$ as well, and it follows from \eqref{tez} that
$$
\max_{v\in\Omega} H(x(t),p(t),p^0,v)=\textrm{Cst} \qquad \forall t\in[0,t_f].
$$
Note that this equality is valid for every (not only for almost every) time $t\in[0,t_f]$ because the function $t\mapsto \max_{v\in\Omega} H(x(t),p(t),p^0,v)$ is Lipschitz.

Note also that, in \eqref{contraintePMP}, the maximum over $\Omega$ exists even when $\Omega$ is not compact. This is part of the result and this is due to the fact that we have assumed that there exists an optimal solution.
\end{remark}

\begin{remark}
If $g$ does not depend on $t$ then \eqref{condannul} says that, roughly, if $t_f$ is free then the (maximized) Hamiltonian vanishes at $t_f$. Note that if $\OCP$ is autonomous then this implies that $H=0$ along every extremal.
\end{remark}

\begin{remark}
If $M_1=\{x\in\R^n\ \vert\ F_1(x)=\cdots=F_p(x)=0\}$, where the functions $F_i$ are of class $C^1$ on $\R^n$, then \eqref{condt2} implies that
$$
\exists \lambda_1,\ldots,\lambda_p\in\R\ \bigm|\  p(t_f)=\sum_{i=1}^p\lambda_i\nabla F_i(x(t_f))+ p^0\nabla_x g(t_f,x(t_f)) .
$$
\end{remark}

\begin{remark}
The minimal time problem corresponds to choose either $f^0=1$ and $g=0$, or $f^0=0$ and $g(t,x)=t$. In both cases it can be checked that the implied transversality conditions coincide.
\end{remark}

\begin{remark}[On the transversality conditions]
Note that if $M_0$ (or $M_1$) is the singleton $\{x_0\}$, which means that the initial point is fixed in \OCP, then the corresponding transversality condition is empty (since the tangent space is then reduced to the singleton $\{0\}$.

At the opposite, if for instance $M_1=\R^n$, which means that the final point is free in \OCP, then the corresponding transversality condition yields that $p(t_f)=p^0\nabla_x g(t_f,x(t_f))$ (since the tangent space is then equal to $\R^n$).
In particular, in that case $p^0$ cannot be equal to $0$, for otherwise we would get $p^0=0$ and $p(t_f)=0$, which contradicts the fact that $(p(\cdot),p^0)$ is nontrivial.
\end{remark}

\subsection{Proof in a simplified context}\label{sec_weakPMP}
In this section, we consider a simplified version of $\OCP$ and we derive a weaker version of the PMP called ``weak PMP". The simplified framework is the following:
\begin{itemize}[parsep=1mm,itemsep=1mm,topsep=1mm]%,leftmargin=*
\item $M_0=\{x_0\}$ and $M_1=\{x_1\}$, where $x_0$ and $x_1$ are two given points of $\R^n$. In other words, we consider a "point to point" control problem.
\item $g=0$ in the definition of the cost \eqref{def_cost}.
\item The final time $t_f=T$ is fixed.\footnote{We are used to denote the final time by $T$ when it is fixed in $\OCP$, and by $t_f$ when it is let free.}
\end{itemize}
These three first simplifications are minor: it is easy to reduce a given optimal control problem to that case (see \cite{Trelat}). In contrast, the following one is major:
\begin{itemize}[parsep=1mm,itemsep=1mm,topsep=1mm]%,leftmargin=*
\item $\Omega=\R^m$, i.e., there are no control constraints; or, if there are some control constraints, we assume that the optimal control $u$ is in the interior of $L^\infty([0,T],\Omega)$ for the topology of $L^\infty([0,T],\R^m)$.
\end{itemize}
The latter assumption is the most important simplification. We will shortly comment further on the difficulties coming from control constraints.

\medskip

%Let us now analyze the simplified $\OCP$ and come finally out with the weak PMP. 
First of all, we note that $\OCP$ is equivalent to the optimization problem
$$
\min_{E_{x_0,T}(v)=x_1}{C_{x_0,T}(v)}
$$
where $E_{x_0,T}$ is the end-point mapping (see Definition \ref{def_Ex0T}) and $C_{x_0,T}$ is the cost defined by \ref{def_cost}.
In this form, this is a nonlinear optimization problem with $n$ equality constraints, in the infinite dimensional space of controls $v\in\mathcal{U}_{x_0,T,\R^m}\subset L^\infty([0,T],\R^m)$.

Let $u\in\mathcal{U}_{x_0,T,\R^m}$ be an optimal control (here, we assume its existence, without making any further assumption on the dynamics). In order to derive first-order necessary conditions for optimality, we apply the well known Lagrange multipliers rule, which we recover as follows. Let us consider Figure \ref{ensaccaug}, in which we draw the range of the mapping $F$ defined by
$
F(v) = \left( E_{x_0,T}(v), C_{x_0,T}(v) \right)
$
with $E_{x_0,T}(v)\in\R^n$ in abscissa and $C_{x_0,T}(v)\in\R$ in ordinate. The range of $F$ is thus seen as a subset of $\R^n\times\R$, whose shape is not important. We are interested in controls steering the system from $x_0$ to $x_1$; on the figure this corresponds to a point that is in the range of $F$, projecting onto $x_1$. Now the optimal control $u$ corresponds on Figure \ref{ensaccaug} to the point $F(u)$, which projects onto $x_1$ and is at the boundary of the range of $F$. In other words, the necessary condition for optimality is:
$$
u\ \textrm{optimal}\ \Rightarrow\ F(u)\in\partial F(L^\infty([0,T],\Omega)).
$$
Indeed, if $F(u)$ were not at the boundary of $F(L^\infty([0,T],\Omega))$ then this would imply that one can find another control, steering the system from $x_0$ to $x_1$ with a lower cost, which would contradict the optimality of $u$.

\begin{figure}[h]
\begin{center}
\resizebox{6.5cm}{!}{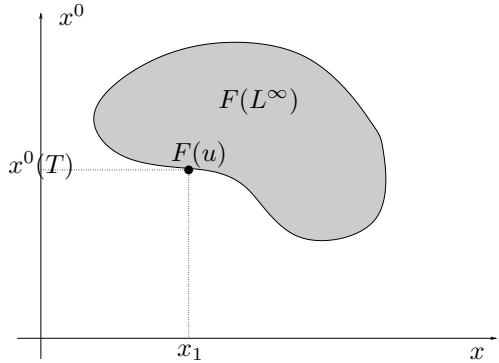}
\end{center}
\caption{Range of the mapping $F$}\label{ensaccaug}
\end{figure}

At this step, we use the important simplification $\Omega=\R^m$. Since $F(u)\in\partial F(L^\infty([0,T],\R^m))$, it follows from an implicit function argument (more precisely, the surjective mapping theorem) that %the linear continuous mapping
$
dF(u): L^\infty([0,T],\R^m)\rightarrow\R^n\times\R
$
is not surjective. Indeed, otherwise, the surjective mapping theorem would imply that $F$ be locally surjective: in other words there would exist a neighborhood of $F(u)$ in $\R^n\times\R$ contained in $F(L^\infty([0,T],\R^m))$, which would contradict the fact that $F(u)\in\partial F(L^\infty([0,T],\R^m))$.
Therefore, $\mathrm{Ran}(dF(u))$ is a proper subspace of $\R^n\times\R$.

\medskip

Note that, when there are some control constraints, the above argument works as well provided $u$ belongs to the interior of $L^\infty([0,T],\Omega)$ for the topology of $L^\infty([0,T],\R^m)$. The argument is however no more valid whenever the control saturates the constraint, that is, whenever for instance the trajectory contains some sub-arc such that $u(t)\in\partial\Omega$. At least, to make it work we would need rather to use an implicit function theorem allowing one to take into account some constraints. Here is actually the main technical difficulty that one has to deal with in order to derive the strong version of the PMP. A usual proof consists of developing needle-like variations (see \cite{Pontryagin}), but except this (important) technical point, the structure of the proof remains the same, in particular an implicit function argument can still be used (see the sketch of proof in \cite{HaberkornTrelat}).

\smallskip

Now, since $\mathrm{Ran}(dF(u))$ is a proper subspace of $\R^n\times\R$, there must exist $\tilde\psi=(\psi,\psi^0)\in\R^n\times\R\setminus\{(0,0)\}$ such that $\tilde\psi\perp\mathrm{Ran}\, dF(u)$, i.e., $\tilde\psi^\top dF(u)=0$ (here, for convenience $dF(u)$ is identified to a matrix with $n+1$ rows). 
In other words, we have obtained the usual Lagrange multipliers relation
\begin{equation}\label{lagmu}
%\langle \psi,dE_{x_0,T}(u).\delta u\rangle+\psi^0dC_{x_0,T}(u).\delta u=0 \qquad \forall \delta u\in L^\infty([0,T],\R^m).
\psi^\top dE_{x_0,T}(u)+\psi^0dC_{x_0,T}(u)=0 .
\end{equation}

Let us now exploit \eqref{lagmu} (or, more exactly, the equation $\tilde\psi^\top dF(u)=0$).
We define a new coordinate $x^0$ and consider the differential equation $\dot x^0(t)=f^0(t,x(t),u(t))$, with the initial condition $x^0(0)=0$. Therefore we have $x^0(T)=C_{x_0,T}(u)$.
We define the \textit{augmented state} $\tilde x=(x,x^0)\in\R^{n+1}$ and the \textit{augmented dynamics}
$\tilde f(t,\tilde x,v) = \begin{pmatrix}f(t,x,v)\\ f^0(t,x,v)\end{pmatrix}$.
We consider the \textit{augmented control system} in $\R^{n+1}$
\begin{equation}\label{aug}
\dot{\tilde x}(t)=\tilde f(t,\tilde x(t),v(t)).
\end{equation}
Note that $\OCP$ is then equivalent to the optimal control problem of steering the system \eqref{aug} from
$\tilde{x}_0=(x_0,0)$ to $\tilde{x}_1=(x_1,x^0(T))$ by minimizing $x^0(T)$.

Since $F(v) = (E_{x_0,T}(v),C_{x_0,T}(v))$, it follows that $F$ is the end-point mapping for the augmented control system \eqref{aug}.
Note that the range of $F$ (drawn on Figure \ref{ensaccaug}) is the accessible set $\tilde{Acc}(\tilde x_0,T)$ of the augmented control system.

\medskip

Using Proposition \ref{prop_diffFrechetE}, the (Fr\'echet) differential $dF(u):L^\infty([0,T],\R^m)\rightarrow\R^n$ is given by
$$
dF(u).\delta u = \int_0^T \tilde R(T,t)\tilde B(t)\delta u(t)\,dt \qquad \forall \delta u\in L^\infty([0,T],\R^m)
$$
where the (augmented) state transition matrix $\tilde R(\cdot,\cdot)$ is defined as the solution of the Cauchy problem $\partial_t\tilde R(t,s)=\tilde A(t)\tilde R(t,s)$, $\tilde R(s,s)=I_{n+1}$, with
$$
\tilde A(t)=\frac{\partial \tilde f}{\partial \tilde x}(t,\tilde x(t),u(t)),\quad \tilde B(t)=\frac{\partial \tilde f}{\partial u}(t,\tilde x(t),u(t)) .
$$
Since $\tilde\psi^\top dF(u).\delta u=0$ for every $\delta u\in L^\infty([0,T],\R^m)$, it follows that
\begin{equation}\label{almostpmp}
\tilde B(t)^\top \tilde R(T,t)^\top \tilde\psi  = 0
\end{equation}
for almost every $t\in[0,T]$. We set $\tilde p(t)=\tilde R(T,t)^\top\tilde\psi$. By derivating with respect to $t$ the relation $\tilde R(T,t)\tilde R(t,T)=I_{n+1}$, is is easy to establish that 
$
\frac{d}{dt}\tilde R(T,t) = -\tilde R(T,t) \tilde A(t).
$
We infer that $\tilde p(\cdot)$ is the unique solution of the Cauchy problem
\begin{equation}\label{ptilde}
\dot{\tilde p}(t) = - \tilde A(t)^\top \tilde p(t),\quad \tilde p(T)=\tilde\psi.
\end{equation}
We are almost done. Let us now come back to the initial coordinates in $\R^n\times\R$. We set $\tilde p(t)=\begin{pmatrix}p(t)\\ p^0(t)\end{pmatrix}$. Since $\tilde f$ does not depend on the (slack) variable $x^0$, we have $\frac{\partial\tilde f}{\partial x^0}=0$ and therefore, using \eqref{ptilde}, 
$$
\begin{pmatrix}\dot{p}(t)\\ \dot{p}^0(t)\end{pmatrix}
=- \begin{pmatrix}
\frac{\partial f}{\partial x}(t,x(t),u(t))^\top & \frac{\partial f^0}{\partial x}(t,x(t),u(t))^\top \\
0 & 0
\end{pmatrix}
\begin{pmatrix}p(t)\\ p^0(t)\end{pmatrix} ,
$$
with $p(T)=\psi$ and $p^0(T)=\psi^0$.
In particular, we have $\dot{p}^0(t)=0$ and thus $p^0=\psi^0$ is a constant.
Defining the Hamiltonian by \eqref{def_H}, the latter equations give \eqref{systPMP}, and from \eqref{almostpmp} we infer \eqref{dHdu=0}. This is the ``weak PMP".
%$H(t,x,p,p^0,u) = \langle p,f(t,x,u)\rangle + p^0f^0(t,x,u)$,
%we easily get that
%\begin{equation} \label{systweakPMP}
%\dot{x}(t)=\frac{\partial H}{\partial p}(t,x(t),p(t),p^0,u(t)) ,  \quad \dot{p}(t)=-\frac{\partial H}{\partial x}(t,x(t),p(t),p^0,u(t)),
%\end{equation}
%and from \eqref{almostpmp} we get that 
%$$
%\frac{\partial H}{\partial u}(t,x(t),p(t),p^0,u(t))=0
%$$
%almost everywhere on $[0,T]$.
%We have obtained what we call the \emph{weak PMP}.
%
%\medskip
%
%We do not encapsulate it into a summarized statement, because this section was just to motivate the general version of the PMP, that we state hereafter.
%Besides of the transversality conditions (that can be easily derived with simple considerations, see \cite{Trelat}), as already said the most technical part is to take into account control constraints. Then, the condition $\frac{\partial H}{\partial u}=0$ becomes a (pointwise) maximization condition of the Hamiltonian.

\begin{remark}
In the above proof, we have constructed the adjoint vector so that $(p(T),p^0)=(\psi,\psi^0)$ is a Lagrange multiplier. It is defined up to scaling (see Remark \ref{rem_normalization} and the subsequent comments).
\end{remark}

\subsection{Generalizations and additional comments}\label{sec_generalizations_PMP}
The PMP withstands many possible generalizations. %Anyway its proof (at least, the one of the weak PMP) is quite simple and can be adapted to many situations.

\medskip
\noindent\textbf{More general transversality conditions.}
In Theorem \ref{PMP}, we have given transversality conditions for "decoupled" terminal conditions $x(0)\in M_0$ and $x(t_f)\in M_1$.
Assume that, instead, we have the coupled terminal conditions $(x(0),x(t_f))\in M$, where $M$ is a subset of $\R^n\times\R^n$.
In this case, using a simple "copy-paste" of the dynamics, it is then easy to prove (see \cite{AgrachevSachkov}) that the transversality conditions become (if they make sense)
$$
\big( -p(0),p(t_f)-p^0\nabla_x g(t_f,x(t_f)) \big)\ \bot\ T_{(x(0),x(t_f))}M .
$$
An important case is the one of periodic terminal conditions $x(0)=x(T)$: then $M=\{(x,x)\ \vert\ x\in\R^n\}$, and, if moreover $g=0$ then $p(0)=p(t_f)$.

\smallskip

Another useful generalization is when $M$ is a general closed subset of $\R^n\times\R^n$, but is not necessarily a manifold, at least, locally around $(x(0),x(t_f))$. In this case, one can still write transversality conditions, by using notions of nonsmooth analysis (see \cite{Clarke,Vinter}), and there holds
$$
\big( p(0),-p(t_f)+p^0\nabla_x g(t_f,x(t_f)) \big) \in N_M(x(0),x(t_f))
$$
where $N_M(x,y)$ is the {\em limiting normal cone} to $M$ at $(x,y)$. This generalized condition can be useful to provide sign conditions on the adjoint vector, whenever the subset $M$ is not smooth.

\medskip
\noindent\textbf{Infinite time horizon.}
The statement of the PMP remains the same when $t_f=+\infty$, under the assumption that the limit of the optimal trajectory $x(t)$ exists when $t\rightarrow+\infty$; in particular, the result then asserts that the limit of $p(t)$ exists (see \cite{LeeMarkus}).

\medskip
\noindent\textbf{State constraints, hybrid optimal control problems.}
Among the most well known and useful generalizations, one can think of the PMP for $\OCP$ with state constraints (see \cite{BrysonHo,Clarke,Pontryagin,Vinter}), for nonsmooth $\OCP$ (see \cite{Clarke,Vinter}), hybrid $\OCP$ (see \cite{GaravelloPiccoli,HaberkornTrelat}), $\OCP$ settled on time scales (see \cite{BourdinTrelat}). There exist several possible proofs of the PMP (see \cite{Dmitruk}), based either on an implicit function argument (as we did here), or on a (Brouwer) fixed point argument (as in the classical book \cite{Pontryagin}), or on a Hahn-Banach separation argument (as in \cite{BressanPiccoli}, or on Ekeland's principle (see \cite{Ekeland}). Each of them may or may not be adapted to such or such generalization.

Let us note that, when dealing with state constraints, in full generality the adjoint vector becomes a measure. The generic situation that one has in mind is the case where this measure has only a finite number of atoms: in this favorable case the adjoint vector is then piecewise absolutely continuous, with possible jumps when touching the state constraint. Unfortunately the structure of the measure might be much more complicated, but such a discussion is outside of the scope of the present manuscript. We refer the reader to \cite{BonnardFaubourgTrelat, BrysonHo, Clarke, HartlSethi, IoffeThikhomirov, Jacobson, Maurer, Pontryagin, Vinter}.

Although it is then not exact, it can be noted that state constraints may be tackled with usual penalization considerations, so as to deal rather with an $\OCP$ without state constraint. In some cases where getting the true optimal trajectory is not the main objective, this may be useful.

\medskip
\noindent\textbf{PMP in infinite dimension.}
We refer to Section \ref{sec_PMP_infinitedim} in Part \ref{part2} and to \cite[Chapter 4]{LiYong} for a generalization of the PMP in infinite dimension. To comment briefly on this extension, we notice that the argument to prove the weak PMP remains valid when replacing $\R^n$ with a Banach space $X$, at the exception of one notable difficulty: in the argument by contraposition, we have to ensure that $\mathrm{Ran}(dF(u))$ is contained in a closed proper subspace of $X\times\R$, i.e., its codimension is $\geq 1$. Here is the main difference with finite dimension. Indeed, the fact that $\mathrm{Ran}(dF(u))$ is a proper subspace of $X\times\R$ is not enough to ensure a separation argument: it could happen that $\mathrm{Ran}(dF(u))\subsetneq X\times\R$ be dense in $X\times\R$. Such a situation corresponds to \emph{approximate controllability}, as we will see in Part \ref{part2}, and in this case the PMP fails to be true (see Example \ref{contreex_PMP} in Section \ref{sec_PMP_infinitedim}).
In few words, a classical sufficient assumption under which the PMP is still valid in infinite dimension is that there is only a finite number of scalar conditions on the final state (finite codimension condition on $M_1$). 
%We refer to \cite[Chapter 4]{LiYong} for a general statement, noting that, in that book, the PMP is proved thanks to the Ekeland variational principle.
Except this additional assumption on $M_1$, the statement in infinite dimension is exactly the same as in Theorem \ref{PMP}.

\medskip
\noindent\textbf{Further comments: second-order conditions.}
Let us insist on the fact that the PMP is a \emph{first-order necessary condition} for optimality.\footnote{This is an elaborated version of the first-order necessary condition $\nabla f(x)=0$ when minimizing a $C^1$ function over $\R^n$!}
As already stressed, the PMP states that every optimal trajectory $x(\cdot)$, associated with a control $u(\cdot)$, is the projection onto $\R^n$ of an extremal $(x(\cdot),p(\cdot),p^0,u(\cdot))$.
However, conversely, an extremal (i.e., a solution of the equations of the PMP) is not necessarily optimal. The study of the optimality status of extremals can be done with the theory of \emph{conjugate points}. More precisely, as in classical optimization where extremal points are characterized by a first-order necessary condition (vanishing of some appropriate derivative), there exists in optimal control a theory of second-order conditions for optimality, which consists of investigating a quadratic form that is the intrinsic second-order derivative of the end-point mapping: if this quadratic form is positive definite then this means that the extremal under consideration is locally optimal (for some appropriate topology), and if it is indefinite then the extremal is not optimal; conversely if the extremal is optimal then this quadratic form is nonnegative. Times at which the index of this quadratic form changes are called conjugate times. The optimality status of an extremal is then characterized by its first conjugate time. We refer to \cite{BonnardCaillauTrelat_COCV2007} (see references therein) for a survey on theory and algorithms to compute conjugate times. Much could be written on conjugate time theory (which has nice extensions in the bang-bang case), but this is beyond of the scope of the present book.

\medskip
\noindent\textbf{Numerical computation.}
The application of the PMP leads to a \textit{shooting problem}, that is a boundary value problem consisting of computing extremals satisfying certain terminal conditions. It can be solved numerically by implementing a Newton method combined with an ODE integrator: this approach is called the \emph{shooting method}.

We do not have room enough, here, to describe numerical methods in optimal control. We refer to \cite{Betts} for a thorough description of so-called direct methods, and to \cite{Trelat_JOTA} for a survey on indirect methods and the way to implement them in practice (see also \cite{Trelat} and references cited therein).

\section{Particular cases and examples}
In this section, we first specify the PMP for two important particular classes of examples: the minimal time problem for linear control systems, which yields the bang-bang principle; linear control systems with a quadratic cost, leading to the well-known ``Linear Quadratic" (LQ) theory.
Finally, we provide several examples of application of the PMP to nonlinear optimal control problems.

\subsection{Minimal time problem for linear control systems}\label{sec241}
Let us assume that the control system is linear, of the form
$$
\dot{x}(t)=A(t)x(t)+B(t)u(t)+r(t)
$$
with the notations and regularity assumptions made in the introduction of Chapter \ref{chap_cont}. Let $x_0\in\R^n$ be an arbitrary initial point, and let $\Omega$ be a compact subset of $\R^m$. For any target point $x_1\in\R^n$, we investigate the problem of steering the system from $x_0$ to $x_1$ in minimal time, under the control constraint $u(t)\in\Omega$.

It can be noticed that, if $x_1$ is accessible from $x_0$, then there exists a minimal time trajectory steering the system from $x_0$ to $x_1$, in a minimal time denoted by $t_f$.
Indeed, by Theorem \ref{access}, $\mathrm{Acc}_\Omega(x_0,t)$ is a compact convex set depending continuously on $t$, thus 
$
t_f = \min \{ t\geq 0\ \mid\ x_1\in\mathrm{Acc}_\Omega(x_0,t) \}.
$

With the notations introduced at the beginning of Chapter \ref{chap_opt}, we have $f(t,x,u)=A(t)x+B(t)u+r(t)$, $f^0(t,x,u)=1$ and $g=0$ (note that we could as well take $f^0=0$ and $g(t,x)=t$).
The Hamiltonian of the optimal control problem is then
$
H(t,x,p,p^0,u) = p^\top A(t)x+ p^\top B(t)u+ p^\top r(t) +p^0 .
$

Let $(x(\cdot),u(\cdot))$ be an optimal trajectory on $[0,t_f]$. According to the PMP, there exist $p^0\leq 0$ and an absolutely continuous mapping $p(\cdot):[0,t_f]\rightarrow\R^n$, with $(p(\cdot),p^0)\neq (0,0)$, such that $\dot{p}(t)=-A(t)^\top p(t)$ for almost every $t\in[0,t_f]$, and the maximization condition yields
\begin{equation}\label{condmaxlineaire}
\langle B(t)^\top p(t),u(t)\rangle=\max_{v\in\Omega}\langle B(t)^\top p(t),v \rangle
\end{equation}
for almost every $t\in[0,t_f]$.

Since the function $v\mapsto \langle B(t)^\top p(t),v\rangle=p(t)^\top B(t)v$ is linear, we expect that the maximum over $\Omega$ be reached at the boundary of $\Omega$ (unless $B(t)^\top p(t)=0$). This is the contents of the \textit{bang-bang principle}.

Let us consider particular but important cases.

\paragraph{Case $m=1$ (scalar control).} 
Let us assume that $m=1$, and that $\Omega=[-a,a]$ with $a>0$. This means that the control must satisfy the constraints $\vert u(t)\vert\leq a$. In that case, $B(t)$ is a vector of $\R^n$, and $\varphi(t)=p(t)^\top B(t)$ is called \textit{switching function}. The maximization condition \eqref{condmaxlineaire} implies that 
$$
u(t)=a\,\mathrm{sign}(\varphi(t)) 
$$
as soon as $\varphi(t)\neq 0$.
Here, we see that the structure of the optimal control $u$ is governed by the switching function. We say that the control is \textit{bang-bang} if the switching function $\varphi$ does not vanish identically on any subset of positive measure of $[0,t_f]$. For instance, this is the case under the assumption: $\varphi(t)=0\Rightarrow\dot\varphi(t)\neq 0$ (because then the zeros of $\varphi$ are isolated). In that case, the zeros of the switching functions are called the switchings of the optimal control. According to the monotonicity of $\varphi$, we see that the optimal control switches between the two values $\pm a$. This is the typical situation of a bang-bang control.

In contrast, if the switching function $\varphi$ vanishes, for instance, along a time subinterval $I$ of $[0,t_f]$, then the maximization condition \eqref{condmaxlineaire} does not provide any immediate information in order to compute the optimal control $u$.
But we can then differentiate with respect to time (if this is allowed) the relation $p(t)^\top B(t)=0$, and try to recover some useful information. This is a usual method in order to prove by contradiction, when it is possible, that optimal controls are bang-bang.
An important example where this argument is successful is the following result, that we let as an exercise (see \cite{LeeMarkus}).

\begin{lemma}
Let us assume that $A(t)\equiv A$, $B(t)\equiv B$, $r(t)\equiv 0$ (autonomous control system), and that the pair $(A,B)$ satisfies the Kalman condition. Then any extremal control is bang-bang, and
\begin{itemize}[parsep=1mm,itemsep=1mm,topsep=1mm]%,leftmargin=*[leftmargin=*]
\item has at most $n-1$ switchings on $[0,+\infty)$ if all eigenvalues of $A$ are real;
\item has an infinite number of switchings on $[0,+\infty)$ if all eigenvalues of $A$ have a nonzero imaginary part. In this case, for every $N\in\N^*$, there exists $x_0\in\R^n$ for which the corresponding minimal time control, steering $x_0$ to $0$, has more than $N$ switchings.
\end{itemize}
\end{lemma}

\paragraph{Case $m=2$ (two scalar controls $u_1$ and $u_2$).} 
Let us assume that $m=2$. In that case, $B(t)=(B_1(t),B_2(t))$, where $B_1(t)$ and $B_2(t)$ are vectors of $\R^n$.
Let us show how to make explicit the extremal controls from the maximization condition of the PMP, for two important constraints very often considered in practice.

\begin{itemize}[parsep=1mm,itemsep=1mm,topsep=1mm]%,leftmargin=*[leftmargin=*]
\item Assume that $\Omega=[-1,1]\times[-1,1]$, the unit square of $\R^2$. This means that the controls $u_1$ and $u_2$ must satisfy the constraints $\vert u_1(t)\vert\leq 1$ and $\vert u_2(t)\vert\leq 1$. As for the case $m=1$, we set $\varphi_i(t)=p(t)^\top B_i(t)$, for $i=1,2$, and the maximization condition \eqref{condmaxlineaire} implies that $u_i(t)=\mathrm{sign}(\varphi_i(t))$ as soon as $\varphi_i(t)\neq 0$, for $i=1,2$.

\item Assume that $\Omega=\bar B(0,1)$, the closed unit ball of $\R^2$. This means that the controls $u_1$ and $u_2$ must satisfy the constraints $u_1(t)^2+u_2(t)^2\leq 1$. Setting again $\varphi_i(t)=p(t)^\top B_i(t)$, for $i=1,2$, the maximization condition \eqref{condmaxlineaire} can be written as
$$
\varphi_1(t) u_1(t) + \varphi_2(t) u_2(t) = \max_{v_1^2+v_2^2\leq 1}\left\langle \begin{pmatrix} \varphi_1(t)\\ \varphi_2(t)\end{pmatrix} , \begin{pmatrix} v_1\\ v_2\end{pmatrix}  \right\rangle
$$
and it follows from the Cauchy-Schwarz inequality that
$$
u_i(t) = \frac{\varphi_i(t)}{\sqrt{\varphi_1(t)^2+\varphi_2(t)^2}}
$$
for $i=1,2$, as soon as $\varphi_1(t)^2+\varphi_2(t)^2\neq 0$.
\end{itemize}
In these two cases, the comments done previously are still in force in the degenerate case where the switching functions vanish identically on some subset of positive measure.
We do not insist on such difficulties at this step.
Note that what is done here with $m=2$ could be written as well for any value of $m$.

\subsection{Linear quadratic theory}\label{sec_LQ}
In this chapter we make an introduction to the well known LQ (linear quadratic) theory, which has many applications in concrete applications, such as Kalman filtering or regulation problems. We first study and solve the basic LQ problem, and then we provide an important application to the tracking problem. For other applications (among which the Kalman filter), see \cite{AndersonMoore, Kailath, KwakernaakSivan, Sontag, Trelat}.

\subsubsection{The basic LQ problem}
We consider the optimal control problem
\begin{equation}\label{systemLQ}
\begin{split}
& \dot{x}(t)=A(t)x(t)+B(t)u(t),\quad x(0)=x_0 \\
& \min \int_0^T \left({x(t)}^\top W(t)x(t)+{u(t)}^\top U(t)u(t) \right) dt +{x(T)}^\top Qx(T)
\end{split}
\end{equation}
where $x_0\in\R^n$ and $T>0$ are fixed (arbitrarily), $W(t)$ and $Q$ are symmetric nonnegative matrices of size $n$, $U(t)$ is a symmetric positive definite matrix of size $m$. The dependence in time of the matrices above is assumed to be $L^\infty$ on $[0,T]$.
The controls are all possible functions of $L^2([0,T],\R^m)$.

We call this problem the basic LQ problem. Note that the final point is let free. The matrices $W(t)$, $U(t)$, and $Q$, are called weight matrices.

We assume that there exists $\alpha>0$ such that
$$
\int_0^T u(t)^\top U(t) u(t) \, dt \geq\alpha\int_0^T {u(t)}^\top u(t) \, dt \qquad\forall u\in L^2([0,T],\R^m).
$$
For instance, this assumption is satisfied if $t\mapsto
U(t)$ is continuous on $[0,T]$. In practice, the weight matrices are often constant.

\begin{theorem}\label{existenceLQ}
There exists a unique optimal solution to \eqref{systemLQ}.
\end{theorem}

This theorem can be proved using classical functional analysis arguments, as in the proof of Theorem \ref{thmfilippov}. The uniqueness comes from the strict convexity of the cost. For a proof, see \cite{Trelat}.

Let us apply the PMP to the basic LQ problem. The Hamiltonian is
$$
H(t,x,p,p^0,u) = p^\top A(t)x+p^\top B(t)u+p^0({x}^\top W(t)x+{u}^\top U(t)u)
$$
and the adjoint equation is 
$
\dot p(t) = -A(t)^\top p(t)-2p^0W(t)x(t) .
$
Since the final point is free, the transversality condition on the final adjoint vector yields $p(T)=2p^0Qx(T)$, and hence necessarily $p^0\neq 0$ (otherwise we would have $(p(T),p^0)=(0,0)$, which is a contradiction with the PMP). According to Remark \ref{rem_normalization}, we choose, here, to normalize the adjoint vector so that $p^0=-\frac{1}{2}$ (this will be convenient when derivating the squares...). Now, since there is no constraint on the control, we have 
$$
0 = \frac{\partial H}{\partial u}(t,x(t),p(t),p^0,u(t)) = B(t)^\top p(t) - U(t)u(t)
$$
and hence $u(t)=U(t)^{-1}B(t)^\top p(t)$.

Summing up, we have obtained that, for the optimal solution $(x(\cdot),u(\cdot))$ of the basic LQ problem, we have 
\begin{equation*}
\begin{split}
\dot x(t) &= A(t) x(t) + B(t) U(t)^{-1}B(t)^\top p(t),\qquad\, x(0)=x_0,\\
\dot p(t) &= -A(t)^\top p(t) + W(t) x(t),\qquad\qquad\qquad p(T) = -Q x(T).
\end{split}
\end{equation*}
At this step, things are already nice, and we could implement a shooting method in order to solve the above problem. But the structure of the problem is so particular that we are actually able, here, to express $p(t)$ linearly in function of $x(t)$. This property is very remarkable but also very specific to that problem. We claim that we can search $p(t)$ in the form $p(t) = E(t)x(t)$. Replacing in the above equations, we easily obtain a relation of the form $R(t)x(t)=0$, with 
$$
R(t) = \dot{E}(t)-W(t)+{A(t)}^\top E(t)+E(t)A(t)+E(t)B(t)U(t)^{-1}{B(t)}^\top E(t)
$$
and $E(T)x(T)=-Qx(T)$. Therefore, ``simplifying by $x$", we see that, if we assume that $R(t)=0$ by definition, then we can go back in the reasoning and infer, by Cauchy uniqueness, that $p(t) = E(t)x(t)$. We have obtained the following result.

\begin{theorem} \label{thmLQ}
The optimal solution $x(\cdot)$ of the basic LQ problem is associated with the control
$$
u(t)=U(t)^{-1} B(t)^\top E(t)x(t)
$$
where $E(t)\in{\cal M}_n(\R)$ is the unique solution on $[0,T]$ of the Riccati matrix differential equation
\begin{equation}\label{riccati}
\begin{split}
\dot{E}(t)&=W(t)-{A(t)}^\top E(t)-E(t)A(t)- E(t)B(t)U(t)^{-1}{B(t)}^\top E(t)\\
E(T)&=-Q 
\end{split}
\end{equation}
\end{theorem}

Actually, there is a difficulty for finishing the proof of that theorem, by proving that the unique solution $E(\cdot)$ of the Cauchy problem \eqref{riccati}, which is a priori defined in a neighborhood of $T$, is indeed well defined over the whole interval $[0,T]$. Indeed, such a Riccati differential equation may produce blow-up, and the well-posedness over the whole $[0,T]$ is not obvious. We do not prove that fact here, and we refer the reader, e.g., to \cite{Trelat} for a proof (which uses the optimality property of $u(\cdot)$).

It can be noted that $E(t)$ is symmetric (this is easy to see by Cauchy uniqueness).
The result above is interesting because, in that way, the problem is completely solved, without having to compute any adjoint vector for instance by means of a shooting method. Moreover the optimal control is expressed in a feedback form, $u(t)=K(t)x(t)$, well adapted to robustness issues.

This is because of that property that the LQ procedures are so much used in practical problems and industrial issues. We are next going to give an application to tracking.

\subsubsection{Tracking problem}
Let us consider the general control system $\dot x(t) = f(t,x(t),u(t))$, with initial condition $x(0)=x_0$, with the regularity assumptions done at the beginning of Chapter \ref{chap_cont}.
Let $t\mapsto\xi(t)$ be a trajectory in $\R^n$, defined on $[0,T]$, and which is arbitrary (in particular, this is not necessarily a solution of the control system). We assume however that $\xi(\cdot)$ is Lipschitz (or, at least, absolutely continuous).
The objective is to design a control $u$ generating a trajectory $x(\cdot)$ that tracks the trajectory $\xi(\cdot)$ in the ``best possible way" (see Figure \ref{figreg}).

\begin{figure}[h]
\begin{center}
\includegraphics[width=5.7cm]{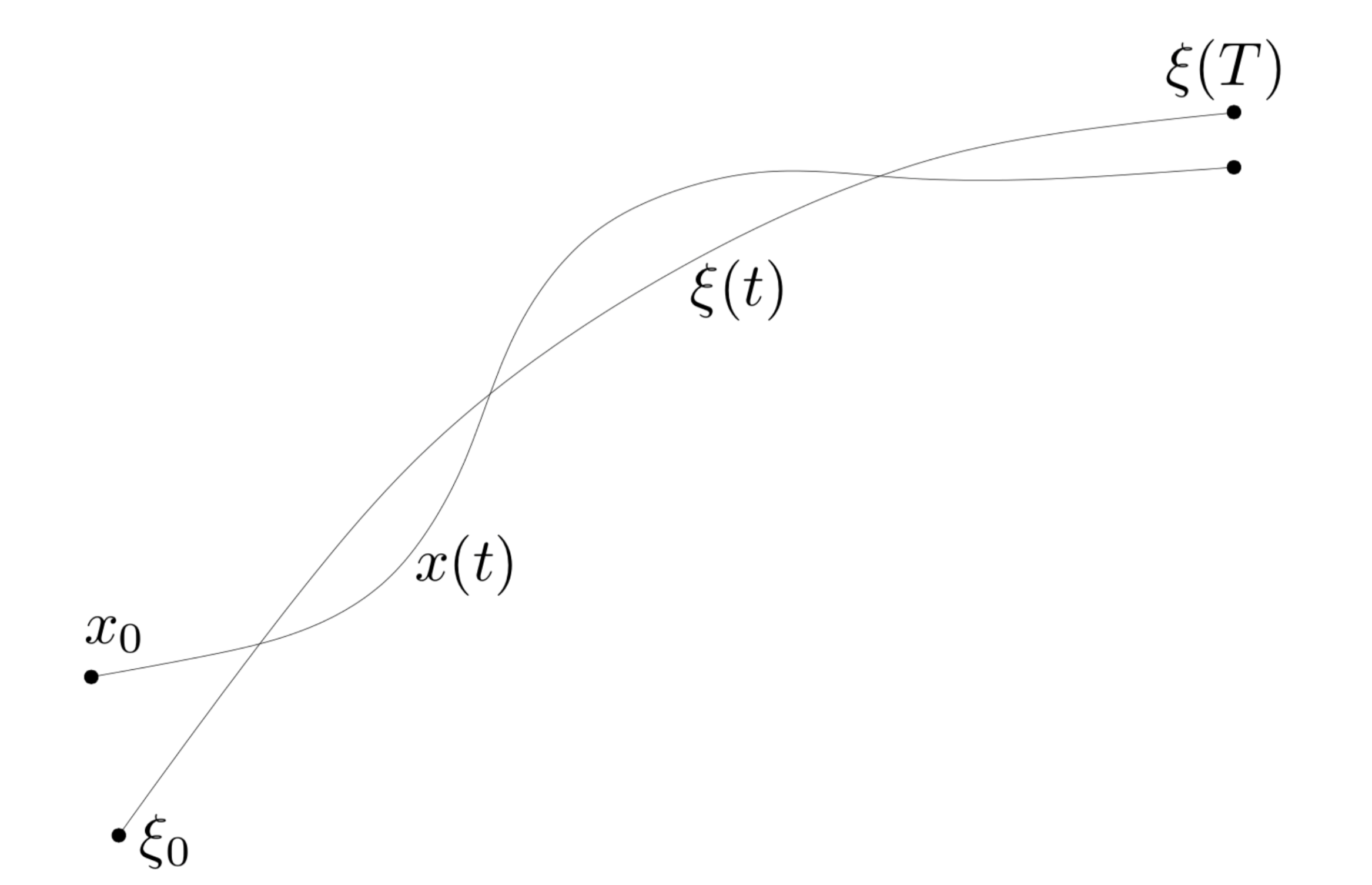}
\end{center}
\caption{Tracking problem}\label{figreg}
\end{figure}

We proceed as follows. 
We set $e(t)=x(t)-\xi(t)$, and we will try to design $u$ so that $e(\cdot)$ remains as small as possible. Using a first-order expansion, we have
$$
\dot e(t) = f(t,\xi(t)+e(t),u(t))-\dot\xi(t)
= A(t)e(t) + B(t)u(t) + r(t)
$$
with 
$$
A(t) = \frac{\partial f}{\partial x}(t,\xi(t),0), \qquad
B(t) = \frac{\partial f}{\partial u}(t,\xi(t),0), 
$$
$$
r(t) = f(t,\xi(t),0)-\dot\xi(t) + \mathrm{o}(e(t),u(t)) .
$$
It seems reasonable to seek a control $u$ minimizing the cost
$$
C(u) = \int_0^T\left( z(t)^\top W(t) z(t) + u(t)^\top U(t) u(t)\right) dt + z(T)^\top Qz(T)
$$
for the control system
$$
\dot z(t) = A(t)z(t) + B(t)u(t) + r_1(t),\quad z(0) = x_0-\xi(0),
$$
with $r_1(t)=f(t,\xi(t),0)-\dot\xi(t)$, where $W(t)$, $U(t)$ and $Q$ are weight matrices that are chosen by the user. We hope that the resulting control will be such that the term $\mathrm{o}(z(t),u(t))$ is small. In any case, the choice above produces a control, which hopefully tracks the trajectory $\xi(t)$ as closely as possible. Note that, when linearizing the system, we have linearized at $u=0$, considering that $u$ will be small. We could have linearized along a given $\bar u(t)$: we then obtain one of the many possible variants of the method.

Let us now solve the above optimal control problem. In order to absorb the perturbation term $r_1$, we consider an augmented system, by adding one dimension. We set
$$
z_1=\begin{pmatrix} z \\ 1 \end{pmatrix},\ 
A_1=\begin{pmatrix} A & r_1 \\ 0 & 0 \end{pmatrix},\ 
B_1=\begin{pmatrix} B \\ 0 \end{pmatrix} ,\ 
Q_1=\begin{pmatrix} Q & 0 \\ 0 & 0 \end{pmatrix},\ 
W_1=\begin{pmatrix} W & 0 \\ 0 & 0 \end{pmatrix}, 
$$
and hence we want to minimize the cost
$$
C(u) = \int_0^T\left( z_1(t)^\top W_1(t) z_1(t) + u(t)^\top U(t) u(t) \right) dt + z_1(T)^\top Q_1z_1(T)
$$
for the control system $\dot{z}_1(t)=A_1(t)z_1(t)+B_1(t)u(t)$, with $z_1(0)$ fixed.
In this form, this is a basic LQ problem, as studied in the previous section. 
According to Theorem \ref{thmLQ}, there exists a unique optimal control, given by $u(t)=U(t)^{-1}B_1(t)^\top E_1(t)z_1(t)$, where $E_1(t)$ is the solution of the Cauchy problem $\dot{E}_1=W_1-{A_1}^\top E_1-E_1A_1-E_1B_1U^{-1}{B_1}^\top E_1$, $E_1(T)=-Q_1$.
Setting
$$
E_1(t)=\begin{pmatrix} E(t) & h(t) \\ h(t)^\top & \alpha(t) \end{pmatrix} 
$$
with $E(t)$ square matrix of size $n$, $h(t)\in \R^n$ and $\alpha(t)\in\R$, we obtain the following result.

\begin{proposition}
The optimal (in the sense above) tracking control is
$$
u(t)=U(t)^{-1}B(t)^\top E(t)(x(t)-\xi(t)) + U(t)^{-1}B(t)^\top h(t)
$$
where
$$
\begin{array}{rll}
\dot{E}&=W-{A}^\top E-EA-EBU^{-1}{B}^\top E, & E(T)=-Q, \\
\dot{h}&=-{A}^\top h-E(f(t,\xi,0) -\dot{\xi})-EBU^{-1}{B}^\top h, & h(T)=0 .
\end{array}
$$
\end{proposition}

It is interesting to note that the control is written in a feedback form $u(t)=K(t)(x(t)-\xi(t))+H(t)$.

\medskip

As said at the beginning of the section, there are many other applications of the LQ theory, and many possible variants. For instance, one can easily adapt the above tracking procedure to the problem of output tracking: in that case we track an observable. It is also very interesting to let the horizon of time $T$ go to $+\infty$. In that case, we can expect to obtain stabilization results. This is indeed the case for instance when one considers a linear autonomous control system (regulation over an infinite horizon); the procedure is referred to as LQR in practice and is very much used for stabilization issues.

In practice, we often make the choice of constant diagonal weight matrices $W(t)=w_0\, I_n$, $U(t)=u_0\, I_m$, and $Q=q_0\, I_n$, with $w_0\geq 0$, $u_0>0$ and $q_0\geq 0$. If $w_0$ is chosen much larger than $u_0$, then it is expected that $\Vert x(t)-\xi(t)\Vert$ will remain small (while paying the price of larger values of $u(t)$). Conversely if $u_0$ is chosen much larger than $w_0$ then it is expected that $u(t)$ will take small values, while the tracking error $\Vert x(t)-\xi(t)\Vert$ may take large values. Similarly, if $q_0$ is taken very large then it is expected (at least, under appropriate controllability assumptions) that $x(T)$ will be close to $\xi(T)$. A lot of such statements, with numerous possible variants, may be established. We refer to \cite{AndersonMoore, Kailath, Khalil, KwakernaakSivan, LeeMarkus, Sontag, Trelat} for (many) more precise results.

\subsection{Examples of nonlinear optimal control problems}

\begin{example}[Zermelo problem]
Let us consider a boat moving with constant speed along a river of constant width $\ell$, in which there is a current $c(y)$ (assumed to be a function of class $C^1$). The movement of the center of mass of the boat is governed by the control system
\begin{equation*}
\begin{split}
\dot{x}(t)&=v\cos u(t)+c(y(t)),\quad\ \ x(0)=0,\\
\dot{y}(t)&=v\sin u(t),\qquad\qquad\qquad y(0)=0,
\end{split}
\end{equation*}
where $v>0$ is the constant speed, and the control is the angle $u(t)$ of the axis of the boat with respect to the axis $(0x)$ (see Figure \ref{fig_zermelo}). We investigate three variants of optimal control problems with the objective of reaching the opposite side: the final condition is $y(t_f)=\ell$, where the final time $t_f$ is free.
\begin{figure}[h]
\begin{center}
{ \resizebox{7cm}{!}{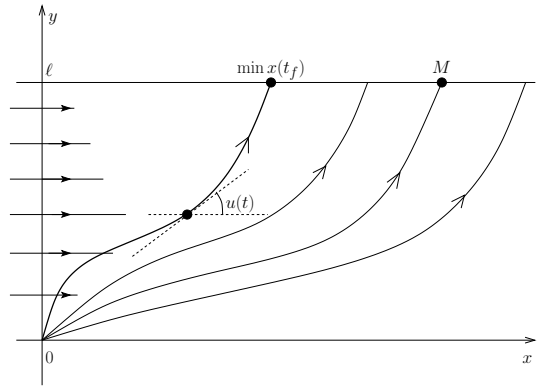} }
\end{center}
\caption{Zermelo problem.}\label{fig_zermelo}
\end{figure}
\begin{enumerate}
\item[1.] Assuming that $c(y)\geq v$ for every $y\in[0,\ell]$ (strong current), compute the optimal control minimizing the drift $x(t_f)$.
\item[2.]  Compute the minimal time control.
\item[3.]  Compute the minimal time control for the problem of reaching a precise point $M=(x_1,\ell)$ of the opposite side.
\end{enumerate}

\medskip

\noindent\textit{1. Reaching the opposite side by minimizing the drift $x(t_f)$.}\\[1mm]
We choose $f^0=0$ and $g(t,x,y)=x$.
The Hamiltonian is 
$$
H = p_x(v\cos u+c(y))+p_yv\sin u .
$$
The adjoint equations are $\dot p_x=0$, $\dot p_y=-p_xc'(y)$. In particular $p_x$ is constant. Since the target is $M_1=\{y=\ell\}$, the transversality condition on the adjoint vector yields $p_x=p^0$.
The maximization condition of the Hamiltonian leads to
$$
\cos u(t) = \frac{p_x}{\sqrt{p_x^2+p_y(t)^2}},\qquad
\sin u(t) = \frac{p_y(t)}{\sqrt{p_x^2+p_y(t)^2}},
$$
for almost every $t$, provided that the function $\varphi(t)=p_x^2+p_y(t)^2$ does not vanish on any subset of positive measure. This condition is proved by contradiction: if $\varphi(t)\equiv 0$ on $I$, then $p_x=0$ and $p_y(t)=0$ on $I$, but then also $p^0=p_x=0$, and we get a contradiction (because the adjoint $(p_x,p_y,p^0)$ must be non trivial).
Finally, since $t_f$ is free and the problem is autonomous, we get that $H=0$ along any extremal, that is, $H=v\sqrt{p_x^2+p_y^2}+p_xc(y)=0$ along any extremal.

We must have $p^0\neq 0$. Indeed, otherwise, $p^0=0$ implies that $p_x=0$, and from $H=0$ we infer that $p_y=0$ as well. This is a contradiction.
Hence, we can take $p^0=-1$, and therefore $p_x=-1$.

From $H=0$, we have $\sqrt{1+p_y^2}=\frac{c(y)}{v}$, and hence $\cos u=-\frac{v}{c(y)}$. Since $c(y)\geq v$, this equation is solvable, and we get
$$u(t)=\textrm{Arccos}\left( -\frac{v}{c(y(t))} \right).$$
Note that we have thus determined the optimal control in a \textit{feedback form}, which is the best possible one (in practice, such a control can be made fully automatic, provided one can measure the position $y$ at any time).

\begin{remark}
The assumption $c(y)\geq v$ means that the current is strong enough. Without this assumption, the optimal control problem consisting of minimizing the drift $x(t_f)$ would be ill-posed: there would not exist any optimal solution (at least, in finite time), because if, for some $y$, we have $c(y)<v$, then, along this $y$, the boat can go against the current towards $x=-\infty$.

We also realize that, if we had not made this assumption, then the above equation would not be solvable. This remark provides a way for showing that the optimal control problem has no solution, by contradiction (recall that the PMP says that \textbf{if} a trajectory is optimal \textbf{then} it must satisfy the various conditions stated in the PMP).
\end{remark}

\begin{remark}\label{rem_zermelo_reachable}
The optimal trajectory, minimizing the lateral deport, is represented on Figure \ref{fig_zermelo}. It is interesting to note that any other trajectory is necessarily at the right of that optimal trajectory. In particular, this gives the reachable set (in any time): the reachable set consists of all points such that $0\leq y\leq \ell$ that are at the right of the optimal trajectory.
\end{remark}

\medskip

\noindent\textit{2. Reaching the opposite side in minimal time.}\\[1mm]
We choose $f^0=1$ and $g=0$.
Now, the Hamiltonian is 
$$
H = p_x(v\cos u+c(y))+p_yv\sin u+p^0,
$$
the adjoint equations are the same as previously, as well as the extremal controls provided that $\varphi(t)\neq 0$. The transversality condition on the adjoint vector is different: we now obtain $p_x=0$. It follows that $p_y$ is constant.
Besides, we still have $H=0$ along any extremal. We claim that $p_y\neq 0$. Indeed, otherwise, $H=0$ implies that $p^0=0$, and then $(p_x,p_y,p^0)=(0,0,0)$, which is a contradiction. Hence, we get that $\cos u(t)=0$ and $\sin u(t)=\textrm{sign}(p_y)$, and thus $u(t)=\frac{\pi}{2}$ for every time $t$ (the sign of $u$ is given by the fact that, at least at the beginning, the boat must leave the initial riverbank with $\dot y>0$).

\begin{remark}
Actually, the fact that the minimal time control is $u=\frac{\pi}{2}$ is obvious by inspecting the equation in $y$. Indeed, since $t_f=\int_0^{t_f}dt=\int_0^\ell \frac{dy}{\dot y}$, it easily follows that, in minimal time, we must have $\dot y=1$.

Note that we do not need to assume, here, that the current is strong enough. The calculations above are valid, whatever the function $c(y)$ may be.
\end{remark}

\medskip

\noindent\textit{3. Reaching a precise point of the opposite side in minimal time.}\\[1mm]
This is a variant of the second problem, in which we lose the transversality condition on the adjoint vector, because the target point is fixed. The constant $p_x$ is thus undetermined at this step. We still have that $H=0$ along any extremal.
By contradiction, it is still true that the function $\varphi$ cannot vanish on any subset of positive measure (indeed otherwise $p_x=0$ and $p_y=0$, and from $H=0$ we get that $p^0=0$: contradiction).

This third variant is interesting because both the normal case $p^0\neq 0$ and the abnormal case $p^0=0$ may occur. Let us analyze them.
\begin{itemize}[parsep=1mm,itemsep=1mm,topsep=1mm]%,leftmargin=*]
\item Normal case: $p^0=-1$. In that case, using that $H=v\sqrt{p_x^2+p_y^2}+p_xc(y)-1=0$ along any extremal, we get  
$$
\cos u(t)= \frac{p_xv}{1-p_xc(y(t))} .
$$
Note that, for this equation to be solvable, $p_x$ must be such that $\vert p_x v\vert\leq \vert 1-p_xc(y(t))\vert$ for every time $t$. We have thus obtained all possible optimal controls, parametrized by $p_x$. The parameter $p_x$ is the ``shooting parameter", that is, the degree of freedom that is required in order to tune the final condition $x(t_f)=x_1$, i.e., in order to reach the target point $M$.

All optimal trajectories are represented on Figure \ref{fig_zermelo} in the case of a strong current. To go further, we could specify a current function $c(y)$, and either implement a shooting method, or try to make explicit computations if this is possible.

\item Abnormal case: $p^0=0$. Using $H=0$, we get $\cos u=-\frac{v}{c(y)}$. In the case where the current is strong enough, we see that we recover exactly the solution of the first variant, that is, the optimal trajectory with minimal drift.

Then, two cases may occur: either the target point $M$ is different from the end-point of the minimal drift trajectory, and then, the trajectory is not solution of our problem and the abnormal does not occur; or, by chance, the target point $M$ exactly coincides with the end-point of the minimal drift trajectory, and then (under the assumption $c(y)\geq v$) the abnormal case indeed occurs, and the optimal trajectory coincides with the minimal drift trajectory.
\end{itemize}
This example is interesting because it gives a very simple situation where we may have an abnormal minimizer and how to interpret it.
\end{example}

\begin{example}[Optimal control of damaging insects by predators.]\label{exo7.3.22}
In order to eradicate as much as possible a population $x_0>0$ of damaging pests, we introduce in the ecosystem a population $y_0>0$ of (nondamaging) predator insects killing the pests.

\medskip

\noindent\textit{First part.}
In a first part, we assume that the predator insects that we introduce are infertile, and thus cannot reproduce themselves. The control consists of the continuous introduction of predator insects. The model is
\begin{equation*}
\begin{split}
\dot{x}(t) &= x(t) ( a-b y(t)),\quad x(0)=x_0,\\
\dot{y}(t) &= -cy(t) + u(t),\quad\ \ y(0)=y_0,
\end{split}
\end{equation*}
where $a>0$ is the reproduction rate of pests, $b>0$ is the predation rate, $c>0$ is the natural death rate of predators. The control $u(t)$ is the rate of introduction of new predators at time $t$. It must satisfy the constraint
$
0\leq u(t)\leq M
$
where $M>0$ is fixed. Let $T>0$ be fixed. We want to minimize, at the final time $T$, the number of pests, while minimizing as well the number of introduced predators. We choose to minimize the cost
$$x(T)+\int_0^Tu(t)\, dt.$$

Throughout, we denote by $p=(p_x,p_y)$ and by $p^0$ the adjoint variables.

\medskip

First, we claim that $x(t)>0$ and $y(t)>0$ along $[0,T]$, for every control $u$.

Indeed, since $u(t)\geq 0$, we have $\dot{y}(t)\geq -cy(t)$, hence $y(t)\geq y_0\mathrm{e}^{-ct}>0$. For $x(t)$, we argue by contradiction: if there exists $t_1\in [0,T]$ such that $x(t_1)=0$, then $x(t)=0$ for every time $t$ by Cauchy uniqueness; this raises a contradiction with the fact that $x(0)=x_0>0$.

The Hamiltonian of the optimal control problem is 
$
H=p_xx(a-by)+p_y(-cy+u)+p^0u
$
and the adjoint equations are $\dot{p}_x=-p_x(a-by)$, $\dot{p}_y=bp_xx+cp_y$. The transversality conditions yield $p_x(T)=p^0$ and $p_y(T)=0$. It follows that $p^0\neq 0$ (otherwise we would have $(p(T),p^0)=(0,0)$, which is a contradiction). In what follows, we set $p^0=-1$.

\medskip

We have 
$
\frac{d}{dt}x(t)p_x(t)=x(t)p_x(t)(a-by(t))-x(t)p_x(t)(a-by(t))=0,
$
hence $x(t)p_x(t)=\mathrm{Cst}=-x(T)$ because $p_x(T)=p^0=-1$. It follows that $\dot{p}_y=-bx(T)+cp_y$, and since $p_y(T)=0$, we infer, by integration, that
$p_y(t) = \frac{b}{c}x(T) ( 1-\mathrm{e}^{c(t-T)} ).$
The maximization condition of the PMP yields %$\max_{0\leq u\leq M} (p_y-1)u$, hence
\begin{equation*}
u(t)=\left\{\begin{array}{lll}
0 & \textrm{if} & p_y(t)-1<0\\
M & \textrm{if} & p_y(t)-1>0
\end{array}\right.
\end{equation*}
unless the function $t\mapsto p_y(t)-1$ vanishes identically on some subinterval. But this is not possible because we have seen above that the function $p_y$ is decreasing. We conclude that the optimal control is bang-bang.
Moreover, at the final time we have $p_y(T)-1=-1$, hence, by continuity, there exists $\varepsilon>0$ such that $p_y(t)-1<0$ along $[T-\varepsilon,T]$, and hence $u(t)=0$ along a subinterval containing the final time. 

We can be more precise: we claim that, actually, the optimal control has at most one switching along $[0,T]$ (and is $0$ at the end).

Indeed, the function $p_y$ is decreasing (because $x(T)>0$, as we have seen at the beginning), hence the function $t\mapsto p_y(t)-1$, which is equal to $-1$ at $t=T$, vanishes at most one time. 
If there is such a switching, necessarily it occurs at some time $t_1\in[0,T]$ such that $p_y(t_1)=1$, which yields
$$
t_1 = T+\frac{1}{c}\mathrm{ln}\left(1-\frac{c}{bx(T)}\right).
$$
Note that this switching can occur only if $t_1>0$ (besides, we have $t_1<T$), hence, only if $x(T)>\frac{c}{b}\frac{1}{1-\mathrm{e}^{-cT}}$. By integrating backwards the equations, we could even express an implicit condition on the initial conditions, ensuring that this inequality be true, hence, ensuring that there is a switching.

\medskip

\noindent\textit{Second part.}
We now assume that the predators that we introduce are fertile, and reproduce themselves with a rate that is proportional to the number of pests. The control is now the death rate of predators. In order to simplify, we assume that the variables are normalized so that all other rates are equal to $1$. The model is
\begin{equation*}
\begin{split}
\dot{x}(t) &= x(t) (1-y(t)),\qquad\quad x(0)=x_0,\\
\dot{y}(t) &= -y(t)(u(t)-x(t)),\quad y(0)=y_0,
\end{split}
\end{equation*}
where the control $u(t)$ satisfies the constraint $0< \alpha\leq u(t)\leq \beta$.
\\
First, as before we have $x(t)>0$ and $y(t)>0$ along $[0,T]$, for every control $u$.

All equilibrium points of the system are given by $x_e=u_e$, $y_e=1$, for every $\alpha\leq u_e\leq\beta$. In the quadrant, we have a whole segment of equilibrium points.

Let us investigate the problem of steering the system in minimal time $t_f$ to the equilibrium point $x(t_f)=a$, $y(t_f)=1$.

The Hamiltonian is 
$
H=p_xx(1-y)-p_yy(u-x)+p^0
$
and the adjoint equations are $\dot{p}_x=-p_x(1-y)-p_yy$, $\dot{p}_y=p_xx+p_y(u-x)$. The transversality condition on the final time gives $H(t_f)=0$, and since the system is autonomous, it follows that the Hamiltonian is constant along any extremal, equal to $0$.

The maximization condition of the PMP is $\underset{0\leq u\leq M}{\max} (-p_yyu)$, which gives, since $y(t)>0$,
\begin{equation*}
u(t)=\left\{\begin{array}{lll}
\alpha & \textrm{if} & p_y(t)>0\\
\beta  & \textrm{if} & p_y(t)<0
\end{array}\right.
\end{equation*}
unless the function $t\mapsto p_y(t)$ vanishes identically along a subinterval. If this is the case then $p_y(t)=0$ for every $t\in I$. Derivating with respect to time, we get that $xp_x=0$ and thus $p_x=0$ along $I$. Therefore, along $I$, we have $H=p^0$, and since $H=0$ we infer that $p^0=0$, which raises a contradiction. Therefore, we conclude that the optimal control is bang-bang.

Along an arc where $u=\alpha$ (resp., $u=\beta$), we compute $\frac{d}{dt}F_\alpha(x(t),y(t))=0$, where
$$F_\alpha(x,y)=x+y-\alpha\,\mathrm{ln}\, x-\mathrm{ln}\, y $$
(resp., $F_\beta$), that is, $F_\alpha(x(t),y(t))$ is constant along such a bang arc.

It can be noted that, formally, this integral of the movement can be obtained by computing $\frac{dy}{dx}=\frac{\dot{y}}{\dot{x}}=\frac{-y}{1-y}\frac{\alpha-x}{x}$ and by integrating this separated variables one-form.

Considering a second-order expansion of $F_\alpha$ at the point $(\alpha,1)$,
$$F_\alpha(\alpha+h,1+k)=\alpha-\alpha\, \mathrm{ln}\, \alpha+1+\frac{1}{2}\left(\frac{h^2}{\alpha}+k^2\right)+\mathrm{o}(h^2+k^2),$$
we see that $F_\alpha$ has a strict local minimum at the point $(\alpha,1)$. Moreover, the function $F_\alpha$ is (strictly) convex, because its Hessian 
$\begin{pmatrix} \frac{a}{x^2} & 0 \\ 0 & \frac{1}{y^2} \end{pmatrix}$
is symmetric positive definite at any point such that $x>0$, $y>0$. It follows that the minimum is global.

For any controlled trajectory (with a control $u$), we have 
$$\frac{d}{dt} F_\alpha(x(t),y(t))=(u(t)-\alpha)(1-y(t)) .$$

Let us prove that there exists $\varepsilon>0$ such that $u(t)=\alpha$, for almost every $t\in[t_f-\varepsilon,t_f]$ (in other words, the control $u$ is equal to $\alpha$ at the end).

Indeed, at the final time, we have either $p_y(t_f)=0$ or $p_y(t_f)\neq0$.

If $p_y(t_f)=0$, then, using the differential equation in $p_y$, we have $\dot{p}_y(t_f)=p_x(t_f)a$. We must have $p_x(t_f)\neq 0$ (otherwise we would get a contradiction, noticing as previously that $H(t_f)=p^0=0$). Hence $\dot{p}_y(t_f)\neq 0$, and therefore the function $p_y$ has a constant sign along some interval $[t_f-\varepsilon,t_f[$. Hence, along this interval, the control is constant, either equal to $\alpha$ or to $\beta$. It cannot be equal to $\alpha$, otherwise, since the function $F_\alpha$ is constant along this arc, and since this arc must reach the point $(\alpha,1)$, this constant would be equal to the minimum of $F_\alpha$, which would imply that the arc is constant, equal to the point $(\alpha,1)$: this is a contradiction because we consider a minimal time trajectory reaching the point $(\alpha,1)$.

If $p_y(t_f)\neq 0$, then the function $p_y$ has a constant sign along some interval $[t_f-\varepsilon,t_f[$, and hence, along this interval, the control is constant, equal either to $\alpha$ or to $\beta$. With a similar reasoning as above, we get $u=\alpha$.

\medskip

Let us now provide a control strategy (which can actually be proved to be optimal) in order to steer the system from any initial point $(x_0,y_0)$ to $(\alpha,1)$.

First, in a neighborhood of the point $(\alpha,1)$, the level sets of the function $F_\alpha$ look like circles. Farther from that point, the level sets look more and more like rectangle triangles, asymptotic to the coordinate axes. Similarly for the level sets of $F_\beta$, with respect to the point $(\beta,1)$ (see Figure \ref{figlevel}).

\begin{figure}[h]
\begin{center}
\includegraphics[width=5.2cm]{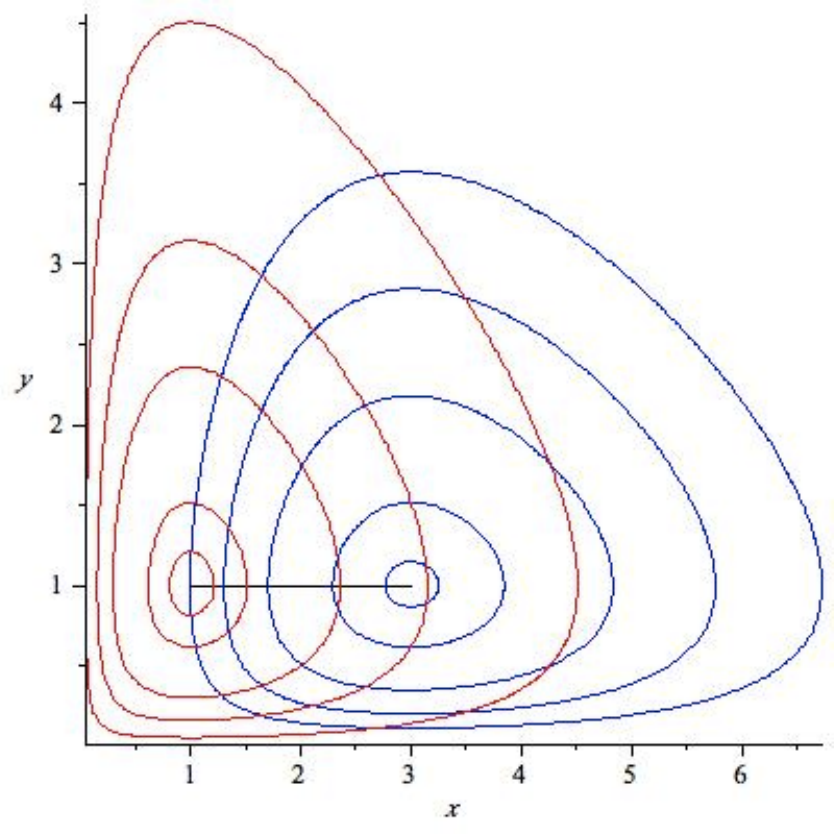}
\end{center}
\caption{Example with $\alpha=1$ and $\beta=3$}\label{figlevel}
\end{figure}

Let us start from a point $(x_0,y_0)$, which is located on the segment joining the two points $(\alpha,1)$ and $(\beta,1)$, that is, such that $\alpha<x_0<\beta$ and $y_0=1$.
We start with the control $u=\alpha$, and we remain along a level set of the function $F_\alpha$ (thus, ``centered" at the point $(\alpha,1)$). 
At some time, we switch on the control $u=\beta$, and then we remain along a level set of $F_\beta$ (thus, ``centered" at the point $(\alpha,1)$) which passes through the target point $(\alpha,1)$.

Now, if we start from any other point $(x_0,y_0)$, we can determine on Figure \ref{figlevel} a sequence of arcs along the level sets, respectively, of $F_\alpha$ and of $F_\beta$, that steers the system from the starting point to the target point.
\end{example}

\begin{example}[The Brachistochrone Problem.]
The objective is to determine what is the optimal shape of a children's slide (between two given altitudes) so that, if a ball slides along it (with zero initial speed), it arrives at the other extremity in minimal time.

This problem can be modeled as the following optimal control problem. In the Euclidean plane, the slide is modeled as a continuous curve, starting from the origin, and arriving at a given fixed point $(x_1,y_1)$, with $x_1>0$. We consider a ball of mass $m>0$ sliding along the curve. We denote by $(x(t),y(t))$ its position at time $t$. The ball is subject to the gravity force $m\vec{g}$ and to the reaction force of the children's slide. At time $t$, we denote by $u(t)$ the (oriented) angle between the unit horizontal vector and the velocity vector $(\dot x(t),\dot y(t))$ of the ball (which is collinear to the tangent to the curve). See Figure \ref{fig_tobog}.

\begin{figure}[h]
\begin{center}
{ \resizebox{7cm}{!}{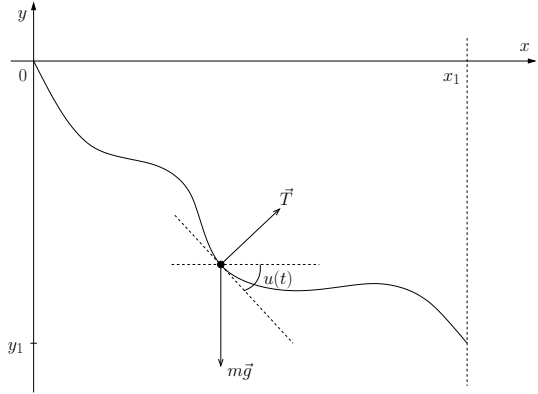} }
\end{center}
\caption{The Brachistochrone Problem, seen as an optimal control problem.}\label{fig_tobog}
\end{figure}

Seeking the curve is equivalent to seeking the angle $u(t)$. Therefore, we stipulate that $u$ is a control. By projecting the equations given by the fundamental principle of dynamics onto the tangent to the curve, we obtain the following control system:
\begin{equation}\label{sys_brach}
\begin{split}
\dot x(t) &= v(t)\cos u(t),\qquad x(0)=0,\quad x(t_f)=x_1, \\
\dot y(t) &= -v(t)\sin u(t),\quad\ y(0)=0,\quad y(t_f)=y_1, \\
\dot v(t) &= g\sin u(t),\quad\qquad\  v(0)=0,\quad v(t_f)\ \textrm{free}.
\end{split}
\end{equation}
The control is $u(t)\in\R$, $g>0$ is a constant. We want to minimize the final time $t_f$.

\medskip

First of all, noticing that $\dot y=-\frac{1}{g} v\dot v$, we get, by integration, that $y(t)=-\frac{1}{2g} v(t)^2$, for every control $u$. This implies that any final point such that $y_1>0$ is not reachable. Therefore, from now on, we will assume that $y_1\leq 0$.

Because of the relation between $y(t)$ and $v(t)$, we can reduce the optimal control problem \eqref{sys_brach} to the minimal time control problem for the following system:
\begin{equation}\label{sys_brach_reduit}
\begin{split}
\dot x(t) &= v(t)\cos u(t),\quad x(0)=0,\quad x(t_f)=x_1>0\ \textrm{fixed}, \\
\dot v(t) &= g\sin u(t),\quad\quad\  v(0)=0,\quad v(t_f)\ \textrm{fixed}.
\end{split}
\end{equation}
Note that, since $y(t_f)=y_1$ is fixed, it follows that $v(t_f)=\pm\sqrt{-2gy_1}$.

Let us apply the Pontryagin maximum principle to the optimal control problem \eqref{sys_brach_reduit}. The Hamiltonian is $H=p_xv\cos u+p_vg\sin u+p^0$. The adjoint equations are $\dot p_x=0$ et $\dot p_v=-p_x\cos u$. In particular, $p_x$ is constant.
Since the final time is free and the problem is autonomous, we have $H=0$ along any extremal. The maximization condition of the Hamiltonian yields
$$
\cos u(t) = \frac{p_x v(t)}{\sqrt{(p_xv(t))^2+(gp_v(t))^2}}, \quad
\sin u(t) = \frac{g p_v(t)}{\sqrt{(p_xv(t))^2+(gp_v(t))^2}}, 
$$
provided that $\varphi(t)=(p_xv(t))^2+(gp_v(t))^2\neq 0$.

We have $\frac{d}{dt}(p_xv(t))=p_xg\sin u(t)$ et $\frac{d}{dt}(gp_v(t))=-p_xg\cos u(t)$.
As a consequence, if $\varphi(t)=0$ for every $t$ in some subset $I$ of positive measure, then $p_x=0$. Therefore $\varphi(t)=(gp_v(t))^2$ and thus $p_v(t)=0$ for every $t\in I$. Since $H=0$, we infer that $p^0=0$. We have obtained $(p_x,p_v(t),p^0)=(0,0,0)$, which is a contradiction.
We conclude that the function $\varphi$ never vanishes on any subset of $[0,t_f]$ of positive measure. Therefore the above expression of the controls is valid almost everywhere.

The maximized Hamiltonian is $H=\sqrt{(p_xv(t))^2+(gp_v(t))^2}+p^0$. Since $H=0$ along any extremal, we infer, by contradiction, that $p^0\neq 0$ (indeed otherwise we would infer that $\varphi\equiv 0$, which leads to a contradiction). Hence, from now on, we take $p^0=-1$.

Since $H=0$ along any extremal, we get that $(p_xv(t))^2+(gp_v(t))^2=1$, and therefore $\cos u(t)=p_x v(t)$ and $\sin u(t)=gp_v(t)$.

\medskip

If $p_x$ were equal to $0$, then we would have $\cos u(t)=0$ and thus $\dot x(t)=0$, and then we would never reach the point $x_1>0$. Therefore $p_x\neq 0$.

Let us now integrate the trajectories.
We have $\dot v=g^2p_v$ and $\dot p_v=-p_x^2v$ and thus $\ddot v+g^2p_x^2v=0$, and since $v(0)=0$ we get that $v(t)=A\sin(gp_xt)$. Since $H=0$ and $v(0)=0$, we have $g^2p_v(0)^2=1$ hence $p_v(0)=\pm \frac{1}{g}$, and thus $\dot v(0)=\pm g$. We infer that $A=\pm\frac{1}{p_x}$, and hence that $v(t)=\pm\frac{1}{p_x}\sin(gp_xt)$. Now, $\dot x=v\cos u=p_xv^2$ and $y=-\frac{1}{2g}v^2$, and by integration we get 
\begin{equation*}
\begin{split}
x(t) &= \frac{1}{2p_x}t-\frac{1}{4gp_x^2}\sin(2gp_xt),\\
y(t) &= -\frac{1}{2gp_x^2}\sin^2(gp_xt) = -\frac{1}{4gp_x^2} \left( 1-\cos(2gp_xt) \right).
\end{split}
\end{equation*}
Note that, since $\dot x=p_xv^2$, we must have $p_x>0$.

Representing in the plane the parametrized curves $(x(t),y(t))$ (with parameters $p_x$ and $t$), we get \textit{cycloid curves}.

\medskip

Let us now prove that there is a unique optimal trajectory joining $(x_1,y_1)$, and it has at most one cycloid arch.

Let us first compute $p_x$ and $t_f$ in the case where $y_1=0$. If $y_1=0$ then $\sin(gp_xt_f)=0$ hence $gp_xt_f=\pi+k\pi$ with $k\in\Z$, but since $t_f$ must be minimal, we must have $k=0$ (and this is what will imply that optimal trajectories have at most one arch). Hence $p_x=\frac{t_f}{2x_1}$. Since $2gp_xt_f=2\pi$, we have $x_1=x(t_f)=\frac{t_f}{2p_x}$, and thus $t_f=\sqrt{\frac{2\pi x_1}{g}}$ and $p_x=\sqrt{\frac{\pi}{2gx_1}}$.

\begin{figure}[h]
\begin{center}
\includegraphics[width=5.7cm]{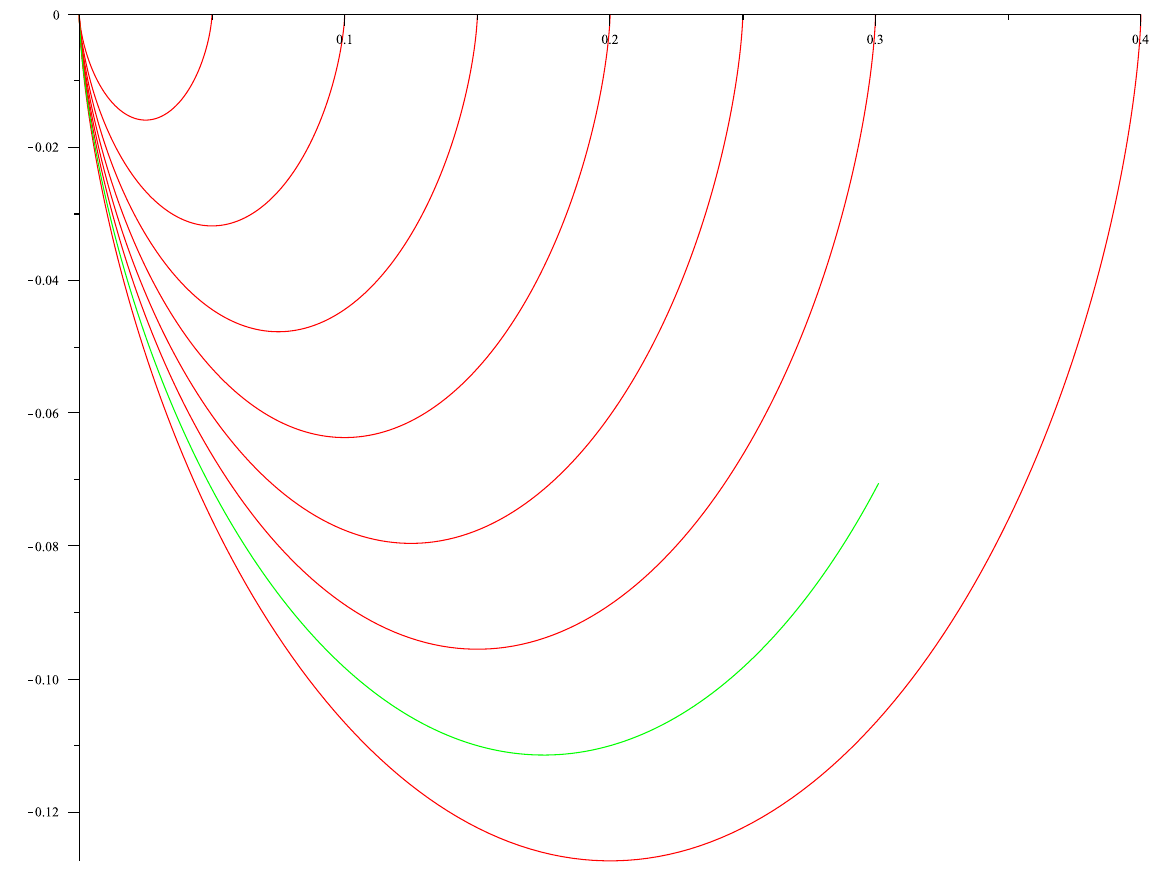}
\end{center}
\caption{Optimal trajectories}\label{fig_brachi}
\end{figure}

On Figure \ref{fig_brachi}, we have represented all optimal trajectories joining points $(x_1,0)$, with $x_1>0$. 

Now, we note that, if a trajectory is optimal on the interval $[0,t_f]$, then, for any $t_1\in]0,t_f[$, it is optimal as well on the interval $[0,t_1]$ for the problem of reaching the point $(x(t_1),y(t_1))$ in minimal time (this is the dynamic programming principle). From that remark, we deduce that any optimal trajectory of the problem \eqref{sys_brach} is the truncation of an optimal trajectory reaching a point of the abscissa axis. In particular, we get the desired uniqueness property, and the fact that such a trajectory has at most one point at which $\dot y=0$ (hence, at most one arch). See Figure \ref{fig_brachi}.

Moreover, if $\dot y(t)=0$ then $2gp_x t=\pi+k\pi$ with $k\in\Z$, and necessarily (by optimality) $k=0$, hence $t=\frac{\pi}{2gp_x}$. Therefore the set of points where $\dot y(t)=0$ is the parametrized curve $x(p_x)=\frac{\pi}{4gp_x^2}$, $y(p_x)=-\frac{1}{2gp_x^2}$, that is the graph $y=-\frac{2}{\pi}x$. Therefore, we have proved that the optimal trajectory $(x(t),y(t))$ reaching $(x_1,y_1)$ is such that
\begin{itemize}[parsep=1mm,itemsep=1mm,topsep=1mm]%,leftmargin=*
\item $y(t)$ passes through a minimum if $y_1>-\frac{2}{\pi}x_1$,
\item $y(t)$ is decreasing if $y_1<-\frac{2}{\pi}x_1$.
\end{itemize}

\begin{remark}
If we investigate the variant of the optimal control problem \eqref{sys_brach}, for which we minimize the final time $t_f$ with $y(t_f)$ free, then in the reduced problem we have moreover that $v(t_f)$ is free, and hence we gain the transversality condition $p_v(t_f)=0$, hence $\dot v(t_f)=0$. This gives $2gp_x t_f=\pi$, in other words, we find exactly the final points of the previously computed optimal trajectories, stopping when $\dot y=0$.

This means that, if $y(t_f)$ is free (with $x_1$ fixed), we minimize the time $t_f$ by choosing, on Figure \ref{fig_brachi}, the arc of cycloid starting from the origin and reaching $x=x_1$ with an horizontal tangent.
\end{remark}

\end{example}

\chapter{Stabilization}\label{chap_stab}

In this chapter, our objective will be to stabilize a possibly unstable equilibrium point by means of a feedback control.

Let $n$ and $m$ be two positive integers.
In this chapter we consider an autonomous control system in $\R^n$
\begin{equation}\label{contsyststab}
\dot{x}(t) = f(x(t),u(t))
\end{equation}
where $f:\R^n\times\R^m\rightarrow\R^n$ is of class $C^1$ with respect to $(x,u)$, and the controls are measurable essentially bounded functions of time taking their values in some measurable subset $\Omega$ of $\R^m$ (set of control constraints).

Let $(\bar x,\bar u)\in\R^n\times\R^m$ be an equilibrium point, that is, $f(\bar x,\bar u)=0$, such that $\bar u\in\mathring{\Omega}$ (interior of $\Omega$). Our objective is to be able to design a feedback control $u(x)$ stabilizing locally the equilibrium $(\bar x,\bar u)$, that is, such that the closed-loop system $\dot x(t)=f(x(t),u(x(t)))$ be locally asymptotically stable at $\bar x$.

\section{Stabilization of autonomous linear systems}
\subsection{Reminders on stability notions}
Consider the linear system $\dot{x}(t)=Ax(t)$, with $A$ a $n\times n$ matrix. The point $0$ is of course an equilibrium point of the system (it is the only one if $A$ is invertible). We have the following well-known result (easy to prove with simple linear algebra considerations).

\begin{theorem}
\begin{itemize}[parsep=1mm,itemsep=1mm,topsep=1mm]%,leftmargin=*
\item If there exists a (complex) eigenvalue $\lambda$ of $A$ such that $\Real(\lambda)>0$, then the equilibrium point $0$ is unstable, meaning that there exists $x_0\in\R^n$ such that the solution of $\dot{x}(t)=Ax(t)$, $x(0)=x_0$ satisfies $\Vert x(t)\Vert\rightarrow+\infty$ as $t\rightarrow+\infty$.
\item If all (complex) eigenvalues of $A$ have negative real part, then $0$ is asymptotically stable, meaning that all solutions of $\dot{x}(t)=Ax(t)$ converge to $0$ as $t\rightarrow+\infty$.
\item The equilibrium point $0$ is stable (i.e., not unstable)\footnote{See also the general definition \ref{def_stable} further.} if and only if all eigenvalues of $A$ have nonpositive real part and if an eigenvalue $\lambda$ is such that $\Real(\lambda)=0$ then $\lambda$ is a simple root\footnote{Equivalently, $\ker(A-\lambda I_n)=\ker(A-\lambda I_n)^2$, or, equivalently, the Jordan decomposition of $A$ does not have any strict Jordan block.} of the minimal polynomial of $A$.
\end{itemize}
\end{theorem}

\begin{definition}
The matrix $A$ is said to be \textit{Hurwitz} if all its eigenvalues have negative real part.
\end{definition}

We are next going to see two classical criteria ensuring that a given matrix (with real coefficients) is Hurwitz. These criteria are particularly remarkable because they are purely algebraic (polynomial conditions on the coefficients of the matrix), and they do not require the calculation of the roots of the characteristic polynomial of the matrix (which is impossible to achieve algebraically in general, for degrees larger than $5$, as is well known from Galois theory).

\paragraph{Routh criterion.}
We consider the complex polynomial
$$
P(z)=a_0z^n+a_1z^{n-1}+\cdots+a_{n-1}z+a_n
$$
with real coefficients $a_i$, and we are going to formulate some conditions under which all roots of this polynomial have negative real part (in that case we also say that $P$ is Hurwitz).
Note that $A$ is Hurwitz if and only if its characteristic polynomial $\chi_A$ is Hurwitz.

\begin{definition}
The \textit{Routh table} is defined as follows:
\begin{equation*}
\begin{split}
& a_0\ \ a_2\ \ a_4\ \ a_6\ \ \cdots \quad \textrm{completed by $0$} \\
& a_1\ \ a_3\ \ a_5\ \ a_7\ \ \cdots \quad \textrm{completed by $0$} \\
& b_1\ \ \, b_2\ \ \, b_3\ \  b_4\ \ \cdots \quad \textrm{where}\ 
b_1=\frac{a_1a_2-a_0a_3}{a_1},\
b_2=\frac{a_1a_4-a_0a_5}{a_1},\ldots \\
& c_1\ \ \,c_2\ \ \,c_3\ \ \,c_4\ \ \cdots \quad \textrm{where}\ 
c_1=\frac{b_1a_3-a_1b_2}{b_1},\
c_2=\frac{b_1a_5-a_1b_3}{b_1},\ldots \\
& \ \vdots\ \ \ \ \vdots\ \ \ \ \vdots\ \ \ \ \vdots
\end{split}
\end{equation*}
The process goes on as long as the first element of the row is not equal to $0$. The process stops when we have built $n+1$ rows.

The Routh table is said to be \textit{complete} if it has $n+1$ rows whose first coefficient is not equal to $0$.
\end{definition}

We have the two following theorems (stated in \cite{Routh}), which can be proved by means of (nonelementary) complex analysis.

\begin{theorem}\label{thm_Routh}
All roots of $P$ have negative real part if and only if the Routh table is complete and the elements in the first column have the same sign.
\end{theorem}

\begin{theorem}
If the Routh table is complete then $P$ has no purely imaginary root, and the number of roots with positive real part is equal to the number of sign changes in the first column.
\end{theorem}

\paragraph{Hurwitz criterion.}
We set $a_{n+1}=a_{n+2}=\cdots=a_{2n-1}=0$, and we define %the square matrix of size $n$
$$
H = \begin{pmatrix}
a_1 & a_3 & a_5 & \cdots & \cdots & a_{2n-1} \\
a_0 & a_2 & a_4 & \cdots & \cdots & a_{2n-2} \\
0   & a_1 & a_3 & \cdots & \cdots & a_{2n-3} \\
0   & a_0 & a_2 & \cdots & \cdots & a_{2n-4} \\
0   &  0  & a_1 & \cdots & \cdots & a_{2n-5} \\
\vdots&\vdots&\ddots&    &        & \vdots   \\
0   & 0   & 0   &   *    & \cdots & a_n
\end{pmatrix}
$$
where $*=a_0$ or $a_1$ according to the parity of $n$.
Let $(H_i)_{i\in\{1,\ldots,n\} }$ be the principal minors of $H$, defined by
$$H_1=a_1,\ 
H_2=\left\vert\begin{array}{cc}
a_1 & a_3 \\ a_0 & a_2 \end{array}\right\vert,\ 
H_3=\left\vert\begin{array}{ccc}
a_1 & a_3 & a_5 \\ a_0 & a_2 & a_4 \\ 0 & a_1 & a_3
\end{array}\right\vert,\ \ldots,\ H_n=\textrm{det}\,H .$$

\begin{theorem}\cite{Hurwitz}
If $a_0>0$, then $P$ is Hurwitz if and only if $H_i>0$ for every $i\in\{1,\ldots,n\}$.
\end{theorem}

\begin{remark}
Assume that $a_0>0$.

If all roots of $P$ have nonpositive real part, then $a_k\geq 0$ and $H_k\geq 0$, for every $k\in\{1,\ldots, n\}$.

If $n\leq 3$, $a_k\geq 0$ and $H_k\geq 0$ for every $k\in\{1,2,3\}$, then all roots of $P$ have nonpositive real part.

A necessary condition for stability is that $a_k\geq 0$ for every $k\in\{1,\ldots,n\}$. This condition is however not sufficient (take $P(z)=z^4+z^2+1$).
\end{remark}

\subsection{Pole-shifting theorem}
\begin{definition}
The linear autonomous control system $\dot{x}(t)=Ax(t)+Bu(t)$, with $x(t)\in\R^n$, $u(t)\in\R^m$, $A$ a $n\times n$ matrix and $B$ a $n\times m$ matrix, is said to be \textit{feedback stabilizable} if there exists a $m\times n$ matrix $K$ (called gain matrix) such that the closed-loop system with the (linear) feedback $u(t)=Kx(t)$,
$$
\dot{x}(t)=(A+BK)x(t)
$$
is asymptotically stable, i.e., equivalently, $A+BK$ is Hurwitz.
\end{definition}

\begin{remark}
This concept is invariant under similar transforms $A_1=PAP^{-1}$, $B_1=PB$, $K_1=KP^{-1}$.
\end{remark}

\begin{theorem}[pole-shifting theorem]\label{thmplacementpoles}
If $(A,B)$ satisfies the Kalman condition $\mathrm{rank}\ K(A,B)=n$, then for every real polynomial $P$ of degree $n$ whose leading coefficient is $1$, there exists a $m\times n$ matrix $K$ such that the characteristic polynomial $\chi_{A+BK}$ of $A+BK$ is equal to $P$.\footnote{Actually, the converse statement is also true.}
\end{theorem}

\begin{corollary}
If the linear control system $\dot{x}(t)=Ax(t)+Bu(t)$ is controllable then it is stabilizable.
\end{corollary}

To prove the corollary, it suffices to take for instance $P(X)=(X+1)^n$ and to apply the pole-shifting theorem.

\begin{proof}[Proof of Theorem \ref{thmplacementpoles}]
We prove the result first in the case $m=1$. It follows from Theorem \ref{thmbrunovski} (Brunovski normal form) that the system is similar to
$$A=\begin{pmatrix}
0  & 1 & \cdots & 0 \\
\vdots & \ddots & \ddots & \vdots \\
0 & \cdots & 0 & 1 \\
-a_n & -a_{n-1} & \cdots & -a_1
\end{pmatrix},
\qquad
B = \begin{pmatrix}
0 \\ \vdots \\ 0 \\ 1 \end{pmatrix} .$$
Setting $K=(k_1\ \cdots\ k_n)$ and $u=Kx$, we have
$$
A+BK=\begin{pmatrix}
0  & 1 & \cdots & 0 \\
\vdots & \ddots & \ddots & \vdots \\
0 & \cdots & 0 & 1 \\
k_1-a_n & k_2-a_{n-1} & \cdots & k_n-a_1
\end{pmatrix}
$$
and thus $\chi_{A+BK}(X) = X^n+(a_1-k_n)X^{n-1}+\cdots+(a_n-k_1).$
Therefore, for every polynomial $P(X)=X^n+\alpha_1X^{n-1}+\cdots+\alpha_n$, it suffices to choose $k_1=a_n-\alpha_n,\ldots,k_n=a_1-\alpha_1$.

Let us now prove that the general case $m\geq 1$ can be reduced to the case $m=1$. We have the following lemma.

\begin{lemma}\label{lem_reducedK}
If $(A,B)$ satisfies the Kalman condition, then there exists $y\in\R^m$ and a $m\times n$ matrix $C$ such that $(A+BC,By)$ satisfies the Kalman condition.
\end{lemma}

The proof of this lemma is done hereafter. 
It follows from Lemma \ref{lem_reducedK} that, for every polynomial $P$ of degree $n$ whose leading coefficient is $1$, there exists a $1\times n$ matrix $K_1$ such that $\chi_{A+BC+ByK_1}=P$, and therefore, defining the $m\times n$ matrix $K=C+yK_1$, we have $\chi_{A+BK}=P$, and Theorem \ref{thmplacementpoles} is proved.
\end{proof}

\begin{proof}[Proof of Lemma \ref{lem_reducedK}]
Let $y\in\R^m$ be such that $By\neq 0$. Let $x_1=By$. 

\smallskip
\noindent\emph{Claim 1:} There exists $x_2\in Ax_1+\mathrm{Ran}(B)$ (and thus there exists $y_1\in\R^m$ such that $x_2=Ax_1+By_1$) such that $\dim(\textrm{Span}(x_1,x_2))=2$.

Indeed, otherwise, $Ax_1+\mathrm{Ran}(B)\subset\R x_1$, hence $Ax_1\in\R x_1$ and $\mathrm{Ran}(B)\subset\R x_1$. Therefore
$
\mathrm{Ran}(AB) = A\mathrm{Ran}(B)\subset\R Ax_1\subset\R x_1
$
and by immediate iteration, $\mathrm{Ran}(A^kB)\subset\R x_1$, for every integer $k$. This implies that
$$
\mathrm{Ran}(B,AB,\ldots,A^{n-1}B)=\mathrm{Ran}(B)+
\mathrm{Ran}(AB)+\cdots+\mathrm{Ran}(A^{n-1}B)  \subset\R x_1
$$
which contradicts the Kalman condition.

\medskip

\noindent\emph{Claim 2:} For every $k\leq n$, there exists $x_k\in Ax_{k-1}+\mathrm{Ran}(B)$ (and thus there exists $y_{k-1}\in\R^m$ such that $x_k=Ax_{k-1}+By_{k-1}$) such that $\dim (E_k)=k$, where $E_k=\mathrm{Span}(x_1,\ldots,x_k)$.

Indeed, otherwise, $Ax_{k-1}+\mathrm{Ran}(B)\subset E_{k-1}$, and hence $Ax_{k-1}\subset E_{k-1}$ and $\mathrm{Ran}(B)\subset E_{k-1}$. Let us then prove that $AE_{k-1}\subset E_{k-1} $. Indeed, note that $Ax_1=x_2-By_1\in E_{k-1}+\mathrm{Ran}(B)\subset E_{k-1}$, and similarly for $Ax_2$, etc, $Ax_{k-2}=x_{k-1}-By_{k-1}\in E_{k-1}+\mathrm{Ran}(B)\subset E_{k-1}$, and finally, $Ax_{k-1}\in E_{k-1}$.

Therefore $\mathrm{Ran}(AB)=A\,\mathrm{Ran}(B)\subset A E_{k-1}\subset E_{k-1}$, and similarly we have
$\mathrm{Ran}(A^iB)\subset E_{k-1}$ for every integer $i$. Hence
$\mathrm{Ran}(B,AB,\ldots,A^{n-1}B)\subset E_{k-1}$,
which contradicts the Kalman condition.

We have thus built a basis $(x_1,\ldots,x_n)$ of $\R^n$. We define the $m\times n$ matrix $C$ by the relations
$Cx_1=y_1$, $Cx_2=y_2$, $\ldots$, $Cx_{n-1}=y_{n-1}$, and $Cx_n$ arbitrary.
Then $(A+BC,x_1)$ satisfies the Kalman condition since $(A+BC)x_1=Ax_1+By_1=x_2$, $\ldots$, $(A+BC)x_{n-1}=Ax_{n-1}+By_{n-1}=x_n$. Lemma \ref{lem_reducedK} is proved.
\end{proof}

\begin{remark}
To stabilize a linear control system in practice, one has the following solutions:
\begin{itemize}[parsep=1mm,itemsep=1mm,topsep=1mm]%,leftmargin=*
\item If $n$ is not too large, one can apply the Routh or Hurwitz criteria and thus determine an algebraic necessary and sufficient condition on the coefficients of $K$ ensuring the desired stabilization property. Note that the characteristic polynomial of $A+BK$ can be computed with a formal computations software like \textit{Maple}.
\item There exist many numerical routines in order to compute numerical gain matrices. In the \textit{Matlab Control Toolbox}, we quote \textit{acker.m}, based on Ackermann's formula (see \cite{Kailath}), limited however to $m=1$ and not very reliable numerically. Better is to use \textit{place.m}, which is a robust pole-shifting routine (see \cite{Kautsky}) based on spectral considerations (but in which the desired poles have to be distinct two by two).
\item Another way consists of applying the LQ theory, elements of which have been given in Section \ref{sec_LQ}, by taking an infinite horizon of time $T=+\infty$ as quickly mentioned at the end of that section (LQR stabilization).
\end{itemize}
\end{remark}

\section{Stabilization of instationary linear systems}
For instationary linear systems $\dot{x}(t)=A(t)x(t)+B(t)u(t)$, the situation is much more complicated and there is no simple and definitive theory as in the autonomous case.

Let us explain which difficulties appear, by considering the system $\dot{x}(t)=A(t)x(t)$ without any control.
A priori one could expect that, if the matrix $A(t)$ is Hurwitz for every $t$, then the system is asymptotically stable. This is however wrong. The statement remains wrong even under stronger assumptions on $A(t)$, such as assuming that there exists $\varepsilon>0$ such that, for every time $t$, every (complex) eigenvalue $\lambda(t)$ of $A(t)$ satisfies $\Re(\lambda(t))\leq-\varepsilon$.
Indeed, for example, take
$$
A(t) = \begin{pmatrix}
-1+a\cos^2t & 1-a\sin t \cos t \\
-1-a\sin t \cos t  &  -1+a\sin^2t
\end{pmatrix}
$$
with $a\in[1,2)$ arbitrary.
Then
%$$
%x(t) = e^{(a-1)t} \begin{pmatrix}
%\cos t \\ -\sin t
%\end{pmatrix}
%$$
$
x(t) = e^{(a-1)t} \begin{pmatrix}
\cos t , -\sin t
\end{pmatrix}^\top
$
is a solution of $\dot{x}(t)=A(t)x(t)$, and does not converge to $0$ whenever $a\geq 1$.
Besides, it can be shown that if $a<1$ then the system is asymptotically stable.

\medskip

Let us explain the reason of this failure. A simple way to understand is the following (not so much restrictive) case. Let us assume that, for every $t$, $A(t)$ is diagonalizable, and that there exists $P(t)$ invertible such that $P(t)^{-1}A(t)P(t)=D(t)$, with $D(t)$ diagonal, and $P(\cdot)$ and $D(\cdot)$ of class $C^1$.
Setting $x(t)=P(t)y(t)$, we get immediately that
$$
\dot{y}(t) = ( D(t) - P(t)^{-1}\dot P(t) ) y(t) .
$$
If the term $P(t)^{-1}\dot P(t)$ were equal to $0$ (as it is the case in the autonomous case), then, obviously, the asymptotic stability would hold true as soon as the eigenvalues (diagonal of $D(t)$) would have negative real parts. But, even if $D(t)$ is Hurwitz, the term $P(t)^{-1}\dot P(t)$ can destabilize the matrix and imply the failure of the asymptotic stability.

In other words, what may imply the divergence is the fact that the eigenvectors (making up the columns of $P(t)$) may evolve quickly in time, thus implying that the norm of $\dot P(t)$ be large.

To end up however with a positive result, it can be noted that, if the matrix $A(t)$ is slowly varying in time, then the norm of the term $P(t)^{-1}\dot P(t)$ is small, and then if one is able to ensure that this norm is small enough with respect to the diagonal $D(t)$, then one can ensure an asymptotic stability result. This is the theory of slowly time-varying linear systems (see \cite{Khalil}).

\section{Stabilization of nonlinear systems}
\subsection{Local stabilization by linearization}
\paragraph{Reminders.}
Consider the continuous dynamical system $\dot{x}(t)=f(x(t))$, where $f:\R^n\rightarrow\R^n$ is of class $C^1$. We denote by $x(\cdot,x_0)$ the unique solution of this system such that $x(0,x_0)=x_0$.
We assume that $\bar{x}$ is an \textit{equilibrium point}, that is, $f(\bar{x})=0$.

\begin{definition}\label{def_stable}
The equilibrium point $\bar{x}$ is said to be \textit{stable} if, for every $\varepsilon>0$, there exists $\delta>0$ such that, for every initial point $x_0$ such that $\Vert x_0-\bar x\Vert\leq \delta$, one has $\Vert x(t,x_0)-\bar x\Vert\leq\varepsilon$ for every $t\geq 0$.
It is said to be \textit{locally asymptotically stable} (in short, LAS) if it is stable and if moreover $x(t,x_0)\rightarrow \bar{x}$ as $t\rightarrow+\infty$ for every $x_0$ in some neighborhood of $\bar x$.
If the neighborhood is the whole $\R^n$ then we speak of \textit{global asymptotic stability} (in short, GAS). If an asymptotic stability result is established in some neighborhood $V$ of $\bar x$, then we say that $\bar x$ is GAS in $V$.
\end{definition}

\begin{theorem}\label{thm_lineariz}
%Assume that $f$ is of class $C^1$.
Let $A$ be the Jacobian matrix of $f$ at the equilibrium point $\bar{x}$.
If all eigenvalues of $A$ have negative real parts (that is, if $A$ is Hurwitz), then $\bar{x}$ is LAS.
If $A$ has an eigenvalue with a positive real part then $\bar{x}$ is not LAS.
\end{theorem}

The above linearization theorem is an easy first result (see Example \ref{Lyapunov_lemma} further for a proof), not saying anything however, for the moment, on the size of the neighborhoods of stability.

\paragraph{Application: local stabilization of nonlinear control systems.}
Consider the general nonlinear control system \eqref{contsyststab}, $\dot x=f(x,u)$, and an equilibrium point $(\bar x,\bar u)\in\R^n\times\mathring{\Omega}$, as settled at the beginning of Chapter \ref{chap_stab}.
Setting $x(t)=\bar x+\delta x(t)$ and $u(t)=\bar u+\delta u(t)$ and keeping the terms of order $1$, we obtain (as already discussed for controllability issues) the linearized system at $(\bar x,\bar u)$,
$$
\delta\dot{x}(t)=A\delta x(t)+B\delta u(t)
$$
where
$$
A=\frac{\partial f}{\partial x}(\bar x,\bar u) \quad\textrm{and}\quad 
B=\frac{\partial f}{\partial u}(\bar x,\bar u).
$$
If one can stabilize the linearized system, that is, find a matrix $K$ of size $m\times n$ such that $A+BK$ is Hurwitz (and take $\delta u=K\delta x$), then Theorem \ref{thm_lineariz} implies a local stabilization result for the nonlinear control system \eqref{contsyststab}. We thus have the following theorem.

\begin{theorem}\label{thm_linearization}
If the pair $(A,B)$ satisfies the Kalman condition, then there exists a matrix $K$ of size $m\times n$ such that the feedback $u = K(x-\bar x)+\bar u$ stabilizes asymptotically the control system \eqref{contsyststab} locally around $(\bar x,\bar u)$: the closed-loop system $\dot{x}(t)=f(x(t),K(x(t)-\bar x)+\bar u)$ is LAS at $\bar x$.
\end{theorem}

Note that the stability neighborhood must be small enough so that the closed-loop control $u$ takes its values in the set $\Omega$.

\begin{example}
Consider the inverted pendulum system given in Example \ref{ex_pendulum}. 
Applying the Routh criterion (Theorem \ref{thm_Routh}) and then Theorem \ref{thm_linearization}, we establish that a sufficient condition on $K=(k_1,k_2,k_3,k_4)$ to stabilize the inverted pendulum locally at the unstable equilibrium $(\bar\xi,0,0,0)^\top$ is
\begin{align*}
& k_1>0, \qquad k_4-k_2L>0, \qquad k_3-k_1L-(m+M)g>0, \\
& k_2((k_4-k_2L)(k_3-k_1L-(m+M)g)-MLgk_2) > k_1(k_4-k_2L)^2.
\end{align*}
\end{example}

\begin{example}
Consider the Maxwell-Bloch system given in Example \ref{ex_Maxwell-Bloch}.
Let us stabilize locally the system at the equilibrium point $\bar x=(a,0,0)$, $\bar u=(0,0)$, where $a\neq 0$ is fixed.

By Example \ref{ex_Maxwell-Bloch}, (2), the linearized system (with $c=0$ and $a\neq 0$) at this point %, $\delta\dot x=A\delta x+B\delta u$, 
is controllable and thus stabilizable with a linear feedback of matrix $K$. %$\delta u = K \delta x$.
We seek a particular matrix $K$ stabilizing the system, of the form $K = \begin{pmatrix} k_1 & 0 & 0 \\ 0 & k_2 & 0 \end{pmatrix}$. By the Routh criterion, it is easy to see that $A+BK$ is Hurwitz if and only if $k_1<0$ and $k_2<0$.
We infer from Theorem \ref{thm_linearization} that the Maxwell-Bloch system is locally stabilizable around $\bar x$, with feedbacks $u_1=k_1(x_1-a)$, $u_2=k_2x_2$ with $k_1,k_2<0$.
\end{example}

\subsection{Global stabilization by Lyapunov theory}

\paragraph{Reminders: Lyapunov and LaSalle theorems.}
Consider the continuous dynamical system $\dot{x}(t)=f(x(t))$, where $f:\R^n\rightarrow\R^n$ is of class $C^1$. We assume that $\bar{x}$ is an \textit{equilibrium point}, that is, $f(\bar{x})=0$.
Let us recall two important results of Lyapunov theory, providing more knowledge on the stability neighborhoods, with the concept of Lyapunov function.

\begin{definition}
Let $\mathcal{D}$ be an open subset of $\R^n$ containing the equilibrium point $\bar x$. The function $V:\mathcal{D}\rightarrow \R$ is called a \textit{Lyapunov function} at $\bar{x}$ on $\mathcal{D}$ if
\begin{itemize}[parsep=1mm,itemsep=1mm,topsep=1mm]%,leftmargin=*
\item $V$ is of class $C^1$ on $\mathcal{D}$;
\item $V(\bar{x})=0$ and $V(x)>0$ for every $x\in\mathcal{D}\setminus\{\bar{x}\}$;
\item $\langle\nabla V(x),f(x)\rangle\leq 0$ for every $x\in\mathcal{D}$. If the inequality is strict on $\mathcal{D}\setminus\{\bar{x}\}$ then the Lyapunov function $V$ is said to be strict.
\end{itemize}
\end{definition}

Note that, along a given trajectory of the dynamical system, one has
$$
\frac{d}{dt}V(x(t))=\langle\nabla V(x(t)),f(x(t))\rangle .
$$
Therefore if $V$ is a Lyapunov function then the value of $V$ is nonincreasing along any trajectory. A Lyapunov function can be seen as a potential well, ensuring stability.

\begin{theorem}[Lyapunov theorem, see \cite{Hale}]
If there exists a Lyapunov function $V$ at $\bar x$ on $\mathcal{D}$ then $\bar x$ is stable, and if $V$ is strict then $\bar x$ is LAS. If $V$ is strict and proper\footnote{$V$ is said to be proper whenever $
V^{-1}([0,L])$ is a compact subset of $\mathcal{D}$, for every $L\in V(\mathcal{D})$; in other words, the inverse image of every compact is compact. When $\mathcal{D}=\R^n$, $V$ is proper if and only if $V(x)\rightarrow+\infty$ as $\Vert x\Vert\rightarrow+\infty$.} then $\bar{x}$ is GAS in $\mathcal{D}$.
\end{theorem}

When a Lyapunov function is not strict then one can be more specific and infer that trajectories converge to some invariant subset. The following result is even more general and does not assume the existence of an equilibrium point.

\begin{theorem}[LaSalle principle, see \cite{Hale}]
Assume that $V:\mathcal{D}\rightarrow[0,+\infty)$ is a proper $C^1$ function such that $\langle\nabla V(x),f(x)\rangle\leq 0$ for every $x\in\mathcal{D}$.
Let $\mathcal{I}$ be the largest subset of $\{x\in\mathcal{D}\
\vert\ \langle\nabla V(x),f(x)\rangle=0\}$that is invariant under the flow (in positive time) of the dynamical system. Then all trajectories starting in $\mathcal{D}$ converge to $\mathcal{I}$, in the sense that
$d(x(t),\mathcal{I})\rightarrow 0$ (Euclidean distance) as $t\rightarrow+\infty$.
\end{theorem}

\begin{remark}
It is interesting to formulate the LaSalle principle in the particular case where the invariant set $\mathcal{I}$ is reduced to a singleton (which must then be an equilibrium point). The statement is then as follows.

Assume that $V$ is a proper Lyapunov function at $\bar x$ on $\mathcal{D}$ and that, if $x(\cdot)$ is a solution of the system such that $\langle\nabla V(x(t)),f(x(t))\rangle=0$ for every $t\geq 0$, then $x(t)=\bar{x}$.
Then the equilibrium point $\bar{x}$ is GAS in $\mathcal{D}$.
\end{remark}

\begin{example}
Let $g:\R\rightarrow\R$ be a function of class $C^1$ such that $g(0)=0$ and $xg(x)>0$ if $x\neq 0$, and satisfying $\int_0^{+\infty} g=+\infty$ and $\int_{-\infty}^0 g=-\infty$. By considering the Lyapunov function $V(x,y)=\frac{1}{2}y^2+\int_0^x g(s)\, ds$, it is easy to prove that the point $x=\dot x=0$ is GAS for the system $\ddot x+\dot x+g(x)=0$ (which has to be written as a first-order system).
\end{example}

\begin{example}
Consider the system in $\R^2$
\begin{equation*}
\dot x = \alpha x - y - \alpha x (x^2+y^2), \qquad
\dot y = x + \alpha y - \alpha y (x^2+y^2),
\end{equation*}
with $\alpha>0$ arbitrary.
The equilibrium point $(0,0)$ is not stable. With the LaSalle principle, it is easy to prove that the unit circle $x^2+y^2=1$ is globally attractive in $\R^2\setminus\{(0,0)\}$, in the sense that a trajectory, starting from any point different from $(0,0)$, converges to the circle in infinite time. Indeed, note that, setting $V(x,y)=\frac{1}{2}(x^2+y^2)$, we have
$$
\frac{d}{dt}V(x(t),y(t)) = \alpha ( x(t)^2+y(t)^2 ) ( 1-x(t)^2-y(t)^2 )
$$
and one can see that $\frac{d}{dt}V(x(t),y(t))$ is positive inside the unit disk (except at the origin), and negative outside of the unit disk. It is then easy to design (by translation) Lyapunov functions inside the punctured unit disk, and outside of the unit disk, and to conclude by the LaSalle principle.
\end{example}

\begin{example}[Lyapunov lemma and applications]\label{Lyapunov_lemma}
Let $A$ be a $n\times n$ real matrix, whose eigenvalues have negative real parts. Then there exists a symmetric positive definite $n\times n$ matrix $P$ such that $A^\top P+PA=-I_n$.
Indeed, it suffices to take $P=\int_0^{+\infty}{e}^{tA^\top}{e}^{tA}\,dt$.

Clearly, the function $V(x)=\langle x,Px\rangle$ is then a strict Lyapunov function for the system $\dot x(t)=Ax(t)$. We recover the fact that $0$ is GAS.

Using a first-order expansion and norm estimates, it is then easy to prove Theorem \ref{thm_lineariz} and even to obtain stability neighborhoods thanks to $V$.
\end{example}

\paragraph{Application to the stabilization of nonlinear control systems.}
As Theorem \ref{thm_lineariz} implied Theorem \ref{thm_linearization}, the Lyapunov and LaSalle theorems can be applied to control systems, providing knowledge on the stability neighborhoods. 
For instance, we get the following statement: consider the nonlinear control system \eqref{contsyststab} and assume that there exists a function $V:\mathcal{D}\rightarrow\R^+$ of class $C^1$, taking positive values in $\mathcal{D}\setminus\{\bar x\}$, such that for every $x\in\mathcal{D}$ there exists $u(x)\in\Omega$ such that $\langle \nabla V(x),f(x,u(x))\rangle < 0$; then the feedback control $u$ stabilizes the control system, globally in $\mathcal{D}$. 

Many other similar statements can be easily derived, based on the Lyapunov or LaSalle theorems.
There exists an elaborate theory of control Lyapunov functions (see, e.g., \cite{Sontag}).
A difficulty in this theory is to ensure a nice regularity of the feedback control. We do not discuss further this difficult question but we mention that it has raised a whole field of intensive researches (see \cite{Coron, Sontag}).

To illustrate the role of the Lyapunov functions in stabilization, and as an application of the idea described above, we next provide a spectacular and widely used (and however very simple) strategy in order to design stabilizing controls thanks to Lyapunov functions.

%\paragraph{Jurdjevic-Quinn method.}

\begin{theorem}[Jurdjevic-Quinn method \cite{JurdjevicQuinn}]\label{thm_JQ}
Consider the control-affine system in $\R^n$\\[-1mm]
$$
\dot{x}(t)=f(x(t))+\sum_{i=1}^mu_i(t)g_i(x(t))
$$
where $f$ and the $g_i$'s are smooth vector fields in $\R^n$. Let $\bar x$ be such that $f(\bar{x})=0$, i.e., $\bar x$ is an equilibrium point of the uncontrolled system (that is, with $u_i=0$).

We assume that there exists a proper Lyapunov function $V$ at $\bar x$ on $\R^n$ for the uncontrolled system, i.e., satisfying
\begin{itemize}[parsep=1mm,itemsep=1mm,topsep=1mm]%,leftmargin=*
\item $V(\bar{x})=0$ and $V(x)>0$ for every $x\in\R^n\setminus\{\bar{x}\}$;
\item $V$ is proper;
\item $L_fV(x)= \langle \nabla V(x),f(x)\rangle \leq 0$ for every $x\in\R^n$;\footnote{The notation $L_fV$ is called the Lie derivative of $V$ along $f$. It is defined by $L_fV(x)=dV(x).f(x)=\langle \nabla V(x),f(x)\rangle$, which is the derivative of $V$ along the direction $f$ at $x$.}
\item the set\\[-1mm]
$$\{x\in\R^n\ \vert\ L_fV(x)=0\ \textrm{and}\
L_f^kL_{g_i}V(x)=0\quad \forall i\in\{1,\ldots,m\}\quad \forall
k\in\N\}$$
is reduced to the singleton $\{\bar{x}\}$.
\end{itemize}
Then the equilibrium point $\bar x$ is GAS in $\R^n$ for the control system in closed-loop with the feedback control defined by $u_i(x)=-L_{g_i}V(x)$, $i=1,\ldots,m$.
\end{theorem}

\begin{proof}
Let $F(x)=f(x)-\sum_{i=1}^m L_{g_i}V(x)g_i(x)$ be the dynamics of the closed-loop system. First of all, we note that $F(\bar{x})=0$, that is, $\bar{x}$ is an equilibrium point for the closed-loop system. Indeed, $V$ is smooth and reaches its minimum at $\bar{x}$, hence $\nabla V(\bar{x})=0$, and therefore $L_{g_i}V(\bar{x})=0$ for $i=1,\ldots,m$. Moreover, we have $f(\bar{x})=0$. Besides, we have
$$
L_FV(x)=\langle\nabla V(x),F(x)\rangle = L_fV(x)-\sum_{i=1}^m
(L_{g_i}V(x))^2 \leq 0
$$
and if $L_FV(x(t))=0$ for every $t\geq 0$, then $L_fV(x(t))=0$ and $L_{g_i}V(x(t))=0$, $i=1,\ldots,m$. Derivating with respect to $t$, we infer that
$$
0=\frac{d}{dt}L_{g_i}V(x(t))=L_fL_{g_i}V(x(t))
$$
since $L_{g_i}V(x(t))=0$. Therefore, clearly, we get that $L_f^kL_{g_i}V(x(t))=0$, for every $i\in\{1,\ldots,m\}$ and for every $k\in\N$. By assumption, it follows that $x(t)=\bar{x}$, and the conclusion follows from the LaSalle principle.
\end{proof}

The idea of the Jurdjevic-Quinn method, that can be seen in the above proof, is very simple. The uncontrolled system has a Lyapunov function, which may be not strict. Then, to get an asymptotic stabilization result, we compute the derivative of $V$ along solutions of the control system, we see that the control enters linearly into the resulting expression, and we design the controls so as to get the desired decrease.

\begin{remark}
It is remarkable that the Jurdjevic-Quinn method also allows one to design globally stabilizing feedback controls, which satisfy moreover some constraints. For instance, in the framework of Theorem \ref{thm_JQ}, let us add the requirement that $\vert u_i\vert\leq 1$, $i=1,\ldots,m$. Then, with the feedback
\begin{equation*}
\begin{split}
u_i &= \textrm{sat}(-1,-L_{g_i}V(x),1)  
= \left\{\begin{array}{lcl}
-1 & \textrm{if} & -L_{g_i}V(x)\leq -1 \\
-L_{g_i}V(x) & \textrm{if} & -1\leq -L_{g_i}V(x)\leq 1 \\
1 & \textrm{if} & -L_{g_i}V(x)\geq 1 
\end{array}\right.
\end{split}
\end{equation*}
the equilibrium point $\bar x$ is GAS.
Indeed, the above proof is easily adapted, and the dynamics $F(x)$ of the closed-loop system is locally Lipschitz.
\end{remark}

The Jurdjevic-Quinn method is much used, for instance, for the stabilization of satellites along a given orbit.
We next give some applications in mathematical biology (control of populations in Lotka-Volterra systems).

\begin{example}
Consider the controlled predator-prey system
\begin{equation*}
\dot{x}=x(1-y+u), \quad \dot{y}=-y(1-x).
\end{equation*}
and the equilibrium point $(x=1,y=1)$. Prove that the function $V(x,y)=x-1-\ln(x)+y-1-\ln(y)$ satisfies all assumptions of Theorem \ref{thm_JQ}, and deduce a feedback control such that the equilibrium point is GAS in $x>0$, $y>0$. Note that the function $x\mapsto x-1-\ln(x)$ is nonnegative on $(0,+\infty)$ and vanishes only at $x=1$.
\end{example}

\begin{example}[Generalized Lotka-Volterra system]
%% inspir\'e de Harrison_1978
Consider the generalized Lotka-Volterra system
$$
\dot N_i = N_i \bigg( b_i + \sum_{j=1}^n a_{ij} N_j \bigg),\quad i=1,\ldots,n.
$$
Consider the equilibrium point $\bar N=(\bar N_1,\ldots,\bar N_n)^T$ defined by $b+A\bar N=0$, where $b=(b_1,\ldots,b_n)^\top$ and $A$ is the square matrix of coefficients $a_{ij}$. Let $c_1,\ldots,c_n$ be some real numbers. Let $C$ be the diagonal matrix whose coefficients are the $c_i$'s. We set
$$
V(N) = \sum_{i=1}^n c_i \left( N_i-\bar N_i - \bar N_i\ln\frac{N_i}{\bar N_i} \right) .
$$
An easy computation shows that
$$
\frac{d}{dt} V(N(t)) = \sum_{i=1}^n c_i (N_i-\bar N_i) (b_i+(AN)_i)
$$
where $(AN)_i$ is the $i^\textrm{th}$ component of the vector $AN$. By noticing that $b_i+(A\bar N)_i=0$, we easily deduce that
$$
\frac{d}{dt} V(N(t)) = \frac{1}{2} \langle N-\bar N , (A^\top C+CA) (N-\bar N) \rangle .
$$
If there exists a diagonal matrix $C$ such that $A^\top C+CA$ be negative definite, then we infer that $\bar N$ is GAS.\footnote{Note that a necessary condition for $A^\top C+CA$ to be negative definite is that the diagonal coefficients $a_{ii}$ of $A$ be negative. If at least one of them is zero then $A^\top C+CA$ is not definite.} Assume for instance that $A$ be skew-symmetric, and take $C=I_n$. Then $V(N(t))$ is constant. We introduce some controls, for instance, in the $n-1$ first equations:
$$
\dot N_i = N_i \bigg( b_i + \sum_{j=1}^n a_{ij} N_j + \alpha_i u_i \bigg),\quad i=1,\ldots,n-1.
$$
Then, we compute $\frac{d}{dt} V(N(t)) = \sum_{i=1}^{n-1} \alpha_i (N_i(t)-\bar N_i) u(t) $. It is then easy to design a feedback control stabilizing globally (by the LaSalle principle) the system to the equilibrium $\bar N$, under the assumption that at least one of the coefficients $a_{in}$, $i=1,\ldots,n-1$, be nonzero, and that $A$ be invertible.

In the particular case $n=2$, it is actually possible to ensure moreover that $u(t)\geq 0$, by playing with the periodicity of the trajectories (as in Example \ref{exo7.3.22}).
\end{example}

\begin{example}
We consider the bilinear control system in $\R^2$
$$
\dot x_1(t) = x_2(t) , \qquad
\dot x_2(t) = -x_1(t) + u(t) x_1(t) 
$$
where the control is subject to the constraint $\vert u\vert\leq 1$. 
Let us stabilize globally this system to $(0,0)$.
Setting $V(x_1,x_2)=\frac{1}{2}(x_1^2+x_2^2)$, we have $\frac{d}{dt}V(x_1(t),x_2(t)) = u(t)x_1(t)x_2(t)$. We choose the feedback $u(x_1,x_2) = \mathrm{sat}(-1,-x_1x_2,1)$. %, we have $\frac{d}{dt}V(x_1(t),x_2(t)) = - \min(\vert x_1(t)x_2(t)\vert, x_1(t)^2x_2(t)^2)\leq 0$. 
To prove the asymptotic stability we apply the LaSalle invariance principle. If $\dot V\equiv 0$ then either $x_1\equiv 0$ (and then by derivation we also have $x_2\equiv 0$), or $x_2\equiv 0$ and then by derivation we also find $x_1\equiv 0$. In all cases the maximal invariant set is $\{0,0\}$, which yields the conclusion.
\end{example}

\begin{example}
Consider the control system
$$
\dot x(t) = -y(t) + v(t)\cos\theta(t) , \quad
\dot y(t) = x(t) + v(t)\sin\theta(t) , \quad
\dot\theta(t) = u(t),
$$
where the controls $u$ and $v$ are subject to the constraints $\vert u\vert\leq 1$ and $\vert v\vert\leq 1$.
Let us stabilize globally this system to $(0,0,0)$.
Setting $V=\frac{1}{2}(x^2+y^2+\theta^2)$, we have $\dot V = v(x\cos\theta+y\sin\theta)+\theta u$. We choose the feedback controls
$$
v = -\mathrm{sat}(-1,x\cos\theta+y\sin\theta,1) \qquad\textrm{et}\qquad u = -\mathrm{sat}(-1,\theta,1)
$$
If $\dot V=0$ along a trajectory then $\theta=0$, $u=0$, $0=x\cos\theta+y\sin\theta=x$, $v=0$, and thus also $0=\dot x=-y$. Therefore the invariant set in the LaSalle principle is reduced to the equilibrium point. This yields the conclusion.
\end{example}

\part{Control in infinite dimension}\label{part2}
%\noindent Use the template \emph{part.tex} together with the Springer document class SVMono (monograph-type books) or SVMult (edited books) to style your part title page and, if desired, a short introductory text (maximum one page) on its verso page in the Springer layout.

\noindent In this part we introduce the control theory for infinite-dimensional systems, that is, control systems
% $\dot{x}(t)=f(x(t),u(t))$ 
where the state $x(t)$ evolves in an infinite-dimensional Banach space. Controlled partial differential equations enter into this category.

As we will see, the tools used to analyze such systems significantly differ from the ones used in finite dimension. The techniques are strongly use tools of functional analysis. The reader should then be quite well acquainted with such knowledge, and we refer to the textbook \cite{Brezis} on functional analysis.

The study of the control of nonlinear partial differential equations is beyond the scope of the present monograph, and we refer the reader to \cite{Coron} for a complete survey.
Throughout this part, except at the end, we will focus on linear autonomous infinite-dimensional control systems of the form
$\dot{x}(t)=Ax(t)+Bu(t)$
where $A$ and $B$ are operators (which can be viewed, at a first step, as infinite-dimensional matrices).

Since such systems involve partial differential equations, throughout this part the state $x(t)$ will be rather denoted as $y(t)$. For PDEs settled on some domain $\Omega$, $y$ is a function of $t$ and $x$, where $t$ is the time and $x$ the spatial variable. The control system considered throughout is
\begin{equation}\label{sys_part2}
\dot y = Ay + Bu
\end{equation}
where $\dot y$ means $\partial_t y(t,x)$ when $y$ is a function of $(t,x)$.

The first step is to define the concept of a solution, which in itself is far from being obvious, in contrast to the finite-dimensional setting. In infinite dimension the exponential of $tA$ is is replaced with the concept of \textit{semigroup}.
Hence, in this part, a whole chapter is devoted to semigroup theory, with the objective of giving a rigorous sense to the solution of \eqref{sys_part2} with $y(0)=y_0$,
$$
y(t) = S(t)y_0 + \int_0^t S(t-s)Bu(s)\, ds
$$
where $S(t)$ is a semigroup, generalizing $\mathrm{e}^{tA}$.

\medskip

There are plenty of ways to introduce the theory of controlled PDEs. Here, one of our main objectives is to provide the general framework in which the Hilbert Uniqueness Method (HUM) of J.-L. Lions can be stated.

\chapter{Semigroup theory}\label{chap_semigroup}

The objective of this chapter is to establish that, in an appropriate functional setting, there is a unique solution of the Cauchy problem
\begin{equation}\label{sys4.1}
\dot{y}(t) = Ay(t)+f(t) , \qquad
y(0) = y_0 ,
\end{equation}
where $A$ is a linear operator on a Banach space $X$, and where $y(t)$ and $f(t)$ evolve in $X$, which is given by
\begin{equation}\label{f4.2}
y(t) = S(t) y_0 + \int_0^t S(t-s)f(s)\, ds
\end{equation}
where $(S(t))_{t\geq 0}$ is the semigroup generated by the operator $A$.

In finite dimension (that is, if $X=\R^n$), this step is easy and one has $S(t)=\mathrm{e}^{tA}$, with the usual matrix exponential. In infinite dimension this step is far from being obvious and requires to define rigorously the concept of (unbounded) operator and of semigroup. The reader can keep in mind the example where the operator $A$ is the Dirichlet-Laplacian, defined on a domain $\Omega$ of $\R^n$.

Most results of the present chapter are borrowed from the textbooks \cite{Nagel} and \cite{Pazy} on semigroup theory and from \cite{TucsnakWeiss}, and are given without proof. 

\medskip

Let us recall several basic notions of functional analysis that are instrumental in what follows (see \cite{Brezis}).

Let $X$ be a Banach space, endowed with a norm denoted by $\Vert\ \Vert_{X}$, or simply by $\Vert\ \Vert$ when there is no ambiguity. Let $Y$ be another Banach space. 
The norm of a bounded (i.e., continuous) linear mapping $g:X\rightarrow Y$ is denoted as well by $\Vert g\Vert$ and is defined as usual by
$$\Vert g\Vert = \sup_{x\in X\setminus\{0\}} \frac{\Vert g(x)\Vert_Y}{\Vert x\Vert_X}.$$
The set of bounded linear mappings from $X$ to $Y$ is denoted by $L(X,Y)$.

The notation $X'$ stands for the (topological) dual of the Banach space $X$, that is, the vector space of all linear continuous mappings $\ell:X\rightarrow\R$ (in other words, $X'=L(X,\R)$). Endowed with the norm of linear continuous forms defined above, it is a Banach space.
The duality bracket is defined as usual by
$\langle\ell,x\rangle_{X',X} = \ell(x)$, for every $\ell\in X'$ and every $x\in X$.

In what follows, the word \textit{operator} is a synonym for \textit{mapping}.
By definition, an unbounded linear operator $A$ from $X$ to $Y$ is a linear mapping $A:D(A)\rightarrow Y$ defined on a vector subspace $D(A)\subset X$ called the \textit{domain} of $A$. The operator $A$ is said to be \textit{bounded}\footnote{Note that the terminology is paradoxal, since an unbounded linear operator can be bounded! Actually, "unbounded operator" usually underlies that $A$ is defined on a domain $D(A)$ that is a proper subset of $X$.} if $D(A)=X$ and if there exists $C>0$ such that $\Vert Ax\Vert_Y\leq C\Vert x\Vert_X$ for every $x\in D(A)$.

The operator $A:D(A)\subset X\rightarrow Y$ is said to be \textit{closed} whenever its \textit{graph}
$$
G(A) = \{ (x,Ax)\ \vert\ x\in D(A) \}
$$
is a closed subset of $X\times Y$. By the closed graph theorem, $A$ is a continuous linear mapping from $X$ to $Y$ if and only if $D(A)=X$ and $G(A)$ is closed.

Let $A:D(A)\subset X\rightarrow Y$ be a densely defined linear operator (that is, $D(A)$ is dense in $X$).
The \textit{adjoint operator} $A^*:D(A^*)\subset Y'\rightarrow X'$ is defined as follows. We set
$$
D(A^*)=\{ z\in Y'\ \vert\ \exists C\geq 0\ \textrm{such that}\ \forall x\in D(A)\ \ \vert\langle z,Ax\rangle_{Y',Y}\vert\leq C\Vert x\Vert_{X} \}.
$$
Then $D(A^*)$ is a vector subspace of $Y'$. For every $z\in D(A^*)$, we define the linear form $\ell_z:D(A)\rightarrow\R$ by $\ell_z(x)=\langle z,Ax\rangle_{Y',Y}$ for every $x\in D(A)$. By definition of $D(A^*)$ we have $\vert\ell_z(x)\vert\leq C\Vert x\Vert_{X}$ for every $x\in D(A)$. Since $D(A)$ is dense in $X$, it follows that the linear form $\ell_z$ can be extended in a unique way to a continuous linear form on $X$, denoted by $\tilde\ell_z\in X'$ (classical continuous extension argument of uniformly continuous mappings on complete spaces). Then we set $A^*z=\tilde\ell_z$. This defines the unbounded linear operator $A^*:D(A^*)\subset Y'\rightarrow X'$, called the \textit{adjoint} of $A$. The fundamental property of the adjoint is that
$$
\langle z, Ax\rangle_{Y',Y} = \langle A^*z, x\rangle_{X',X}
$$
for every $x\in D(A)$ and every $z\in D(A^*)$.
Note that:
\begin{itemize}[parsep=1mm,itemsep=0mm,topsep=1mm]%,leftmargin=*
\item $A$ is bounded if and only if $A^*$ is bounded, and in this case their norms are equal;
\item $A^*$ is closed;
\item $D(A^*)$ is not necessarily dense in $Y'$ (even if $A$ is closed), however if $A$ is closed and if $Y$ is reflexive then $D(A^*)$ is dense in $Y'$. % (see \cite{Brezis}).
\end{itemize}

In the case where $X$ is a Hilbert space, we identify $X'$ with $X$. A densely defined linear operator $A:D(A)\subset X\rightarrow X$ is said to be \textit{self-adjoint} (resp. \textit{skew-adjoint}) whenever $D(A^*)=D(A)$ and $A^*=A$ (resp., $A^*=-A$). Note that self-adjoint and skew-adjoint operators are necessarily closed.

More generally, given two Hilbert spaces $X$ and $Z$ such that $Z\hookrightarrow X$, i.e., $Z$ is continuously embedded in $X$, we have $X'\hookrightarrow Z'$. Now we can decide, by the Riesz theorem, to identify $X$ with $X'$; in this case, we have the triple $Z\hookrightarrow X\hookrightarrow Z'$ (but then we cannot identify $Z$ with $Z'$).
Then, for any $x\in X\subset Z'$ and any $z\in Z\subset X$, we have $\langle x,z\rangle_{Z',Z} = (x,z)_X$, i.e., the duality bracket $\langle\ ,\ \rangle_{Z',Z}$, standing for the application of a linear continuous mapping on $Z$ (that is an element of $Z'$) to an element of $Z$ is identified to the scalar product on $X$, when both elements are in $X$. We say that $X$ is the \emph{pivot space}.

\medskip

Throughout this part, we consider a Banach space $X$. An operator on $X$ will mean an (unbounded) linear operator $A:D(A)\subset X\rightarrow X$. In practice, most of unbounded operators used to model systems of the form \eqref{sys4.1} are operators $A:D(A)\subset X\rightarrow X$ that are closed and whose domain $D(A)$ is dense in $X$.

When an integral is considered over the Banach space $X$ (like in \eqref{f4.2}), it is understood that it is in the usual sense of the Bochner integral (see \cite{Rudin,TucsnakWeiss}).

\section{Homogeneous Cauchy problems}
We first focus on the homogeneous Cauchy problem, that is \eqref{sys4.1} with $f=0$, with the objective of giving a sense to the unique solution $y(t)=S(t)y_0$.

\subsection{Semigroups of linear operators}
In the sequel, the notation $\mathrm{id}_X$ stands for the identity mapping on $X$.

\begin{definition}\label{def_semigroup}
A \textit{$C_0$ semigroup of bounded linear operators} on $X$ is a one-parameter family $(S(t))_{t\geq 0}$ of bounded linear mappings $S(t)\in L(X)$ such that
\begin{enumerate}[parsep=1mm,itemsep=1mm,topsep=1mm]%,leftmargin=*
\item $S(0)=\mathrm{id}_X$;
\item $S(t+s)=S(t)S(s)$ for all $(t,s)\in[0,+\infty)^2$ (semigroup property);
\item $\displaystyle\lim_{t\rightarrow 0,\ t>0}\ S(t)y=y$ for every $y\in X$.
\end{enumerate}
The linear operator $A:D(A)\subset X\rightarrow X$, defined by
$$
Ay=\lim_{t\rightarrow 0^+}\frac{S(t)y-y}{t}
$$
on the domain $D(A)$ that is the set of $y\in X$ such that the above limit (computed in $X$, i.e., with the norm $\Vert\cdot\Vert_X$) exists,
is called the \textit{infinitesimal generator} of the semigroup $(S(t))_{t\geq 0}$.

The semigroup is said to be a group if the second property holds true for all $(t,s)\in\R^2$.
\end{definition}

\begin{proposition}\cite[Section 1.2]{Pazy}\label{propclosed}
Let $(S(t))_{t\geq 0}$ be a $C_0$ semigroup. Then
\begin{itemize}[parsep=1mm,itemsep=1mm,topsep=1mm]%,leftmargin=*
\item the mapping $t\in [0,+\infty)\mapsto S(t)y$ is continuous for every $y\in X$;
\item $A$ is closed and $D(A)$ is dense in $X$;
\item $S(t)y\in D(A)$ and $\dot S(t)y=AS(t)y$ for all $y\in D(A)$ and $t>0$.
\end{itemize}
\end{proposition}

This proposition shows that the notion of $C_0$ semigroup is adapted to solve the homogeneous Cauchy problem.

%Note that (it is easy to see) if $S(t)$ is a $C_0$ semigroup then the mapping $t\mapsto S(t)x$ is continuous for every $x\in X$. Moreover, for every $x\in D(A)$, one has $S(t)x\in D(A)$ and $\frac{d}{dt}S(t)x=AS(t)x$ for every $t\geq 0$. Indeed, $\frac{S(h)-\mathrm{id}_X}{h} S(t) x = S(t) \frac{S(h)-\mathrm{id}_X}{h} x \rightarrow S(t) Ax$ as $h\rightarrow 0$, whence $S(t)x \in D(A)$ and the conclusion.
%This simple property already shows that the notion of $C_0$ semigroup is adapted to solve the homogeneous Cauchy problem.

Actually, more generally, semigroups are defined with the two first items of Definition \ref{def_semigroup} (see \cite{Pazy}). The additional third property characterizes so-called $C_0$ semigroups (also called strongly continuous semigroups).
It is a simple convergence property. If the $C_0$ semigroup $(S(t))_{t\geq 0}$ satisfies the stronger (uniform convergence) property $\lim_{t\rightarrow 0,\ t>0}\ \Vert S(t)-\mathrm{id}_X\Vert=0$, then it is said to be \textit{uniformly continuous}.
The following result is however proved in \cite[Theorem 1.2]{Pazy}:

\begin{quotation}
\noindent\textit{A linear operator $A:D(A)\rightarrow X$ is the infinitesimal generator of a uniformly continuous semigroup if and only if $A$ is bounded and $D(A)=X$. In that case, moreover, $S(t) = \mathrm{e}^{tA} = \sum_{n=0}^{+\infty} \frac{t^n}{n!}A^n$.}
\end{quotation}

\noindent This result shows that, as soon as a given control process in infinite dimension involves unbounded operators (i.e., $D(A)\subsetneq X$, like for instance a PDE with a Laplacian), then the underlying semigroup is not uniformly continuous. Actually as soon as an operator involves a derivation then it is unbounded. In what follows we focus on $C_0$ semigroups.

\begin{proposition}\cite[Section 1.2]{Pazy}
Let $(S(t))_{t\geq 0}$ be a $C_0$ semigroup. There exist $M\geq 1$ and $\omega\in\R$ such that 
\begin{equation}\label{inegsemigroup}
\Vert S(t)\Vert\leq M\mathrm{e}^{\omega t} \qquad\forall t\geq 0 .
\end{equation}
We say that $(S(t))_{t\geq 0}\in\mathcal{G}(M,\omega)$.
The infimum $\omega^*$ of all possible real numbers $\omega$ such that \eqref{inegsemigroup} is satisfied for some $M\geq 1$ is the growth bound of the semigroup, and is given by
$$
\omega^* = \inf_{t>0} \frac{1}{t} \ln \Vert S(t)\Vert.
$$
\end{proposition}

This proposition shows that $C_0$ semigroups have at most an exponential growth in norm. This property is similar to what happens in finite dimension.

\begin{definition}
Let $A:D(A)\rightarrow X$ be a linear operator on $X$ defined on the domain $D(A)\subset X$. The \textit{resolvent set} $\rho(A)$ of $A$ is defined as the set of complex numbers $\lambda$ such that $\lambda\, \mathrm{id}_X - A:D(A)\rightarrow X$ is invertible and $(\lambda\, \mathrm{id}_X - A)^{-1}:X\rightarrow X$ is bounded (we say it is boundedly invertible).
The \textit{resolvent} of $A$ is defined by
$R(\lambda,A)=(\lambda\, \mathrm{id}_X-A)^{-1}$,
for every $\lambda\in\rho(A)$.
\end{definition}

Notice the so-called \textit{resolvent identity}, often instrumental in some proofs:
$$
R(\lambda,A)-R(\mu,A)=(\mu-\lambda)R(\lambda,A)R(\mu,A)\qquad \forall(\lambda,\mu)\in\rho(A)^2 .
$$

\begin{remark}
If $(S(t))_{t\geq 0}\in\mathcal{G}(M,\omega)$ then $\{ \lambda\in\C\ \vert\ \Re\lambda>\omega \}\subset\rho(A)$, and
\begin{equation}\label{laplacetransform}
R(\lambda,A)y=(\lambda\, \mathrm{id}_X-A)^{-1}y=\int_0^{+\infty}\mathrm{e}^{-\lambda t} S(t) y\, dt
\end{equation}
for every $x\in X$ and every $\lambda\in\C$ such that $\Re\lambda>\omega$ (Laplace transform). Indeed, integrating by parts, one has
$
\lambda\int_0^{+\infty} \mathrm{e}^{-\lambda t} S(t)\, dt = \mathrm{id}_X + A \int_0^{+\infty} \mathrm{e}^{-\lambda t} S(t)\, dt 
$
and thus 
$
(\lambda\,\mathrm{id}_X - A) \int_0^{+\infty} \mathrm{e}^{-\lambda t} S(t)\, dt = \mathrm{id}_X .
$
\end{remark}

Note that, using the expression \eqref{laplacetransform} of $R(\lambda,A)$ with the Laplace transform, it follows that, if $(S(t))_{t\geq 0}\in\mathcal{G}(M,\omega)$ then $\Vert R(\lambda,A)\Vert \leq \frac{M}{\Re\lambda-\omega}$ for every $\lambda\in\C$ such that $\Re\lambda>\omega$. This can be iterated, by derivating $R(\lambda,A)$ with respect to $\lambda$, using the resolvent formula and the Laplace transform, and this yields the estimates $\Vert R(\lambda,A)^n\Vert \leq \frac{M}{(\Re\lambda-\omega)^n}$ for every $n\in\N^*$ and for every $\lambda\in\C$ such that $\Re\lambda>\omega$.
Actually, we have the following general result.

\begin{theorem}\cite[Section 1.5, Theorem 5.2]{Pazy}\label{thm4.1}
A linear operator $A:D(A)\subset X\rightarrow X$ is the infinitesimal generator of a $C_0$ semigroup $(S(t))_{t\geq 0}\in\mathcal{G}(M,\omega)$ if and only if the following conditions are satisfied:
\begin{itemize}[parsep=1mm,itemsep=1mm,topsep=1mm]%,leftmargin=*
\item $A$ is closed and $D(A)$ is dense in $X$;
\item $(\omega,+\infty)\subset\rho(A)$ and
$\Vert R(\lambda,A)^n\Vert\leq\frac{M}{(\Re\lambda-\omega)^n}$
for every $n\in\N^*$ and every $\lambda\in\C$ such that $\Re\lambda >\omega$.
\end{itemize}
\end{theorem}

\paragraph{\bf Particular case: contraction semigroups.}

\begin{definition}
Let $(S(t))_{t\geq 0}$ be a $C_0$ semigroup. Assume that $(S(t))_{t\geq 0}\in\mathcal{G}(M,\omega)$ for some $M\geq 1$ and $\omega\in\R$.
If $\omega\leq 0$ and $M=1$ then $(S(t))_{t\geq 0}$ is said to be a \textit{semigroup of contractions}.
\end{definition}

Semigroups of contractions are of great importance and cover many applications. They are mostly considered in many textbooks (such as \cite{Brezis,CazenaveHaraux}) and in that case Theorem \ref{thm4.1} takes the more specific following forms, which are the well-known Hille-Yosida and Lumer-Phillips theorems.

\begin{theorem}[Hille-Yosida theorem, {\cite[Section 1.3, Theorem 3.1]{Pazy}}]\label{thm_HilleYosida}
A linear operator $A:D(A)\subset X\rightarrow X$ is the infinitesimal generator of a $C_0$ semigroup of contractions if and only if the following conditions are satisfied:
\begin{itemize}[parsep=1mm,itemsep=1mm,topsep=1mm]%,leftmargin=*
\item $A$ is closed and $D(A)$ is dense in $X$;
\item $(0,+\infty)\subset\rho(A)$ and
$\Vert R(\lambda,A)\Vert\leq\frac{1}{\lambda}$ for every $\lambda>0$.
\end{itemize}
\end{theorem}

\begin{remark}
The latter condition can equivalently be replaced by:
$\{\lambda\in\C\ \vert\ \Re\lambda>0 \}\subset\rho(A)$ and $\Vert R(\lambda,A)\Vert\leq\frac{1}{\Re\lambda}$ for every $\lambda\in\C$ such that $\Re\lambda>0$.
\end{remark}

\begin{remark}
Let $A:D(A)\rightarrow X$ generating a $C_0$ semigroup $(S(t))_{t\geq 0}\in\mathcal{G}(1,\omega)$. Then the operator $A_\omega=A-\omega\,\mathrm{id}_X$ (having the same domain) is the infinitesimal generator of $S_\omega(t)=\mathrm{e}^{-\omega t}S(t)$ which is a semigroup of contractions (and conversely). In particular, we obtain the following corollary.

\begin{corollary}
A linear operator $A:D(A)\subset X\rightarrow X$ is the infinitesimal generator of a $C_0$ semigroup $(S(t))_{t\geq 0}\in\mathcal{G}(1,\omega)$ if and only if the following conditions are satisfied:
\begin{itemize}[parsep=1mm,itemsep=1mm,topsep=1mm]%,leftmargin=*
\item $A$ is closed and $D(A)$ is dense in $X$;
\item $(\omega,+\infty)\subset\rho(A)$ and
$\Vert R(\lambda,A)\Vert\leq\frac{1}{\lambda-\omega}$ for every $\lambda>\omega$.
\end{itemize}
\end{corollary}
\end{remark}

Before providing the statement of the Lumer-Phillips theorem, which is another characterization of $C_0$ semigroups of contractions, let us recall some important definitions.

For every $y\in X$, we define
$F(y) = \{ \ell\in X' \ \vert\ \langle\ell,y\rangle_{X',X} = \Vert y\Vert_X^2 = \Vert\ell\Vert_{X'}^2 \}$.
It follows from the Hahn-Banach theorem that $F(y)$ is nonempty.
In the important case where $X$ is a Hilbert space, one has $y\in F(y)$ (identifying $X'$ with $X$).

\begin{definition}
The operator $A:D(A)\subset X\rightarrow X$ is said to be:
\begin{itemize}[parsep=1mm,itemsep=1mm,topsep=1mm]%,leftmargin=*
\item \textit{dissipative} if for every $y\in D(A)$ there exists an element $\ell\in F(y)$ such that $\Re\langle\ell,Ay\rangle_{X',X}\leq 0$;

\item \textit{m-dissipative} if it is dissipative and $\mathrm{Ran}(\mathrm{id}_X-A)=(\mathrm{id}_X-A)D(A)=X$.
\end{itemize}
\end{definition}

If $X$ is a Hilbert space, then $A$ is dissipative if and only if $\Re(y,Ay)_{X}\leq 0$ for every $y\in D(A)$, where $(\ ,\ )_X$ is the scalar product of $X$.

Other names are often used in the existing literature (see \cite{Brezis,CazenaveHaraux,Rudin}): $A$ is dissipative if and only if $A$ is \textit{accretive}, if and only if $-A$ is \textit{monotone}, and $A$ is m-dissipative if and only if $-A$ is \textit{maximal monotone} (the letter \textit{m} stands for \textit{maximal}).

\begin{remark}
If $A:D(A)\rightarrow X$ is m-dissipative then $\mathrm{Ran}(\lambda\,\mathrm{id}_X-A)=X$ for every $\lambda>0$.
\end{remark}

\begin{remark}
Let $A:D(A)\rightarrow X$ be a m-dissipative operator.
If $X$ is reflexive\footnote{There is a canonical injection $\iota:X\rightarrow X''$ (the bidual of the Banach space $X$), defined by $\langle \iota y,\ell\rangle_{X'',X'} = \langle\ell,y\rangle_{X',X}$ for every $y\in X$ and every $\ell\in X'$, which is a linear isometry, so that $X$ can be identified with a subspace of $X''$. The Banach space $X$ is said to be \textit{reflexive} whenever $\iota(X)=X''$; in this case, $X''$ is identified with $X$ with the isomorphism $\iota$.} then $A$ is closed and densely defined (i.e., $D(A)$ is dense in $X$).
\end{remark}

\begin{theorem}[Lumer-Phillips theorem, {\cite[Section 1.4, Theorem 4.3]{Pazy}}]\label{thm_LumerPhillips}
Let $A:D(A)\subset X\rightarrow X$ be a densely defined closed linear operator. Then $A$ is the infinitesimal generator of a $C_0$ semigroup of contractions if and only if $A$ is m-dissipative.
\end{theorem}

Note that it is not necessary to assume that $A$ is closed and densely defined in this theorem if $X$ is reflexive.
A statement which is very often useful is the following one (see \cite[Section 1.4, Corollary 4.4]{Pazy}).

\begin{proposition}\label{prop4.3}
Let $A:D(A)\subset X\rightarrow X$ be a densely defined operator. If $A$ is closed and if both $A$ and $A^*$ are dissipative then $A$ is the infinitesimal generator of a $C_0$ semigroup of contractions; the converse is true if $X$ is moreover reflexive.
\end{proposition}

\begin{remark}
If $A$ is skew-adjoint then it is closed and dissipative. Actually, $A$ is skew-adjoint if and only if $A$ and $-A$ are m-dissipative (see \cite[Proposition 3.7.2]{TucsnakWeiss}).
\end{remark}

\begin{example}\label{ex_heat}
Let $\Omega$ be an open subset of $\R^n$.
The Dirichlet-Laplacian operator $\triangle_D$ is defined on $D(\triangle_D) = \{ f\in H^1_0(\Omega)\ \vert\ \triangle f\in L^2(\Omega) \}$ by $\triangle_Df = \triangle f$ for every $f\in D(\triangle_D)$, where $\triangle$ is the usual Laplacian differential operator.
Note that $H^1_0(\Omega)\cap H^2(\Omega)\subset D(\triangle_D)$ and that, in general, the inclusion is strict. However, if $\Omega$ is an open bounded subset with $C^2$ boundary, or if $\Omega$ is a convex polygon of $\R^2$, then $D(\triangle_D) = H^1_0(\Omega)\cap H^2(\Omega)$ (see \cite[Chapter 1.5]{Grisvard}, see also \cite[Theorem 3.6.2 and Remark 3.6.6]{TucsnakWeiss}).

The operator $\triangle_D$ is self-adjoint and dissipative in $X=L^2(\Omega)$, hence by Proposition \ref{prop4.3} it generates a semigroup of contractions, called the \textit{heat semigroup}.

\medskip

The Dirichlet-Laplacian can be as well defined in the space $X=H^{-1}(\Omega)$ (which is the dual of $H^1_0(\Omega)$ with respect to the pivot space $L^2(\Omega)$). In that case, assuming that $\Omega$ is an open bounded subset of $\R^d$, one has $D(\triangle_D) = H^1_0(\Omega)$, and $\triangle_D:H^1_0(\Omega)\rightarrow H^{-1}(\Omega)$ is an isomorphism (see \cite[Section 3.6]{TucsnakWeiss}).
\end{example}

\begin{example}\label{ex_wave}
Anticipating a bit, let us study the operator underlying the wave equation. With the notations of Example \ref{ex_heat}, we define the operator
$$
A = \begin{pmatrix}
0 & \mathrm{id}_{H^1_0(\Omega)} \\
\triangle_D & 0
\end{pmatrix}
$$
on the domain $D(A) = D(\triangle_D) \times H^1_0(\Omega)$, in the Hilbert space $X=H^1_0(\Omega)\times L^2(\Omega)$.
Then it is easy to see that $A$ is closed, densely defined, skew-adjoint and thus m-dissipative, as well as $-A$. Hence $A$ and $-A$ both generate a semigroup of contractions, therefore $A$ generates a group of contractions. The fact that it is a group reflects the fact that the wave equation is time-reversible.
\end{example}

\begin{example}
Let $X$ be a Hilbert space, let $(e_n)_{n\in\N^*}$ be a Hilbert basis of $X$, and let $(\lambda_n)_{n\in\N^*}$ be a sequence of real numbers such that $\sup_{n\geq 1}\lambda_n<+\infty$ (this is satisfied if $\lambda_n\rightarrow -\infty$ as $n\rightarrow+\infty$). We define the operator
$$
Ay = \sum_{n=1}^{+\infty} \lambda_n (y,e_n)_Xe_n\quad\textrm{on}\quad
D(A) = \Big\{ y\in X\ \bigm|\ \sum_{n=1}^{+\infty}\lambda_n^2(y,e_n)_X^2 <+\infty \Big\}.
$$
Let us prove that $A$ is self-adjoint and generates the $C_0$ semigroup defined by
$$
S(t) y =\sum_{n=1}^{+\infty} \mathrm{e}^{\lambda_n t} (y,e_n)_X e_n.
$$

Firstly, noting that $X_p\subset D(A)$ for every $p\in\N^*$, with $X_p = \{ y\in X\ \mid\ \exists p\in \N^*\ \ \forall n\geq p\ \ (y,e_n)_X=0 \}$ and that $\cup_{p\in\N^*} X_p$ is dense in $X$, it follows that $D(A)$ is dense in $X$.

Secondly, let us prove that $A$ is closed. Let $(y_p)_{p\in\N^*}$ be a sequence of $D(A)$ such that $y_p\rightarrow y\in X$ and $A y_p\rightarrow z\in X$ as $p$ tends to $+\infty$. In particular $(y_p)_{p\in\N^*}$ is bounded in $X$ and thus there exists $M>0$ such that $\sum_{n=1}^{N} \lambda_n^2 (y_p,e_n)_X^2\leq M$, for every $p\in \N^*$ and every $N\in\N^*$. Letting $p$ tend to $+\infty$ then yields that $\sum_{n=1}^{+\infty} \lambda_n^2 (y,e_n)_X^2\leq M$, and thus $y\in D(A)$. It remains to prove that $z=Ay$. Since $Ay_p=\sum_{n=1}^{+\infty}\lambda_n(y_p,e_n)_Xe_n$, and since, for every $n\in\N^*$, $(Ay_p,e_n)_X = \lambda_n (y_p,e_n)_X$ converges to $\lambda_n(y,e_n)_X=(Ay,e_n)_X$, it follows that $Ay_p$ converges weakly to $Ay$ in $X$. Then by uniqueness of the limit it follows that $z=Ay$.

Now, let us prove that $\lambda\,\mathrm{id}_X-A$ is boundedly invertible if and only if $\inf_{n\geq 1}\vert\lambda-\lambda_n\vert>0$.
If $\inf_{n\geq 1}\vert\lambda-\lambda_n\vert>0$, we set
$$
A_\lambda y=\sum_{n=1}^{+\infty} \frac{1}{\lambda-\lambda_n} (y,e_n)_X e_n
$$
for every $y\in X$. Clearly, $A_\lambda:X\rightarrow X$ is linear and bounded, and one has $\mathrm{Ran}A_\lambda\subset D(A)$ and $(\lambda\,\mathrm{id}_X-A)A_\lambda=A_\lambda(\lambda\,\mathrm{id}_X-A)=\mathrm{id}_X$, hence $\lambda\in\rho(A)$ and $A_\lambda = (\lambda\,\mathrm{id}_X-A)^{-1}$.
Conversely, if $(\lambda\,\mathrm{id}_X-A)$ is boundedly invertible, then for every $n\in\N^*$ there exists $y_n\in X$ such that $(\lambda\,\mathrm{id}_X-A)y_n=e_n$, and the sequence $(y_n)_{n\in\N^*}$ is bounded. One has $y_n=\frac{1}{\lambda-\lambda_n}e_n$ and hence $\inf_{n\geq 1}\vert\lambda-\lambda_n\vert>0$.

It follows from these arguments that if $\inf_{n\geq 1}\vert\lambda-\lambda_n\vert>0$ then
$$
R(\lambda,A)^p y = \sum_{n=1}^{+\infty} \frac{1}{(\lambda-\lambda_n)^p} (y,e_n)_X e_n
$$
and hence 
$$
\Vert R(\lambda,A)^p \Vert \leq \sup_{n\geq 1} \frac{1}{\vert\lambda-\lambda_n\vert^p} = \left( \sup_{n\geq 1} \frac{1}{\vert\lambda-\lambda_n\vert} \right)^p .
$$

Let $\omega\geq\sup_{n\geq 1}\lambda_n$. Then for every $\lambda>\omega$ one has $\inf_{n\geq 1}\vert\lambda-\lambda_n\vert\geq\lambda-\omega$ and hence $\sup_{n\geq 1}\frac{\lambda-\omega}{\vert\lambda-\lambda_n\vert}\leq 1$.
Then, the conclusion follows from the Hille-Yosida theorem.

This example can be applied to a number of situations. Indeed any self-adjoint operator having a compact inverse (like the Dirichlet-Laplacian) is diagonalizable and the general framework of this example can then be applied.
\end{example}

\begin{example}
In order to model the 1D transport equation $\partial_t y+\partial_x y=0$, with $0\leq x\leq 1$, we define $X=L^2(\Omega)$ and the operator $Ay=-\partial_x y$ on the domain $D(A)=\{ y\in X\ \vert\ \partial_x y\in X,\ y(0)=0\}$. It is easy to prove that $A$ is closed and densely defined, and that $A$ and $A^*$ are dissipative (it can also be seen that $A-\lambda\,\mathrm{id}_X$ is surjective), and hence $A$ generates a $C_0$ semigroup of contractions.
\end{example}

Recall that (as said in the introduction of the chapter), given a densely defined operator $A:D(A)\rightarrow X$, the adjoint $A^*:D(A^*)\rightarrow X'$ is a closed operator, and if moreover $A$ is closed and $X$ is reflexive then $D(A^*)$ is dense in $X'$. Concerning the semigroup properties, note that, given a $C_0$ semigroup $(S(t))_{t\geq 0}$ on $X$, the adjoint family $(S(t)^*)_{t\geq 0}$ is a family of bounded operators on $X'$ satisfying the semigroup property but is not necessarily a $C_0$ semigroup.

\begin{proposition}\cite[Section 1.10, Corollary 10.6]{Pazy}\label{propS*}
If $X$ is reflexive and if $(S(t))_{t\geq 0}$ is a $C_0$ semigroup on $X$ with generator $A$ then $(S(t)^*)_{t\geq 0}$ is a $C_0$ semigroup on $X'$ with generator $A^*$.
\end{proposition}

\subsection{The Cauchy problem}\label{sec_Cauchy}
\subsubsection{Classical solutions}
Let $A:D(A)\subset X\rightarrow X$ be a densely defined linear operator on the Banach space $X$. Consider the Cauchy problem
\begin{equation}\label{cp}
\begin{split}
\dot y(t)&=Ay(t),\qquad  t>0,\\
y(0)&=y_0\in D(A).
\end{split}
\end{equation}
As a consequence of Proposition \ref{propclosed} we have the following result.

\begin{theorem}\label{thm4.4}
Assume that $A$ is the infinitesimal generator of a $C_0$ semigroup $(S(t))_{t\geq 0}$ on $X$. Then the Cauchy problem \eqref{cp} has a unique solution $y\in C^0([0,+\infty);D(A))\cap C^1((0,+\infty);X)$ given by $y(t)=S(t)y_0$ for every $t\geq 0$.
Moreover, the differential equation $\dot y(t)=Ay(t)$ makes sense in $X$.
\end{theorem}

This solution is often called \textit{strong solution} in the existing literature.

\begin{example}\label{exheat}
Let $\Omega$ be a bounded open subset of $\R^n$ having a Lipschitz boundary (to be able to define a trace). 
Having in mind Example \ref{ex_heat}, let us apply Theorem \ref{thm4.4} to the Dirichlet heat equation.
The Cauchy problem
\begin{equation*}
\partial_t y=\triangle y\ \ \textrm{in}\ \Omega,\qquad
y_{\vert\partial\Omega}=0,\qquad
y(0)=y_0\in H^1_0(\Omega),
\end{equation*}
has a unique solution $y\in C^0([0,+\infty);H^1_0(\Omega))\cap C^1((0,+\infty);H^{-1}(\Omega))$. Moreover, there exist $M\geq 1$ and $\omega<0$ such that
$\Vert y(t,\cdot)\Vert_{L^2(\Omega)}\leq M\mathrm{e}^{\omega t}\Vert y_0(\cdot)\Vert_{L^2(\Omega)}$.
\end{example}

\begin{example}\label{exwave}
Let $\Omega$ be a bounded open subset of $\R^n$ having a Lipschitz boundary. Having in mind Example \ref{ex_wave}, let us apply Theorem \ref{thm4.4} to the Dirichlet wave equation. The Cauchy problem
\begin{equation*}
\partial_{tt} y=\triangle y\ \ \textrm{in}\ \Omega,\qquad
y_{\vert\partial\Omega}=0,\qquad
y(0)=y_0\in H^1_0(\Omega),\ \partial_t y(0)=y_1\in L^2(\Omega),
\end{equation*}
has a unique solution 
$$y\in C^0([0,+\infty);H^1_0(\Omega))\cap C^1((0,+\infty);L^2(\Omega)) \cap C^2((0,+\infty);H^{-1}(\Omega)).$$
Moreover, we have conservation of energy: $\Vert \partial_ty(t)\Vert_{H^{-1}(\Omega)}^2+\Vert y(t)\Vert_{L^2(\Omega)}^2
=\Vert y_1\Vert_{H^{-1}(\Omega)}^2+\Vert y_0\Vert_{L^2(\Omega)}^2$ (by integration by parts).

If $\partial\Omega$ is $C^2$ and if $y_0\in H^1_0(\Omega)\cap H^2(\Omega)$ and $y_1\in H^1_0(\Omega)$ then 
$$
y\in C^0([0,+\infty); H^2(\Omega)\cap H^1_0(\Omega))\cap C^1((0,+\infty);H^1_0(\Omega)) \cap C^2((0,+\infty);L^2(\Omega))
$$
and $\Vert \partial_ty(t)\Vert_{L^2(\Omega)}^2+ \Vert y(t)\Vert_{H^1_0(\Omega)}^2 =\Vert y_1\Vert_{L^2(\Omega)}^2+ \Vert y_0\Vert_{H^1_0(\Omega)}^2$.
\end{example}

\begin{remark}\label{rem_regular}
Concerning the regularity of the solutions of the heat equation (Example \ref{exheat}) and of the wave equation (Example \ref{exwave}), we can actually be much more precise, by expanding the solutions as a series with the eigenfunctions of the Dirichlet-Laplacian. Indeed, let $(\phi_j)_{j\in\N^*}$ be a Hilbert basis of $L^2(\Omega)$, consisting of eigenfunctions of the Dirichlet-Laplacian, corresponding to the eigenvalues $(\lambda_j)_{j\in\N^*}$. %It is well known that $\phi_j$ is (real) analytic in the open set $\Omega$.

For the heat equation of Example \ref{exheat}, if $y_0=\sum_{j=1}^{+\infty}a_j\phi_j\in L^2(\Omega)$ then $y(t,x)=\sum_{j=1}^{+\infty}a_je^{\lambda_jt}\phi_j(x)$ is a function of $(t,x)$ of class $C^\infty$ for $t>0$ (see \cite{CazenaveHaraux}), and for every $t>0$ fixed, the function $x\mapsto y(t,x)$ is (real) analytic on the open set $\Omega$ (see \cite{Nelson}).
This reflects the \textit{smoothing effect} of the heat equation.

For the wave equation, there is no smoothing effect, but smoothness or analyticity properties can also be established for appropriate initial conditions (see \cite{Nelson}).

These remarks show that the regularity properties obtained by the general semigroup theory may be much improved when using the specific features of the operator under consideration (see also Remark \ref{rem_regular2} further).
\end{remark}

\begin{remark}\label{rem4.4}
If $y_0\in X\setminus D(A)$ then in general $y(t)=S(t)y_0\notin D(A)$, and hence $y(t)$ is not a solution of \eqref{cp} in the above sense. Actually, $y(t)$ is solution of $\dot y(t)=Ay(t)$ in a weaker sense, by replacing $A$ with an extension on $A$ to $X$, as we are going to see next. 
\end{remark}

%%%%%%%%%%%%%%%%%%%%%%%%%%%%%%%%%%

\subsubsection{Weak solutions}
The objective of this section is to define an extension of the Banach space $X$, and extensions  of $C_0$ semigroups on $X$ which will provide weaker solutions. A good reference for this part is the textbook \cite{TucsnakWeiss}.

Let $(S(t))_{t\geq 0}\in\mathcal{G}(M,\omega)$ be a $C_0$ semigroup on $X$, of generator $A:D(A)\rightarrow X$. Let $\beta\in\rho(A)$ (if $X$ is real, consider such a real number $\beta$).

\begin{definition}\label{def_Xm1}
Let $X_1$ denote the Banach space $D(A)$, equipped with the norm $\Vert y\Vert_{X_1}=\Vert (\beta\,\mathrm{id}_X-A)y\Vert_X$, and let $X_{-1}$ denote the completion of $X$ with respect to the norm $\Vert y\Vert_{X_{-1}}=\Vert (\beta\,\mathrm{id}_X-A)^{-1}y\Vert_X =\Vert R(\beta,A)y\Vert_X$.
\end{definition}

Note that, by definition, $\beta\,\mathrm{id}_X-A:X_1\rightarrow X$ and $(\beta\,\mathrm{id}_X-A)^{-1}:X\rightarrow X_1$ are surjective isometries (unitary operators).
It is then easy to see that the norm $\Vert\ \Vert_{1}$ on $X_1$ is equivalent to the graph norm $\Vert y\Vert_G=\Vert y\Vert_X+\Vert Ay\Vert_X$. Therefore, from the closed graph theorem, $(X_1,\Vert\ \Vert_{1})$ is a Banach space, and we clearly get an equivalent norm by considering any other $\beta'\in\rho(A)$.

Similarly, the space $X_{-1}$ does not depend on the specific value of $\beta\in\rho(A)$, in the sense that we get an equivalent norm by considering any other $\beta'\in\rho(A)$. Indeed, for all $(\beta,\beta')\in\rho(A)^2$, we have $(\beta\,\mathrm{id}_X-A)(\beta'\,\mathrm{id}_X-A)^{-1}=\mathrm{id}_X+(\beta-\beta')(\beta'\,\mathrm{id}_X-A)^{-1}$, hence
$$
(\beta'\,\mathrm{id}_X-A)^{-1} = (\beta\,\mathrm{id}_X-A)^{-1} + (\beta-\beta') (\beta\,\mathrm{id}_X-A)^{-1} (\beta'\,\mathrm{id}_X-A)^{-1}
$$
(resolvent identity), and moreover $(\beta\,\mathrm{id}_X-A)^{-1}$ and $(\beta'\,\mathrm{id}_X-A)^{-1}$ commute. The conclusion follows easily.

\begin{remark}
The injections $X_1\hookrightarrow X \hookrightarrow X_{-1}$ are, by definition, continuous and dense. They are moreover compact as soon as $\beta\,\mathrm{id}_X-A$ has a compact inverse (i.e., as soon as $A$ has compact resolvents).
\end{remark}

\begin{example}\label{ex4.7}
Let $\Omega$ be an open bounded subset of $\R^n$ having a $C^2$ boundary, and consider the Dirichlet-Laplacian $\triangle_D$ on $X=L^2(\Omega)$ defined on $D(\triangle_D)=H^1_0(\Omega)\cap H^2(\Omega)$. Then $X_1=D(\triangle_D)=H^1_0(\Omega)\cap H^2(\Omega)$ and, as will follow from Theorem \ref{thm4.5} below, $X_{-1}=(H^1_0(\Omega)\cap H^2(\Omega))'$, where the dual is taken with respect to the pivot space $X=L^2(\Omega)$.
\end{example}

Let us now provide a general theorem allowing one to identify the space $X_{-1}$.
Since $A^*$ is closed, $D(A^*)$ endowed with the norm $\Vert z\Vert_{X_1}=\Vert (\beta\,\mathrm{id}_{X'}-A^*)z\Vert_{X'}$ where $\beta\in \rho(A^*)=\rho(A)$, is a Banach space.

\begin{theorem}\label{thm4.5}
If $X$ is reflexive then $X_{-1}$ is isomorphic to $D(A^*)'$, where the dual is taken with respect to the pivot space $X$.
\end{theorem}

\begin{proof}
We begin by recalling the following general fact: if $E$ and $F$ are two Banach spaces with a continuous injection $E\hookrightarrow F$, then we have a continuous injection $F'\hookrightarrow E'$. From this general fact, since $D(A^*)\subset X'$ with a continuous injection, it follows that $X''\subset D(A^*)'$ (with a continuous injection). Using the canonical injection from $X$ to $X''$, it follows that every element of $X$ is (identified with) an element of $D(A^*)'$.
Let us prove that $\Vert y \Vert_{X_{-1}} = \Vert y \Vert_{D(A^*)'}$ for every $y\in X$. By definition, we have
\begin{equation*}
\begin{split}
\Vert y\Vert_{X_{-1}} &= \Vert (\beta\,\mathrm{id}_X-A)^{-1}y\Vert_X \\
&= \sup \left\{  \langle f, (\beta\,\mathrm{id}_X-A)^{-1}y \rangle_{X',X}  \ \mid\ f\in X',\ \Vert f\Vert_{X'}\leq 1 \right\} \\
&= \sup \left\{ \langle (\beta\,\mathrm{id}_{X'}-A^*)^{-1} f, y \rangle_{X',X}  \ \mid\ f\in X',\ \Vert f\Vert_{X'}\leq 1 \right\} .
\end{split}
\end{equation*}
Using the canonical injection of $X$ in $X''$ (and not yet the fact that $X$ is reflexive), $y$ can be considered as an element of $X''$, and then
$$
\Vert y\Vert_{X_{-1}} = \sup \left\{ \langle y, (\beta\,\mathrm{id}_{X'}-A^*)^{-1} f \rangle_{X'',X'}  \ \mid\ f\in X',\ \Vert f\Vert_{X'}\leq 1 \right\} .
$$
Besides, by definition we have
$$
\Vert y\Vert_{D(A^*)'} = \sup \left\{ \langle y,z \rangle_{D(A^*)',D(A^*)} \ \mid\ z\in D(A^*),\ \Vert z\Vert_{D(A^*)}\leq 1 \right\} .
$$
In this expression we make a change of variable: for every $z\in D(A^*)$ such that $\Vert z\Vert_{D(A^*)}\leq 1$, there exists $f\in X'$ such that $z=(\beta\,\mathrm{id}_{X'}-A^*)^{-1}f$, and since $\Vert z\Vert_{D(A^*)} = \Vert (\beta\,\mathrm{id}_{X'}-A^*)z\Vert_{X'}=\Vert f\Vert_{X'}$ it follows that $\Vert f\Vert_{X'}\leq 1$. Therefore
$$
\Vert y\Vert_{D(A^*)'} = \sup \left\{ \langle y,(\beta\,\mathrm{id}_{X'}-A^*)^{-1}f \rangle_{D(A^*)',D(A^*)} \ \mid\ f\in X',\ \Vert f\Vert_{X'}\leq 1 \right\} .
$$
In the above duality bracket, since $y\in X''$ and $(\beta\,\mathrm{id}_{X'}-A^*)^{-1}f\in X'$ we can as well use the duality bracket $\langle\ ,\ \rangle_{X'',X'}$. Hence $\Vert y \Vert_{X_{-1}} = \Vert y \Vert_{D(A^*)'}$.

To conclude the proof, it remains to note that $X_1$ is dense in $X$ and that $X\simeq X''$ is dense in $D(A^*)'$. It is a general fact that $X_1$ is dense in $X$ (see Proposition \ref{propclosed}). The fact that $X$ is dense in $D(A^*)'$ is ensured by the reflexivity assumption: indeed, since $X$ is reflexive it follows that $D(A^*)$ is dense in $X'$ (with a continuous injection), and hence $X\simeq X''$ is dense in $D(A^*)'$. %The proof is done.
\end{proof}

\begin{theorem}
The operator $A:D(A)=X_1\rightarrow X$ can be extended to an operator $A_{-1}:D(A_{-1})=X\rightarrow X_{-1}$, and the $C_0$ semigroup $(S(t))_{t\geq 0}$ on $X$ extends to a semigroup $(S_{-1}(t))_{t\geq 0}$ on $X_{-1}$, generated by $A_{-1}$.
\end{theorem}

\begin{proof}
Note that the operator $A:D(A)\rightarrow X$ is continuous when one endows $D(A)$ with its norm, because $\Vert Ay\Vert_X \leq \Vert y\Vert_{G} \leq C\Vert y\Vert_{1}$ as already said. 
By definition of the norm in $X_{-1}$, we have easily
\begin{equation*}
\begin{split}
\Vert Ay\Vert_{X_{-1}} = \Vert (\beta\,\mathrm{id}_X-A)^{-1}Ay\Vert_X
&= \Vert y - \beta(\beta\,\mathrm{id}_X-A)^{-1} y\Vert_X \\
&\leq \Vert y\Vert_X + \vert\beta\vert\Vert(\beta\,\mathrm{id}_X-A)^{-1} y\Vert_X 
\end{split}
\end{equation*}
for every $y\in D(A)$, and since $(\beta\,\mathrm{id}_X-A)^{-1}$ is bounded it follows that there exists some constant $C_1>0$ such that $\Vert Ay\Vert_{X_{-1}} \leq C_1 \Vert y\Vert_X$ for every $y\in D(A)$. Therefore the operator $A$ has a unique extension $A_{-1}:D(A_{-1})=X\rightarrow X_{-1}$ that is continuous for the respective norms. The fact that $D(A_{-1})=X$ with equivalent norms follows from the equality
$$
\Vert y\Vert_X = \Vert (\beta\,\mathrm{id}_X-A)^{-1}(\beta\,\mathrm{id}_X-A)y\Vert_X = \Vert (\beta\,\mathrm{id}_X-A)y\Vert_{X_{-1}} = \Vert y\Vert_{D(A_{-1})}
$$
for every $y\in D(A)$, and by density this holds true for every $y\in X$.
\end{proof}

\begin{remark}
Note that, by density, if $A$ is m-dissipative then $A_{-1}$ is m-dissipative as well, and hence $(S_{-1}(t))_{t\geq 0}$ is a $C_0$ semigroup of contractions.
\end{remark}

The following result follows from Theorem \ref{thm4.4}, giving an answer to the question raised in Remark \ref{rem4.4}.

\begin{corollary}
For every $y_0\in X$, the Cauchy problem $\dot{y}(t)=A_{-1}y(t)$, $y(0)=y_0$ has a unique solution $y\in C^0([0,+\infty);X)\cap C^1((0,+\infty);X_{-1})$ given by $y(t)=S(t)y_0=S_{-1}(t)y_0$ for every $t\geq 0$.
\end{corollary}

Note that, here, the differential equation $\dot{y}(t)=A_{-1}y(t)$ is written in $X_{-1}$. In particular, the derivative is computed in $X_{-1}$, with the norm $\Vert\cdot\Vert_{X_{-1}}$.

In other words, with respect to Theorem \ref{thm4.4}, for a given $y_0\in X$, $y(t)=S(t)y_0$ (often called \textit{mild solution} in the existing literature) is still a solution of $\dot y(t)=Ay(t)$ (now written in $X_{-1}$) provided $A$ is replaced with its extension $A_{-1}$.
Note that this weaker solution is a strong solution for the operator $A_{-1}$ in the Banach space $X_{-1}$. For these reasons, we shall not insist on naming solutions "strong", "mild" or "weak". What is important is to make precise the Banach spaces in which we are working.

Note that the above concept of weak solution corresponds to solutions sometimes defined by transposition. Indeed, if $X$ is reflexive then $X_{-1}\simeq D(A^*)'$ (see Theorem \ref{thm4.5}), and hence, considering the differential equation $\dot y(t)=A_{-1}y(t)$ in the space $X_{-1}$ means that
$$
\langle \dot y(t),\varphi\rangle_{D(A^*)',D(A^*)} = \langle A_{-1} y(t),\varphi\rangle_{D(A^*)',D(A^*)} \qquad\forall \varphi\in D(A^*) .
$$
This concept of solution by transposition is often encountered in the existing literature (see, e.g., \cite{Coron} for control issues).

\begin{example}
Let $\Omega\subset\R^n$ be a bounded open set with $C^2$ boundary. 
\begin{itemize}[parsep=1mm,itemsep=1mm,topsep=1mm]%,leftmargin=*
\item The Cauchy problem $\partial_t{y}=\triangle y$ in $\Omega$, $y_{\vert\partial\Omega}=0$, $y(0)=y_0\in L^2(\Omega)$,
has a unique solution 
$$
y\in C^0([0,+\infty);L^2(\Omega))\cap C^1((0,+\infty);(H^{1}_0(\Omega)\cap H^2(\Omega))') .
$$
Moreover, there exist $M\geq 1$ and $\omega\in\R$ (actually, $\omega<0$) such that
$\Vert y(t)\Vert_{L^2(\Omega)}\leq M\mathrm{e}^{\omega t}\Vert
y_0\Vert_{L^2(\Omega)}$ for every $t\geq 0$.

\item Consider the Cauchy problem $\partial_{tt}{y}=\triangle y$ in $\Omega$, $y_{\vert\partial\Omega}=0$, $y(0)=y_0$, $\partial_t{y}(0)=y_1$.
\begin{itemize}[parsep=1mm,itemsep=1mm,topsep=1mm]%,leftmargin=*
\item If $y_0\in H^{-1}(\Omega)$ and $y_1\in (H^{1}_0(\Omega)\cap H^2(\Omega))'$, then there is a unique solution
$$y\in C^0([0,+\infty);H^{-1}(\Omega))\cap C^1((0,+\infty); (H^{1}_0(\Omega)\cap H^2(\Omega))' ) .$$
\item If $y_0\in L^2(\Omega)$ and $y_1\in H^{-1}(\Omega)$, then there is a unique solution
$$y\in C^0([0,+\infty);L^2(\Omega))\cap C^1((0,+\infty); H^{-1}(\Omega)) .$$
\end{itemize}
%Indeed the wave equation can be written as the first-order system
%$$\partial_t \begin{pmatrix} y\\ \partial_t y\end{pmatrix} = A \begin{pmatrix} y\\ \partial_t y\end{pmatrix}\qquad\textrm{with}\qquad A = \begin{pmatrix} 0 & \mathrm{id} \\ A_0 & 0 \end{pmatrix}, $$
%where 
\end{itemize}
\end{example}

\subsection{Scale of Banach spaces}\label{sec:scale}
We can generalize the previous framework and adopt a more abstract (and probably simpler, at the end) point of view.

The construction of $X_1$ and of $X_{-1}$ can indeed be iterated, and leads to a sequence of Banach spaces $(X_n)_{n\in\Z}$ (called "tower of Sobolev spaces" in \cite{Nagel}).
For positive integers $n$, the operator $A_n:D(A_n)=D(A^{n+1})\rightarrow D(A^n)$ is the restriction of $A$ to $D(A^{n+1})$.

The construction can even be generalized in order to obtain a continuous \textit{scale of Banach spaces} $(X_\alpha)_{\alpha\in\R}$, with the property that if $\alpha_1>\alpha_2$ then the canonical injection $X_{\alpha_1}\hookrightarrow X_{\alpha_2}$ is continuous and dense, and is compact as soon as the resolvent of $A$ is compact.
We refer to \cite[Sections 3.6 and 3.9]{Staffans} for this general construction (where these spaces are called \textit{rigged spaces}) and for further properties (see also \cite[Chapter II, Section 5, c]{Nagel}). %In the sequel we will briefly refer to this continuous scale of Banach spaces.
%% see also \cite[Chapter 2, Section 4.3]{LiYong} for analytic operators

The Banach space $X_\alpha$, with $\alpha$ an arbitrary real number, can be defined for instance by symbolic calculus with real powers of the resolvent of $A$ and complex integrals (see \cite{Nagel, Kato, Pazy, Staffans}), or by Banach spaces interpolation (see \cite{Lions_interp, LionsMagenes, Triebel}), or by Fourier transform when it is possible (see \cite{Kato}), or with a Hilbert basis when $X$ is a Hilbert space and $A$ is diagonalizable (see Remark \ref{rem41} below).
For instance, the construction with the fractional powers of the resolvent goes as follows, in few words (see \cite[Section 3.9]{Staffans}), provided $A$ generates a $C_0$ semigroup. 
%% see also \cite[Section 2.6]{Pazy} but Pazy restricts mainly to analytic semigroups, while Staffans considers general C_0 semigroups 
%% Another construction is given in \cite[Chapter II, Section 5, c]{Nagel}, under a slight assumption on $A$ (satisfied if $A$ generates a $C_0$ semigroup).
Given $\beta\in\rho(A)$ with $\Real(\beta)>\omega$, given any $\alpha>0$ we define\footnote{This formula extrapolates the Laplace transform formula $\frac{1}{z^\alpha} = \frac{1}{\Gamma(\alpha)} \int_0^{+\infty} t^{\alpha-1}e^{-tz}\, dt$ valid for any $z\in\C$ such that $\Real(z)>0$.}
\begin{equation}\label{def_poweralpha}
(\beta\,\mathrm{id}_X-A)^{-\alpha} = \frac{1}{\Gamma(\alpha)}\int_0^{+\infty} t^{\alpha-1}e^{-\beta t} S(t)\, dt 
\end{equation}
and then we define $(\beta\,\mathrm{id}_X-A)^\alpha = ((\beta\,\mathrm{id}_X-A)^{-\alpha})^{-1}$ on the domain $\mathrm{Ran}((\beta\,\mathrm{id}_X-A)^{-\alpha})$.
We also define $(\beta\,\mathrm{id}_X-A)^0=\mathrm{id}$.
We denote by 
$$
X_\alpha=\mathrm{Ran}((\beta\,\mathrm{id}_X-A)^{-\alpha})=(\beta\,\mathrm{id}_X-A)^{-\alpha}(X)
$$
the Banach space endowed with the norm $\Vert y\Vert_{X_\alpha} = \Vert (\beta\,\mathrm{id}_X-A)^{\alpha} y\Vert_X$. 
The Banach space $X_{-\alpha}$ is defined as the completion of $X$ for the norm $\Vert y\Vert_{X_{-\alpha}} = \Vert (\beta\,\mathrm{id}_X-A)^{-\alpha} y\Vert_X$. Accordingly, we set $X_0=X$, endowed with the norm of $X$.
We have thus defined the scale of Banach spaces $(X_\alpha)_{\alpha\in\R}$.
The construction does not depend on the specific choice of $\beta\in\rho(A)$.

In this general framework, the operator 
$$
A_\alpha:D(A_\alpha)=X_{\alpha+1}\rightarrow X_\alpha
$$
(with $\alpha\in\R$), which is either the restriction or the extension of $A:D(A)\rightarrow X$ (with $X_0=X$) according to the sign of $\alpha$, generates the $C_0$ semigroup $(S_\alpha(t))_{t\geq 0}$.

Hereafter, when it is clear from the context, we skip the index $\alpha$ in $S_\alpha(t)$ or in $A_\alpha$, when referring to the restriction or extension of $S(t)$ or of $A$ to $X_\alpha$.

Note that, for any $\alpha_1,\alpha_2\in\R$, 
$$
(\beta\,\mathrm{id}_X-A)^{\alpha_1-\alpha_2} : X_{\alpha_1} \rightarrow X_{\alpha_2}
$$ 
is a surjective isometry (unitary operator), where $A$ denotes here (without the index) the appropriate restriction or extension of the operator $A$.

The spaces $X_\alpha$ are interpolation spaces between the spaces $X_n$ with integer indices. It can be noted that there exists $C>0$ such that, for every $n\in\Z$, for every $\alpha\in[n,n+1]$, we have 
$$
\Vert y\Vert_{X_\alpha}\leq C \Vert y\Vert_{X_n}^{n+1-\alpha} \Vert y\Vert_{X_{n+1}}^{\alpha-n} 
\qquad\forall y\in X_{n+1}
$$
(see \cite[Lemma 3.9.8]{Staffans}). This is an interpolation inequality, as in \cite{LionsMagenes}. 
Replacing the operator $A$ with any real power of $\beta\,\mathrm{id}_X-A$, we infer from those inequalities the following more general interpolation inequalities (see \cite[Chapter II, Section 5, Theorem 5.34]{Nagel}): given any real numbers $\alpha<\beta<\gamma$, there exists $C(\gamma-\alpha)>0$ (only depending on $\gamma-\alpha$) such that
\begin{equation}\label{interpolation_EngelNagel}
\Vert y\Vert_{X_\beta}\leq C(\gamma-\alpha) \Vert y\Vert_{X_\alpha}^{\frac{\gamma-\beta}{\gamma-\alpha}} \Vert y\Vert_{X_\gamma}^{\frac{\beta-\alpha}{\gamma-\alpha}}
\qquad \forall y\in X_\gamma . 
\end{equation}

When $X$ is reflexive, the operator $A^*:D(A^*)\rightarrow X'$ generates the $C_0$ semigroup $(S(t)^*)_{t\geq 0}$ (see Proposition \ref{propS*}). By the construction above where $A$ is replaced with $A^*$, there exists a scale of Banach spaces denoted by $(X^*_\alpha)_{\alpha\in\R}$, with $X^*_0=X'$.
Similarly as in Theorem \ref{thm4.5}, we have 
\begin{equation}\label{Xalphastar}
X_{-\alpha} = (X^*_\alpha)'\qquad\forall\alpha\in\R
\end{equation}
where the dual is taken with respect to the pivot space $X$.

There exist plenty of other constructions of (different) interpolation spaces, such as Favard or abstract H\"older spaces (see \cite[Chapter II, Section 5, b]{Nagel}).

\paragraph{Cauchy problems.}
The above general construction allows one to generalize in a wide sense the concept of strong or weak solution. As a consequence of Theorem \ref{thm4.4}, we have the following result.

\begin{theorem}
The Cauchy problem
\begin{equation}\label{cauchyalpha}
\dot{y}(t)=Ay(t), \quad y(0)=y_0\in X_\alpha,
\end{equation}
has a unique solution 
$
y\in C^0([0,+\infty);X_\alpha)\cap C^1((0,+\infty);X_{\alpha-1})
$
given by $y(t)=S(t)y_0$ for every $t\geq 0$. Here, we have skipped the index $\alpha$, but it is understood that $A=A_{\alpha-1}$ in \eqref{cauchyalpha}; the differential equation is written in $X_{\alpha-1}$ and the derivative is computed with respect to the norm $\Vert\cdot\Vert_{X_{\alpha-1}}$.
\end{theorem}

\begin{remark}\label{rem41}
It is interesting to perform the construction of the scale of Banach spaces in the following important case of diagonalizable operators.

Assume that $X$ is a Hilbert space and that $\mathcal{A}:D(\mathcal{A})\rightarrow X$ is a self-adjoint positive operator with $\mathcal{A}^{-1}$ compact (for instance, the negative of the Dirichlet-Laplacian on a bounded domain). 
Then there exists a normalized Hilbert basis $(e_j)_{j\in\N^*}$ of eigenvectors of $\mathcal{A}$, associated with eigenvalues $(\lambda_j)_{j\in\N^*}$. One has
$$
\mathcal{A}y = \sum_{j=1}^{+\infty} \lambda_j (e_j,y)_X e_j\quad\textrm{on}\quad
D(\mathcal{A}) = \Big\{ y\in X \ \bigm|\ \sum_{j=1}^{+\infty} \lambda_j^2 (e_j,y)_X^2 < +\infty \Big\} 
$$
and then $\mathcal{A}^\alpha$ is defined for every $\alpha\in\R$ in a spectral way by
$$
\mathcal{A}^\alpha y = \sum_{j=1}^{+\infty} \lambda_j^\alpha (e_j,y)_X e_j\quad\textrm{on}\quad
D(\mathcal{A}^\alpha) = \Big\{ y\in X \ \bigm|\ \sum_{j=1}^{+\infty} \lambda_j^{2\alpha} (e_j,y)_X^2 < +\infty \Big\}.
$$
Note that we have used a calligraphic $\mathcal{A}$, to avoid the confusion with the operator $A=-\mathcal{A}$ that one can consider in the differential equation $\dot y(t)=Ay(t)$.
\end{remark}

\begin{example}\label{example_chain_Dirichlet}
Let us consider the negative of the Dirichlet-Laplacian $\mathcal{A}=-\triangle_D$ defined on $D(\mathcal{A})=\{ y \in H^1_0(\Omega)\ \mid\ \triangle y \in L^2(\Omega)\}$, with $X=L^2(\Omega)$, where $\Omega$ is a bounded open subset of $\R^n$ with $C^2$ boundary. We have $D(\mathcal{A})=H^2(\Omega)\cap H^1_0(\Omega)$ (see Example \ref{ex_heat}) and $X_{-1}=(H^1_0(\Omega)\cap H^2(\Omega))'$ (see Example \ref{ex4.7}) where the dual is taken with respect to the pivot space $X=L^2(\Omega)$.
We can define $\mathcal{A}^{1/2}=\sqrt{-\triangle}$ in a spectral way as above.

Assuming that the boundary of $\Omega$ is of class $C^\infty$, the spaces $D(\mathcal{A}^{j/2})$, for $j\in\N$, called \emph{Dirichlet spaces}, are the Sobolev spaces with (the so-called) Navier boundary conditions, defined by
\begin{equation*}
\begin{split}
D(\mathcal{A}^{1/2}) & = \{ y \in H^1(\Omega) \ \mid\  y_{\vert\partial\Omega}=0 \}  = H^1_0(\Omega) , \\
D(\mathcal{A}) & = \{ y \in H^2(\Omega) \ \mid\   y_{\vert\partial\Omega}=0 \} = H^1_0(\Omega)\cap H^2(\Omega) ,  \\
D(\mathcal{A}^{3/2}) &= \{ y \in H^3(\Omega) \ \mid\   y_{\vert\partial\Omega}=(\triangle  y)_{\vert\partial\Omega}=0 \} , \\
D(\mathcal{A}^2) &= \{ y \in H^4(\Omega) \ \mid\   y_{\vert\partial\Omega}=(\triangle  y)_{\vert\partial\Omega}=0 \} , \\
D(\mathcal{A}^{5/2}) &= \{ y \in H^5(\Omega) \ \mid\   y_{\vert\partial\Omega}=(\triangle  y)_{\vert\partial\Omega}=(\triangle^2  y)_{\vert\partial\Omega}=0 \} , \\
D(\mathcal{A}^3) &= \{ y \in H^6(\Omega) \ \mid\   y_{\vert\partial\Omega}=(\triangle  y)_{\vert\partial\Omega}=(\triangle^2  y)_{\vert\partial\Omega}=0 \} ,
\end{split}
\end{equation*}
etc; in other words,
$$
D(\mathcal{A}^{j/2}) = \left\{ y \in H^j(\Omega) \ \bigm|\   y_{\vert\partial\Omega}=(\triangle  y)_{\vert\partial\Omega}=\cdots = (\triangle^{\left[\frac{j-1}{2}\right]}  y)_{\vert\partial\Omega}=0 \right\} 
$$
for every $j\in\N^*$, where $\left[\ \right]$ is the floor function. Moreover, the operator $\mathcal{A}^{j/2}:D(\mathcal{A}^{j/2})\rightarrow L^2(\Omega)$ is an isomorphism (see \cite{TucsnakWeiss} for other properties). 
It can be noted that 
$$
\Vert \mathcal{A}^{j/2}y\Vert_{L^2(\Omega)} = 
\left\{\begin{array}{ll}
\Vert(-\triangle)^{j/2} y\Vert_{L^2(\Omega)} & \textrm{if $j$ is even}, \\
\Vert(-\triangle)^{j/2} y\Vert_{H^1_0(\Omega)}=\Vert(-\triangle)^{(j+1)/2} y\Vert_{L^2(\Omega)} %=\Vert\nabla\triangle^{j/2} y\Vert_{L^2(\Omega)} 
& \textrm{if $j$ is odd}.
\end{array}\right.
$$
Omitting the indices, we have the scale of Hilbert spaces
$$
\cdots \xrightarrow{\mathcal{A}^{1/2}}{} \! D(\mathcal{A}) \xrightarrow{\mathcal{A}^{1/2}}{} \! D(\mathcal{A}^{1/2}) \xrightarrow{\mathcal{A}^{1/2}}{} \! L^2(\Omega) \xrightarrow{\mathcal{A}^{1/2}}{} \! D(\mathcal{A}^{1/2})' \xrightarrow{\mathcal{A}^{1/2}}{} \! D(\mathcal{A})' \xrightarrow{\mathcal{A}^{1/2}}{}  \!\cdots
$$
with $D(\mathcal{A}^{1/2})'=H^{-1}(\Omega)$ and $D(\mathcal{A})'=(H^1_0(\Omega)\cap H^2(\Omega))'$ (with respect to the pivot space $L^2(\Omega)$). All mappings $\mathcal{A}^{1/2}=\sqrt{-\triangle}$, between the corresponding spaces, are isometric isomorphims.

As in the previous remark, we can even define $X_\alpha=D(\mathcal{A}^\alpha)$ (and their duals) in a spectral way, for any $\alpha\in\R$, thus obtaining the scale $(X_\alpha)_{\alpha\in\R}$ of Dirichlet spaces associated with the Dirichlet-Laplacian. By interpolation theory (see \cite{LionsMagenes}), for every $\alpha\in[0,1)$, we have $X_\alpha=H^{2\alpha}_0(\Omega)$ if $\alpha\neq 1/4$ and $X_{1/4}=H^{1/2}_{00}(\Omega)$ (Lions-Magenes space).
\end{example}

\begin{remark}
Using Proposition \ref{propS*}, if $X$ is reflexive then all these results can be stated as well for the adjoint operator $A^*$ and the adjoint $C_0$ semigroup $S(t)^*$.
\end{remark}

\section{Nonhomogeneous Cauchy problems}\label{chap4_sec4.2}
Let $y_0\in X$. We consider the Cauchy problem
\begin{equation}\label{cp1}
\dot{y}(t)=Ay(t)+f(t), \qquad y(0)=y_0,
\end{equation}
where $A:D(A)\rightarrow X$ generates a $C_0$ semigroup $(S(t))_{t\geq 0}$ on $X$.

\begin{proposition}\label{thm32}
If $y_0\in D(A)$ and $f\in L^p_\mathrm{loc}([0,+\infty),D(A))$ with $1\leq p\leq +\infty$, then \eqref{cp1} has a unique solution $y\in C^0([0,+\infty);D(A))\cap W^{1,p}_\mathrm{loc}([0,+\infty),X)$ (often referred to as \textit{strong solution} of \eqref{cp1}) given by
\begin{equation}\label{eqq1}
y(t)=S(t)y_0+\int_0^t S(t-s)f(s)ds.
\end{equation}
Morever, the differential equation \eqref{cp1} makes sense in $X$.
\end{proposition}

\begin{proof}
The function $y$ defined by \eqref{eqq1} is clearly a solution of \eqref{cp1}. To prove uniqueness, let $y_1$ and $y_2$ be two solutions. Then $z=y_1-y_2$ is solution of $\dot z(t)=Az(t)$, $z(0)=0$. Since $\frac{d}{ds} S(t-s)z(s) = -S(t-s)Az(s) + S(t-s)Az(s) = 0$ for every $s\in[0,t]$, it follows that $0=S(t)z(0)=S(0)z(t)=z(t)$. %, whence the conclusion.
\end{proof}

Note that, if $f\in L^p_\mathrm{loc}([0,+\infty),X)$, then \eqref{eqq1} still makes sense. 
Note also that, using the extension of $A$ (and of $S(t)$) to $X_{-1}$, Proposition \ref{thm32} implies that, if $y_0\in X$ and $f\in L^p_\mathrm{loc}([0,+\infty),X)$, then \eqref{cp1} has a unique solution $y\in C^0([0,+\infty),X)\cap W^{1,p}_\mathrm{loc}([0,+\infty),X_{-1})$ given as well by the Duhamel formula \eqref{eqq1} (and often referred to as \textit{mild solution} of \eqref{cp1}), and the differential equation \eqref{cp1} is written in $X_{-1}$ (see, e.g., \cite[Chapter 4]{Pazy}). Moreover, for every $T>0$ there exists $K_T>0$ (not depending on $y_0$ and $f$) such that
$$
\Vert y(t)\Vert_X \leq K_T ( \Vert y_0\Vert_X + \Vert f\Vert_{L^p([0,T],X)} ).
$$

More generally, using the general scale of Banach spaces $(X_\alpha)_{\alpha\in\R}$ mentioned previously, we have the following result (see \cite{Nagel} or \cite[Theorem 3.8.2]{Staffans}).

\begin{proposition}\label{propXalpha}
If $f\in L^p_\mathrm{loc}([0,+\infty),X_{\alpha})$ for some $\alpha\in\R$ and $1\leq p\leq +\infty$, then for every $y_0\in X_\alpha$ the Cauchy problem \eqref{cp1} has a unique solution 
$$
y\in C^0([0,+\infty),X_\alpha)\cap W^{1,p}_\mathrm{loc}([0,+\infty),X_{\alpha-1})
$$
given as well by \eqref{eqq1} (called strong solution in $X_{\alpha}$ in \cite{Staffans,TucsnakWeiss}).
Here, we have $A=A_{\alpha-1}$ in the equation \eqref{cp1} which is written in $X_{\alpha-1}$ almost everywhere, and the integral in \eqref{eqq1} is done in $X_{\alpha}$ (with $S(\cdot)=S_{\alpha}(\cdot)$ in the integral).
\end{proposition}

Proposition \ref{thm32} corresponds to $\alpha=1$. 

\begin{remark}
The regularity stated in Proposition \ref{propXalpha} is sharp in general.
Given $y_0\in X_{\alpha+1}$, the condition $f\in C^0([0,+\infty);X_\alpha)$ does not ensure that $y\in C^0([0,+\infty);X_{\alpha+1})\cap C^1((0,+\infty);X_{\alpha})$ (unless the semigroup is analytic, see \cite[Chapter VI, Section 7.b, Corollary 7.17]{Nagel}). 
Indeed, for $y_0=0$ and for a given $y_1\in X_\alpha$, the solution of the Cauchy problem $\dot y(t)=Ay(t)+S(t)y_1$, $y(0)=0$ is $y(t)=\int_0^t S(t-s)S(s)y_1\, ds = tS(t)y_1$. Hence, if $y_1\in X_\alpha\setminus X_{\alpha+1}$ then $S(t)y_1$ may not belong to $X_{\alpha+1}$. 

It can be however noted that if $f$ is more regular in time then the solution gains some regularity with respect to the space variable. More precisely we have the following (sometimes useful) result (see \cite{Nagel,Pazy,TucsnakWeiss}).

\begin{lemma}\label{lem_moreregular}
If $y_0\in X_{\alpha+1}$ and $f\in W^{1,p}_\mathrm{loc}([0,+\infty),X_\alpha)$ then \eqref{cp1} has a unique solution 
$$
y\in C^0([0,+\infty);X_{\alpha+1})\cap C^1((0,+\infty),X_\alpha)\cap W^{2,p}_\mathrm{loc}([0,+\infty),X_{\alpha-1})
$$
given by \eqref{eqq1}.
\end{lemma}

The assumption on $f$ can even be weakened if $X_{\alpha-1}$ is reflexive, and then it suffices to assume that $f$ is Lipschitz continuous with values in $X_{\alpha-1}$.
\end{remark}

\begin{remark}\label{rem_regular2}
Let $\Omega$ be a bounded open subset of $\R^n$ with $C^2$ boundary.
Let us consider the Cauchy problem
\begin{equation*}
\partial_t y=\triangle y+f\ \ \textrm{in}\ \Omega,\qquad
y_{\vert\partial\Omega}=0,\qquad
y(0)=y_0\in L^2(\Omega),
\end{equation*}
with $f\in L^2((0,+\infty)\times\Omega)$. The general theory implies that there exists a unique solution
$$
y\in C^0([0,+\infty),L^2(\Omega))\cap H^1([0,+\infty), (H^2(\Omega)\cap H^1_0(\Omega))') .
$$
Actually, by using a spectral expansion as in Remark \ref{rem_regular}, it is easy to prove that 
$$
y\in L^2([0,+\infty),H^1_0(\Omega))\cap H^1([0,+\infty),H^{-1}(\Omega)) ,
$$
which is more precise because this set is contained in $C^0([0,+\infty),L^2(\Omega))$. Moreover, if $y_0\in H^1_0(\Omega)$, then we have the improved regularity 
$$
y\in L^2([0,+\infty),H^2(\Omega)\cap H^1_0(\Omega))\cap H^1([0,+\infty),L^2(\Omega))\subset C^0([0,+\infty),H^1_0(\Omega))
$$
(see also \cite[Chapter 7.1]{Evans} where these regularity properties are established by using Galerkin approximations for more general elliptic operators).
%
%As in Remark \ref{rem_regular}, this remark shows again that the regularity properties obtained by the general semigroup theory may be much improved when using the specific features of the operator under consideration.
\end{remark}

\chapter{Linear control systems in Banach spaces}\label{chap_diminfinie}
Throughout the chapter, we consider the linear autonomous control system
\begin{equation}\label{eqE}
\begin{split}
\dot{y}(t)&=Ay(t)+Bu(t)\\
y(0)&=y_0
\end{split}
\end{equation}
where the state $y(t)$ belongs to a Banach space $X$, $y_0\in X$, the control $u(t)$ belongs to a Banach space $U$, $A:D(A)\rightarrow X$ is the generator of a $C_0$ semigroup $(S(t))_{t\geq 0}\in\mathcal{G}(M,\omega)$ on $X$, and $B\in L(U,X_{-1})$. The space $X_{-1}$ has been defined in the previous chapter.

The control operator $B$ is said to be \textit{bounded} if $B\in L(U,X)$, and is called \textit{unbounded} otherwise (this is the standard wording, although it is a bit ambiguous since $B$ is bounded as an operator from $U$ in $X_{-1}$). 
Unbounded operators appear naturally when dealing with boundary or pointwise control systems.
Other choices could be made for the control operator, and we can more generally assume that $B\in L(U,X_{-\alpha})$ for some $\alpha\geq 0$. We will comment on that further, with the concept of degree of unboundedness.

A priori if $u\in L^1_{\mathrm{loc}}(0,+\infty;U)$ then $Bu\in L^1_{\mathrm{loc}}(0,+\infty;X_{-1})$, and since $y_0\in X$, it follows from the results of Section \ref{chap4_sec4.2} that \eqref{eqE} has a unique solution $y\in C^0([0,+\infty);X_{-1})\cap W^{1,1}_\mathrm{loc}(0,+\infty;X_{-2})$, given by
\begin{equation}\label{defyfaible}
y(t;y_0,u)=S(t)y_0+L_tu
\end{equation}
where
\begin{equation}\label{def_Lt}
L_tu=\int_0^t S(t-s)Bu(s)\, ds.
\end{equation}
Moreover, the differential equation in \eqref{eqE} is written in $X_{-2}$. The integral \eqref{def_Lt} is done in $X_{-1}$.

Note that, of course, if $B\in L(U,X)$ is bounded, then the regularity moves up a rung: \eqref{eqE} has a unique solution $y\in C^0([0,+\infty);X)\cap W^{1,1}_\mathrm{loc}(0,+\infty;X_{-1})$ given as well by \eqref{defyfaible}, and the differential equation is written in $X_{-1}$.

\medskip

For a general (unbounded) control operator $B\in L(U,X_{-1})$, it is desirable to have conditions under which all solutions of \eqref{eqE} take their values in $X$, that is, under which the situation is as when the control operator is bounded.

Such control operators will be said to be \textit{admissible}. The admissibility property says that the control system is well posed in $X$ (note that it is always well posed in $X_{-1}$).

Of course, the notion of admissibility depends on the time-regularity of the inputs $u$. Since it will be characterized by duality, it is necessary, here, to fix once for all the class of controls.

In what follows, and in view of the Hilbert Uniqueness Method, we will actually deal with controls $u\in L^2([0,T],U)$ (for some $T>0$ arbitrary). Of course, we have $L^2([0,T],U)\subset L^1([0,T],U)$. Also, the duality will be easier to tackle in $L^2$ (although easy modifications can be done in what follows to deal with $L^p$, at least for $1< p\leq +\infty$, see \cite{Staffans} for exhaustive results).

Hence, from now on, the space of controls is $L^2([0,T],U)$.

\medskip

In this chapter, after having defined admissible operators, we will introduce different concepts of controllability, and show that they are equivalent, by duality, to some observability properties. Finally, we will explain the Hilbert Uniqueness Method (in short, HUM) introduced by J.-L. Lions in \cite{Lions_SIREV,Lions_HUM} in order to characterize the spaces where exact controllability holds true.

Most of this chapter is borrowed from \cite{TucsnakWeiss} (see also \cite{Lions_HUM,Staffans,Zabczyk}).

\section{Admissible control operators}\label{sec_admissible}
As said previously, we have a priori the inclusion $\mathrm{Ran}(L_T)\subset X_{-1}$, for every $T>0$, and the fact that $L_T\in L(L^2([0,T],U),X_{-1})$.

\subsection{Definition}
\begin{definition}
A control operator $B\in L(U,X_{-1})$ is said to be \textit{admissible} for the $C_0$ semigroup $(S(t))_{t\geq 0}$ if there exists $T>0$ such that $\mathrm{Ran}(L_T)\subset X$.
\end{definition}

\begin{lemma}\label{lemequivadmissible}
The following properties are equivalent:
\begin{itemize}[parsep=1mm,itemsep=1mm,topsep=1mm]%,leftmargin=*
\item There exists $T>0$ such that $\mathrm{Ran}(L_T)\subset X$.
\item For every $T>0$, one has $\mathrm{Ran}(L_T)\subset X$.
\item For every $T>0$, one has $L_T\in L(L^2([0,T],U),X)$.
\item All solutions \eqref{defyfaible} of \eqref{eqE}, with $y_0\in X$ and $u\in L^2([0,T],U)$, take their values in $X$.
\end{itemize}
\end{lemma}

\begin{proof}
Assume that $\mathrm{Ran}(L_T)\subset X$. Let us prove that $\mathrm{Ran}(L_T)\subset X$ for every $t>0$.

Let $t\in(0,T)$ arbitrary. For every control $u\in L^2([0,t],U)$, we define the control $\tilde u\in L^2([0,T],U)$ by $\tilde u(s) = 0$ for $s\in[0,T-t]$ and $\tilde u(s) = u(s-T+t)$ for $s\in[T-t,T]$.
Then, we have
$L_T \tilde u %= \int_0^T S(T-s)B \tilde u(s) \, ds 
= \int_{T-t}^T S(T-s)B u(s-T+t) \, ds = \int_0^t S(t-\tau)B u(\tau)\, d\tau = L_t u$ (with $\tau = s-T+t$).
It follows that if $\mathrm{Ran}(L_T)\subset X$ then $\mathrm{Ran}(L_T)\subset X$, for every $t\in(0,T)$.

Before proving the statement for $t>T$, let us note that, for every $u\in L^2(0,2T;U)$, we have
$
L_{2T}u = \int_0^{2T} S(2T-t)Bu(t)\, dt = \int_0^{T} S(2T-t)Bu(t)\, dt +  \int_T^{2T} S(2T-t)Bu(t)\, dt
= S(T)\int_0^{T} S(T-t)Bu(t)\, dt +  \int_0^{T} S(T-s)Bu(s+T)\, ds = S(T) L_T u_1 + L_T u_2,
$
with the controls $u_1$ and $u_2$ defined by $u_1(t) = u(t)$ and $u_2(t)=u(t+T)$ for almost every $t\in[0,T]$.
It follows that if $\mathrm{Ran}(L_T)\subset X$ then $\mathrm{Ran}(L_{2T})\subset X$, and by immediate iteration, this implies as well that $\mathrm{Ran}(L_{kT})\subset X$ for every $k\in\N^*$.

Now, let $t>T$ arbitrary, and let $k\in\N^*$ be such that $kT>t$. Since $\mathrm{Ran}(L_{kT})\subset X$, it follows from the first part of the proof that $\mathrm{Ran}(L_T)\subset X$.

It remains to prove that, if $\mathrm{Ran}(L_T)\subset X$, then $L_T\in L(L^2([0,T],U),X)$. Note first that the operator $L_T$ is closed. Indeed, we have
$L_T u = (\beta\,\mathrm{id}_X-A)\int_0^T S(T-t) (\beta\,\mathrm{id}_X-A)^{-1} B u(t) \, dt$,
for every $u\in L^1([0,T],U)$, with $\beta\in\rho(A)$ arbitrary. By definition of $X_{-1}$, the operator $(\beta\,\mathrm{id}_X-A)^{-1} B$ is linear and continuous from $U$ to $X$. Since $A$ is closed (according to Proposition \ref{propclosed}), it follows that $L_T$ is closed.
A priori, the graph of $L_T$ is contained in $X_{-1}$. Under the assumption that $\mathrm{Ran}(L_T)\subset X$, this graph is contained in $X$. Moreover, this graph is closed because the operator $L_T$ is closed. Then the fact that $L_T\in L(L^2([0,T],U),X)$ follows from the closed graph theorem.
\end{proof}

Note that, obviously, every bounded control operator $B\in L(U,X)$ is admissible. The question is however nontrivial for an unbounded control operator.

Classical examples of bounded control operators are obtained when one considers a controlled PDE with an internal control, that is, a control system of the form $\dot y(t)=Ay(t) + \chi_\omega u$, with $A:D(A)\rightarrow X=L^2(\Omega)$, where $\Omega$ is a domain of $\R^n$ and $\omega$ is a measurable subset of $\Omega$.

Unbounded control operators appear for instance when one considers a control acting along the boundary of $\Omega$ (see further for examples).

\begin{remark}
Note that, if $B$ is admissible, then the solution $y$ of \eqref{eqE} takes its values in $X$, and the equation $\dot{y}(t)=Ay(t)+Bu(t)$ is written in the space $X_{-1}$, almost everywhere on $[0,T]$.
The solution $y$ has the regularity $y\in C^0([0,T];X)\cap H^1([0,T],X_{-1})$ whenever $u\in L^2([0,T],U)$.
Note also that, in the term $L_Tu$, the integration is done in $X_{-1}$, but the result is in $X$ whenever $B$ is admissible.
\end{remark}

\begin{remark}\label{rem_Lp}
As said in the introduction, we have assumed that the class of controls is $L^2([0,T],U)$. We can define as well the concept of admissibility within the class of controls $L^p([0,T],U)$, for some $p\geq 1$, but we obtain then a different concept, called $p$-admissibility (for instance if $X$ is reflexive and $p=1$ then every admissible operator is necessarily bounded, see \cite[Theorem 4.8]{Weiss_SICON1989}). Here, we restrict ourselves to $p=2$ (in particular in view of HUM), which is the most usually encountered case.
\end{remark}

\subsection{Dual characterization of the admissibility}
Let us compute the adjoint of the operator $L_T\in L(L^2([0,T],U),X_{-1})$, and then derive a dual characterization of the admissibility property.

\begin{lemma}\label{lem_LT*}
Assume that $X$ and $U$ are reflexive. The adjoint $L_T^*$ satisfies $L_T^*\in L(D(A^*),L^2([0,T],U'))$, and is given by
$
(L_T^*z)(t) = B^* S(T-t)^* z
$
for every $z\in D(A^*)$ and for almost every $t\in[0,T]$.
\end{lemma}

\begin{proof}
Since $X$ is reflexive, we have $X_{-1}=D(A^*)'$ (see Theorem \ref{thm4.5}).
Since $L_T$ is a linear continuous operator from $L^2([0,T],U)$ to $D(A^*)'$, the adjoint $L_T^*$ is a linear continuous operator from $D(A^*)''$ to $L^2([0,T],U)'$.

On the one part, note that $L^2([0,T],U)'=L^2([0,T],U')$ because $U$ is reflexive.
On the other part, let us prove that $D(A^*)$ is reflexive (and hence, that $D(A^*)''=D(A^*)$). According to the Kakutani theorem (see \cite{Brezis}), it suffices to prove that the closed unit ball of $D(A^*)$ is compact for the weak topology $\sigma(D(A^*),D(A^*)')$.
We have, for some $\beta\in\rho(A)$,
\begin{equation*}
\begin{split}
B_{D(A^*)} &= \{ z\in D(A^*)\ \mid\ \Vert z\Vert_{D(A^*)} = \Vert (\beta\mathrm{id}_{X'}-A^*) z\Vert_{X'} \leq 1 \} \\
&= \{ (\beta\mathrm{id}_{X'}-A^*)^{-1} f\ \mid\ f\in X',\ \Vert f\Vert_{X'} \leq 1 \} \\
&= (\beta\mathrm{id}_{X'}-A^*)^{-1} B_{X'}
\end{split}
\end{equation*}
and since $X'$ is reflexive, the closed unit ball $B_{X'}$ is compact for the weak topology $\sigma(X',X'')$. Hence $D(A^*)$ is reflexive.

Therefore, $L_T^*\in L(D(A^*),L^2([0,T],U'))$.

Let $u\in L^2([0,T],U)$ and $z\in D(A^*)$. We have, by definition, and using the duality brackets with respect to the pivot space $X$,
$$
\langle L_Tu,z\rangle_{D(A^*)',D(A^*)} = \langle L_T^*z,u\rangle_{L^2([0,T],U'),L^2([0,T],U)}.
$$
Note that, here, we have implicitly used the fact that $U$ is reflexive. Now, noticing that $B\in L(U,D(A^*)')$ and hence that $B^*\in L(D(A^*),U')$, we have
\begin{equation*}
\begin{split}
\langle L_Tu,z\rangle_{D(A^*)',D(A^*)} &= \Big\langle \int_0^T S(T-t)Bu(t)\, dt , z\Big\rangle_{D(A^*)',D(A^*)} \\
&= \int_0^T \langle S(T-t)Bu(t) , z\rangle_{D(A^*)',D(A^*)}  \, dt \\
&= \int_0^T \langle B^*S(T-t)^*z,u(t)\rangle_{U',U}  \, dt \\
&= \left\langle t\mapsto B^*S(T-t)^*z, t\mapsto u(t)\right\rangle_{L^2([0,T],U)',L^2([0,T],U)}  
\end{split}
\end{equation*}
and the conclusion follows.
\end{proof}

The following proposition, providing a dual characterization of admissibility, is an immediate consequence of Lemmas \ref{lemequivadmissible} and \ref{lem_LT*}. Indeed in the admissible case we have $L_T\in L(L^2([0,T],U),X)$ and equivalently $L_T^*\in L(X',L^2([0,T],U)')$.\footnote{Note that, in the case where the operator $B\in L(U,X)$ is bounded, we always have $L_T\in L(L^2([0,T],U),X)$ and hence $L_T^*\in L(X',L^2([0,T],U)')$. Moreover, if $U$ is reflexive then $L(X',L^2([0,T],U)')=L(X',L^2([0,T],U'))$.}

\begin{proposition}\label{charactdualadm}
Assume that $X$ and $U$ are reflexive. 
The control operator $B\in L(U,X_{-1})$ (with $X_{-1}\simeq D(A^*)'$) is admissible if and only if, for some $T>0$ (and equivalently, for every $T>0$) there exists $K_T>0$ such that
\begin{equation}\label{ineg_adm}
\int_0^T \Vert B^*S(T-t)^*z\Vert_{U'}^2\, dt \leq K_T \Vert z\Vert_{X'}^2\qquad \forall z\in D(A^*).
\end{equation}
\end{proposition}

\begin{remark}
The inequality \eqref{ineg_adm} is called an \textit{admissibility inequality}. Establishing such an inequality is a way to prove that a control operator is admissible. Showing such energy-like inequalities is a classical issue in PDEs (Strichartz inequalities for instance).

Once again, we stress that the admissibility property means that the control system $\dot y(t)=Ay(t)+Bu(t)$ is \textit{well posed} in $X$, which means here that, for a control $u\in L^2([0,T],U)$ and an initial condition $y(0)\in X$, the corresponding solution $y(t)$ stays in $X$ indeed (and does not go in a wider space like $X_{-1}$). 

The concept of well-posedness of a PDE is in general a difficult issue. In finite dimension, this kind of difficulty does not exist, but in the infinite-dimensional setting, showing the admissibility of $B$ may already be a challenge (at least, for an unbounded control operator).
Examples are provided further.
\end{remark}

\begin{remark}
The inequality \eqref{ineg_adm} says that the operator $B^*\in L(D(A^*),U')$ is an admissible observation operator for the semigroup $S^*(t)$ (see \cite{TucsnakWeiss,Weiss_IJM1989}).
\end{remark}

\begin{remark}\label{rem_Lebext}
Note that the inequality \eqref{ineg_adm} is stated for every $z\in D(A^*)$. Of course, the norm $\Vert z\Vert_{X'}^2$ has a sense for $z$ belonging to the larger space $X'$, and it is natural to ask whether the inequality \eqref{ineg_adm} can be written for every $z\in X'$.

The question has been studied in \cite{Weiss_IJM1989}. If $z\in X'$ then $S(T-t)^*z\in X'$ and then we cannot apply $B^*$ to this element. Actually, we can replace $B^*$ with its \textit{$\Lambda$-extension}, defined by
$$
B^*_\Lambda z=\lim_{\lambda\rightarrow +\infty}B^*\lambda(\lambda \mathrm{id}_X-A^*)^{-1}z
$$
also called \textit{strong Yosida extension} in \cite[Definition 5.4.1]{Staffans} (note that $(\mathrm{id}_X-A^*)^{-1}z\in D(A^*)$ for every $z\in X'$) and defined on the domain $D(B^*_\Lambda)$ which is the set of $z$ for which the above limit exists.
Then Proposition \ref{charactdualadm} still holds true with $B^*$ replaced with $B^*_\Lambda$, with the inequality \eqref{ineg_adm} written for every $z\in X'$.

Actually, in the context of Lemma \ref{lem_LT*}, $L_T^*$ is given by $(L_T^*z)(t)=B_\Lambda^*S(T-t)^*z$, for every $z\in X'$ and for almost every $t\in[0,T]$.
\end{remark}

\subsection{Degree of unboundedness of the control operator}\label{sec_degree_unboundedness}
%As mentioned in the introduction, so far we have focused on control operators $B\in L(U,X_{-1})$. 
Recall that a control operator $B$ is said to be a \emph{bounded} if $B\in L(U,X)$.
When $X\subsetneq\mathrm{Ran}(B)$, it is said to be $\emph{unbounded}$ (although $B$ is bounded as an operator from $U$ to $X_{-1}$, if we have assumed that $B\in L(U,X_{-1})$).

Now more generally, using the scale of Banach spaces $(X_\alpha)_{\alpha\in\R}$ constructed in Section \ref{sec:scale}, we can define ``unbounded control" operators such that $B\in L(U,X_{-\alpha})$, for some $\alpha>0$. 
This leads to the notion of degree of unboundedness of a control operator (see \cite{RebarberWeiss,Staffans}, see also \cite[Chapter 4, page 138]{TucsnakWeiss}).

\begin{definition}\label{def_degree_unboundedness}
The \textit{degree of unboundedness} $\alpha(B)\geq 0$ of the control operator $B$ (with respect to the spaces $X$ and $U$) is the infimum of the (not necessarily closed!) set of $\alpha\geq 0$ such that $B\in L(U,X_{-\alpha})$, i.e., $(\beta\,\mathrm{id}_X-A)^{-\alpha}B\in L(U,X)$ for some arbitrary $\beta\in\rho(A)$ satisfying $\Real(\beta)>\omega$, where $(\beta\,\mathrm{id}_X-A)^{-\alpha}$ is defined by \eqref{def_poweralpha}.
In other words, $B\in L(U,X_{-\alpha(B)-\varepsilon})$ for every $\varepsilon>0$ (but not necessarily for $\varepsilon=0$).

Equivalently,\footnote{This second definition of $\alpha(B)$ is given in \cite{RebarberWeiss}. The equivalence with the first one follows from \cite[Chapter II, Section 5, Propositions 5.12, 5.14 and 5.33]{Nagel}.}
$\alpha(B)$ is equal to the infimum of the set of $\alpha\geq 0$ for which there exists $C_\alpha>0$ such that
$$
\Vert (\lambda\mathrm{id}_X-A)^{-1}B\Vert \leq \frac{C_\alpha}{\lambda^{1-\alpha}}\qquad \forall \lambda>\omega .
$$
\end{definition}

When $X$ is reflexive, $\alpha(B)$ is the infimum of the set of $\alpha\geq 0$ such that $B^*\in L(X_\alpha^* , U')$ (where $X_\alpha^*=D((\beta\,\mathrm{id}_X-A^*)^\alpha)$ thanks to \eqref{Xalphastar}), or equivalently, of the set of $\alpha\geq 0$ for which there exists $C_\alpha>0$ such that
\begin{equation}\label{inegalpha}
\Vert B^* z\Vert_{U'}\leq C_\alpha \Vert (\beta\,\mathrm{id}_X-A^*)^\alpha z\Vert_{X'} \qquad\forall z\in X_\alpha^* .
\end{equation}
Note that \eqref{inegalpha} may fail for $\alpha=\alpha(B)$.
%% Cf Rebarber Weiss SCL 2000

Throughout the chapter, as in most of the existing literature, we consider control operators such that $\alpha(B)\leq 1$. This covers the most usual applications. Internal controls are bounded control operators (thus, $\alpha(B)=0$). For boundary controls, in general we have $0<\alpha(B)\leq 1$ (see further, and see \cite{LasieckaTriggiani,TucsnakWeiss} for many examples).

%With this notion of degree of unboundedness, it is interesting to note the following result.

\begin{lemma}\label{lem_admissible}
\begin{enumerate}
\item Assume that $X$ and $U$ are reflexive Banach spaces. If $B$ is admissible then $B\in L(U,X_{-1/2})$ (and thus $\alpha(B)\leq 1/2$).
\item Assume that $X$ is a Hilbert space and that $A$ is self-adjoint. If $B\in L(U,X_{-1/2})$ then $B$ is admissible.
\end{enumerate}
\end{lemma}

%Actually, the second point of the lemma is true under the more general assumption that the semigroup generated by $A$ is either analytic and normal or invertible (see \cite{Weiss_1991}).

\begin{proof}
Let us prove the first point (adapted from \cite[Proposition 4.4.6]{TucsnakWeiss}).

First of all, by an obvious change of variable, we have $L_{t+s}=S(s)L_t+L_s$, for all $s,t\geq 0$. It follows that $L_n=(S(n-1)+\cdots+S(1)+\mathrm{id}_X)L_1$, for every $n\in\N^*$. Since $(S(t))_{t\geq 0}\in\mathcal{G}(M,\omega)$, we have $\Vert S(k)\Vert\leq Me^{k\omega}$ for every integer $k$, and therefore, we infer that $\Vert L_n\Vert\leq K_n\Vert L_1\Vert$, with $K_n=M\frac{e^{\omega n}-1}{e^{\omega}-1}$ if $\omega>0$, $K_n=Mn$ if $\omega=0$, and $K_n=M\frac{1}{1-e^{\omega}}$ if $\omega<0$. Here, the norm $\Vert\cdot\Vert$ stands for the norm of bounded operators from $L^2_\mathrm{loc}(0,+\infty;U')$ to $X$ (note that $B$ is assumed to be admissible).

Besides, for $0<t_1<t_2$ arbitrary, by taking controls that are equal to $0$ on $(t_1,t_2)$, we easily prove that $\Vert L_{t_1}\Vert\leq \Vert L_{t_2}\Vert$.

Now, for an arbitrary $T>0$, let $n\in\N$ be such that $n\leq T<n+1$. Writing $L_T=S(T-n)L_n+L_{T-n}$, we get that $\Vert L_T\Vert\leq Me^{\omega(T-n)}\Vert L_n\Vert+\Vert L_1\Vert$.

It finally follows that $\Vert L_T\Vert\leq Ke^{\omega T}$ if $\omega>0$, $\Vert L_T\Vert\leq KT$ if $\omega=0$, and $\Vert L_T\Vert\leq K$ if $\omega<0$, for some constant $K>0$ that does not depend on $T$.

By duality, we have the same estimates on the norm of $L_T^*$.

For every $z\in D(A^*)$, we define $(\Psi z)(t)=B^*S(t)^*z$. It follows from the above estimates (by letting $T$ tend to $+\infty$) that, for every $\alpha>\omega$, the function $t\mapsto e^{-\alpha t}(\Psi z)(t)$ belongs to $L^2(0,+\infty;U')$.

Let us consider the Laplace transform of $\Psi z$. 
On the one part, we have, by an easy computation as in \eqref{laplacetransform},
$
\mathcal{L}(\Psi z)(s) = \int_0^{+\infty} e^{-st} (\Psi z)(t)\, dt = B^*(s\,\mathrm{id}_X-A^*)z
$
for every $s\in\rho(A)$ such that $\Real(s)>\omega$.
On the other part, writing $\mathcal{L}(\Psi z)(s) = \int_0^{+\infty} e^{(\alpha-s)t} e^{-\alpha t}(\Psi z)(t)\, dt$ and applying the Cauchy-Schwarz inequality, we get
$$
\Vert \mathcal{L}(\Psi z)(s)\Vert_{U'} \leq \frac{1}{\sqrt{2(\Real(s)-\alpha)}} \Vert t\mapsto e^{-\alpha t}(\Psi z)(t)\Vert_{L^2(0,+\infty;U')}.
$$
%The conclusion follows.
The first point is proved.
% source: livre de Tucsnak Weiss, adaptation de la Proposition 4.4.6 page 128. Cf aussi page 138.

%\medskip

Let us prove the second point (adapted from \cite[Proposition 5.1.3]{TucsnakWeiss}).
By definition, there exists $C>0$ such that 
$$
\Vert B^*z\Vert_{U'}^2\leq C \Vert (\beta\,\mathrm{id}_X-A^*)^{1/2} z\Vert_{X'}^2 = C (z,(\beta\,\mathrm{id}_X-A)z)_X
$$
for every $z\in D(A)$ (we have used that $A=A^*$), where $(\cdot,\cdot)_X$ is the scalar product in $X$. Applying this inequality to $z(t)=S^*(T-t)\psi$, multiplying by $e^{2\beta t}$, and integrating over $[0,T]$, we get
$$ 
\int_0^T e^{2\beta t} \Vert B^*S^*(T-t)\psi\Vert_{U'}^2\, dt \leq C \int_0^T e^{2\beta t} (z(t),(\beta\,\mathrm{id}_X-A)z(t))_X \, dt .
$$
Since $\dot z(t)=-Az(t)$ and $z(T)=\psi$, we have
\begin{equation*}
\begin{split}
\frac{1}{2} \frac{d}{dt} \left( e^{2\beta t}\Vert z(t)\Vert_X^2 \right) &=  e^{2\beta t} \left( \beta\Vert z(t)\Vert_X^2 - (z(t),Az(t))_X \right) \\
&=  e^{2\beta t} ( z(t), (\beta\,\mathrm{id}_X-A)z(t) )_X
\end{split}
\end{equation*}
and therefore we get
$$ 
\int_0^T e^{2\beta t} \Vert B^*S^*(T-t)\psi\Vert_{U'}^2\, dt \leq \frac{C}{2} \int_0^T \frac{d}{dt} \left( e^{2\beta t}\Vert z(t)\Vert_X^2 \right) \, dt 
\leq \frac{C}{2} e^{2\beta T} \Vert\psi\Vert_X^2.
$$
The lemma follows.
% Source: livre de Tucsnak Weiss, adaptation de la Proposition 5.1.3 page 143
\end{proof}

\subsection{Examples}

\subsubsection{Dirichlet heat equation with internal control}\label{exsec}
Let $\Omega\subset\R^n$ be a bounded open set with $C^2$ boundary, and let $\omega\subset\Omega$ be an open subset. Consider the internally controlled Dirichlet heat equation
\begin{equation*}
\partial_ty=\triangle y+\chi_\omega u\ \ \textrm{in}\ \Omega,\qquad
y_{\vert\partial\Omega}=0,\qquad
y(0)=y_0\in L^2(\Omega).
\end{equation*}
We set $X=L^2(\Omega)$, and we consider the operator $A=\triangle_D:D(A)\rightarrow X$, where $D(A)=X_1=H^1_0(\Omega)\cap H^2(\Omega)$.
The operator $A$ is self-adjoint, and $X_{-1}=D(A^*)'=(H^1_0(\Omega)\cap H^2(\Omega))'$ with respect to the pivot space $L^2(\Omega)$.
The control operator $B$ is defined as follows: for every $u\in U=L^2(\omega)$, $Bu\in L^2(\Omega)$ is the extension of $u$ by $0$ to the whole $\Omega$. It is bounded and therefore admissible (by Lemma \ref{lem_admissible}), which means (by Proposition \ref{charactdualadm}) that, for every $T>0$, there exists $K_T>0$ such that
$$
\int_0^T \int_\omega \psi(t,x)^2\, dx\, dt\leq K_T\Vert\psi(0)\Vert^2_{L^2(\Omega)}
$$
for every solution of $\partial_t\psi=\triangle\psi\ \ \textrm{in}\ \Omega$, $\psi_{\vert\partial\Omega}=0$, with $\psi(0)\in H^2(\Omega)\cap H^1_0(\Omega)$.
The above inequality can also be established by using that $\psi(t)=S(t)\psi(0)$ with $S(t)\in L(X)$.

\subsubsection{Heat equation with Dirichlet boundary control}\label{sec_heatdirbound}
Let $\Omega\subset\R^n$ be a bounded open set with Lipschitz boundary. Consider the heat equation with Dirichlet boundary control
\begin{equation}\label{boundarycontrolheat}
\partial_ty=\triangle y\ \ \textrm{in}\ \Omega,\qquad
y_{\vert\partial\Omega}=u,\qquad
y(0)=y_0\in L^2(\Omega).
\end{equation}
We set $U=L^2(\partial\Omega)$, $X=H^{-1}(\Omega)$, and we consider the self-adjoint operator $A=\triangle_D:D(A)\rightarrow X$ where $D(A)=X_1=H^1_0(\Omega)$.
%The operator $A$ is self-adjoint, hence $D(A^*)=D(A)$. 
We have $X_{-1}=D(A^*)'=D(A)'$ with respect to the pivot space $H^{-1}(\Omega)$. Note that if $\partial\Omega$ is regular enough then $X_{-1}$ is the dual of $A^{-1}(H^1_0(\Omega)) = \{y\in H^3(\Omega)\ \mid\ y_{\vert\partial\Omega}=(\triangle y)_{\vert\partial\Omega}=0\}$ with respect to the pivot space $L^2(\Omega)$ (see Example \ref{example_chain_Dirichlet}).

Let us express the control operator $B\in L(U,D(A^*)')$. 
We preliminarily recall (see Example \ref{example_chain_Dirichlet}, with $\mathcal{A}=-\triangle_D=-A$) that, for every $f\in H^{-1}(\Omega)$, $\Vert f\Vert_{H^{-1}(\Omega)} = \Vert (-\triangle_D)^{-1/2} f\Vert_{L^2(\Omega)}$ and 
$$
(f,g)_{H^{-1}(\Omega)} = ( (-\triangle_D)^{-1/2}f , (-\triangle_D)^{-1/2} g)_{L^2(\Omega)} = -(A^{-1}f, g)_{L^2(\Omega)}
$$
for all $f\in H^{-1}(\Omega)$ and $g\in L^2(\Omega)$.
Taking a solution $y$ regular enough, associated with a control $u$ (for instance, of class $C^1$),
since the differential equation $\dot y=Ay+Bu$ is written in $X_{-1}$, by definition we have
$$
\langle \dot y,\phi\rangle_{X_{-1},X_1} = \langle Ay,\phi\rangle_{X_{-1},X_1}+ \langle Bu,\phi\rangle_{X_{-1},X_1}
\qquad\forall\phi\in X_1 .
$$
The duality bracket is considered with respect to the pivot space $X=H^{-1}(\Omega)$, hence $\langle \dot y,\phi\rangle_{X_{-1},X_1} = -(\dot y,A^{-1}\phi)_{L^2(\Omega)}$ and $\langle Ay,\phi\rangle_{X_{-1},X_1} = (Ay,\phi)_{H^{-1}(\Omega)} = -(y,\phi)_{L^2(\Omega)}$. 
Besides, using \eqref{boundarycontrolheat} and integrating by parts (Green formula), we have
$$
(\partial_ty,\psi)_{L^2(\Omega)} = (y,\triangle\psi)_{L^2(\Omega)} - \Big( u , \frac{\partial\psi}{\partial\nu} \Big)_{L^2(\partial\Omega)} \qquad \forall \psi\in H^1_0(\Omega)\cap H^2(\Omega) .
$$
Taking $\psi=-A^{-1}\phi$, we have $\triangle\psi=-\phi$ on $\Omega$ and, by identification,
$$
\langle Bu,\phi\rangle_{X_{-1},X_1}
= \Big( u  , \frac{\partial}{\partial\nu}(A^{-1}\phi) \Big)_{L^2(\partial\Omega)}
\qquad \forall u\in L^2(\partial\Omega)\quad \forall \phi\in H^1_0(\Omega) ,
$$
which defines $B$ by transposition: since $\langle Bu,\phi\rangle_{X_{-1},X_1} = (u, B^*\phi)_U$, we infer that
$B^*\phi = \frac{\partial}{\partial\nu}_{\vert\partial\Omega}(A^{-1}\phi)$ for every $\phi\in H^1_0(\Omega)=D(A^*)$.

When $\partial\Omega$ is $C^2$, we can express the operator $B$ by using the \emph{Dirichlet map} $D$ (see \cite{LionsMagenes} and \cite[Chapters 10.6 and 10.7]{TucsnakWeiss}), which is the linear operator defined as follows:
%as the adjoint of $D^*\in L(L^2(\Omega),L^2(\partial\Omega))$ given by $D^*g=\frac{\partial}{\partial\nu}_{\vert\partial\Omega}(A^{-1}g)$, for every $g\in L^2(\Omega)$. Equivalently, 
given any $v\in L^2(\partial\Omega)$, $Dv$ is the unique solution in the sense of distributions of the Laplace equation $\triangle(Dv)=0$ in $\Omega$ such that $(Dv)_{\vert\partial\Omega}=v$. 
Note that $Dv\in C^\infty(\Omega)$ by hypoellipticity.
Integrating by parts (Green formula), we have $( Dv , \triangle\varphi )_{L^2(\Omega)} = \big( v , \frac{\partial\varphi}{\partial\nu} \big)_{L^2(\partial\Omega)}$ for every $\varphi\in H^1_0(\Omega)\cap H^2(\Omega)$, and thus, setting $\phi=\triangle_D\varphi = A\varphi$, we have $( Dv , \phi )_{L^2(\Omega)} = \big( v , \frac{\partial}{\partial\nu}(A^{-1}\phi) \big)_{L^2(\partial\Omega)}$ for every $\phi\in L^2(\Omega)$. 
This implies that $D:L^2(\partial\Omega)\rightarrow L^2(\Omega)$ (but $D$ is not surjective).
%
%We identify $L^2(\partial\Omega)'=L^2(\partial\Omega)$, but the dual $L^2(\Omega)'$ is taken with respect to the pivot space $X=H^{-1}(\Omega)$ and is thus identified with $H^1_0(\Omega)\cap H^2(\Omega)$. Hence $D^*\in L(H^1_0(\Omega)\cap H^2(\Omega),L^2(\partial\Omega))$.
%Since $( Dv , g )_{L^2(\Omega)} = - ( Dv , Ag )_{H^{-1}(\Omega)} = - \langle Ag , Dv \rangle_{X_{-1},X_1}$
%
% and that its adjoint $D^*\in L(L^2(\Omega),L^2(\partial\Omega))$ is given by $D^*g=\frac{\partial}{\partial\nu}_{\vert\partial\Omega}(A^{-1}g)$ for every $g\in L^2(\Omega)$.
%
%
Then
$\langle Bu,\phi\rangle_{X_{-1},X_1} = ( Du , \phi )_{L^2(\Omega)} = - ( Du , A\phi )_{H^{-1}(\Omega)} = - \langle ADu,\phi\rangle_{X_{-1},X_1}$ and therefore 
$$
B=-AD
$$
where we consider the extension $A:L^2(\Omega)\rightarrow (H^2(\Omega)\cap H^1_0(\Omega))'$ (dual with respect to $L^2(\Omega)$). Note that $X_{1/2} = L^2(\Omega)$ and that $X_{-1/2} = (H^2(\Omega)\cap H^1_0(\Omega))'$.
In particular, we have $B\in L(U,X_{-1/2})$ (and $B^*\in L(X_{1/2},U)$), and thus $\alpha(B)\leq 1/2$.

Actually, by using the finer fact that $D:L^2(\partial\Omega)\rightarrow H^{1/2}(\Omega)$ (see \cite{LionsMagenes}, see also \cite[Chapter 3.1]{LasieckaTriggiani}, it can be proved that $\alpha(B)=1/4$. 
%% voir en particulier LT Remark 3.1.4 page 186 : espace de Lions-Magenes. 

It follows from Lemma \ref{lem_admissible} that $B$ is admissible; equivalently, by Proposition \ref{charactdualadm}, for every $T>0$, there exists $K_T>0$ such that
$$
\int_0^T\left\Vert\frac{\partial\psi}{\partial\nu}_{\vert\partial\Omega} (t)\right\Vert_{L^2(\partial\Omega)}^2dt \leq K_T\Vert\psi(0)\Vert^2_{H^1_0(\Omega)}
$$
for every solution of $\partial_t\psi=\triangle\psi\ \ \textrm{in}\ \Omega$, $\psi_{\vert\partial\Omega}=0$, 
with $\psi(0)\in %H^2(\Omega)\cap 
H^1_0(\Omega)$.
This result says that the heat equation with boundary control \eqref{boundarycontrolheat} is well posed in the state space $X=H^{-1}(\Omega)$, with $U=L^2(\partial\Omega)$. 

Note that \eqref{boundarycontrolheat} is not well posed in the state space $L^2(\Omega)$, meaning that, for $y^0\in L^2(\Omega)$ and $u\in L^2([0,T],\partial\Omega)$, the solution $y$ of \eqref{boundarycontrolheat} may fail to belong to $C^0([0,T];L^2(\Omega))$ (even in dimension one). Actually, for every $T>0$,
$$
\sup \bigg\{ \int_0^T\left\Vert\frac{\partial\psi}{\partial\nu}_{\vert\partial\Omega} (t)\right\Vert_{L^2(\partial\Omega)}^2dt \ \mid\ \psi(0)\in L^2(\Omega), \ \Vert\psi(0)\Vert_{L^2(\Omega)}=1 \bigg\} = +\infty
$$
(see \cite{Lions_HUM,LionsMagenes}).
If we take $X=L^2(\Omega)$ then $B^*\phi = -\frac{\partial\phi}{\partial\nu}_{\vert\partial\Omega}$ for every $\phi\in H^2(\Omega)\cap H^1_0(\Omega)$ and $B=-AD\in L(U,D(A^{3/4+\varepsilon})')$ for every $\varepsilon>0$ (see \cite[Chapter 3.1]{LasieckaTriggiani}), but the continuity property fails for $\varepsilon=0$ (even in dimension one). We have then $\alpha(B)=3/4$ and $B$ is not admissible by Lemma \ref{lem_admissible}.

% voir LT_book1991 p. 17 (section 2.3.1): equation de la chaleur avec controle Dirichlet : l'operateur de controle n'est pas admissible. Ils referent Lions p. 217. C'est redit page 51 (section 6.1). Voir aussi Neumann page 53.

\subsubsection{Heat equation with Neumann boundary control}
We replace in \eqref{boundarycontrolheat} the Dirichlet control with the Neumann control $\frac{\partial y}{\partial\nu}_{\vert\partial\Omega}=u$. In this case, we set $X=L^2(\Omega)$, $U=L^2(\partial\Omega)$, we consider the operator $A=\triangle_N$ defined on $D(A)=\{ y \in H^2(\Omega)\ \mid\ \frac{\partial y}{\partial\nu}_{\vert\partial\Omega}=0\}$, and we obtain $B^*\phi = \phi_{\vert\partial\Omega}$ and $B=-AN$, where $N$ is the Neumann map. We do not provide any details. Actually, we have $B\in L(U,(D(A^{1/4+\varepsilon}))')$ for every $\varepsilon>0$ (see \cite[Chapter 3.3]{LasieckaTriggiani}), thus $\alpha(B)=1/4$, and hence $B$ is admissible by Lemma \ref{lem_admissible}.

% Voir aussi LT_book1991 page 53 (section 6.2).

\subsubsection{Second-order equations}
Another typical example is provided by second-order equations. The framework is the following (see \cite{TucsnakWeiss}).
Let $H$ be a Hilbert space, and $A_0:D(A_0)\rightarrow H$ be
self-adjoint and positive. Recall that $D(A_0^{1/2})$ is the
completion of $D(A_0)$ with respect to the norm $\Vert y\Vert_{D(A_0^{1/2})}=\sqrt{\langle A_0y,y\rangle_H}$, and that $D(A_0)\subset D(A_0^{1/2})\subset H$, with continuous and dense embeddings (see also Remark \ref{rem41}). We set $X=D(A_0^{1/2})\times H$, and we define the skew-adjoint operator $A:D(A)\rightarrow X$ on $D(A)=D(A_0)\times D(A_0^{1/2})$ by
$$
A=\begin{pmatrix} 0 & I \\ -A_0 & 0 \end{pmatrix}.
$$

Let $U$ be a Hilbert space and let $B_0\in L(U,D(A_0^{1/2})')$, where $D(A_0^{1/2})'$ is the dual of $D(A_0^{1/2})$ with respect to the pivot space $H$.
The second-order control system
$$
\partial_{tt}y+A_0y=B_0u %,\qquad  y(0)=y_0,\quad \partial_ty(0)=y_1.
$$
can be written in the form
$$
\frac{\partial}{\partial t} 
\begin{pmatrix} y\\ \partial_ty\end{pmatrix}
= A \begin{pmatrix} y\\ \partial_ty\end{pmatrix} +Bu  
\qquad\textrm{with}\qquad
B=\begin{pmatrix} 0\\ B_0\end{pmatrix}.
$$
We have
$X_{-1}=D(A^*)'=H\times D(A_0^{1/2})'$ with respect to the pivot space $X$, where $D(A_0^{1/2})'$ is the dual
of $D(A_0^{1/2})$ with respect to the pivot space $H$.
Moreover we have $B\in L(U,H\times D(A_0^{1/2})')$ and
$$
B^*=\begin{pmatrix} 0\\ B_0^*\end{pmatrix} \in L(D(A_0)\times D(A_0^{1/2}),U) .
$$

\begin{proposition}\label{prop_waveadm}
The following statements are equivalent:
\begin{itemize}[parsep=1mm,itemsep=1mm,topsep=1mm]%,leftmargin=*
\item $B$ is admissible.
\item There exists $K_T>0$ such that every solution of
\begin{equation*}
\partial_{tt}\psi+A_0\psi=0,\qquad
\psi(0)\in D(A_0),\quad \partial_t\psi(0)\in D(A_0^{1/2})
\end{equation*}
satisfies
$\displaystyle
\int_0^T\Vert B_0^*\partial_t\psi(t)\Vert_{U'}^2\,dt\leq K_T\left(\Vert\psi(0)\Vert_{D(A_0^{1/2})}^2+\Vert\partial_t\psi(0)\Vert_H^2\right).
$
\item There exists $K_T>0$ (the same constant) such that every solution of
\begin{equation*}
\partial_{tt}\psi+A_0\psi=0,\qquad
\psi(0)\in H,\quad \partial_t\psi(0)\in D(A_0^{1/2})'
\end{equation*}
satisfies
$\displaystyle
\int_0^T\Vert B_0^*\psi(t)\Vert_{U'}^2\,dt \leq  K_T\left(\Vert\psi(0)\Vert_{H}^2+\Vert\partial_t\psi(0)\Vert_{D(A_0^{1/2})'}^2\right).
$
\end{itemize}
\end{proposition}

\begin{example}\label{examplewave}
Consider the wave equation with Dirichlet boundary control
\begin{equation*}
\partial_{tt}y=\triangle y\quad\textrm{in}\ \Omega,\qquad
y_{\vert\partial\Omega}=u,
\end{equation*}
where $\Omega$ is a bounded open subset of $\R^n$ with $C^2$ boundary. We set $H=H^{-1}(\Omega)$, and we take $A_0=-\triangle_D:D(A_0)=H^1_0(\Omega)\rightarrow H$ (isomorphism). We have $D(A_0^{1/2})=L^2(\Omega)$, and the dual space $D(A_0^{1/2})'$ (with respect to the pivot space $H=H^{-1}(\Omega)$) is equal to the dual space $(H^2(\Omega)\cap H^1_0(\Omega))'$ (with respect to the pivot space $L^2(\Omega)$).
The state space is 
$$
X=D(A_0^{1/2})\times H = L^2(\Omega)\times H^{-1}(\Omega)
$$
and we have 
$$
X_1=D(A) = H^1_0(\Omega)\times L^2(\Omega), \qquad
X_{-1} = H^{-1}(\Omega)\times (H^2(\Omega)\cap H^1_0(\Omega))' .
$$
The spaces $X_\alpha$ can be characterized easily.
Setting $U=L^2(\partial\Omega)$, the controlled wave equation is written as $y_{tt}=-A_0y+B_0u$ in $D(A_0^{1/2})'$, where 
$$
B_0^*\phi = \frac{\partial}{\partial\nu}_{\vert\partial\Omega}(A_0^{-1}\phi) \qquad \forall \phi\in L^2(\Omega) ,
$$
or, equivalently, $B_0=A_{-1}D \in L(U,D(A_0^{1/2})')$ where $D$ is the Dirichlet mapping.
Then $B\in L(U,X_{-1/2})$, but this is the limit case of Lemma \ref{lem_admissible}.
It is however true that $B$ is admissible, i.e., for every $T>0$ there exists $K_T>0$ such that
$$
\int_0^T\left\Vert\frac{\partial\psi}{\partial\nu}(t)\right\Vert^2_{L^2(\partial\Omega)}dt\leq
K_T\left(\Vert\psi(0)\Vert^2_{H_0^1(\Omega)}+\Vert\partial_t\psi(0)\Vert^2_{L^2(\Omega)}\right)
$$
for every solution of $\partial_{tt}\psi=\triangle\psi$ in $\Omega$, $\psi_{\vert\partial\Omega}=0$.
This is the \emph{hidden regularity property} for the Dirichlet wave equation, proved in \cite{Lions_SIREV, Lions_HUM, LionsMagenes} (see also \cite[Theorem 7.1.3]{TucsnakWeiss}) by using \emph{multipliers}. 
The multiplier method consists of multiplying the evolution equations by adequate functions and then using integrations by parts.
\end{example}

% Wave equation: LT_book1991 
% Remark 3.4 page 27,   section 7.1 p. 71

\begin{example}\label{examplewaveinternal}
Consider the Dirichlet wave equation with internal control 
$$
\partial_{tt}y=\triangle y+\chi_\omega u\quad\textrm{in}\ \Omega, \qquad y_{\vert\partial\Omega}=0
$$
on a bounded open subset $\Omega$ with Lipschitz boundary. In this case, we set $H=L^2(\Omega)$, we take $A_0=-\triangle_D:D(A_0)=H^1_0(\Omega)\cap H^2(\Omega)\rightarrow H$ (isomorphism). We have $D(A_0^{1/2})=H^1_0(\Omega)$, and the dual space $D(A_0^{1/2})'$ (with respect to the pivot space $H=L^2$) is equal to the dual space $H^{-1}(\Omega)$.
The state space is $X=D(A_0^{1/2})\times H = H^1_0(\Omega)\times L^2(\Omega)$, and we have $X_1=D(A) = H^1_0(\Omega)\cap H^2(\Omega)\times H^1_0(\Omega)$ and $X_{-1} = L^2(\Omega)\times H^{-1}(\Omega)$. 
Setting $U=L^2(\omega)$, the bounded control operator $B_0\in L(U,X)$ is such that, for every $u\in U$, $Bu$ is the extension of $u$ by $0$ to the whole $\Omega$. Its admissibility (which is obvious) means that, for every $T>0$, there exists $K_T>0$ such that
$$
\int_0^T \int_\omega \psi(t,x)^2\, dx\, dt\leq
K_T\left(\Vert\psi(0)\Vert^2_{L^2(\Omega)}+\Vert\partial_t\psi(0)\Vert^2_{H^{-1}(\Omega)}\right)
$$
for every solution of $\partial_{tt}\psi=\triangle\psi$ in $\Omega$, $\psi_{\vert\partial\Omega}=0$.
\end{example}

We refer to \cite{LasieckaTriggiani, TucsnakWeiss} (and references cited therein) for many other examples.

\section{Controllability}
We consider the linear control system \eqref{eqE}. % with the framework given in the introduction of the chapter. 
We do not assume that $B$ is admissible.

\subsection{Definitions}
Let us define the concept of controllability.
A priori, the most natural concept is to require that for a given time $T$, for all $y_0$ and $y_1$ in $X$, there exists a control $u\in L^2([0,T],U)$ and a solution of \eqref{eqE} such that $y(0)=y_0$ satisfies $y(T)=y_1$.
In finite dimension, a necessary and sufficient condition is the Kalman condition. In infinite dimension, new difficulties appear.
Indeed, let us consider a heat equation settled on a domain $\Omega$ of $\R^n$, with either an internal or a boundary control. Due to the smoothing effect (see Remark \ref{rem_regular}, see also \cite{CazenaveHaraux}), whatever the regularity of the initial condition and of the control may be, the solution $y(t,\cdot)$ is a smooth function (of $x$) as soon as $t>0$, outside of the control domain. It is therefore hopeless to try to reach a final target $y(T)=y_1\in L^2(\Omega)$ in general (unless $y_1$ is smooth enough, so as to belong to the range of the heat semigroup).
However, for such a parabolic equation, it makes sense to reach either $y(T)=0$, or to "almost reach" any $y_1\in L^2(\Omega)$.
This motivates the following definitions.

\begin{definition}\label{defnotionscont}
Let $T>0$ arbitrary. The control system \eqref{eqE} is said to be:
\begin{itemize}[parsep=1mm,itemsep=1mm,topsep=1mm]%,leftmargin=*
\item \textit{exactly controllable} in (the state space) $X$ in time $T$ if, for all $(y_0,y_1)\in X^2$, there exists $u\in L^2([0,T],U)$ such that the solution \eqref{defyfaible} of \eqref{eqE} satisfies $y(T;y_0,u)=y_1$;
\item \textit{approximately controllable} in $X$ in time $T$ if, for all $(y_0,y_1)\in X^2$, for every $\varepsilon>0$, there exists $u\in L^2([0,T],U)$ such that $\Vert y(T;y_0,u)-y_1\Vert_X\leq\varepsilon$;
\item \textit{exactly null controllable} in $X$ in time $T$ if, for every $y_0\in X$, there exists
$u(\cdot)\in L^2([0,T],U)$ such that $y(T;y_0,u)=0$.
\end{itemize}
\end{definition}

\begin{remark}\label{remprelimdualitycontobs}
Using the fact that $y(T)=S(T)y_0+L_Tu$ (see \eqref{defyfaible}), with $L_T$ defined by \eqref{def_Lt}, we make the following remarks:
\begin{itemize}[parsep=1mm,itemsep=1mm,topsep=1mm]%,leftmargin=*
\item The control system \eqref{eqE} is exactly controllable in $X$ in time $T$ if and only if $\mathrm{Ran}(L_T)=X$.
In particular, if this is true, then $B$ must be admissible and thus $\alpha(B)\leq 1/2$.
\item The control system \eqref{eqE} is approximately controllable in $X$ in time $T$ if and only if $\mathrm{Ran}(L_T)\cap X$ is dense in $X$.
\item The control system \eqref{eqE} is exactly null controllable in $X$ in time $T$ if and only if $\mathrm{Ran}(S(T))\subset\mathrm{Ran}(L_T)$.
\end{itemize}
\end{remark}

\begin{remark}
Note that, if $\mathrm{Ran}(L_T)=X$ for some $T>0$, then $\mathrm{Ran}(L_t)=X$ for every $t\geq T$. Indeed, taking (as in the proof of Lemma \ref{lemequivadmissible}) controls such that $u=0$ on $(0,t-T)$, we have 
$L_t u %= \int_0^t S(t-s) u(s) \, ds 
= \int_{t-T}^t S(t-s) u(s) \, ds = \int_0^T S(T-\tau)u(\tau+t-T)\, d\tau = L_Tu(\cdot+t-T)$. %The conclusion follows.
This shows that if the control system \eqref{eqE} is exactly controllable in time $T$ then it is exactly controllable in any time $t\geq T$.
\end{remark}

\begin{remark}
We speak of approximate null controllability in time $T$ when one takes the target $y_1=0$ in Definition \ref{defnotionscont}; equivalently, $\mathrm{Ran}(S(T))$ is contained in the closure of $\mathrm{Ran}(L_T)$. Approximate controllability and approximate null controllability (in time $T$) coincide when $S(T)^*$ is injective, i.e., when $\mathrm{Ran}(S(T))$ is dense in $X$ (see \cite{TrelatWangXu} for finer results).

There are other notions of controllability, depending on the context and on the needs, for instance: spectral controllability, controllability to finite-dimensional subspaces, controllability to trajectories (see, e.g., \cite{Coron, TucsnakWeiss} and references therein).
\end{remark}

\subsection{Duality controllability -- observability}\label{sec_duality}
As for the admissibility, we are going to provide a dual characterization of the controllability properties. To this aim, it suffices to combine Remark \ref{remprelimdualitycontobs} with the following general lemma of functional analysis (see \cite{Brezis} for the first part and \cite{LiYong,Zabczyk} for the last part).

\begin{lemma}\label{lemgenanafonc}
Let $X$ and $Y$ be Banach spaces, and let $F\in L(X,Y)$. Then:
\begin{itemize}[parsep=1mm,itemsep=1mm,topsep=1mm]%,leftmargin=*
\item $\mathrm{Ran}(F)$ is dense in $Y$ (that is, $F$ is "approximately surjective") if and only if $F^*\in L(Y',X')$ is one-to-one, that is: for every $z\in Y'$, if $F^*z=0$ then $z=0$.
\item $\mathrm{Ran}(F)=Y$ (that is, $F$ is surjective) if and only if $F^*\in L(Y',X')$ is bounded below, in the sense that there exists $C>0$ such that $\Vert F^*z\Vert_{X'}\geq C\Vert z\Vert_{Y'}$ for every $z\in Y'$.
\end{itemize}
Let $X$, $Y$ and $Z$ be Banach spaces, with $Y$ reflexive, and let $F\in L(X,Z)$ and $G\in L(Y,Z)$. Then $\mathrm{Ran}(F)\subset \mathrm{Ran}(G)$ if and only if there exists $C>0$ such that $\Vert F^*z\Vert_{X'}\leq C\Vert G^*z\Vert_{Y'}$ for every $z\in Z'$.
\end{lemma}

It is interesting to stress the difference with the finite-dimensional setting, in which a proper subset cannot be dense. The fact that, in infinite dimension, a proper subset may be dense explains the fact that the notion of approximate controllability is distinct from the notion of exact controllability.

Now, applying Lemma \ref{lemgenanafonc} to the operators $L_T$ and $S(T)$, we get, with Remark \ref{remprelimdualitycontobs}, the following result.

\begin{theorem}\label{dualityondes}
Assume that $X$ and $U$ are reflexive.
Let $T>0$ arbitrary. The control system \eqref{eqE} is:
\begin{itemize}[parsep=1mm,itemsep=1mm,topsep=1mm]%,leftmargin=*
\item exactly controllable in $X$ in time $T$ if and only if there exists $C_T>0$ such that 
\begin{equation}\label{inegobs}
\int_0^T \Vert B^*S^*(T-t)z\Vert_{U'}^2\,dt \geq C_T\Vert z\Vert_{X'}^2\qquad\forall z\in D(A^*);
\end{equation}

\item %The control system \eqref{eqE} is 
approximately controllable in $X$ in time $T$ if and only if %the following implication holds true:
\begin{equation}\label{prolongunique}
\forall z\in D(A^*)\quad \forall t\in[0,T]\quad B^*S^*(T-t)z=0\quad \Rightarrow\ z=0;
\end{equation}

\item %The control system \eqref{eqE} is 
exactly null controllable in $X$ in time $T$ if and only if there exists $C_T>0$ such that
\begin{equation}\label{obscont}
\int_0^T \Vert B^*S^*(T-t)z\Vert_{U'}^2\, dt \geq C_T\Vert S(T)^*z\Vert_{X'}^2\qquad\forall z\in D(A^*).
\end{equation}
\end{itemize}
\end{theorem}

\begin{remark}\label{rem51}
The inequalities \eqref{inegobs} and \eqref{obscont} are called \textit{observability inequalities}. As in Remark \ref{rem_Lebext}, they can be written for every $z\in X'$, provided that $B^*$ be replaced with its $\Lambda$-extension $B^*_\Lambda$. The largest constant $C_T>0$ such that \eqref{inegobs} (or \eqref{obscont}) holds true is called the \textit{observability constant}.

Like the admissibility inequality, an observability inequality is an energy-like inequality, but in the converse sense.
Proving observability inequalities is a challenging issue in general for PDEs. We will give some examples further.
\end{remark}

\begin{remark}
In the context of PDEs, \eqref{prolongunique} corresponds to a \textit{unique continuation} property, often established thanks to Holmgren's theorem (see \cite[Section 1.8]{Lions_HUM}).
\end{remark}

\begin{remark}
Setting $\varphi(t) = S^*(T-t)z$, we have
$$
\dot\varphi(t) = -A^*\varphi(t),\qquad \varphi(T)=z,
$$
and the properties above can be interpreted in terms of $\varphi(t)$ which is the adjoint vector in an infinite-dimensional version of the PMP (see Section \ref{sec_PMP_infinitedim}).
Usually, we rather consider $\psi(t)=\varphi(T-t)=S^*(t)z$, and hence we have
$$
\dot\psi(t) = A^*\psi(t),\qquad \psi(0)=z.
$$
This is the adjoint equation.

In terms of this adjoint equation, the approximate controllability property is equivalent to the unique continuation property: $B^*\psi(t)=0$ for every $t\in[0,T]$ implies $\psi(\cdot)=0$.
The exact controllability is equivalent to the observability inequality
$$
\int_0^T \Vert B^*\psi(t)\Vert_{U'}^2\, dt \geq C_T \Vert\psi(0)\Vert_{X'}^2
$$
for every solution of the adjoint equation, and the exact null controllability is equivalent to the observability inequality
$$
\int_0^T \Vert B^*\psi(t)\Vert_{U'}^2\, dt \geq C_T \Vert\psi(T)\Vert_{X'}^2.
$$
\end{remark}

\begin{remark}
As announced in Remarks \ref{remGT} and \ref{rem_auton} (in Chapter \ref{chap_cont}, Section \ref{sec_controllability_time-varying}), the observability inequality \eqref{inegobs} is the infinite-dimensional version of the observability inequality \eqref{inegobsdimfinie_autonomous} obtained in the finite-dimensional setting.

Note that, in finite dimension, the properties \eqref{inegobs} and \eqref{prolongunique} are equivalent, whereas, in the infinite-dimensional setting, there is a deep difference, due to the fact that a proper subset of an infinite-dimensional space may be dense.

\paragraph{Gramian operator.}
As in Remark \ref{remGT} and in Theorem \ref{controlabiliteinstationnaire}, we can similarly define the Gramian operator in the present context of Banach spaces.
Its general definition is the following. Assume that $X$ and $U$ are reflexive. Recall that, in general, we have $L_T\in L(L^2([0,T],U),D(A^*)')$ and $L_T^*\in L(D(A^*),L^2([0,T],U'))$ (see the proof of Lemma \ref{lem_LT*}). Identifying $U\simeq U'$ and $L^2([0,T],U)\simeq L^2([0,T],U')$, we define the Gramian operator
\begin{equation}\label{def_GT}
G_T = L_T L_T^* = \int_0^T S(T-t)BB^*S(T-t)^* \, dt \  \in L(D(A^*),D(A^*)') 
\end{equation}
where, in the formula above, $BB^*$ is to be understood as $BJB^*$ where $J:U'\rightarrow U$ is the canonical isomorphism.

If the control operator $B$ is admissible then $L_T\in L(L^2([0,T],U),X)$ and $L_T^*\in L(X',L^2([0,T],U'))$, and therefore $G_T \in L(X',X)$. The expression of $G_T$ is still given by \eqref{def_GT} when applied to some element $z\in D(A^*)$. Using Remark \ref{rem_Lebext}, it can be noted that the expression of $G_T$ on the whole space $X'$ is given by
\begin{equation*}
G_T = \int_0^T S(T-t)B_\Lambda B_\Lambda^*S(T-t)^* \, dt.
\end{equation*}

If the control system \eqref{eqE} is exactly controllable in time $T$, then, using \eqref{inegobs} and Remark \ref{rem51}, it follows that 
$\left\langle G_Tz,z\right\rangle_{D(A^*)',D(A^*)}\geq C_T\Vert z\Vert_{X'}^2$
for every $z\in D(A^*)$; the converse inequality is satisfied if $B$ is admissible.
In other words, we have the following lemma.

\begin{lemma}\label{lemGT}
The control operator $B$ is admissible and the control system \eqref{eqE} is exactly controllable in time $T$ if and only if $G_T:X'\rightarrow X$ is an isomorphism satisfying $C_T\Vert z\Vert_{X'}^2\leq \langle G_T z,z\rangle_{X',X}\leq K_T\Vert z\Vert_{X'}^2$ for every $z\in X'$.
\end{lemma}
%This lemma will naturally lead to the Hilbert Uniqueness Method described in the next section.
\end{remark}

\subsection{Hilbert Uniqueness Method (HUM)}
The Hilbert Uniqueness Method (in short, HUM; see \cite{Lions_SIREV,Lions_HUM}) is based on Lemma \ref{lemGT} by noticing that, in the context of this lemma, the norm $\Vert\cdot\Vert_{X'}$ is equivalent to the norm given by $(\langle G_T z,z\rangle_{X',X})^{1/2}$. This gives a characterization of the state space $X$ in which we have exact controllability.

HUM can then be stated as follows. Let $Y$ be a reflexive Banach space, let $A:D(A)\rightarrow Y$ be an operator generating a $C_0$ semigroup and let $(Y_\alpha)_{\alpha\in\R}$ be the associated scale of Banach spaces. Let $U$ be a fixed reflexive Banach space and let $B\in L(U,Y_{-\alpha})$ be a control operator. Let $Z$ be the completion of $D(A^*)$ for the norm $(\langle G_T z,z\rangle_{D(A^*)',D(A^*)})^{1/2}$, and let $X$ be a Banach space such that $X'=Z$.
Then $X$ is the Banach space for which Lemma \ref{lemGT} is satisfied, i.e., for which $B$ is admissible and the control system \eqref{eqE} is exactly controllable in time $T$ in the state space $X$.

HUM may as well be restated in the other way round: the (reflexive Banach) state space $X$ is fixed and one wants to characterize the control Banach space $U$ for which admissibility and exact controllability are satisfied.

\paragraph{HUM functional.}
In the framework of Lemma \ref{lemGT}, the so-called HUM functional $J$, defined by
$$
J(z) = \frac{1}{2}\langle G_Tz,z\rangle_{X',X} + \langle z,S(T)y_0-y_1\rangle_{X',X}\qquad\forall z\in X'
$$
is smooth and coercive in $X'$, hence $J$ has a unique minimizer $\bar z$, satisfying 
$$
0 = \nabla J(\bar z) = G_T\bar z + S(T)y_0-y_1 .
$$
Defining the so-called HUM control by $\bar u(t) = B^*S(T-t)^*\bar z = (L_T\bar z)(t)$, the above equality says that $S(T)y_0+L_T\bar u=y_1$, i.e., $y(T;y_0,\bar u)=y_1$. In other words, the control $\bar u$ steers the control system \eqref{eqE} from $y_0$ to $y_1$ in time $T$. Actually, the control $\bar u$ is even the minimal $L^2$ norm control realizing this controllability property (see \cite{Lions_SIREV,Lions_HUM}): this can also be seen by observing that, when wanting to solve the overdetermined equation $L_Tu=y_1-S(T)y_0$, the control of minimal $L^2$ norm is given by $u = L_T^\#(y_1-S(T)y_0)$ where $L_T^\#$ is the pseudo-inverse of $L_T$ (this is indeed a well known property of the pseudo-inverse); since $L_T^\#=L_T^*(L_TL_T^*)^{-1} = L_T^*G_T^{-1}$, the claim follows.

\paragraph{HUM for exact null controllability.}
When wanting to realize an exact null controllability result for a control system that is not exactly controllable (like the heat equation), of course the conclusion of Lemma \ref{lemGT} does not hold. In terms of the Gramian operator, the observability inequality \eqref{obscont} is written as $\left\langle G_Tz,z\right\rangle_{D(A^*)',D(A^*)} \geq C_T \Vert S(T)^*z\Vert_{X'}^2$ for every $z\in D(A^*)$.

We can however still write the HUM functional as above (with an additional care) and determine the minimal $L^2$ norm control steering the control system \eqref{eqE} to $0$ in time $T$. The HUM functional $J$ is defined as above, for every $z\in D(A^*)$, with the duality bracket $\left\langle G_Tz,z\right\rangle_{D(A^*)',D(A^*)}$ for the first term. The functional $J$ is however not coercive in $X'$. To recover such a property, we define the Banach space $\mathcal{X}$ as the completion of $D(A^*)$ for the norm $(\left\langle G_Tz,z\right\rangle_{D(A^*)',D(A^*)})^{1/2}$. Note that the space $\mathcal{X}$ is in general much larger than $D(A^*)$ and may even fail to be a space of distributions (see \cite{Lions_HUM}). Anyway, there is a unique minimizer $\bar z\in\mathcal{X}$ of $J$, satisfying therefore $0 = \nabla J(\bar z) = G_T\bar z + S(T)y_0$, and then the HUM control $\bar u(t) = B^*S(T-t)^*\bar z = (L_T\bar z)(t)$ steers the control system \eqref{eqE} to $0$ in time $T$, and is the control of minimal $L^2$ norm doing so.

This approach provides a generalization of Theorem \ref{controlabiliteinstationnaire} (in Section \ref{sec_controllability_time-varying}) to  infinite-dimensional autonomous linear control systems.

\subsection{Example: the wave equation}
The typical (and historical) example of application of HUM is the wave equation, either with an internal control or with a (Dirichlet or Neumann) boundary control. In 1D, the analysis is easy thanks to Fourier series, as elaborated below. The multi-D case is much more complicated and can be treated thanks to microlocal analysis (see comments further).

\subsubsection{1D wave equation with Dirichlet boundary control}
Let $T>0$ and $L>0$ be fixed. We consider the 1D wave equation with Dirichlet boundary control at the right-boundary:
\begin{equation}\label{ondes}
\begin{split}
& \partial_{tt}y = \partial_{xx}y,\qquad\qquad\qquad\qquad\qquad\quad t\in(0,T),\, x\in (0,L),\\
& y(t,0)=0,\ y(t,L)=u(t),\qquad\quad\qquad t\in(0,T), \\
& y(0,x)=y_0(x),\ \partial_ty(0,x)=y_1(x),\quad\ x\in(0,L),
\end{split}
\end{equation}
where the state at time $t\in [0,T]$ is $(y(t,\cdot),\partial_t y(t,\cdot))$ and the control is $u(t)\in\R$.
Let us establish that this equation is exactly controllable in time $T$ in the space $L^2(0,L)\times H^{-1}(0,L)$ with controls $u\in L^2(0,T)$ if and only if $T\geq 2L$.

By Theorem \ref{dualityondes} (see also Example \ref{examplewave}), this is equivalent to establishing the following observability inequality: there exists $C_T>0$ such that any solution of
\begin{equation}\label{ondesadjoint}
\partial_{tt}\psi = \partial_{xx}\psi, \qquad \psi(t,0)=\psi(t,L)=0,
\end{equation}
such that $(\psi(0),\partial_t\psi(0))\in H^1_0(0,L)\times L^2(0,L)$ satisfies
\begin{equation}\label{inegobsici}
\int_0^T \vert\partial_x\psi(t,L)\vert^2\, dt \geq %C_T \int_0^L \left(  \vert\partial_x\psi(0,x)\vert^2 + \vert\partial_t\psi(0,x)\vert^2 \right) dx = 
C_T \left( \Vert\psi(0)\Vert_{H^1_0(0,L)}^2 + \Vert\partial_t\psi(0)\Vert_{L^2(0,L)}^2 \right)
\end{equation}
Given any $T\geq 2L$, let us establish \eqref{inegobsici} by using spectral expansions (Fourier series). 
We expand the solutions of \eqref{ondesadjoint} as
$$
\psi(t,x) = \sum_{k=1}^\infty \frac{L}{k\pi}\left( a_k\cos\frac{k\pi t}{L} + b_k\sin\frac{k\pi t}{L} \right)\sin\frac{k\pi x}{L} 
$$
with $(a_k)_{k\in\N^*}\in\ell^2(\R)$ and $(b_k)_{k\in\N^*}\in\ell^2(\R)$, so that
$\psi(0,x) = \sum_{k=1}^\infty \frac{L}{k\pi}a_k\sin\frac{k\pi x}{L}$ and $\partial_t\psi(0,x) = \sum_{k=1}^\infty b_k \sin\frac{k\pi x}{L}
$ and thus
$$
\Vert\psi(0)\Vert_{H^1_0}^2 + \Vert\partial_t\psi(0)\Vert_{L^2}^2
= \int_0^L \left(  \vert\partial_x\psi(0,x)\vert^2 + \vert\partial_t\psi(0,x)\vert^2 \right) dx
= \frac{L}{2}\sum_{k=1}^\infty (a_k^2+b_k^2) .
$$
Then
\begin{equation*}
\begin{split}
\int_0^T \vert\partial_x\psi(t,L)\vert^2\, dt &\geq \int_0^{2L} \vert\partial_x\psi(t,L)\vert^2\, dt \\
&= \int_0^{2L} \left\vert  \sum_{k=1}^\infty (-1)^k \left( a_k\cos\frac{k\pi t}{L} + b_k\sin\frac{k\pi t}{L} \right)  \right\vert^2 dt \\
&= \sum_{j,k=1}^\infty (-1)^{j+k} \int_0^{2L} \left(a_j\cos\left(\frac{j\pi t}{L}\right)+b_j\sin\left(\frac{j\pi t}{L}\right)\right) \times\\
& \qquad\qquad\qquad\qquad\qquad\left(a_k\cos\left(\frac{k\pi t}{L}\right)+b_k\sin\left(\frac{k\pi t}{L}\right)\right) \, dt \\
&= L\sum_{k=1}^\infty (a_k^2+b_k^2)
= 2 \left( \Vert\psi(0)\Vert_{H^1_0(0,L)}^2 + \Vert\partial_t\psi(0)\Vert_{L^2(0,L)}^2 \right) ,
\end{split}
\end{equation*}
and \eqref{inegobsici} is proved.

Using similar Fourier series expansions, we see that admissibility property (of the Dirichlet control operator) is satisfied for any $T>0$: by Proposition \ref{prop_waveadm} (see also Example \ref{examplewave}), equivalently, for any $T>0$ there exists $K_T>0$ such that 
\begin{equation*}%\label{inegadmici}
\int_0^T \vert\partial_x\psi(t,L)\vert^2\, dt \leq K_T \left( \Vert\psi(0)\Vert_{H^1_0(0,L)}^2 + \Vert\partial_t\psi(0)\Vert_{L^2(0,L)}^2 \right) .
\end{equation*}
for any solution of \eqref{ondesadjoint}.
Indeed, $\int_0^T \vert\partial_x\psi(t,L)\vert^2\, dt \leq \int_0^{2nL} \vert\partial_x\psi(t,L)\vert^2\, dt$ for some $n\in\N^*$ we perform the same expansion as above. 

We conclude that, for $T\geq 2L$, we have the double inequality
$$
C_T \Vert(\psi(0),\partial_t\psi(0))\Vert_{H^1_0\times L^2}^2 
\leq \int_0^T \vert\partial_x\psi(t,L)\vert^2\, dt 
\leq K_T \Vert(\psi(0),\partial_t\psi(0))\Vert_{H^1_0\times L^2}^2 
$$
for all solutions of \eqref{ondesadjoint}, saying that $\big(\int_0^T \vert\partial_x\psi(t,L)\vert^2\, dt \big)^{1/2}$ is a norm, equivalent to the norm of $H^1_0(0,L)\times L^2(0,L)$.
This illustrates Lemma \ref{lemGT}. The term $\int_0^T \vert\partial_x\psi(t,L)\vert^2\, dt$ stands for the Gramian.

Let us finally prove that controllability is lost if $T<2L$.
Let $\delta>0$ be such that $T\leq 2L-2\delta$. We consider a solution of \eqref{ondesadjoint} such that $\psi(T/2,\cdot)$ and $\partial_x\psi(T/2,\cdot)$ are supported in $(0,\delta)$. Then, using the fact that the support of any solution of the wave equation propagates at speed $1$, it follows that the observability inequality \eqref{inegobs} is not satisfied.

Note that the same argument shows that, for $T<2L$, the wave equation is not approximately controllable either.

\subsubsection{1D Dirichlet wave equation with internal control}
Instead of \eqref{ondes}, we consider
\begin{equation}\label{ondes_interne}
\begin{split}
& \partial_{tt}y = \partial_{xx}y + \chi_\omega u,\qquad\qquad\qquad\quad\ \ \, t\in(0,T),\, x\in (0,L),\\
& y(t,0)=0,\ y(t,L)=0,\qquad\quad\qquad\quad  t\in(0,T), \\
& y(0,x)=y_0(x),\ \partial_ty(0,x)=y_1(x),\quad\ x\in(0,L),
\end{split}
\end{equation}
where the control is $u(t,x)\in\R$ and $\omega\subset(0,L)$ is a measurable subset of positive Lebesgue measure.
Let us establish that \eqref{ondes_interne} is exactly controllable in time $T\geq 2L$ in the space $H^1_0(0,L)\times L^2(0,L)$ with controls $u\in L^2((0,T)\times\omega)$.

By Theorem \ref{dualityondes}, this is equivalent to establishing the following observability inequality:
for every $T\geq 2L$, for every $\omega\subset(0,L)$ measurable of positive measure, there exists $C_T(\omega)>0$ such that
\begin{equation}\label{obsintwave}
\int_0^T\int_\omega \phi(t,x)^2 \, dx\, dt \geq C_T(\omega)\left( \Vert \phi(0)\Vert_{L^2(0,L)}^2 + \Vert \partial_t\phi(0)\Vert_{H^{-1}(0,L)}^2\right)
\end{equation}
for every solution $\phi$ of the adjoint equation
\begin{equation}\label{adjwav}
\partial_{tt}\phi-\partial_{xx}\phi=0 , \qquad\phi(t,0)=\phi(t,\pi)=0 .
\end{equation}
We consider solutions $\phi$ of \eqref{adjwav} expanded as
$$
\phi(t,x) = \sum_{j=1}^\infty\left( a_j\cos\left(\frac{j\pi t}{L}\right) + b_j\sin\left(\frac{j\pi t}{L}\right) \right)\sin\left(\frac{j\pi x}{L}\right)
$$
with $(a_j)_{j\in\N^*}\in\ell^2(\R)$ and $(b_j)_{j\in\N^*}\in\ell^2(\R)$, so that 
$\phi(0,x) = \sum_{j=1}^\infty a_j\sin\left(\frac{j\pi x}{L}\right)$ and $\partial_t\phi(0,x) = \sum_{j=1}^\infty \frac{j\pi}{L} b_j\sin\left(\frac{j\pi x}{L}\right)$ and thus
$$
\Vert \phi(0)\Vert_{L^2(0,L)}^2 + \Vert \partial_t\phi(0)\Vert_{H^{-1}(0,L)}^2 = \frac{L}{2}\sum_{j=1}^\infty (a_j^2+b_j^2) .
$$
For every $T\geq 2L$, we have
$\int_0^T\int_\omega \phi(t,x)^2 \, dx\, dt \geq \int_0^{2L}\int_\omega \phi(t,x)^2 \, dx\, dt $
and
\begin{equation*}
\begin{split}
& \int_0^{2L}\int_\omega \phi(t,x)^2 \, dx\, dt \\
=&\ \frac{L}{\pi}\sum_{j,k=1}^\infty \int_0^{2\pi} (a_j\cos(js)+b_j\sin(js))(a_k\cos(ks)+b_k\sin(ks)) \, ds \\
& \qquad\qquad\qquad\qquad \times \int_\omega \sin\left(\frac{j\pi x}{L}\right) \sin\left(\frac{k\pi x}{L}\right) dx \\
=&\ L \sum_{j=1}^\infty (a_j^2+b_j^2) \int_\omega \sin^2\left(\frac{j\pi x}{L}\right) dx .
\end{split}
\end{equation*}
We now give two ways to infer the observability inequality \eqref{obsintwave}.

\medskip
\noindent\underline{First way.}
We observe that $\sin^2\left(\frac{j\pi x}{L}\right)\rightharpoonup \frac{1}{2}$ as $j\rightarrow+\infty$ (in weak $L^2$ topology). It follows that, for any measurable subset $\omega\subset(0,L)$ of positive measure, there exists $C(\omega)>0$ such that
$$
\int_\omega \sin^2\left(\frac{j\pi x}{L}\right) dx\geq C(\omega)\qquad\forall j\in\N^* ,
$$
and then \eqref{obsintwave} follows.

Note that we have used here an information on the highfrequency eigenfunctions $\phi_j(x)=\sqrt{\frac{2}{L}}\sin\left(\frac{j\pi x}{L}\right)$. %, more precisely we have used the crucial fact that all probability measures

\medskip
\noindent\underline{Second way.} We have the following lemma.

\begin{lemma}\label{lemperiago}
Given any measurable subset $\omega\subset(0,L)$ of positive Lebesgue measure $\vert\omega\vert>0$, we have
$$
\int_\omega \sin^2\left(\frac{j\pi x}{L}\right) dx \geq  \frac{1}{2} \left( \vert\omega\vert - \frac{L}{\pi} \sin\left( \frac{\pi}{L}\vert\omega\vert\right)\right)  \qquad \forall j\in\N^*.
$$
\end{lemma}

\begin{proof}
For a \emph{fixed} integer $j$, consider the problem of minimizing the functional
$K_j(\omega')=\int_{\omega'} \sin^2\left(\frac{j\pi x}{L}\right) dx$
over all possible measurable subsets $\omega'\subset(0,\pi)$ s.t. $\vert\omega'\vert=\vert\omega\vert$. 
Identifying the minima (zeros) of $\sin^2\left(\frac{j\pi x}{L}\right)$, using a bathtub principle argument, it is quite obvious to see that there exists a unique (up to zero measure subsets) optimal set, characterized as a level set of the function $x\mapsto\sin^2\left(\frac{j\pi x}{L}\right)$, which is
$$
\omega_j^{\textnormal{inf}}=\left(0,\frac{\vert\omega\vert}{2j}\right)\ \bigcup\ \bigcup_{k=1}^{j-1}\ \left(\frac{kL}{j}-\frac{\vert\omega\vert}{2j},\frac{kL}{j}+\frac{\vert\omega\vert}{2j}\right)\ \bigcup\ \left(L-\frac{\vert\omega\vert}{2j},L\right)
$$
and we have
\begin{multline*}
\int_{\omega_j^{\textnormal{inf}}}\sin^2\left(\frac{j\pi x}{L}\right)\,dx  =  2j\int_0^{\vert\omega\vert/2j}\sin^2\left(\frac{j\pi x}{L}\right)\,dx \\
=  \frac{2L}{\pi} \int_0^{\frac{\pi}{2L}\vert\omega\vert}\sin^2u\, du
=  \frac{1}{2} \left( \vert\omega\vert - \frac{L}{\pi} \sin\left( \frac{\pi}{L}\vert\omega\vert\right)\right) 
\end{multline*}
for any $j$. Since this value does not depend on $j$, the lemma follows.
\end{proof}

We infer from that lemma that
$$
\int_0^T\int_\omega \phi(t,x)^2 \, dx\, dt \geq \left( \vert\omega\vert - \frac{L}{\pi} \sin\left( \frac{\pi}{L}\vert\omega\vert\right)\right) \left( \Vert \phi(0)\Vert_{L^2}^2 + \Vert \partial_t\phi(0)\Vert_{H^{-1}}^2 \right) 
$$
which gives the observability inequality \eqref{obsintwave}. It is interesting to note that, for $T=2L$, the inequality is sharp and thus the observability constant is
$$
C_{T=2L}(\omega) = \vert\omega\vert - \frac{L}{\pi} \sin\left( \frac{\pi}{L}\vert\omega\vert\right).
$$

Recalling the admissibility property in Example \ref{examplewaveinternal}, we have obtained that, for $T\geq 2\pi$, $\big(\int_0^T\int_\omega \phi(t,x)^2 \, dx\, dt\big)^{1/2}$ is a norm, equivalent to the norm of $L^2(0,L)\times H^{-1}(0,L)$, which illustrates Lemma \ref{lemGT}.

As in the boundary control case, controllability is lost if $T$ is too small: using the finite speed propagation of the wave equation, it suffices to consider a solution supported in $(0,L)\setminus\omega$ over a small enough time interval. 

For instance, if $\omega=(a,b)\subset(0,L)$ then the minimal controllability time is $T=2\max(a,L-b)$.

\subsubsection{Multi-D Dirichlet wave equation with internal control}
Consider the internally controlled Dirichlet wave equation
\begin{equation}\label{ondes_multiD}
\begin{split}
& \partial_{tt}y = \triangle y + \chi_\omega u,\qquad\qquad\qquad\qquad t\in(0,T),\, x\in\Omega,\\
& y(t,x)=0,\qquad\quad\qquad\qquad\qquad\qquad  t\in(0,T),\ x\in\partial\Omega \\
& y(0,x)=y_0(x),\ \partial_t y(0,x)=y_1(x),\quad\, x\in\Omega,
\end{split}
\end{equation}
on a bounded open domain $\Omega\subset\R^n$ having a $C^2$ boundary, and internal control on a measurable subset $\omega\subset\Omega$ of positive Lebesgue measure.

The admissibility (well-posedness) has been seen in Example \ref{examplewaveinternal}.

It is proved in \cite{BardosLebeauRauch} that, if $\omega$ is open and if the pair $(\omega,T)$ satisfies the \emph{Geometric Control Condition} (GCC), then \eqref{ondes_multiD} is exactly controllable in time $T$ in the space $H^1_0(\Omega)\times L^2(\Omega)$ with controls $u\in L^2((0,T)\times\omega)$;
equivalently, by Theorem \ref{dualityondes}, there exists $C_T(\omega)>0$ such that
\begin{equation}\label{obsintwave_multiD}
\int_0^T\int_\omega \phi(t,x)^2 \, dx\, dt \geq C_T(\omega)\left( \Vert \phi(0)\Vert_{L^2(\Omega)}^2 + \Vert \partial_t\phi(0)\Vert_{H^{-1}(\Omega)}^2\right)
\end{equation}
for every solution $\phi$ of the adjoint equation $\partial_{tt}\phi-\triangle\phi=0$ with $\phi=0$ along the boundary of $\Omega$. The GCC stipulates that any geodesic ray, propagating in $\Omega$ (seen as a billiard) at speed $1$ and reflecting at the boundary according to the laws of classical optics (see Figure \ref{fig_GCC}), meets the open set $\omega$ within time $T$. 
\begin{figure}[h]
\centerline{\includegraphics[width=9cm]{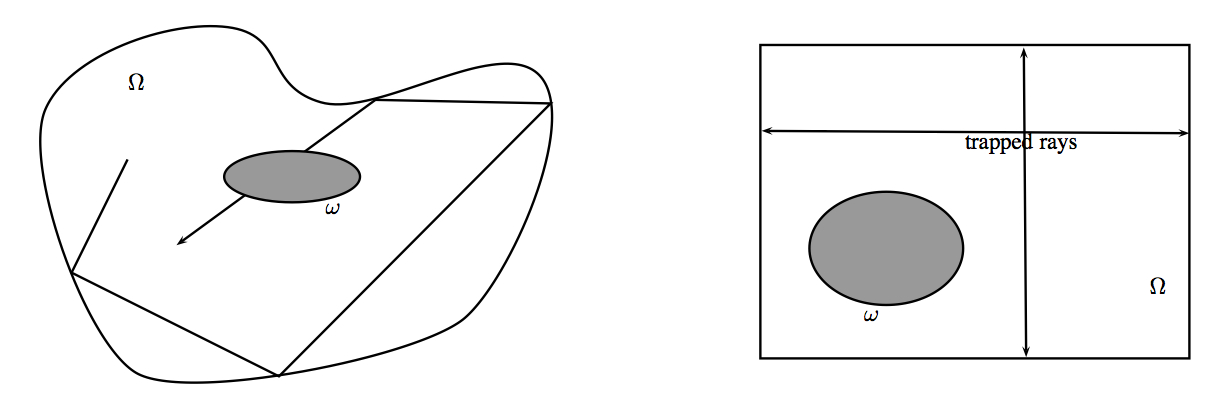}}
\caption{Geodesic rays propagating in $\Omega$.}\label{fig_GCC}
\end{figure}
On this figure, on the right, $\Omega$ is a square in the plane and $\omega$ is an internal disk: the GCC is never satisfied because of the existence of trapped rays (bouncing balls). 
The proof of \eqref{obsintwave_multiD} requires significantly more elaborate tools than in the 1D case: microlocal analysis, propagation of singularities, defect measures.
The ``almost-equivalence" between GCC and the observability inequality \eqref{obsintwave_multiD} is studied in \cite{HumbertPrivatTrelat_CPDE2019}.

We have obtained that, under GCC, $\big(\int_0^T\int_\omega \phi(t,x)^2 \, dx\, dt\big)^{1/2}$ is a norm, equivalent to the norm of $L^2(\Omega)\times H^{-1}(\Omega)$. This illustrates Lemma \ref{lemGT}. 

There are similar results for the boundary control case.
Note that, as initially developed in \cite{Lions_HUM}, under stronger geometric conditions on $\omega$, there are more elementary proofs based on the multiplier method (see \cite{Komornik} and \cite[(7.0.3) and Theorem 7.4.1]{TucsnakWeiss}). Carleman estimates can also be used, with the advantage of allowing to tackle lower-order and/or low regularity terms (see \cite{FursikovImanuvilov}). %; in some sense they are a more sophisticated version of the multiplier method.

\subsection{Example: the heat equation}
\subsubsection{Dirichlet heat equation with internal control}
Let $\Omega\subset\R^n$ be a bounded open set having a $C^2$ boundary, and let $\omega\subset\Omega$ be an open subset. Like in Section \ref{exsec}, we consider the heat equation with internal control and Dirichlet boundary conditions
\begin{equation*}
\partial_ty=\triangle y+\chi_\omega u\quad\textrm{in}\ \Omega,\qquad
y_{\vert\partial\Omega}=0,\qquad
y(0)=y_0\in L^2(\Omega).
\end{equation*}
The admissibility property of the control operator, as discussed in Section \ref{exsec}, is obvious since $B$ is bounded.

Due to the smoothing effect of the heat semigroup ($e^{t\triangle}y_0$ is smooth on $\Omega$ for every $t>0$, whatever the regularity of $y_0$ may be), the above heat equation cannot be exactly controllable in $X=L^2(\Omega)$. But it is approximately controllable in $X$ and exactly null controllable in $X$ in any time $T>0$.

Indeed, by Theorem \ref{dualityondes}, the approximate controllability property is equivalent to the following property: given any $T>0$, given any solution $\psi$ of 
\begin{equation}\label{adjheat}
\partial_t\psi=\triangle\psi\quad\textrm{in}\ \Omega,\qquad
\psi_{\vert\partial\Omega}=0 
\end{equation}
such that $\psi(0)\in H^2(\Omega)\cap H^1_0(\Omega)$, 
if $\psi(t,x)=0$ for all $t\in[0,T]$ and $x\in\omega$ then $\psi\equiv 0$. This unique continuation property is a consequence of Holmgren's theorem (see \cite[Section 1.8]{Lions_HUM}).

Exact null controllability is equivalent, by Theorem \ref{dualityondes}, to the following observability inequality: given any $T>0$, there exists $C_T(\omega)>0$ such that
\begin{equation}\label{inegobsheat}
\int_0^T \int_\omega \psi(t,x)^2\, dx\, dt \geq C_T(\omega) \Vert\psi(T)\Vert_X^2
\end{equation}
for any solution of \eqref{adjheat}.
The observability inequality \eqref{inegobsheat} has been established in \cite{FursikovImanuvilov, Imanuvilov, LebeauRobbiano}. Nowadays, it seems that the most powerful tool in order to establish such inequalities in the parabolic setting is the \emph{Carleman estimates} (see \cite[Chapter 9]{TucsnakWeiss} or \cite[Section 2.5]{Coron}). Even in 1D, the proof of \eqref{inegobsheat} by a Carleman estimate is quite technical.
It can however be noted that, in 1D, the exact null controllability property can also be proved thanks to harmonic analysis considerations, by applying the \emph{moment method} (see Section \ref{sec_moment}). 

\subsubsection{Heat equation with Dirichlet boundary control}
Like in Section \ref{sec_heatdirbound}, we consider the heat equation with Dirichlet boundary control
\begin{equation}\label{heatbounddir}
\partial_ty=\triangle y\ \ \textrm{in}\ \Omega,\qquad
y_{\vert\partial\Omega}=u,\qquad
y(0)=y_0\in L^2(\Omega),
\end{equation}
where $\Omega\subset\R^n$ is a bounded open subset having a $C^2$ boundary.
We have seen in that section that, setting $X=H^{-1}(\Omega)$ and $U=L^2(\partial\Omega)$, the control operator is admissible, but that admissibility is not true if one takes $X=L^2(\Omega)$.

According to \cite[Chapter 11.5]{TucsnakWeiss}, the heat equation \eqref{heatbounddir} is exactly null controllable in $X=H^{-1}(\Omega)$ in any time $T>0$; equivalently, for every $T>0$ there exists $C_T>0$ such that
\begin{equation}\label{obs_heat_internal}
\int_0^T\left\Vert\frac{\partial\psi}{\partial\nu}_{\vert\partial\Omega} (t)\right\Vert_{L^2(\partial\Omega)}^2dt 
\geq C_T\Vert\psi(T)\Vert^2_{H^1_0(\Omega)}
\end{equation}
for any solution of $\partial_t\psi=\triangle\psi\ \ \textrm{in}\ \Omega$, $\psi_{\vert\partial\Omega}=0$, 
with $\psi(0)\in % H^2(\Omega)\cap 
H^1_0(\Omega)$.
%As in the internal case, the above observability inequality has been established in \cite{Imanuvilov, LebeauRobbiano}. Nowadays, it seems that the most powerful in order to establish such inequalities in the parabolic setting is the \emph{Carleman estimates} . 

In most of the existing literature (see, e.g., \cite{Imanuvilov, LebeauRobbiano}) one can find the observability inequality \eqref{obs_heat_internal} with the $L^2$ norm at the right-hand side, thus saying by duality that the heat equation \eqref{heatbounddir} is exactly null controllable in $L^2(\Omega)$ in any time $T>0$. Actually, it is exactly controllable in any $H^s(\R)$, for any $s\in\R$ (and in particular for any $s<0$): indeed, taking any $y_0\in H^s(\Omega)$, considering an appropriate extension $S(t)$ of the Dirichlet heat semigroup, one has $S(t)y_0\in H^1_0(\Omega)\cap C^\infty(\Omega)$ for any $t>0$ and in particular for $t=T/2$; then apply the controllability property in time $T/2$ to steer $S(T/2)y_0$ to $0$. Then, by duality, the observability inequality \eqref{obs_heat_internal} remains true when replacing the $H^1_0$ norm at the right-hand side with the norm of $X_\alpha$, for any $\alpha\in\R$, where $(X_\alpha)_{\alpha\in\R}$ is the family of Dirichlet spaces constructed in Example \ref{example_chain_Dirichlet}.

%heat equation with Dirichlet control: Li Yong page 361. Le pb de regularite est explique page 366, ainsi que dans Lions (vieux bouquin) page 217.
%
%wave equation: voir Li Yong page 284
%
%Pour rappel:
%
%- Equation chaleur avec controle Dirichlet: bien pos\'ee dans $ H^{-1}$, mais pas dans $L^2$. Ce qui veut dire que, pour une donn\'ee initiale $L^2$ et un contr\^ole $L^2$, la solution vit dans l'espace $H^{-1}$, mais pas forc\'ement dans l'espace $L^2$.
% 
%- Pourtant, elle est exactement contr\^olable \`a z\'ero dans $L^2$

\subsection{Pontryagin maximum principle}\label{sec_PMP_infinitedim}
Formally, HUM is obtained by applying the PMP to the LQ optimal control problem consisting of steering the control system $\dot y=Ay+Bu$ from $y(0)=y_0$ to $y(T)=y_1$ in time $T$, by minimizing the cost functional $\int_0^T\Vert u(t)\Vert_U^2\, dt$. We have anyway to be careful there.
Indeed, as shortly discussed in Section \ref{sec_generalizations_PMP}, the generalization of the PMP to the infinite-dimensional setting, done for instance in \cite[Chapter 4]{LiYong} (and proved, in this book, by using the Ekeland variational principle), requires in general that the final state $y(T)$ is subject to a finite number only of scalar constraints. There are counterexamples to the statement of the PMP whenever $y(T)=y_1\in X$ when $\dim X=+\infty$, i.e., when there are an infinite number of final scalar constraints (such counterexamples are easy to design by considering systems enjoying approximate but not exact controllability: see Example \ref{contreex_PMP} hereafter). Nevertheless, under exact controllability properties, the PMP is valid for LQ optimal control problems.

More generally, as mentioned in Section \ref{sec_generalizations_PMP}, the PMP is generalized to infinite dimension, with the same statement as in Theorem \ref{PMP}, under the following assumption: 
\begin{quote}
\textit{There exists $z_1\in \overline{\mathrm{Conv}}(M_1)$ such that $\mathrm{Span}(M_1-z_1)$ is of finite codimension in $X$.}
\end{quote}
Roughly speaking this means that we can impose only a finite number of scalar constraints on $y(T)$.

We do not give more details since, with this additional finite codimensionality assumption, the statement is the same.

\begin{example}\label{contreex_PMP}
Following \cite[Chapter 4]{LiYong}, let us design an example of an optimal control problem in infinite dimension on which the expected PMP statement fails. Let $X$ be an infinite-dimensional separable Hilbert space. Consider the control system $\dot y(t)=Ay(t)+bu(t)$, with initial condition $y(0)=0$, where $A:D(A)\rightarrow X$ is an operator generating a $C_0$ semigroup $(S(t))_{t\geq 0}$ and where $b\in X$ is fixed; here, the control $u$ is a real-valued function. 

The idea is to make assumptions ensuring that we can find a point $y_1\in X$ that can be reached from $0$ in time $1$ with only one control.
The constant control $u=1$ steers in time $1$ the control system to $y_1=y(1)=\int_0^1 S(t)b\, dt = (S(1)-\mathrm{id})A^{-1}b$ (the latter formula is obtained by assuming that $A$ is invertible).
Let us assume that $A$ is self-adjoint negative, so that there exists a Hilbert basis $(\phi_j)_{j\in\N^*}$ of eigenfunctions, i.e., $A\phi_j=-\lambda_j\phi_j$ with $\lambda_j>0$ for every $j\in\N^*$.

For any control $\bar u$ steering the system from $y(0)=0$ to $y(1)=y_1$, we must have $\int_0^1 S(1-t)b(\bar u(t)-1)\, dt=0$, i.e., expanding $b=\sum_{j\in\N^*}b_j\phi_j$,
$$
\sum_{j\in\N^*} \int_0^1 e^{-(1-t)\lambda_j}(\bar u(t)-1)\, dt\ b_j\phi_j = 0 .
$$
Assuming that $b_j\neq 0$ for every $j\in\N^*$, we infer that $\int_0^1 e^{t\lambda_j}(\bar u(t)-1)\, dt=0$ for every $j\in\N^*$.
Let us now further assume that $\sum_{j\in\N^*}1/\lambda_j=+\infty$. Then, by the M\"untz-Sz\'asz theorem (see Remark \ref{rem_Muntz} in Section \ref{sec_moment} further), the family $(e^{t\lambda_j})_{j\in\N^*}$ is complete in $L^2(0,1)$. It then follows that $\bar u=1$. We have thus proved that, under the above assumptions, $\bar u=1$ is the unique solution steering the system from $y(0)=0$ to $y(1)=y_1$.

Therefore, the control $\bar u=1$ is optimal for any cost functional (and, by the way, it must be abnormal but we will not use this fact). 

Let us consider the cost functional 
$$
C(u) = \int_0^1 \left( \langle a,y(t)\rangle_X + cu(t)\right) dt ,
$$
where $a\in X$ and $c\in\R$ are fixed, and let us assume that the controls are subject, for instance, to the constraint $\vert u(t)\vert\leq 2$ for almost every $t\in[0,1]$. The Hamiltonian of the optimal control problem is
$$
H(y,p,p^0,u) = \langle p,Ay\rangle_X+\langle p,b\rangle_Xu+p^0\langle a,y\rangle_X+p^0cu.
$$
Let us prove, by contradiction, that the statement of the PMP is not satisfied for the optimal control $\bar u=1$. Otherwise, there would exist an adjoint $p(\cdot)$ satisfying $\dot p=-Ap-p^0a$ and thus, by integration, 
$p(t) %= S(1-t)^*p(1) + p^0 \int_t^1 S(1-s)^*a\, ds
= S(1-t)p(1) + p^0(S(1-t)-\mathrm{id})A^{-1}a$.
Besides, the condition $\frac{\partial H}{\partial u}=0$ gives $\langle p,b\rangle_X+p^0c=0$ on $[0,1]$, i.e., 
$$
\langle p(1)+p^0A^{-1}a, S(1-t)b\rangle_X + p^0(c-\langle A^{-1}a,b\rangle_X) = 0\qquad \forall t\in[0,1], 
$$
%for every $t\in[0,1]$, 
with
$S(1-t)b = \sum_{j\in\N^*} e^{-(1-t)\lambda_j}b_j\phi_j$. By linear independence of the exponential functions, we infer that all terms are equal to $0$. But then, assuming that $c\neq \langle A^{-1}a,b\rangle_X$, it follows that $p(1)=0$ and $p^0=0$, which is a contradiction. %This contradiction shows that the PMP does not hold for the above optimal control problem. 
\end{example}

\section{Further results}

\subsection{Kalman condition in infinite dimension}\label{sec_Kalman_infinitedimension}
It is interesting to mention that the unique continuation property implies an infinite-dimensional version of the Kalman condition. A simple sufficient condition is the following.

\begin{lemma}\label{Kalman_infinitedimension}
We assume that $X$ is reflexive and that $B\in L(U,X)$ is a bounded control operator.
We set
$$
U_\infty = \Big\{ u\in U\ \bigm|\ Bu\in \bigcap_{k=1}^{+\infty} D(A^k) \Big\}.
$$
If the set $\mathcal{K}_T=\mathrm{Span} \{ A^k Bu\ \vert\ u\in U_\infty,\ k\in\N  \}$ is dense in $X$ then the control system \eqref{eqE} is approximately controllable in any time $T>0$.
\end{lemma}

%To prove this statement, we first prove that the control system \eqref{eqE} is approximately controllable in time $T$ if and only if \eqref{prolongunique} holds.

Note that the set $\mathcal{K}_T$ is the infinite-dimensional version of the image of the Kalman matrix in finite dimension.

\begin{proof}
We use the equivalence between approximate controllability and \eqref{prolongunique}. Note that, in \eqref{prolongunique}, it suffices to take $z$ in any dense subspace of $D(A^*)$.
Let $z\in \cap_{k=1}^{+\infty} D((A^*)^n)$ (which is dense in $D(A^*)$) be such that $B^*S(T-t)^*z=0$ for every $t\in[0,T]$. Then, by successive derivations with respect to $t$, and taking $t=T$, we obtain $B^*(A^k)^*z=0$, hence $\langle z,A^kBu\rangle_{X',X}=0$, and therefore $z=0$ because $\mathcal{K}_T$ is dense in $X$.
\end{proof}

We refer the reader to \cite{Triggiani_SICON1976} for a more precise result (and an almost necessary and sufficient condition).

\subsection{Necessary conditions for exact controllability}
Let us assume that $X$ is of infinite dimension, and let us provide general conditions under which the control system \eqref{eqE} is never exactly controllable in finite time $T$, with controls in $L^2([0,T],U)$. We have already seen that exact controllability implies that the control operator $B$ is admissible (and thus $\alpha(B)\leq 1/2$).

\begin{lemma}
Under any of the following assumptions:
\begin{itemize}[parsep=1mm,itemsep=1mm,topsep=1mm]%,leftmargin=*
\item $B\in L(U,X)$ is compact; % (i.e., the image under $B$ of the unit ball of $U$ is relatively compact);
\item $B\in L(U,X)$ is bounded and $S(t)$ is compact for every $t>0$;
\item $X\simeq X'$ and $U\simeq U'$ are Hilbert spaces, $-A$ is a self-adjoint positive operator with compact inverse, and $B\in L(U,X_{-1/2})$ (and thus $B$ is admissible and $\alpha(B)\leq 1/2$, see Section \ref{sec_degree_unboundedness}); 
\end{itemize}
the control system \eqref{eqE} is not exactly controllable in any time $T>0$ (with controls in $L^2([0,T],U)$).
\end{lemma}

For instance, $B$ is compact if $U$ is finite dimensional. Then, the first point implies that it is impossible to control exactly an infinite-dimensional system with a finite number of controls (see \cite[Theorem 4.1.5]{CurtainZwart}). % page 145.

The second point applies for instance to the heat equation with internal control.

The third point applies for instance to the heat equation with Neumann boundary control.

\begin{proof}
%Having in mind the definition \eqref{def_Lt} of the operator $L_T$, for an arbitrary $N\in\N^*$ we set $t_i=iT/N$, $i=0\ldots,N$, and we define the operator
%$$ L_{T,N}u=\sum_{i=0}^{N-1} S(T-t_i)B\int_{t_i}^{t_{i+1}}u(t)\, dt.$$
%
%From the inequality
%\begin{equation*}
%\Vert L_{T,N}u-L_Tu\Vert_X \leq 
%\sum_{i=0}^{N-1} \int_{t_i}^{t_{i+1}} \Vert S(T-t_i)-S(T-t)\Vert\Vert B\Vert \Vert u(t)\Vert_U\, dt ,
%\end{equation*}
%and from the continuity of the semigroup, it is clear that the sequence $(L_{T,N})_{N\in\N^*}$ converges strongly to $L_T$. 
%
%Besides, since the range of $L_{T,N}$ is clearly finite-dimensional, it follows that the operator $L_{T,N}:L^2([0,T],U)\rightarrow X$ is compact.
%
In the first case where $B$ is compact, it is easy to see that the operator $L_T$ is compact. The conclusion follows since $X$ is infinite dimensional, by the Riesz compactness lemma. %, the control system \eqref{eqE} cannot be exactly controllable in any time $T>0$ (with controls in $L^2([0,T],U)$).

\medskip

In the second case (compact semigroup), for every $\varepsilon>0$, we define the operator $L_{T,\varepsilon}:L^2([0,T],U)\rightarrow X$ by $L_{T,\varepsilon}u=\int_0^{T-\varepsilon}S(T-t)Bu(t)\, dt$.
Clearly, $L_{T,\varepsilon}$ converges strongly $L_T$ as $\varepsilon$ tends to $0$.
Besides, using $S(T-t)=S(\varepsilon)S(T-\varepsilon-t)$, we get that $ L_{T,\varepsilon} = S(\varepsilon) L_{T-\varepsilon}$, and hence we infer that $L_{T,\varepsilon}$ is a compact operator, for every $\varepsilon>0$.
Therefore $L_T$ is compact and the conclusion follows as previously.

\medskip

In the third case,
recall that $X_{1/2}=D((-A)^{1/2})$ is the completion of $D(A)$ for the norm $\sqrt{-(Ax,x)_X}$, and $X_{-1/2}=X_{1/2}'$ with respect to the pivot space $X$ (see also Remark \ref{rem41}).
Let $(\phi_j)_{j\in\N^*}$ an orthonormal basis of (unit) eigenvectors of $-A$ associated with eigenvalues $\lambda_j>0$, with $\lambda_j\rightarrow +\infty$ as $j\rightarrow +\infty$. Firstly, we have
\begin{equation}\label{ineq21:00}
\begin{split}
\int_0^T \Vert B^*S(T-t)^*\phi_j\Vert_U^2 \, dt 
& = \int_0^T e^{-\lambda_j(T-t)}\Vert B^*\phi_j\Vert_U^2 \, dt  \\
&\sim \frac{1}{\lambda_j}\Vert B^*\phi_j\Vert_U^2 = \Vert B^*(-A)^{-1/2}\phi_j\Vert_U^2
\end{split}
\end{equation}
as $j\rightarrow+\infty$.
Secondly, since the operator $(-A)^{-1/2}$ is compact, it follows that the operator $B^*(-A)^{-1/2}\in L(X,U)$ is compact as well.
Thirdly, we claim that the sequence $(\phi_j)_{j\in\N^*}$ converges to $0$ for the weak topology of $X$. Indeed, since $\sum_{j=1}^{+\infty}  (x,\phi_j)_X^2 <+\infty$ for every $x\in X$ (by Parseval for instance), it follows that $(x,\phi_j)_X\rightarrow 0$ as $j\rightarrow+\infty$, for every $x\in X$, whence the claim. %. This is exactly the desired weak convergence property.

Since $B^*(-A)^{-1/2}\in L(X,U)$ is compact and since $\phi_j\rightharpoonup 0$, it follows that $B^*(-A)^{-1/2}\phi_j$ converges strongly to $0$. Then, from \eqref{ineq21:00}, we infer that the observability inequality \eqref{inegobs} does not hold true.
%We conclude that the control system \eqref{eqE} cannot by exactly controllable in any time $T>0$ (with controls in $L^2([0,T],U)$).
We refer to \cite[Proposition 9.1.1]{TucsnakWeiss} for such arguments.
%% cf Tucsnak Weiss page 296.
\end{proof}

\subsection{Moment method}\label{sec_moment}
The moment method, relying on harmonic analysis results, has been used in the 70s to establish the first exact controllability results, essentially in 1D (see \cite[Sections 4 and 7]{Russell}). 
To explain the moment method, let us consider the 1D heat equation on $(0,\pi)$ with internal control and with Dirichlet boundary conditions
$$
\partial_t y = \partial_{xx} y + \chi_\omega u, \qquad y(t,0)=y(t,\pi)=0 , \qquad y(0)=y_0\in L^2(0,\pi).
$$
The eigenfunctions $\sqrt{\frac{2}{\pi}}\sin(jx)$ of the Dirichlet-Laplacian, associated with the eigenvalues $\lambda_j=-j^2$, make an orthonormal basis of $L^2(0,\pi)$.
Expanding in series, we have
$y_0(x) = \sum_{j=1}^{+\infty} a_j\sin(jx)$ with $(a_j)_{j\in\N^*}\in\ell^2(\R)$, and writing $y(t,x)= \sum_{j=1}^{+\infty} y_j(t)\sin(jx)$, we get $\dot y_j(t) = -j^2 y_j(t) + \int_\omega u(t,x)\sin(jx)\, dx$ and thus
$$
y_j(T) = e^{-j^2T}a_j+\int_0^T e^{-j^2(T-t)}\int_\omega u(t,x)\sin(jx) \, dx\, dt.
$$
In order to realize the exact null controllability in time $T$, we wish to find some controls $u$ such that
\begin{equation}\label{eq1}
\forall j\in\N^*\quad \int_0^T\int_\omega u(t,x) e^{-j^2(T-t)}\sin(jx)\, dx\, dt = -a_j e^{-j^2T}.
\end{equation}
By the M\"untz-Sz\'asz theorem, there exists a sequence $(\theta_T^j)_{j\in\N^*}$ in $L^2(0,T)$, spanning a proper subspace of $L^2(0,T)$, that is biorthogonal to the sequence of functions $t\mapsto e^{-j^2t}$, $j\in\N^*$, i.e.,
$\int_0^T e^{-j^2t}\theta_T^k(t)\, dt = \delta_{jk}$ for all $j,k\in\N^*$ (see \cite{MicuZuazua} for fine properties). 
We search controls $u$ satisfying \eqref{eq1}, of the form
%\begin{equation}\label{eq2}
%u(t,x) = \sum_{k,\ell=1}^{+\infty} b_{k\ell}\theta_T^k(T-t)\sin(\ell x).
%\end{equation}
%Plugging into \eqref{eq1}, there must hold
%$$
%\forall j\in\N^*\qquad \sum_{\ell=1}^{+\infty} b_{j\ell}\int_\omega\sin(jx)\sin(\ell x)\, dx=-a_je^{-j^2T}.
%$$
%We make the \textit{particular choice}: $b_{j\ell}=0$ whenever $j\neq\ell$. This amounts to searching $u$ of the less general form
$$
u(t,x) = \sum_{k=1}^{+\infty} b_k\theta_T^k(T-t)\sin(kx) ,
$$
which gives $b_j\int_\omega\sin^2(jx)\, dx=-a_je^{-j^2T}$ for every $j\in\N^*$.
We conclude that
$$
u(t,x) = -\sum_{k=1}^{+\infty} a_k e^{-k^2T} \theta_T^k(T-t)\frac{\sin(kx)}{\int_\omega\sin^2(ky)\, dy}.
$$
Thanks to Lemma \ref{lemperiago}, this function is well defined and is in $L^2((0,T)\times(0,\pi))$.

\paragraph{Abstract generalization.}
Let $X$ and $U$ be Hilbert spaces. Let $A:D(A)\rightarrow X$ be a densely defined operator, assumed to be self-adjoint and of compact inverse. Let $(\phi_j)_{j\in\N^*}$ be a Hilbert basis of $X$ consisting of eigenvectors of $A$. Note that $\phi_j\in X_1=D(A)=D(A^*)$ for every $j\in\N^*$.
Let $B\in L(U,D(A^*)')$ be an admissible control operator. We consider the control system \eqref{eqE} with $y(0)=y_0=\sum_{j\in\N^*} a_j\phi_j$.
Expanding $y(t)= \sum_{j=1}^{+\infty} y_j(t)\phi_j$, we have $\dot y_j=\lambda_j y_j+\langle Bu,\phi_j\rangle_{X_{-1},X_1}$ and thus
\begin{equation}\label{momentabstract}
\forall j\in\N^*\quad y_j(T) = e^{\lambda_jT}a_j+\int_0^T e^{\lambda_j(T-t)} (u(t),B^*\phi_j)_U \, dt,
\end{equation}
and we want to solve $y_j(T)=0$ for every $j\in\N^*$.

We assume that there exists a sequence $(\theta_T^j)_{j\in\N^*}$ in $L^2(0,T)$ that is biorthogonal to the family $\Lambda$ of functions $t\mapsto e^{\lambda_jt}$, $j\in\N^*$ (see Remark \ref{rem_Muntz} below). 
We search $u$ in the particular form
$$
u(t)=\sum_{k\in\N^*} b_k\theta_T^k(T-t) B^*\phi_k .
$$
Then, solving $y_j(T)=0$ for every $j\in\N^*$ is equivalent to requiring that
$b_j\Vert B^*\phi_j\Vert^2_U=-a_je^{\lambda_jT}$, and thus,
$$
u(t)=-\sum_{k\in\N^*} a_k e^{\lambda_kT}\theta_T^k(T-t) \frac{B^*\phi_k}{\Vert B^*\phi_k\Vert^2_U} .
$$
Showing that such a series gives a well-defined function requires to establish lower estimates of $\Vert B^*\phi_k\Vert^2_U$.

\begin{remark}\label{rem_Muntz}
Such a biorthogonal sequence $(\theta_T^j)_{j\in\N^*}$ exists if and only if the family $\Lambda$ is \textit{minimal}, that is, every element $t\mapsto e^{-\lambda_jt}$ lies outside of the closure in $L^2(0,T)$ of the vector space spanned by all other elements $t\mapsto e^{-\lambda_kt}$, with $k\neq j$.
If this condition is fulfilled, then the biorthogonal sequence $(\theta_T^j)_{j\in\N^*}$ is uniquely determined if and only if the family $\Lambda$ is complete in $L^2(0,T)$.
Note anyway that the biorthogonal sequence is difficult to construct in practice.

It is well known, by the M\"untz-Sz\'asz theorem, that the family $\Lambda$ is complete in $L^2(0,T)$ (but not independent) if and only if 
$\sum_{j\in\N^*}\frac{1}{\Real(\lambda_j)+\lambda}=+\infty$
for some real number $\lambda$ such that $\Real(\lambda_j)+\lambda>0$ for every $j\in\N^*$ (for instance, $\lambda=-\Real (\lambda_1)+1$).
At the opposite, if this series is convergent then the closure of the span of $\Lambda$ is a proper subspace of $L^2(0,T)$, moreover $\Lambda$ is minimal and thus a biorthogonal sequence exists. 

Then, here, we are led to assume that the series is convergent, which is a quite strong restriction on the parabolic system under consideration.
\end{remark}

\paragraph{Relationship with HUM.}
Within the previous abstract general framework, let us solve the moment equations $y_j(T)=0$ for every $j\in\N^*$, in another way: 
using \eqref{momentabstract}, it suffices to search a control $u$ such that
$$
\big( u , (t,x)\mapsto e^{\lambda_j(T-t)}(B^*\phi_j)(x) \big)_{L^2([0,T],U)} = -e^{\lambda_jT}a_j\qquad \forall j\in\N^*
$$
and, generalizing the previous approach, we can solve this moment problem by using, if it exists, a biorthogonal sequence $(u_j)_{j\in\N^*}$ to the sequence of (time-space) functions $(t,x)\mapsto e^{\lambda_j(T-t)}(B^*\phi_j)(x)$, i.e., noting that $S(T-t)^*\phi_k = e^{\lambda_k(T-t)}\phi_k$, a sequence satisfying
\begin{equation}\label{biorthhum}
\left( u_j , B^*S(T-t)^*\phi_k \right)_{L^2([0,T],U)} = \delta_{jk} \qquad \forall j,k\in\N^* .
\end{equation}
Note that, when such a family exists, $u_j$ is a control steering the control system \eqref{eqE} from $-e^{-\lambda_jT}\phi_j$ to $0$ in time $T$ (this is related to the notion of \emph{spectral controllability}).

There are plenty of ways for designing such controls $u_j$, when this is possible. 
For every $j\in\N^*$, let $u_j\in L^2([0,T],U)$ be the (unique) HUM control steering the initial condition $-e^{-\lambda_jT}\phi_j$ to $0$ (we assume that this is possible): we have $u_j(t)=B^*S(T-t)^*\psi_j$ where $G_T\psi_j=\phi_j$ (note that $u_j$ is the control of minimal $L^2$ norm, and $\Vert u_j\Vert^2_{L^2([0,T],U)}=(G_T\psi_j,\psi_j)_U=(\phi_j,\psi_j)_U$).
Then, obviously, \eqref{biorthhum} holds.

Now, the (unique) HUM control $u$ such that $y(T)=0$ is given by $u(t)=B^*S(T-t)^*\psi$ with $S(T)y_0+G_T\psi=0$.
Since $y_0=\sum_{j\in\N^*} a_j\phi_j$, it easily follows by linearity that, formally,
$$
u = -\sum_{j=1}^{+\infty} a_j e^{\lambda_jT}u_j.
$$
%Hence, the HUM control can be interpreted as resulting of a moment method where one considers the time-space biorthogonal sequence $(u_j)_{j\in\N^*}$ as above with $u_j$ of minimal $L^2$ norm.
Of course, all above computations are formal, and it may be difficult to establish the convergence of the series in practical examples.

\subsection{Equivalence between observability and exponential stability}
The following result is a generalization of the main result of \cite{Haraux}.

\begin{theorem}
Let $X$ be a Hilbert space, let $A:D(A)\rightarrow X$ be a densely defined skew-adjoint operator, let $B$ be a bounded self-adjoint nonnegative operator on $X$.
We have equivalence of:
\begin{enumerate}
\item There exist $T>0$ and $C>0$ such that every solution of the conservative\footnote{It is said to be conservative because, since $A$ is skew-adjoint, we have $\Vert\phi(t)\Vert_X=\mathrm{Cst}=\Vert\phi(0\Vert_X$ for every $t\in\R$.} equation 
$\dot \phi(t)+A\phi(t)=0$
satisfies the observability inequality
$$\Vert \phi(0)\Vert_X^2 \leq C\int_0^{T}\Vert B^{1/2}\phi(t)\Vert_X^2 \, dt .$$
\item There exist $C_1>0$ and $\delta>0$ such that every solution of the damped equation
$\dot y(t)+Ay(t)+By(t)=0$
satisfies
$$
E_y(t) \leq C_1 E_y(0)\mathrm{e}^{-\delta t} \qquad \forall t\geq 0 ,
$$
where 
$E_y(t) = \frac{1}{2} \Vert y(t)\Vert_X^2 $.
\end{enumerate}
\end{theorem}

\begin{proof}
Let us first prove that the first property implies the second one.
%: we want to prove that every solution of $\dot y+Ay+B y=0$ satisfies
%$$
%E_y(t) = \frac{1}{2} \Vert y(t)\Vert_X^2 \leq E_y(0)\mathrm{e}^{-\delta t} = \frac{1}{2} \Vert y(0)\Vert_X^2  \mathrm{e}^{-\delta t}.
%$$
Let $y$ be a solution of the damped equation.
Let $\phi$ be the solution of $\dot\phi+A\phi=0$, $\phi(0)=y(0)$.
Setting $\theta=y-\phi$, we have $\dot\theta+A\theta+By=0$, $\theta(0)=0$. Then, taking the scalar product with $\theta$, since $A$ is skew-adjoint, we get 
$(\dot\theta+By,\theta)_X=0$.
But, setting $E_\theta(t)=\frac{1}{2} \Vert \theta(t)\Vert_X^2 $, we have
$ \dot E_\theta = -(By,\theta)_X$.
Then, integrating a first time over $[0,t]$, and then a second time over $[0,T]$, since $E_\theta(0)=0$, we get
\begin{multline*}
\int_0^T E_\theta(t)\,dt = - \int_0^T\int_0^t ( By(s),\theta(s) )_X \,ds\, dt \\
= -\int_0^T (T-t) ( B^{1/2}y(t),B^{1/2}\theta(t) )_X\, dt ,
\end{multline*}
where we have used the Fubini theorem.
Hence, using the Cauchy-Schwarz inequality and then the Young inequality $ab\leq\frac{\alpha}{2}a^2+\frac{1}{2\alpha}b^2$ with $\alpha=2T \Vert B^{1/2}\Vert$, we infer that
\begin{equation*}
\begin{split}
\frac{1}{2}\int_0^T  \Vert \theta(t)\Vert_X^2 \, dt
&\leq T \Vert B^{1/2}\Vert \int_0^T \Vert  B^{1/2}y(t)\Vert_X\Vert\theta(t)\Vert_X \, dt \\
&\leq T^2 \Vert B^{1/2}\Vert^2 \int_0^T \Vert  B^{1/2}y(t)\Vert_X^2 \, dt +  \frac{1}{4}\int_0^T \Vert\theta(t)\Vert_X^2 \, dt,
\end{split}
\end{equation*}
and therefore
$$
\int_0^T \Vert\theta(t)\Vert_X^2 \, dt \leq 4 T^2 \Vert B^{1/2}\Vert^2 \int_0^T \Vert  B^{1/2}y(t)\Vert_X^2 \, dt.
$$
Now, since $\phi=y-\theta$, it follows that
\begin{equation*}
\begin{split}
\int_0^T \Vert B^{1/2}\phi(t)\Vert_X^2 \, dt 
& \leq 2 \int_0^T \Vert B^{1/2}y(t)\Vert_X^2 \, dt + 2\int_0^T \Vert B^{1/2}\theta(t)\Vert_X^2 \, dt \\
& \leq (2+8 T^2 \Vert B^{1/2}\Vert^4) \int_0^T \Vert B^{1/2} y(t)\Vert_X^2 \, dt.
\end{split}
\end{equation*}
Finally, since
$$
E_y(0)=E_\phi(0)=\frac{1}{2} \Vert \phi(0)\Vert_X^2 \leq \frac{C}{2}\int_0^{T}\Vert B^{1/2}\phi(t)\Vert_X^2 \, dt
$$
it follows that
$
E_y(0) \leq C (1+4 T^2 \Vert B^{1/2}\Vert^4) \int_0^T \Vert B^{1/2} y(t)\Vert_X^2 \, dt.
$
Besides, one has $E_y'(t)=-\Vert B^{1/2} y(t)\Vert_X^2$, and then
$\int_0^T \Vert B^{1/2} y(t)\Vert_X^2 dt=E_y(0)-E_y(T)$.
Therefore
$$E_y(0) \leq C (1+4 T^2 \Vert B^{1/2}\Vert^4) (E_y(0)-E_y(T))=C_1(E_y(0)-E_y(T))$$
and hence
$$
E_y(T)\leq \frac{C_1-1}{C_1} E_y(0) = C_2 E_y(0),
$$
with $C_2<1$.

Actually this can be done on every interval $[kT,(k+1)T]$, and it yields
$E_y((k+1)T)\leq C_2 E_y(kT)$ for every $k\in\N$, and hence $E_y(kT)\leq E_y(0)C_2^k$.

For every $t\in [kT,(k+1)T)$, noting that $k=\left[\frac{t}{T}\right]> \frac{t}{T}-1$, and that $\ln\frac{1}{C_2}>0$, it follows that
$$
C_2^k=\exp(k\ln C_2)=\exp\Big(-k\ln\frac{1}{C_2}\Big)\leq
%\exp(\frac{-t}{T}\ln\frac{1}{C_2}+\ln\frac{1}{C_2})=
\frac{1}{C_2}\exp\Big(\frac{-\ln\frac{1}{C_2}}{T} t\Big)
$$
and hence $E_y(t)\leq E_y(kT)\leq \delta E_y(0) \exp(-\delta t)$ for some $\delta>0$.

\medskip
Let us now prove the converse: assume the exponential decrease, and let us prove the observability property.
Let $\phi$ be a solution of the conservative equation.
Let $y$ be a solution of the damped equation such that $y(0)=\phi(0)$.

From the exponential decrease inequality, one has
\begin{equation}\label{orl1841}
\int_0^T \Vert B^{1/2} y(t)\Vert_X^2 \, dt = E_y(0)-E_y(T) \geq (1-C_1\mathrm{e}^{-\delta T})E_y(0) = C_2 E_y(0),
\end{equation}
and for $T>0$ large enough there holds $C_2=1-C_1\mathrm{e}^{-\delta T}>0$.

Then we make the same proof as before, starting from
$\dot\phi+A\phi=0$, that we write in the form $\dot\phi+A\phi +B\phi=B\phi$,
and considering the solution of $\dot y + Ay + By=0$, $y(0)=\phi(0)$.
Setting $\theta=\phi-y$, we have $\dot\theta+A\theta+B\theta=B\phi$, $\theta(0)=0$.
Taking the scalar product with $\theta$, since $A$ is skew-adjoint, we get
$(\dot\theta+B\theta,\theta)_X=(B\phi,\theta)_X$, 
and therefore
$\dot E_\theta +( B\theta,\theta)_X = ( B\phi,\theta)_X$.
Since $( B\theta,\theta)_X=\Vert B^{1/2}\theta\Vert_X\geq 0$, it follows that
$ \dot E_\theta \leq ( B\phi,\theta)_X$.
As before we apply $\int_0^T\int_0^t$ and hence, since $E_\theta(0)=0$,
$$
\int_0^T E_\theta(t) \, dt \leq \int_0^T \! \int_0^t ( B\phi(s),\theta(s))_X \, ds\, dt = \int_0^T (T-t) ( B^{1/2}\phi(t),B^{1/2}\theta(t))_X \, dt.
$$
Thanks to the Young inequality, we get, exactly as before,
\begin{equation*}
\begin{split}
\frac{1}{2}\int_0^T  \Vert \theta(t)\Vert_X^2 \, dt
&\leq T \Vert B^{1/2}\Vert \int_0^T \Vert  B^{1/2}\phi(t)\Vert_X \Vert\theta(t)\Vert_X \, dt \\
&\leq T^2 \Vert B^{1/2}\Vert^2 \int_0^T \Vert  B^{1/2}\phi(t)\Vert_X^2 \,dt +  \frac{1}{4}\int_0^T \Vert\theta(t)\Vert_X^2 \, dt ,
\end{split}
\end{equation*}
and finally,
$$
\int_0^T \Vert\theta(t)\Vert_X^2 \, dt \leq 4 T^2 \Vert B^{1/2}\Vert^2 \int_0^T \Vert  B^{1/2}\phi(t)\Vert_X^2 \,dt.
$$
Now, since $y=\phi-\theta$, it follows that
\begin{equation*}
\begin{split}
\int_0^T \Vert B^{1/2} y(t)\Vert_X^2 \, dt 
& \leq 2 \int_0^T \Vert B^{1/2} \phi(t)\Vert_X^2 \, dt + 2\int_0^T \Vert B^{1/2}\theta(t)\Vert_X^2 \, dt \\
& \leq (2+8 T^2 \Vert B^{1/2}\Vert^4) \int_0^T \Vert B^{1/2} \phi(t)\Vert_X^2 \, dt .
\end{split}
\end{equation*}
Now, using \eqref{orl1841} and noting that $E_y(0)=E_\phi(0)$,
we infer that
$$ C_2 E_\phi(0) \leq (2+8 T^2 \Vert B^{1/2}\Vert^4) \int_0^T \Vert B^{1/2} \phi(t)\Vert_X^2 dt.$$
This is the desired observability inequality.
\end{proof}

\begin{remark}
This result says that the observability property for a linear conservative equation is equivalent to the exponential stability property for the same equation in which a linear damping has been added. This result has been written in \cite{Haraux} for second-order equations, but the proof works exactly in the same way for more general first-order systems, as shown here. The above proof uses in a crucial way the fact that the operator $B$ is bounded. We refer to \cite{AmmariTucsnak} for a generalization for unbounded operators with degree of unboundedness $\leq 1/2$, and only for second-order equations, with a proof using Laplace transforms, under a condition on the unboundedness of $B$ that is 
%not easy to check (related to ``hidden regularity" results). 
related to ``hidden regularity" results. 
For instance this works for waves with a nonlocal operator $B$ corresponding to a Dirichlet condition, in the state space $L^2\times H^{-1}$, but not for the usual Neumann one, in the state space $H^1\times L^2$ (except in 1D).
\end{remark}

\subsection{1D semilinear heat equation}
Let $L>0$ be fixed and let $f:\R\rightarrow\R$ be a function of class $C^2$ such that $f(0)=0$. We consider the 1D semilinear heat equation
\begin{equation} \label{eqcont0}
\partial_t y = \partial_{xx} y+f(y),  \qquad  y(t,0)=0,\ y(t,L)=u(t),
\end{equation}
where the state is $y(t,\cdot):[0,L]\rightarrow\R$ and the control is $u(t)\in \R$.

We want to design a feedback control locally stabilizing \eqref{eqcont0} asymptotically to $0$.
Note that this cannot be global, because we can have other steady-states (a steady-state is a function $y\in C^2(0,L)$ such that $y''(x)+f(y(x))=0$ on $(0,L)$ and $y(0)=0$).
By the way, here, without loss of generality we consider the steady-state $0$.

Let us first note that, for every $T>0$, \eqref{eqcont0} is well posed in the Banach space $Y_T = L^2([0,T],H^2(0,L))\cap H^1([0,T],L^2(0,L))$, which is continuously embedded in $L^\infty((0,T)\times(0,L))$.\footnote{Indeed, considering $v\in L^2([0,T],H^2(0,L))$ with $v_t\in H^1([0,T],L^2(0,L))$, writing $v=\sum_{j,k}c_{jk}e^{ijt}e^{ikx}$, we have
$$
\sum_{j,k}\vert c_{jk}\vert\leq \bigg( \sum_{j,k} \frac{1}{1+j^2+k^4} \bigg)^{1/2} \bigg( \sum_{j,k} (1+j^2+k^4)\vert c_{jk}\vert^2 \bigg)^{1/2} 
$$
and these series converge, whence the embedding, allowing to give a sense to $f(y)$.

Now, if $y_1$ and $y_2$ are solutions of \eqref{eqcont0} on $[0,T]$, then $y_1=y_2$. Indeed,  $v=y_1-y_2$ is solution of $v_t=v_{xx}+a\, v$, $v(t,0)=v(t,L)=0$, $v(0,x)=0$, with $a(t,x)=g(y_1(t,x),y_2(t,x))$ where $g$ is a function of class $C^1$. We infer that $v=0$.}

First of all, in order to end up with a Dirichlet problem, we set $z(t,x)=y(t,x)-\frac{x}{L}u(t)$. Assuming (for the moment) that $u$ is differentiable, we set $v(t)=u'(t)$, and we consider in the sequel $v$ as a control. We also assume that $u(0)=0$. Then we have
\begin{equation}\label{reducedproblem2}
\partial_t z=\partial_{xx} z+f'(0)z+\frac{x}{L}f'(0)u-\frac{x}{L}v+r(t,x),  \qquad
z(t,0)=z(t,L)=0,  
\end{equation}
with $z(0,x)=y(0,x)$ and (by performing a second-order Taylor expansion of $f$ with integral remainder)
\begin{equation*}%\label{reste}
r(t,x)= \left(z(t,x)+\frac{x}{L}u(t)\right)^2\int_0^1 (1-s)f''\left(sz(s,x)+s\frac{x}{L}u(s)\right)ds.
\end{equation*}
Note that, given $B>0$ arbitrary, there exist positive constants $C_1$ and $C_2$ such that, if $\vert u(t)\vert\leq B$ and $\Vert z(t,\cdot)\Vert_{L^\infty(0,L)}\leq B$, then
\begin{equation*}%\label{estimreste}
\Vert r(t,\cdot)\Vert_{L^\infty(0,L)}\leq C_1(u(t)^2 +\Vert z(t,\cdot)\Vert_{L^\infty(0,L)}^2) \leq C_2(u(t)^2 +\Vert z(t,\cdot)\Vert_{H^1_0(0,L)}^2) .
\end{equation*}
In the sequel, $r(t,x)$ will be considered as a remainder.

We define the operator $A=\triangle+f'(0)\mathrm{id}$ on $D(A) = H^2(0,L)\cap H_0^1(0,L)$, so that \eqref{reducedproblem2} is written as
\begin{equation}\label{reducedproblem3}
\dot u=v,\qquad \partial_t z=Az+au+bv+r,  \qquad z(t,0)=z(t,L)=0,  
\end{equation}
with $a(x) = \frac{x}{L}f'(0)$ and $b(x)=-\frac{x}{L}$.

Since $A$ is self-adjoint and has a compact resolvent, there exists a Hilbert basis $(e_j)_{j\geq 1}$ of $L^2(0,L)$, consisting of eigenfunctions $e_j\in H^1_0(0,L)\cap C^2([0,L])$ of $A$, associated with eigenvalues $(\lambda_j)_{j\geq 1}$ such that $-\infty<\cdots<\lambda_n<\cdots<\lambda_1$ and $\lambda_n\rightarrow-\infty$ as $n\rightarrow+\infty$.

Any solution $z(t,\cdot)\in H^2(0,L)\cap H^1_0(0,L)$ of \eqref{reducedproblem2}, as long as it is well defined, can be expanded as a series $z(t,\cdot)=\sum_{j=1}^{\infty}z_j(t)e_j(\cdot)$ (converging in $H_0^1(0,L)$), and then we have, for every $j\geq 1$,
\begin{equation}\label{sysdiag}
\dot z_j(t) = \lambda_j z_j(t) + a_j u(t) + b_j v(t) + r_j(t),
\end{equation}
with
$$
a_j= \frac{f'(0)}{L}\int_0^L xe_j(x)\, dx, \quad
b_j= -\frac{1}{L}\int_0^L xe_j(x)\, dx, \quad
r_j(t)=\int_0^L r(t,x)e_j(x)\, dx.
$$
Setting, for every $n\in\N^*$,
\begin{equation*}
X_n(t)=\begin{pmatrix} u(t) \\ z_1(t) \\ \vdots \\ z_n(t)
\end{pmatrix}, \
A_n=\begin{pmatrix}
0         &       0         & \cdots &    0           \\
a_1 & \lambda_1 & \cdots &    0           \\
\vdots    &  \vdots         & \ddots &   \vdots       \\
a_n &  0              & \cdots & \lambda_n
\end{pmatrix} , \
B_n=\begin{pmatrix} 1 \\ b_1 \\ \vdots \\ b_n \end{pmatrix} , \ 
R_n(t)=\begin{pmatrix} 0 \\ r_1(t) \\ \vdots \\ r_n(t) \end{pmatrix},
\end{equation*}
we have then
\begin{equation*}%\label{sysn}
\dot X_n(t) = A_nX_n(t) + B_n v(t) + R_n(t).
\end{equation*}

\begin{lemma}
For every $n\in\N^*$, the pair $(A_n,B_n)$ satisfies the Kalman condition.
\end{lemma}

\begin{proof}
We compute
\begin{equation}\label{deter}
\mathrm{det} \left( B_n, A_nB_n, \ldots, A_n^{n}B_n \right)
= \prod_{j=1}^{n}(a_j+\lambda_jb_j)\
\mathrm{VdM}(\lambda_1,\ldots,\lambda_n) 
\end{equation}
where $\mathrm{VdM}(\lambda_1,\ldots,\lambda_n)$ is a Van der Monde determinant, and thus is never equal to zero since the $\lambda_i$, $i=1\ldots n$, are pairwise distinct.
On the other part, using the fact that each $e_j$ is an eigenfunction of $A$ and belongs to $H^1_0(0,L)$, we compute
\begin{equation*}
a_j+\lambda_jb_j
= \frac{1}{L} \int_0^L x ( f'(0)-\lambda_j)e_j(x) \, dx
= -\frac{1}{L} \int_0^L x e_j''(x) \, dx
= - e_j'(L),
\end{equation*}
and this quantity is never equal to zero since $e_j(L)=0$ and $e_j$ is a nontrivial solution of
a linear second-order scalar differential equation.
Therefore the determinant \eqref{deter} is never equal to zero.
\end{proof}

By the pole-shifting theorem (Theorem \ref{thmplacementpoles}), there exists $K_n=\left( k_0,\ldots,k_n \right)$ such that the matrix $A_n+B_nK_n$ has $-1$ as an eigenvalue of multiplicity $n+1$. Moreover, by the Lyapunov lemma (see Example \ref{Lyapunov_lemma}), there exists a symmetric positive definite matrix $P_n$ of size $n+1$ such that
\begin{equation*}%\label{poleshifting}
P_n\left(A_n+B_nK_n\right) + \left(A_n+B_nK_n\right)^\top P_n = -I_{n+1} .
\end{equation*}
Therefore, as shown in Example \ref{Lyapunov_lemma}, the function defined by $V_n(X) = X^\top P_n X$ for any $X\in\R^{n+1}$ is a Lyapunov function for the closed-loop system $\dot X_n(t) = (A_n+B_nK_n)X_n(t)$: along this system we have $\frac{d}{dt} V_n(X_n(t)) = -\Vert X_n(t)\Vert_2^2$.
Here, $\Vert\ \Vert_2$ stands for the Euclidean norm of $\R^{n+1}$.

\medskip

Let $\gamma>0$ and $n\in\N^*$ to be chosen later. For every $u\in\R$ and every $z\in H^2(0,L)\cap H^1_0(0,L)$, we set
\begin{equation}\label{defLyapV1216}
V(u,z)=\gamma\, X_n^\top P_n X_n - \frac{1}{2}\langle z,Az\rangle_{L^2(0,L)}
= \gamma\, X_n^\top P_n X_n - \frac{1}{2}\sum_{j=1}^\infty \lambda_jz_j^2
\end{equation}
where $X_n=(u,z_1,\ldots,z_n)^\top\in\R^{n+1}$ and $z_j=\langle z,e_i\rangle_{L^2(0,L)}$ for every $j$.

Using that $\lambda_n\rightarrow-\infty$ as $n\rightarrow+\infty$, it is clear that, choosing $\gamma>0$ and $n\in\N^*$ large enough, we have $V(u,z)>0$ for all $(u,z)\in\R\times (H^2(0,L)\cap H^1_0(0,L))\setminus\{(0,0)\}$. More precisely, there exist positive constants $C_3$, $C_4$, $C_5$ and $C_6$ such that
\begin{multline*}%\label{equiv2}
C_3\left( u^2+\Vert z\Vert_{H^1_0(0,L)}^2\right)\ \leq\ V(u,z)\ \leq\ C_4\left( u^2+\Vert z\Vert_{H^1_0(0,L)}^2\right),\\
V(u,z)\ \leq\ C_5\left( \Vert X_n\Vert_2^2 + \Vert Az\Vert_{L^2(0,L)}^2 \right) ,\qquad\qquad
\gamma C_6\Vert X_n\Vert_2^2\ \leq\ V(u,z) ,
\end{multline*}
for all $(u,z)\in\R\times (H^2(0,L)\cap H^1_0(0,L))$. 

\medskip

Our objective is now to prove that $V$ is a Lyapunov function for the system \eqref{reducedproblem3} in closed-loop with the feedback control $v=K_nX_n$ and $u$ defined by $\dot u=v$ and $u(0)=0$.
We compute
\begin{multline}\label{dVuzdt}
\frac{d}{dt} V(u(t),z(t))=-\gamma\,\Vert X_n(t)\Vert_2^2-\Vert Az(t,\cdot)\Vert_{L^2}^2-\langle
Az(t,\cdot),a(\cdot)\rangle_{L^2}u(t) \\ 
\qquad\qquad\qquad\qquad -\langle Az(t,\cdot),b(\cdot)\rangle_{L^2}K_nX_n(t) -\langle Az(t,\cdot),r(t,\cdot)\rangle_{L^2} \\
+\gamma\left(R_n(t)^\top P_nX_n(t)+X_n(t)^\top P_nR_n(t)\right) .
\end{multline}
Let us estimate the terms at the right-hand side of \eqref{dVuzdt}. Under the a priori estimates $\vert u(t)\vert\leq B$ and $\Vert z(t,\cdot)\Vert_{L^\infty(0,L)}\leq B$, there exist positive constants $C_7$, $C_8$ and $C_9$ such that (dropping the dependence in $t$)
\begin{equation*}
\begin{split}
& \vert \langle Az,a\rangle_{L^2}u \vert+\vert \langle Az,b\rangle_{L^2}K_nX_n \vert \leq \frac{1}{4} \Vert Az\Vert_{L^2}^2+C_7\Vert X_n\Vert_2^2  , \\
& \vert\langle Az,r\rangle_{L^2}\vert \leq \frac{1}{4} \Vert Az\Vert_{L^2}^2+C_8V^2 ,\qquad\qquad
\Vert R_n\Vert_\infty \leq \frac{C_2}{C_3} V ,  \\
& \vert\gamma\left(R_n^\top P_nX_n+X_n^\top P_nR_n\right)\vert \leq \frac{C_2}{C_3\sqrt{C_6}}\sqrt{\gamma}\,V^{3/2} .
\end{split}
\end{equation*}
We infer that, if $\gamma>0$ is large enough, then there exist $C_{10},C_{11}>0$ such that $\frac{d}{dt} V \leq -C_{10}V+C_{11}V^{3/2}$.
We easily conclude the local asymptotic stability of the system \eqref{reducedproblem3} in closed-loop with the control $v=K_nX_n$.

\begin{remark}
The above local asymptotic stability can be achieved with other procedures, for instance, by using the Riccati theory (see \cite{Zabczyk} for Riccati operators in the parabolic case). However, the procedure developed above is much more efficient because it consists of stabilizing a finite-dimensional part of the system, namely, the part that is not naturally stable. We refer to \cite{CoronTrelat} for examples and for more details. Actually, it is proved in that reference that, thanks to such a strategy, one can pass from any steady-state to any other one, provided that the two steady-states belong to a same connected component of the set of steady-states: this is a partially global exact controllability result.
\end{remark}

The main idea used above\footnote{This idea has been used as well to treat other parabolic problems, and even hyperbolic: it has been as well used in \cite{CoronTrelat_CCM2006} for the 1D semilinear equation
$$
\partial_{tt}y = \partial_{xx} y+f(y),  \qquad  y(t,0)=0,\ y_x(t,L)=u(t).
$$
%with the same assumptions on $f$ as before. 
We first note that, if $f(y)=cy$ is linear (with $c\in L^\infty(0,L)$), then, taking $u(t) = -\alpha \partial_t y(t,L)$ with $\alpha>0$ yields an exponentially decrease of the energy $\int_0^L ( \partial_t y(t,x)^2+\partial_xy(t,x)^2 )\, dt$, and moreover, the eigenvalues of the corresponding operator have a real part tending to $-\infty$ as $\alpha\rightarrow 1$. Therefore, in the general case, if $\alpha$ is sufficiently close to $1$ then at most a finite number of eigenvalues may have a nonnegative real part. Using a Riesz spectral expansion, the above spectral method yields a feedback based on a finite number of modes, which stabilizes locally the semilinear wave equation, asymptotically to equilibrium.}
is the following fact, already used in \cite{Russell}. Considering the linearized system with no control, we have an infinite-dimensional linear system that can be split, through a spectral decomposition, in two parts: the first part is finite-dimensional, and consists of all spectral modes that are unstable (meaning that the corresponding eigenvalues have nonnegative real part); the second part is infinite-dimensional, and consists of all spectral modes that are asymptotically stable (meaning that the corresponding eigenvalues have negative real part).
The idea used here then consists of focusing on the finite-dimensional unstable part of the system, and to design a feedback control in order to stabilize that part. Then, we plug this control in the infinite-dimensional system, and we have to check that this feedback indeed stabilizes the whole system (in the sense that it does not destabilize the other infinite-dimensional part). This is the role of the Lyapunov function $V$ defined by \eqref{defLyapV1216}.

%The extension to general systems \eqref{contsys} is quite immediate, at least in the parabolic setting under appropriate spectral assumptions (see \cite{SchmidtTrelat} for Couette flows and \cite{CoronTrelatVazquez} for Navier-Stokes equations).

\end{document}